\documentclass[11pt]{article}
\usepackage{amsmath,amsthm,amssymb}
\usepackage[usenames,dvipsnames]{xcolor}
\usepackage{enumerate}
\usepackage{graphicx}
\usepackage{cite}
\usepackage{comment}
\usepackage{oands}
\usepackage{tikz}
\usepackage{changepage}
\usepackage{bbm}
\usepackage{mathtools}
\usepackage[margin=1in]{geometry} 
\usepackage[pagewise,mathlines]{lineno}
\usepackage{appendix}
\usepackage{stmaryrd}
\usepackage{multicol}
\usepackage{microtype}
\usepackage[colorinlistoftodos]{todonotes}

\usepackage[pdftitle={Uniqueness of the critical and supercritical Liouville quantum gravity metrics},
  pdfauthor={Jian Ding and Ewain Gwynne},
colorlinks=true,linkcolor=NavyBlue,urlcolor=RoyalBlue,citecolor=PineGreen,bookmarks=true,bookmarksopen=true,bookmarksopenlevel=2,unicode=true,linktocpage]{hyperref}

\theoremstyle{plain}
\newtheorem{thm}{Theorem}[section]

\newtheorem{lem}[thm]{Lemma}
\newtheorem{prop}[thm]{Proposition}

\newtheorem{prob}[thm]{Problem}

\def\@rst #1 #2other{#1}
\newcommand\MR[1]{\relax\ifhmode\unskip\spacefactor3000 \space\fi
  \MRhref{\expandafter\@rst #1 other}{#1}}
\newcommand{\MRhref}[2]{\href{http://www.ams.org/mathscinet-getitem?mr=#1}{MR#2}}

\theoremstyle{definition}
\newtheorem{defn}[thm]{Definition}
\newtheorem{remark}[thm]{Remark}

\numberwithin{equation}{section}

\newcommand{\dsb}{\begin{adjustwidth}{2.5em}{0pt}
\begin{footnotesize}}
\newcommand{\dse}{\end{footnotesize}
\end{adjustwidth}}

\newcommand{\ssb}{\begin{adjustwidth}{2.5em}{0pt}}
\newcommand{\sse}{\end{adjustwidth}}

\newcommand{\aryb}{\begin{eqnarray*}}
\newcommand{\arye}{\end{eqnarray*}}
\def\alb#1\ale{\begin{align*}#1\end{align*}}
\def\allb#1\alle{\begin{align}#1\end{align}}
\newcommand{\eqb}{\begin{equation}}
\newcommand{\eqe}{\end{equation}}
\newcommand{\eqbn}{\begin{equation*}}
\newcommand{\eqen}{\end{equation*}}

\newcommand{\BB}{\mathbbm}
\newcommand{\ol}{\overline}
\newcommand{\ul}{\underline}
\newcommand{\op}{\operatorname}

\newcommand{\im}{\operatorname{Im}}
\newcommand{\re}{\operatorname{Re}}
\newcommand{\frk}{\mathfrak}
\newcommand{\eqD}{\overset{d}{=}}
\newcommand{\ep}{\varepsilon}
\newcommand{\rta}{\rightarrow}

\newcommand{\wt}{\widetilde}
\newcommand{\wh}{\widehat}
\newcommand{\mcl}{\mathcal}

\newcommand{\bdy}{\partial}

\newcommand{\el}{l}

\newcommand{\crit}{{\mathrm{crit}}}

\newcommand{\ccM}{{\mathbf{c}_{\mathrm M}}}

\let\originalleft\left
\let\originalright\right
\renewcommand{\left}{\mathopen{}\mathclose\bgroup\originalleft}
\renewcommand{\right}{\aftergroup\egroup\originalright}

\title{Uniqueness of the critical and supercritical Liouville quantum gravity metrics}  
 \date{ }
 \author{
\begin{tabular}{c} Jian Ding\footnote{School of Mathematical Sciences, Peking University, Beijing {\rm 100871}, China, \url{dingjian@math.pku.edu.cn} }  \\[-5pt]\small Peking University \end{tabular}
\begin{tabular}{c} Ewain Gwynne\footnote{Dept.\ of Mathematics, University of Chicago, 5734 S University Ave, Chicago {\rm 60637}, USA, \url{ewain@uchicago.edu} } \\[-5pt]\small University of Chicago \end{tabular} 
}

\begin{document}

\maketitle

\begin{abstract}
We show that for each ${\mathbf c}_{\mathrm M} \in [1,25)$, there is a unique metric associated with Liouville quantum gravity (LQG) with matter central charge ${\mathbf c}_{\mathrm M}$. An earlier series of works by Ding-Dub\'edat-Dunlap-Falconet, Gwynne-Miller, and others showed that such a metric exists and is unique in the subcritical case ${\mathbf c}_{\mathrm M} \in (-\infty,1)$, which corresponds to coupling constant $\gamma \in (0,2)$. 
The critical case ${\mathbf c}_{\mathrm M} = 1$ corresponds to $\gamma=2$ and the supercritical case ${\mathbf c}_{\mathrm M} \in (1,25)$ corresponds to $\gamma \in \mathbb C$ with $|\gamma| = 2$. 

Our metric is constructed as the limit of an approximation procedure called Liouville first passage percolation, which was previously shown to be tight for $\mathbf c_{\mathrm M} \in [1,25)$ by Ding and Gwynne (2020). In this paper, we show that the subsequential limit is uniquely characterized by a natural list of axioms. This extends the characterization of the LQG metric proven by Gwynne and Miller (2019) for $\mathbf c_{\mathrm M} \in (-\infty,1)$ to the full parameter range $\mathbf c_{\mathrm M} \in (-\infty,25)$.  

Our argument is substantially different from the proof of the characterization of the LQG metric for $\mathbf c_{\mathrm M} \in (-\infty,1)$. In particular, the core part of the argument is simpler and does not use confluence of geodesics.  
\end{abstract}

\medskip

\newcommand{\Cupper}{{\hyperref[eqn-bilip-def]{\mathfrak C_*}}}
\newcommand{\Clower}{{\hyperref[eqn-bilip-def]{\mathfrak c_*}}}
\newcommand{\Cmid}{{\mathfrak c}'}
\newcommand{\Cmed}{{\mathfrak C}'}

\newcommand{\Kopt}{\beta}
\newcommand{\Kep}{{\delta_0}}
\newcommand{\Kann}{{\alpha}}
\newcommand{\Karound}{{A}} 
\newcommand{\geoExp}{{\hyperref[lem-hit-ball-phi]{\theta}}}
 
\newcommand{\Fr}{{\hyperref[sec-block-event]{\mathsf F}}}

\newcommand{\Er}{{\mathsf E}}
\newcommand{\Ur}{{\mathsf U}}
\newcommand{\Vr}{{\mathsf V}}
\newcommand{\fr}{{\mathsf f}}

\newcommand{\Cmax}{{\mathsf M}} 
\newcommand{\Cacross}{{\mathsf a}}
\newcommand{\Caround}{{\mathsf A}}
\newcommand{\Ctube}{{\mathsf L}}
\newcommand{\Crn}{{\mathsf K}}
\newcommand{\Ctime}{{\mathsf b}}
\newcommand{\Cinc}{{\mathsf c}} 

\newcommand{\llambda}{{\hyperref[eqn-small-const]{\lambda}}}
 
\newcommand{\Aendpt}{{\hyperref[lem-endpt-ball]{\mathsf t}}}
\newcommand{\sr}{{{\mathsf s}}}
\newcommand{\vr}{{\mathsf v}}
\newcommand{\ur}{{\mathsf u}}
\newcommand{\Hr}{{\mathsf H}}
\newcommand{\Aloc}{{\hyperref[lem-endpt-ball]{\mathsf S}}}
\newcommand{\pr}{{\hyperref[lem-endpt-ball]{\mathsf p}}}

\tableofcontents

\section{Introduction}
\label{sec-intro}

\subsection{Overview}
\label{sec-overview}

Liouville quantum gravity (LQG) is a one-parameter family of random fractal surfaces which originated in the physics literature in the 1980s~\cite{polyakov-qg1,david-conformal-gauge,dk-qg} as a class of canonical models of random geometry in two dimensions.  
One possible choice of parameter is the \emph{matter central charge} $\ccM \in (-\infty,25)$. 
Heuristically speaking, for an open domain $U\subset \BB C$, an LQG surface with matter central charge $\ccM$ is a sample from ``the uniform measure on Riemannian metric tensors $g$ on $U$, weighted by $(\det\Delta_g)^{-\ccM/2}$", where $\Delta_g$ denotes the Laplace-Beltrami operator. This definition is far from rigorous, e.g., because the space of Riemannian metric tensors on $U$ is infinite-dimensional, so there is not an obvious notion of a uniform measure on this space. However, there are various ways of defining LQG surface rigorously, as we discuss just below. 

\begin{defn} \label{def-cc-phases}
We refer to LQG with $\ccM \in (-\infty,1)$, $\ccM  =1$, and $\ccM \in (1,25)$ as the \emph{subcritical}, \emph{critical}, and \emph{supercritical} phases, respectively.
\end{defn}

See Table~\ref{fig-c-phases-table} for a summary of the three phases. 
One way to define LQG rigorously in the subcritical and critical phases is via the \emph{David-Distler-Kawai (DDK) ansatz}. The DDK ansatz states that for $\ccM \in (-\infty,1]$, the Riemannian metric tensor associated with an LQG surface takes the form
\eqb \label{eqn-lqg-tensor}
g = e^{\gamma h} \, (dx^2 + dy^2) ,\quad \text{where $\gamma \in (0,2]$ satisfies} \quad \ccM = 25 - 6\left(\frac{2}{\gamma} + \frac{\gamma}{2} \right)^2 .
\eqe
Here, $dx^2 + dy^2$ denotes the Euclidean metric tensor on $U$ and $h$ is a variant of the Gaussian free field (GFF) on $U$, the most natural random generalized function on $U$.
We refer to~\cite{shef-gff,pw-gff-notes,bp-lqg-notes} for more background on the GFF.  

The Riemannian metric tensor in~\eqref{eqn-lqg-tensor} is still not well-defined since the GFF is not a function, so $e^{\gamma h}$ does not make literal sense. Nevertheless, it is possible to rigorously define various objects associated with~\eqref{eqn-lqg-tensor} using regularization procedures. To do this, one considers a family of continuous functions $\{h_\ep\}_{\ep > 0}$ which approximate $h$, then takes an appropriate limit of objects defined using $h_\ep$ in place of $h$. Objects which have been constructed in this manner include the LQG area and length measures~\cite{shef-kpz,rhodes-vargas-log-kpz,kahane}, Liouville Brownian motion~\cite{grv-lbm,berestycki-lbm}, the correlation functions for the random ``fields" $e^{\alpha h}$ for $\alpha  \in \BB R$~\cite{krv-dozz}, and the distance function (metric) associated with~\eqref{eqn-lqg-tensor}, at least for $\ccM < 1$~\cite{dddf-lfpp,gm-uniqueness}. 

LQG in the subcritical and critical phases is expected, and in some cases proven, to describe the scaling limit of various types of random planar maps. For example, in keeping with the above heuristic definition, LQG with $\ccM \in (-\infty,1]$ should describe the scaling limit of random planar maps sampled with probability proportional to $(\det\Delta)^{-\ccM/2}$, where $\Delta$ is the discrete Laplacian. We refer to~\cite{berestycki-lqg-notes,gwynne-ams-survey,ghs-mating-survey} for expository articles on subcritical and critical LQG. 

The supercritical phase $\ccM \in (1,25)$ is much more mysterious than the subcritical and critical phases, even from the physics perspective. In this case, the DDK ansatz does not apply. In fact, the parameter $\gamma$ from~\eqref{eqn-lqg-tensor} is complex with $|\gamma| = 2$, so attempting to directly analytically continue formulas from the subcritical case to the supercritical case often gives nonsensical complex answers. 
It is expected that supercritical LQG still corresponds in some sense to a random geometry related to the GFF. 
However, until very recently there have been few mathematically rigorous results for supercritical LQG. 
See~\cite{ghpr-central-charge} for an extensive discussion of the physics literature and various conjectures concerning LQG with $\ccM \in (1,25)$. 

The purpose of this paper is to show that in the critical and supercritical phases, i.e., when $\ccM \in [1,25)$, there is a canonical metric (distance function) associated with LQG. 
This was previously established in the subcritical phase $\ccM \in (-\infty,1)$ in the series of papers~\cite{dddf-lfpp,local-metrics,lqg-metric-estimates,gm-confluence,gm-uniqueness}.
Our results resolve~\cite[Problems 7.17 and 7.18]{gm-uniqueness}, which ask for a metric associated with LQG for $\ccM \in [1,25)$. 

\begin{figure}[ht!]
\begin{center}
\includegraphics[width=1\textwidth]{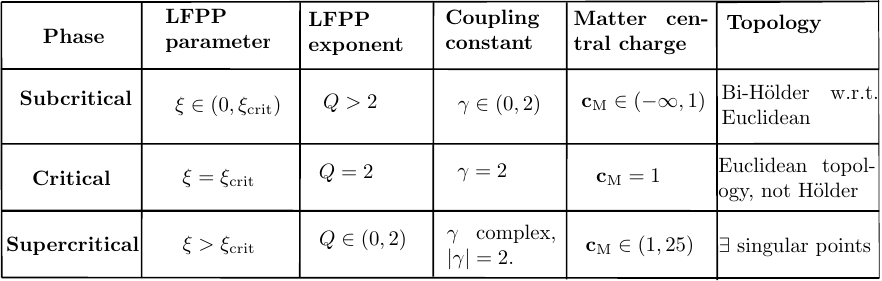} 
\caption{\label{fig-c-phases-table} 
Comparison of the different phases of LQG. This paper proves that the LQG metric is unique in the critical and supercritical phases. The bi-H\"older continuity w.r.t.\ to the Euclidean metric in the subcritical phase is proven in~\cite{lqg-metric-estimates}. The statement that the critical LQG metric induces the Euclidean topology, but is not H\"older continuous, is proven in~\cite{dg-critical-lqg}. 
}
\end{center}
\end{figure}

This paper builds on~\cite{dg-supercritical-lfpp}, which proved the tightness of an approximation procedure for the metric when $\ccM \in [1,25)$ (using~\cite{lfpp-pos} and some estimates from~\cite{dddf-lfpp} which also work for the critical/supercritical cases), and~\cite{pfeffer-supercritical-lqg}, which proved various properties of the subsequential limits. The analogs of these works in the subcritical case are~\cite{dddf-lfpp} and~\cite{lqg-metric-estimates}, respectively.
We will also use one preliminary lemma which was proven in~\cite{dg-confluence} (Lemma~\ref{lem-hit-ball}), but we will not need the main result of~\cite{dg-confluence}, i.e., the confluence of geodesics property.

Our results are analogous to those of~\cite{gm-uniqueness}, which proved uniqueness of the subcritical LQG metric. We will prove that the subsequential limiting metrics in the critical and supercritical cases are uniquely characterized by a natural list of axioms. 
However, our proof is very different from the argument of~\cite{gm-uniqueness}, for two main reasons.
\begin{itemize}
\item A key input in~\cite{gm-uniqueness} is \emph{confluence of geodesics}, which says that two LQG geodesics with the same starting point and different target points typically coincide for a non-trivial initial interval of time~\cite{gm-confluence}. We replace the core part of the argument in~\cite{gm-uniqueness}, which corresponds to~\cite[Section 4]{gm-uniqueness}, by a simpler argument which does not use confluence of geodesics (Section~\ref{sec-counting}). Instead, our argument is based on counting the number of events of a certain type which occur. Confluence of geodesics was proven for the critical and supercritical LQG metrics in~\cite{dg-confluence}, but it is not needed in this paper.  
\item There are many additional difficulties in our proof, especially in Section~\ref{sec-construction}, arising from the fact that the metrics we work with are not continuous with respect to the Euclidean metric, or even finite-valued. 
\end{itemize}
The first point reduces the complexity of this paper as compared to~\cite{gm-uniqueness}, whereas the second point increases it. The net effect is that our argument is overall longer than~\cite{gm-uniqueness}, but conceptually simpler and requires less external input. We note that all of our arguments apply in the subcritical phase as well as the critical and supercritical phases, so this paper also gives a new proof of the results of~\cite{gm-uniqueness}. 
\bigskip

\noindent \textbf{Acknowledgments.} We thank an anonymous referee for helpful comments on an earlier version of this article. J.D.\ was partially supported by NSF grants DMS-1757479 and DMS-1953848. E.G.\ was partially supported by a Clay research fellowship. 

\subsection{Convergence of Liouville first passage percolation}
\label{sec-lfpp}

For concreteness, throughout this paper we will restrict attention to the whole-plane case. We let $h$ be the whole-plane Gaussian free field with the additive constant chosen so that its average over the unit circle is zero. Once the LQG metric for $h$ is constructed, it is straightforward to construct metrics associated with variants of the GFF on other domains via restriction and/or local absolute continuity; see~\cite[Remark 1.5]{gm-uniqueness}. 

As in the subcritical case, the construction of our metric uses an approximation procedure called \emph{Liouville first passage percolation} (LFPP). 
To define LFPP, we first introduce a family of continuous functions which approximate $h$. 
For $s > 0$ and $z \in\BB C$, let $p_s(z) = \frac{1}{2\pi s} \exp\left( - \frac{|z|^2}{2s} \right)$ be the heat kernel. For $\ep >0$, we define a mollified version of the GFF by
\eqb \label{eqn-gff-convolve}
h_\ep^*(z) := (h*p_{\ep^2/2})(z) = \int_{\BB C} h(w) p_{\ep^2/2} (z-w) \, dw ,\quad \forall z\in \BB C  ,
\eqe
where the integral is interpreted in the sense of distributional pairing. We use $p_{\ep^2/2}$ instead of $p_\ep$ so that the variance of $h_\ep^*(z)$ is $\log\ep^{-1} + O_\ep(1)$. 
 
We now consider a parameter $\xi > 0$, which will shortly be chosen to depend on the matter central charge $\ccM$ (see~\eqref{eqn-cc-xi}). \emph{Liouville first passage percolation} (LFPP) with parameter $\xi$ is the family of random metrics $\{D_h^\ep\}_{\ep > 0}$ defined by
\eqb \label{eqn-lfpp}
D_h^\ep(z,w) := \inf_{P : z\rta w} \int_0^1 e^{\xi h_\ep^*(P(t))} |P'(t)| \,dt ,\quad \forall z,w\in\BB C 
\eqe
where the infimum is over all piecewise continuously differentiable paths $P:[0,1]\rta\BB C$ from $z$ to $w$.  
To extract a non-trivial limit of the metrics $D_h^\ep$, we need to re-normalize. We (somewhat arbitrarily) define our normalizing factor by
\eqb \label{eqn-gff-constant}
\frk a_\ep := \text{median of} \: \inf\left\{ \int_0^1 e^{\xi h_\ep^*(P(t))} |P'(t)| \,dt  : \text{$P$ is a left-right crossing of $[0,1]^2$} \right\} ,
\eqe  
where a left-right crossing of $[0,1]^2$ is a piecewise continuously differentiable path $P : [0,1]\rta [0,1]^2$ joining the left and right boundaries of $[0,1]^2$. 
We do not know the value of $\frk a_\ep$ explicitly. The best currently available estimates are given in~\cite[Theorem 1.11]{dg-polylog}. 

More generally, the definition~\eqref{eqn-lfpp} of LFPP also makes sense when $h$ is a \emph{whole-plane GFF plus a bounded continuous function}, i.e., a random distribution of the form $\wt h + f$, where $\wt h$ is a whole-plane GFF and $f$ is a (possibly random and $\wt h$-dependent) bounded continuous function. 

In terms of LFPP, the main result of this paper gives the convergence of the metrics $\frk a_\ep^{-1} D_h^\ep$ for each $\xi > 0$. 
For values of $\xi$ corresponding to the supercritical case $\ccM \in (1,25)$, the limiting metric is not continuous with respect to the Euclidean metric. 
Hence we cannot expect convergence with respect to the uniform topology. 
Instead, as in~\cite{dg-supercritical-lfpp}, we will work with the topology of the following definition.

\begin{defn} \label{def-lsc} 
Let $X\subset \BB C$. 
A function $f : X \times X \rta \BB R \cup\{-\infty,+\infty\}$ is \emph{lower semicontinuous} if whenever $(z_n,w_n) \in X\times X$ with $(z_n,w_n) \rta (z,w)$, we have $f(z,w) \leq \liminf_{n\rta\infty} f(z_n,w_n)$. 
The \emph{topology on lower semicontinuous functions} is the topology whereby a sequence of such functions $\{f_n\}_{n\in\BB N}$ converges to another such function $f$ if and only if
\begin{enumerate}[(i)]
\item Whenever $(z_n,w_n) \in X\times X$ with $(z_n,w_n) \rta (z,w)$, we have $f(z,w) \leq \liminf_{n\rta\infty} f_n(z_n,w_n)$.
\item For each $(z,w)\in X\times X$, there exists a sequence $(z_n,w_n) \rta (z,w)$ such that $f_n(z_n,w_n) \rta f(z,w)$. 
\end{enumerate}
\end{defn}

It follows from~\cite[Lemma 1.5]{beer-usc} that the topology of Definition~\ref{def-lsc} is metrizable (see~\cite[Section 1.2]{dg-supercritical-lfpp}). 
Furthermore,~\cite[Theorem 1(a)]{beer-usc} shows that the metric inducing this topology can be taken to be separable. 

\begin{thm} \label{thm-lfpp-conv}
Let $h$ be a whole-plane GFF, or more generally a whole-plane GFF plus a bounded continuous function. For each $\xi  >0$, the re-scaled LFPP metrics $\frk a_\ep^{-1} D_h^\ep$ converge in probability with respect to the topology on lower semicontinuous functions on $\BB C\times \BB C$ (Definition~\ref{def-lsc}). 
The limit $D_h$ is a random metric on $\BB C$, except that it is allowed to take on infinite values.
\end{thm}

To make the connection between Theorem~\ref{thm-lfpp-conv} and the LQG metric, we need to discuss the LFPP distance exponent $Q$. It was shown in~\cite[Proposition 1.1]{dg-supercritical-lfpp} that for each $\xi > 0$, there exists $Q = Q(\xi) > 0$ such that 
\eqb  \label{eqn-Q-def}
\frk a_\ep = \ep^{1 - \xi Q  + o_\ep(1)} ,\quad \text{as} \quad \ep \rta 0 . 
\eqe 
The existence of $Q$ is proven via a subadditivity argument, so the exact relationship between $Q$ and $\xi$ is not known. However, it is known that $Q \in (0,\infty)$ for all $\xi  > 0$ and $Q$ is a continuous, non-increasing function of $\xi$~\cite{dg-supercritical-lfpp,lfpp-pos}. See also~\cite{gp-lfpp-bounds,ang-discrete-lfpp} for bounds for $Q$ in terms of $\xi$. 

As we will discuss in more detail below, LFPP with parameter $\xi$ is related to LQG with matter central charge
\eqb \label{eqn-cc-xi}
\ccM = \ccM(\xi )  = 25 - 6Q(\xi)^2 .
\eqe
The function $\xi\mapsto Q(\xi)$ is continuous and $Q(\xi) \rta \infty$ as $\xi \rta 0$ and $Q(\xi) \rta 0$ as $\xi\rta \infty$~\cite[Proposition 1.1]{dg-supercritical-lfpp}. So, the formula~\eqref{eqn-cc-xi} shows that there is a value of $\xi$ corresponding to each $\ccM \in (-\infty,25)$. Furthermore, $\xi \mapsto Q(\xi)$ is strictly decreasing on $(0,0.7)$, so the function $\xi\mapsto\ccM(\xi)$ is injective on this interval. We expect that it is in fact injective on all of $(0,\infty)$, which would mean that there is a one-to-one correspondence between $\xi$ and $\ccM$.\footnote{One way to prove the injectivity of $\xi\mapsto \ccM(\xi)$ would be to show that if $\xi$ and $\ccM$ are related as in~\eqref{eqn-cc-xi}, then $\xi$ is the distance exponent for the dyadic subdivision model in~\cite{ghpr-central-charge} with parameter $\ccM$: indeed, this would give an inverse to the function $\xi \mapsto \ccM(\xi)$. We expect that this can be proven using similar arguments to the ones used to related LFPP and Liouville graph distance in~\cite{dg-lqg-dim}, see also the discussion of LFPP in~\cite[Section 2.3]{ghpr-central-charge}.}

The relation between $\xi$ and $\ccM$ in~\eqref{eqn-cc-xi} is not explicit since the dependence of $Q$ on $\xi$ is not known explicitly. 
The only exact relation between $\ccM$ and $\xi$ which we know is that $\ccM = 0$ corresponds to $\xi = 1/\sqrt 6$. This is equivalent to the fact that the Hausdorff dimension of LQG with $\gamma=\sqrt{8/3}$ is 4. See~\cite{dg-lqg-dim} for details. 

From~\eqref{eqn-cc-xi}, we see that $Q(\xi) = 2$ corresponds to the critical value $\ccM = 1$, which motivates us to define
\eqb
\xi_\crit := \inf\{\xi > 0 : Q(\xi) = 2\} .
\eqe
It follows from~\cite[Proposition 1.1]{dg-supercritical-lfpp} that $\xi_\crit$ is the unique value of $\xi$ for which $Q(\xi) = 2$ and from~\cite[Theorem 2.3]{gp-lfpp-bounds} that $\xi_\crit  \in [0.4135 , 0.4189]$. We have $Q > 2$ for $\xi < \xi_\crit$ and $Q \in(0,2)$ for $\xi > \xi_\crit$. 

\begin{defn} \label{def-xi-phases}
We refer to LFPP with $\xi <\xi_\crit$, $\xi = \xi_\crit$, and $\xi > \xi_\crit$ as the \emph{subcritical}, \emph{critical}, and \emph{supercritical} phases, respectively.
\end{defn}

By~\eqref{eqn-cc-xi}, the three phases of LFPP correspond exactly to the three phases of LQG in Definition~\ref{def-cc-phases}.

Theorem~\ref{thm-lfpp-conv} has already been proven in the subcritical phase $\xi < \xi_\crit$ (but this paper simplifies part of the proof). 
Indeed, it was shown by Ding, Dub\'edat, Dunlap, and Falconet~\cite{dddf-lfpp} that in this case the re-scaled LFPP metrics $\frk a_\ep^{-1} D_h^\ep$ are tight with respect to the topology of uniform convergence on compact subsets of $\BB C\times \BB C$, which is a stronger topology than the one in Definition~\ref{def-lsc}. 
Subsequently, it was shown by Gwynne and Miller~\cite{gm-uniqueness}, building on~\cite{local-metrics,lqg-metric-estimates,gm-confluence}, that the subsequential limit is unique. This was done by establishing an axiomatic characterization of the limiting metric. 

The limiting metric in the subcritical phase induces the same topology on $\BB C$ as the Euclidean metric, but has very different geometric properties. This metric can be thought of as the Riemannian distance function associated with the Riemannian metric tensor~\eqref{eqn-lqg-tensor}, where $\ccM \in (-\infty,1)$ and $\xi$ are related as in~\eqref{eqn-cc-xi}. The relation between $\ccM$ and $\xi$ can equivalently be expressed as $\gamma = \xi d(\xi)$, where $\gamma \in (0,2)$ is as in~\eqref{eqn-lqg-tensor} and $d(\xi) > 2$ is the Hausdorff dimension of the limiting metric~\cite{dg-lqg-dim,gp-kpz}. 
See~\cite{ddg-metric-survey} for a survey of results about the subcritical LQG metric (and some previous results in the critical and supercritical cases).

In the critical and supercritical cases, Theorem~\ref{thm-lfpp-conv} is new. 
We previously showed in~\cite{dg-supercritical-lfpp} that for all $\xi > 0$, the metrics $\{\frk a_\ep^{-1} D_h^\ep\}_{\ep > 0}$ are tight with respect to the topology on lower semicontinuous functions. 
The contribution of the present paper is to show that the subsequential limit is unique. 
We will do this by proving that the limiting metric is uniquely characterized by a list of axioms analogous to the one in~\cite{gm-uniqueness} (see Theorems~\ref{thm-strong-uniqueness} and~\ref{thm-weak-uniqueness}).  

In the critical case $\xi = \xi_\crit$, the limiting metric $D_h$ induces the same topology as the Euclidean metric~\cite{dg-critical-lqg}, and can be thought of as the Riemannian distance function associated with critical ($\gamma=2$) LQG. We refer to~\cite{powell-gmc-survey} for a survey of results concerning the critical LQG \emph{measure}. 

In the supercritical case $\xi > \xi_\crit$, the limiting metric in Theorem~\ref{thm-lfpp-conv} does not induce the Euclidean topology on $\BB C$. Rather, a.s.\ there exists an uncountable, Euclidean-dense set of \emph{singular points} $z\in \BB C$ such that
\eqb \label{def-singular}
D_h(z,w) = \infty ,\quad \forall w\in \BB C\setminus \{z\} .
\eqe
However, for each fixed $z\in\BB C$, a.s.\ $z$ is not a singular point, so the set of singular points has zero Lebesgue measure. Moreover, any two non-singular points lie at finite $D_h$-distance from each other~\cite{dg-supercritical-lfpp}. One can think of singular points as infinite ``spikes" which $D_h$-rectifiable paths must avoid. 

If we let $\{h_\ep\}_{\ep > 0}$ be the circle average process for the GFF~\cite[Section 3.1]{shef-kpz}, then the set of singular points is (almost) the same as the set of points $z\in\BB C$ which have \emph{thickness} greater than $Q$, in the sense that
\eqb \label{eqn-singular-thick}
\limsup_{\ep\rta 0} \frac{h_\ep(z)}{\log\ep^{-1}} > Q .
\eqe
See~\cite[Proposition 1.11]{pfeffer-supercritical-lqg} for a precise statement. 
It is shown in~\cite{hmp-thick-pts} that a.s.\
\eqbn
\limsup_{\ep\rta 0} h_\ep(z) / \log\ep^{-1} \in [-2,2] ,\quad \forall z\in\BB C ,
\eqen
which explains why $\xi_\crit$ (which corresponds to $Q = 2$) is the critical threshold for singular points to exist.

\begin{remark}[Conjectured random planar map connection] \label{remark-planar-map}
In the subcritical case, the LQG metric is conjectured to describe the scaling limit of various types of random planar maps, equipped with their graph distance, with respect to the Gromov-Hausdorff topology (see~\cite[Section 1.3]{gm-uniqueness}). 
This conjecture naturally extends to the critical case. In particular, the critical LQG metric should be the Gromov-Hausdorff scaling limit of random planar maps sampled with probability proportional to the partition function of, e.g., the discrete Gaussian free field, the O(2) loop model, the critical 4-state Potts model, or the critical Fortuin-Kasteleyn model with parameter $q=4$~\cite{shef-burger,ghs-mating-survey,ahps-critical-mating}.

A naive guess in the supercritical case is that the LQG metric for $\ccM \in (1,25)$ should describe the scaling limit of random planar maps sampled with probability proportional to $(\det\Delta)^{-\ccM/2}$, where $\Delta$ is the discrete Laplacian. This guess appears to be false, however, since numerical simulations and heuristics suggest that such planar maps converge in the scaling limit to trees (see~\cite[Section 2.2]{ghpr-central-charge} and the references therein). 
Rather, in order to get supercritical LQG in the limit, one should consider planar maps sampled with probability proportional to $(\det\Delta)^{-\ccM/2}$ which are in some sense ``allowed to have infinitely many vertices". We do not know how to make sense of such maps rigorously. However,~\cite{ghpr-central-charge} defines a random planar map which should be in the same universality class: it is the adjacency graph of a dyadic tiling of $\BB C$ by squares which all have the same ``$\ccM$-LQG size" with respect to an instance of the GFF. See~\cite{ghpr-central-charge} for further discussion. 
\end{remark}

\subsection{Characterization of the LQG metric}
\label{sec-strong}

Since we already know that LFPP is tight for all $\xi  >0$~\cite{dg-supercritical-lfpp}, in order to prove Theorem~\ref{thm-lfpp-conv} we need to show that the subsequential limit is unique. To accomplish this, we will prove that for each $\xi > 0$, there is a unique (up to multiplication by a deterministic positive constant) metric satisfying certain axioms. That is, we will extend the characterization result of~\cite{gm-uniqueness} to the supercritical case. To state our axioms, we first need some preliminary definitions. 

\begin{defn} \label{def-metric-properties}
Let $(X,d)$ be a metric space, with $d$ allowed to take on infinite values. 
\begin{itemize}
\item A \emph{curve (a.k.a.\ a path)} in $(X,d)$ is a continuous function $P : [a,b] \rta X$ for some interval $[a,b]$.
\item
For a curve $P : [a,b] \rta X$, the \emph{$d$-length} of $P$ is defined by 
\eqbn
\op{len}(P;d) :=  \sup_{T} \sum_{i=1}^{\# T} d(P(t_i) , P(t_{i-1})) 
\eqen
where the supremum is over all partitions $T : a= t_0 < \dots < t_{\# T} = b$ of $[a,b]$. Note that the $d$-length of a curve may be infinite. In particular, the $d$-length of $P$ is infinite if there are times $s,t\in[a,b]$ such that $d(P(s) , P(t)) = \infty$.
\item
We say that $(X,d)$ is a \emph{length space} if for each $x,y\in X$ and each $\ep > 0$, there exists a curve of $d$-length at most $d(x,y) + \ep$ from $x$ to $y$. 
If $d(x,y) < \infty$, a curve from $x$ to $y$ of $d$-length \emph{exactly} $d(x,y)$ is called a \emph{geodesic}. 
\item
For $Y\subset X$, the \emph{internal metric of $d$ on $Y$} is defined by
\eqb \label{eqn-internal-def}
d(x,y ; Y)  := \inf_{P \subset Y} \op{len}\left(P ; d \right) ,\quad \forall x,y\in Y 
\eqe 
where the infimum is over all curves $P$ in $Y$ from $x$ to $y$. 
Note that $d(\cdot,\cdot ; Y)$ is a metric on $Y$, except that it is allowed to take infinite values.  
\item
If $X \subset \BB C$, we say that $d$ is a \emph{lower semicontinuous metric} if the function $(x,y) \rta d(x,y)$ is lower semicontinuous w.r.t.\ the Euclidean topology.  
We equip the set of lower semicontinuous metrics on $X$ with the topology on lower semicontinuous functions on $X \times X$, as in Definition~\ref{def-lsc}, and the associated Borel $\sigma$-algebra.
\end{itemize}
\end{defn}

The axioms which characterize our metric are given in the following definition.

\begin{defn}[LQG metric]
\label{def-metric0}
Let $\mcl D'$ be the space of distributions (generalized functions) on $\BB C$, equipped with the usual weak topology.   
For $\xi > 0$, a \emph{(strong) LQG metric with parameter $\xi$} is a measurable function $h\mapsto D_h$ from $\mcl D'$ to the space of lower semicontinuous metrics on $\BB C$ with the following properties.\footnote{We do not care how $D$ is defined on any subset of $\mcl D'$ which has probability zero for the distribution of any whole-plane GFF plus a continuous function.} Let $h$ be a \emph{GFF plus a continuous function} on $\BB C$: i.e., $h$ is a random distribution on $\BB C$ which can be coupled with a random continuous function $f$ in such a way that $h-f$ has the law of the whole-plane GFF. Then the associated metric $D_h$ satisfies the following axioms. 
\begin{enumerate}[I.]
\item \textbf{Length space.} Almost surely, $(\BB C,D_h)$ is a length space. \label{item-metric-length} 
\item \textbf{Locality.} Let $U\subset\BB C$ be a deterministic open set. 
The $D_h$-internal metric $D_h(\cdot,\cdot ; U)$ is a.s.\ given by a measurable function of $h|_U$.  \label{item-metric-local}
\item \textbf{Weyl scaling.} For a continuous function $f : \BB C \rta \BB R$, define
\eqb \label{eqn-metric-f}
(e^{\xi f} \cdot D_h) (z,w) := \inf_{P : z\rta w} \int_0^{\op{len}(P ; D_h)} e^{\xi f(P(t))} \,dt , \quad \forall z,w\in \BB C ,
\eqe 
where the infimum is over all $D_h$-rectifiable paths from $z$ to $w$ in $\BB C$ parametrized by $D_h$-length (we use the convention that $\inf \emptyset = \infty$).
Then a.s.\ $ e^{\xi f} \cdot D_h = D_{h+f}$ for every continuous function $f: \BB C \rta \BB R$. \label{item-metric-f}
\item \textbf{Scale and translation covariance.} Let $Q$ be as in~\eqref{eqn-Q-def}. For each fixed deterministic $r > 0$ and $z\in\BB C$, a.s.\ \label{item-metric-coord0}
\eqb
 D_h \left( ru + z , r v + z \right) = D_{h(r\cdot + z)  +Q\log r}(u,v)  , \quad \forall u,v\in\BB C  .
\eqe    
\item \textbf{Finiteness.} Let $U \subset\BB C$ be a deterministic, open, connected set and let $K_1 , K_2\subset U$ be disjoint, deterministic, compact, connected sets which are not singletons. Almost surely, $ D_h(K_1,K_2;U) < \infty$.  \label{item-metric-finite}
\end{enumerate}
\end{defn} 

Definition~\ref{def-metric0} is nearly identical to the analogous definition in the subcritical case~\cite[Section 1.2]{gm-uniqueness}, except we only require the metric to be lower semicontinuous, rather than requiring it to induce the Euclidean topology. 
Because we allow $D_h$ to take infinite values, we need to include a finiteness condition (Axiom~\ref{item-metric-finite}) to rule out metrics which assign infinite distance to too many pairs of points. For example, if we defined $D_h$ for every distribution $h$ by $D_h(z,w) = 0$ if $z=w$ and $D_h(z,w) =  \infty$ if $z\not=w$, then $h\mapsto D_h$ would satisfy all of the conditions of Definition~\ref{def-metric0} except for Axiom~\ref{item-metric-finite}.

Axioms~\ref{item-metric-length}, \ref{item-metric-local}, and~\ref{item-metric-f} are natural from the heuristic that the LQG metric should be given by ``integrating $e^{\xi h}$ along paths, then taking an infimum over paths". 
We remark that if $h$ is a GFF plus a continuous function and $D_h$ is a weak LQG metric, then a.s.\ the Euclidean metric is continuous with respect to $D_h$~\cite[Proposition 1.10]{pfeffer-supercritical-lqg} (but $D_h$ is not continuous w.r.t.\ the Euclidean metric if $\xi > \xi_\crit$). Consequently, a.s.\ every path of finite $D_h$-length is Euclidean continuous.  

Axiom~\ref{item-metric-coord0} is the metric analog of the LQG coordinate change formula from~\cite[Section 2]{shef-kpz}, but restricted to translation and scaling. 
Following~\cite{shef-kpz}, we can think of the pairs $(\BB C , D_h)$ and $(\BB C , h(r\cdot + z)  +Q\log r)$ as representing two different parametrizations of the same LQG surface. Axiom~\ref{item-metric-coord0} implies that the metric is an intrinsic function of the LQG surface, i.e., it is invariant under changing coordinates to a different parametrization. We do not assume that the metric is covariant with respect to rotations in Definition~\ref{def-metric0}: this turns out to be a consequence of the other axioms (see Proposition~\ref{prop-rotation}). 

The following theorem extends~\cite[Theorem 1.2]{gm-uniqueness} to the critical and supercritical phases. 

\begin{thm} \label{thm-strong-uniqueness} 
For each $\xi> 0$, there is an LQG metric $D$ with parameter $\xi$ such that the limiting metric of Theorem~\ref{thm-lfpp-conv} is a.s.\ equal to $D_h$ whenever $h$ is a whole-plane GFF plus a bounded continuous function. 
Furthermore, this LQG metric is unique in the following sense. If $D$ and $\wt D$ are two LQG metrics with parameter $\xi$, then there is a deterministic constant $C>0$ such that a.s.\ $\wt D_h = C   D_h$ whenever $h$ is a whole-plane GFF plus a continuous function. 
\end{thm}

Theorem~\ref{thm-strong-uniqueness} tells us that for every $\ccM \in (-\infty,25)$, there is an essentially unique\footnote{
Strictly speaking, we only show that there is a unique LQG metric with parameter $\xi$ for each $\xi \in (0,\infty)$. 
In order to deduce that the metric with central charge $\ccM$ is unique we would need to know that $\xi\mapsto \ccM(\xi)$ is injective.
We expect that this injectivity is not hard to prove, but a proof of has so far only been written down for $\xi \in (0,0.7)$. See the discussion just after~\eqref{eqn-cc-xi}.}
metric associated with LQG with matter central charge $\ccM$ (recall the non-explicit relation between $\xi$ and $\ccM$ from~\eqref{eqn-cc-xi}). 
The deterministic positive constant $C$ from Theorem~\ref{thm-strong-uniqueness} can be fixed in various ways. For example, we can require that the median of the $D_h$-distance between the left and right sides of the unit square is 1 in the case when $h$ is a whole-plane GFF normalized so that its average over the unit circle is 0. 
Due to~\eqref{eqn-gff-constant}, the limit of LFPP has this normalization. 

Theorem~\ref{thm-strong-uniqueness} implies that the LQG metric is covariant with respect to rotation, not just scaling and translation. See~\cite[Remark 1.6]{gm-uniqueness} for a heuristic discussion of why we do not need to assume rotational invariance in Definition~\ref{def-metric0}. 
 
\begin{prop}  \label{prop-rotation}
Let $\xi > 0$ and let $D $ be an LQG metric with parameter $\xi$. Let $h$ be a whole-plane GFF plus a continuous function and let $\omega \in \BB C$ with $|\omega | =1$. Almost surely, 
\eqb
D_h(u,v) = D_{h(\omega\cdot)}(\omega^{-1} u ,\omega^{-1} v) ,\quad\forall u ,v \in \BB C .
\eqe
\end{prop}
\begin{proof}
Define $D_h^{(\omega)}(u,v) :=  D_{h(\omega\cdot)}(\omega^{-1} u ,\omega^{-1} v)$.
It is easily verified that $D^{(\omega)}$ satisfies the conditions of Definition~\ref{def-metric0}, so Theorem~\ref{thm-strong-uniqueness} implies that there is a deterministic constant $C >0$ such that a.s.\ $D^{(\omega)}_h = C D_h$ whenever $h$ is a whole-plane GFF plus a continuous function.
To check that $C = 1$, consider the case when $h$ is a whole-plane GFF $h$ normalized so that its average over the unit circle is 0.
Then the law of $h$ is rotationally invariant, so $\BB P[D_h(0,\bdy\BB D) > R] = \BB P[D_h^{(\omega)}(0,\bdy\BB D) > R]$ for every $R > 0$. 
Therefore $C  =1$. 
\end{proof}

Proposition~\ref{prop-rotation} implies that $D_h$ is covariant with respect to complex affine maps. 
It is natural to expect that $D_h$ is also covariant with respect to general conformal maps, in the following sense.
Let $U,\wt U \subset \BB C$ be open and let $\phi : U \rta \wt U$ be a conformal map. 
Then it should be the case that a.s.\
\eqb \label{eqn-metric-coord}
D_h(\phi(u) , \phi(v) ;  \wt U ) = D_{h\circ\phi + Q\log|\phi'|}(u,v ; U) ,\quad \forall u ,v \in   U . 
\eqe 
In the subcritical case, the coordinate change relation~\eqref{eqn-metric-coord} was proven in~\cite{gm-coord-change}. We expect that the proof there can be adapted to treat the critical and supercritical cases as well.

Various properties of the LQG metric $D_h$ for $\ccM \in [1,25)$ have already been established in the literature. 
For example, for $\ccM \in (1,25)$ a.s.\ each $D_h$-metric ball $\mcl B$ centered at a non-singular point is not $D_h$-compact~\cite[Proposition 1.14]{gp-kpz}, but the boundaries of the connected components of $\BB C\setminus \mcl B$ are $D_h$-compact and are Jordan curves~\cite[Theorem 1.4]{dg-confluence}.
Furthermore, one has a confluence property for LQG geodesics~\cite[Theorem 1.6]{dg-confluence} and a version of the Knizhnik-Polyakov-Zamolodchikov (KPZ) formula, which relates Hausdorff dimensions with respect to $D_h$ and the Euclidean metric~\cite[Theorem 1.15]{pfeffer-supercritical-lqg}. 
Simulations of supercritical LQG metric balls and geodesics can be found in~\cite{dg-supercritical-lfpp,dg-confluence,ddg-metric-survey}.

There are many open problems related to the LQG metric for $\ccM \in [1,25)$.  
A list of open problems concerning LQG with $\ccM \in (1,25)$ can be found in~\cite[Section 6]{ghpr-central-charge}.  
Moreover, most of the open problems for the LQG metric with $\ccM \in (-\infty,1)$ from~\cite[Section 7]{gm-uniqueness} are also interesting for $\ccM \in [1,25)$. 
Here, we mention one open problem which has not been discussed elsewhere.

\begin{prob} \label{prob-limit}
Let $D_h^{(\xi)}$ denote the LQG metric with parameter $\xi$. 
Does $D_h^{(\xi)}$, appropriately re-scaled, converge in some topology as $\xi\rta\infty$ (equivalently, $\ccM \rta 25$)? 
Even if one doesn't have convergence of the whole metric, can anything be said about the limits of $D_h^{(\xi)}$-metric balls, geodesics, etc.?
\end{prob}

\subsection{Weak LQG metrics}
\label{sec-weak}

In this subsection we will introduce a notion of weak LQG metric for general $\xi > 0$ (Definition~\ref{def-metric}), which is similar to Definition~\ref{def-metric0} but with Axiom~\ref{item-metric-coord0} replaced by a weaker condition. Our notion of a weak LQG metric first appeared in~\cite{pfeffer-supercritical-lqg}. We will then state a uniqueness theorem for weak LQG metrics (Theorem~\ref{thm-weak-uniqueness}) and explain why our other main theorems (Theorems~\ref{thm-lfpp-conv} and~\ref{thm-strong-uniqueness}) follow from this theorem. A similar notion of weak LQG metrics was used in the proof of uniqueness of the subcritical LQG metric~\cite{lqg-metric-estimates,gm-uniqueness}.

To motivate the definition of weak LQG metrics, we first observe that every possible subsequential limit of the re-scaled LFPP metrics $\frk a_\ep^{-1} D_h^\ep$ satisfies Axioms~\ref{item-metric-length}, \ref{item-metric-local}, and~\ref{item-metric-f} in Definition~\ref{def-metric0}.  
This is intuitively clear from the definition, and not too hard to check rigorously (see~\cite[Section 2]{pfeffer-supercritical-lqg}). 
It is also easy to see that every possible subsequential limit of LFPP satisfies Axiom~\ref{item-metric-coord0} for $r = 1$ (i.e., it satisfies the coordinate change formula for translations). 
However, it is far from obvious that the subsequential limits satisfy Axiom~\ref{item-metric-coord} when $r\not=1$. The reason is that re-scaling space changes the value of $\ep$ in~\eqref{eqn-lfpp}: for $\ep ,r>0$, one has~\cite[Lemma 2.6]{lqg-metric-estimates}
\eqbn
D_h^\ep(rz, rw) = r D_{h(r\cdot)}^{\ep/r}(z,w)  ,\quad \forall z,w\in \BB C.
\eqen
So, since we only have subsequential limits of $\frk a_\ep^{-1} D_h^\ep$, we cannot directly deduce that the subsequential limit satisfies an exact spatial scaling property. 

Because of the above issue, we do not know how to check Axiom~\ref{item-metric-coord0} for subsequential limits of LFPP directly. 
Instead, we will prove a stronger uniqueness statement than the one in Theorem~\ref{thm-strong-uniqueness}, under a weaker list of axioms which can be checked for subsequential limits of LFPP. We will then deduce from this stronger uniqueness statement that the weaker list of axioms implies the axioms in Definition~\ref{def-metric0} (Lemma~\ref{lem-weak-to-strong}).  

An \emph{annular region} is a bounded open set $A\subset\BB C$ such that $A$ is homeomorphic to an open, closed, or half-open Euclidean annulus. If $A$ is an annular region, then $\bdy A$ has two connected components, one of which disconnects the other from $\infty$. We call these components the outer and inner boundaries of $A$, respectively.

\begin{defn}[Distance across and around annuli] \label{def-around-across} 
Let $d$ be a length metric on $\BB C$. 
For an annular region $A \subset\BB C$, we define $d\left(\text{across $A$}\right)$ to be the $d $-distance between the inner and outer boundaries of $A$.
We define $d \left(\text{around $A$}\right)$ to be the infimum of the $d $-lengths of paths in $A$ which disconnect the inner and outer boundaries of $A$. 
\end{defn}

\noindent
Note that both $d(\text{across $A$})$ and $d(\text{around $A$})$ are determined by the internal metric of $d$ on $A$. 
Distances around and across Euclidean annuli play a similar role to ``hard crossings" and ``easy crossings" of $2\times 1$ rectangles in percolation theory. 
One can get a lower bound for the $d$-length of a path in terms of the $d$-distances across the annuli that it crosses. On the other hand, one can ``string together" paths around Euclidean annuli to get upper bounds for $d$-distances.
The following is (almost) a re-statement of~\cite[Definition 1.6]{pfeffer-supercritical-lqg}. 
 
\begin{defn}[Weak LQG metric]
\label{def-metric}
Let $\mcl D'$ be as in Definition~\ref{def-metric}. 
For $\xi > 0$, a \emph{weak LQG metric with parameter $\xi$} is a measurable function $h\mapsto D_h$ from $\mcl D'$ to the space of lower semicontinuous metrics on $\BB C$ which satisfies properties~\ref{item-metric-length} (length metric), \ref{item-metric-local} (locality), and~\ref{item-metric-f} (Weyl scaling) from Definition~\ref{def-metric0} plus the following two additional properties.
\begin{enumerate}[I$'$.]
\addtocounter{enumi}{3} 
\item \textbf{Translation invariance.} For each deterministic point $z \in \BB C$, a.s.\ $D_{h(\cdot + z)} = D_h(\cdot+ z , \cdot+z)$.  \label{item-metric-translate}
\item \textbf{Tightness across scales.} Suppose that $h$ is a whole-plane GFF and let $\{h_r(z)\}_{r > 0, z\in\BB C}$ be its circle average process. 
Let $A\subset \BB C$ be a deterministic Euclidean annulus.
In the notation of Definition~\ref{def-around-across}, the random variables
\eqbn
r^{-\xi Q} e^{-\xi h_r(0)} D_h\left( \text{across $r A$} \right) \quad \text{and} \quad
r^{-\xi Q}  e^{-\xi h_r(0)} D_h\left( \text{around $r A$} \right)
\eqen
and the reciprocals of these random variables for $r>0$ are tight.   \label{item-metric-coord}  
\end{enumerate}
\end{defn}

\newcommand{\reftranslate}{{\hyperref[item-metric-translate]{IV$'$}}}
\newcommand{\refcoord}{\hyperref[item-metric-translate]{V$'$}}

We think of Axiom~\refcoord as a substitute for Axiom~\ref{item-metric-coord0} of Definition~\ref{def-metric0}. 
Indeed, Axiom~\refcoord does not give an exact spatial scaling property, but it still allows us to get estimates for $D_h$ which are uniform across different Euclidean scales. 

It was shown in~\cite[Theorem 1.7]{pfeffer-supercritical-lqg} that every subsequential limit of the re-scaled LFPP metrics $\frk a_\ep^{-1} D_h^\ep$ is a weak LQG metric in the sense of Definition~\ref{def-metric}. Actually,~\cite{pfeffer-supercritical-lqg} allows for a general family of scaling constants $\{\frk c_r\}_{r > 0}$ in Axiom~\refcoord in place of $r^{\xi Q}$, but it was shown in~\cite[Theorem 1.9]{dg-polylog} that one can always take $\frk c_r = r^{\xi Q}$. 
So, our definition is equivalent to the one in~\cite{pfeffer-supercritical-lqg}. 

From the preceding paragraph and the tightness of $\frk a_\ep^{-1} D_h^\ep$~\cite{dg-supercritical-lfpp}, we know that there exists a weak LQG metric for each $\xi > 0$. 
Most of this paper is devoted to the proof of the uniqueness of the weak LQG metric. 

\begin{thm} \label{thm-weak-uniqueness} 
For each $\xi > 0$, the weak LQG metric is unique in the following sense. If $D$ and $\wt D$ are two weak LQG metrics with parameter $\xi$, then there is a deterministic constant $C>0$ such that a.s.\ $D_h = C\wt D_h$ whenever $h$ is a whole-plane GFF plus a continuous function. 
\end{thm}

Let us now explain why Theorem~\ref{thm-weak-uniqueness} is sufficient to establish our main results, Theorems~\ref{thm-lfpp-conv} and~\ref{thm-strong-uniqueness}. 
We first observe that every strong LQG metric is a weak LQG metric. 

\begin{lem} \label{lem-strong-to-weak}
For each $\xi > 0$, each strong LQG metric (Definition~\ref{def-metric0}) is a weak LQG metric (Definition~\ref{def-metric}). 
\end{lem}
\begin{proof}
Let $D$ be a strong LQG metric. It is immediate from Axiom~\ref{item-metric-coord} of Definition~\ref{def-metric0} with $r=1$ that $D$ satisfies translation invariance (Axiom~\reftranslate). We need to check Axiom~\refcoord. 
To this end, let $h$ be a whole-plane GFF normalized so that $h_1(0) = 0$. Weyl scaling (Axiom~\ref{item-metric-f}) together with conformal covariance (Axiom~\ref{item-metric-coord0}) gives
\eqb \label{eqn-metric0-scale}
r^{-\xi Q} e^{-\xi h_r(0)} D_h(r\cdot,r\cdot) = D_{h(r\cdot) - h_r(0)} (\cdot,\cdot) \eqD D_h(\cdot,\cdot),  
\eqe
where the equality in law is due to the scale invariance of the law of $h$, modulo additive constant. 

To get tightness across scales, it therefore suffices to show that for each fixed Euclidean annulus $A$, a.s.\ $D_h(\text{across $A$})$ and $D_h(\text{around $A$})$ are finite and positive.
Our finiteness condition Axiom~\ref{item-metric-finite} easily implies that these two quantities are a.s.\ finite. 
To see that they are a.s.\ positive, it suffices to show that for any two deterministic, disjoint, Euclidean-compact sets $K_1,K_2\subset\BB C$, a.s.\ $D_h(K_1,K_2) > 0$. 
Indeed, on the event $\{D_h(K_1,K_2) = 0\}$ we can find sequences of points $z_n \in K_1$ and $w_n \in K_2$ such that $D_h(z_n,w_n) \rta 0$. 
After possibly passing to a subsequence, we can arrange that $z_n \rta z \in K_1$ and $w_n \rta w \in K_2$. By the lower semicontinuity of $D_h$, we get $D_h(z,w) =0$. 
Since $z$ and $w$ are distinct and $D_h$ is a metric (not a pseudometric) this implies that $\BB P[D_h(K_1,K_2) = 0] = 0$. 
\end{proof}

Theorem~\ref{thm-weak-uniqueness} implies that one also has the converse to Lemma~\ref{lem-strong-to-weak}. 

\begin{lem} \label{lem-weak-to-strong}
For each $\xi > 0$, every weak LQG metric is a strong LQG metric in the sense of Definition~\ref{def-metric0}. 
\end{lem}
\begin{proof}[Proof of Lemma~\ref{lem-weak-to-strong} assuming Theorem~\ref{thm-weak-uniqueness}]
Let $D$ be a weak LQG metric. 
It is clear that $z$ satisfies Axioms~\ref{item-metric-length}, \ref{item-metric-local}, \ref{item-metric-f}, and~\ref{item-metric-finite} of Definition~\ref{def-metric0}.
To show that $D$ is a strong LQG metric, we need to check Axiom~\ref{item-metric-coord0} of Definition~\ref{def-metric0} in the case when $z = 0$ (note that we already have translation invariance from Definition~\ref{def-metric}). 
To this end, for $b > 0$ let 
\eqb
D^{(b)}_h(\cdot,\cdot) :=   D_{h( b  \cdot  ) + Q\log b } (   \cdot  / b   ,   \cdot / b )  .
\eqe
If $h$ is a whole-plane GFF with $h_1(0) = 0$ then by the scale invariance of the law of $h$, modulo additive constant, we have $h(b\cdot) - h_b(0) \eqD h$. 
Consequently, if $h$ is a whole-plane GFF plus a continuous function, then $h(b\cdot ) + Q\log b$ is also a whole-plane GFF plus a continuous function.
Hence $D_h^{(b)}$ is well-defined.

We need to show that a.s.\ $D_h^{(b)} = D_h$. We will prove this using Theorem~\ref{thm-weak-uniqueness}. 
We first claim that $D_h^{(b)}$ is a weak LQG metric. 
It is easy to check that $D^{(b)}$ satisfies Axioms~\ref{item-metric-length}, \ref{item-metric-local}, \ref{item-metric-f}, and~\reftranslate in Definition~\ref{def-metric}. To check Axiom~\refcoord, we use Weyl scaling (Axiom~\ref{item-metric-f}) to get that 
\alb
&r^{-\xi Q} e^{-\xi h_r(0)} D^{(b)}_h (r \cdot , r\cdot) \notag\\
&\qquad =  e^{-\xi (h_r(0) - h_{ r / b}(0))} e^{\xi h_b(0)} \times   (r/b)^{-\xi Q}  e^{-\xi h_{  r /b}(0)}  D_{h( b \cdot ) - h_b(0) } \left( (r/b)  \cdot , (r/b) \cdot \right)  .
\ale
In the case when $h$ is a whole-plane GFF, the random variables $h_r(0) - h_{ r /b}(0)$ and $h_b(0)$ are each centered Gaussian with variance $\log \max\{b,1/b\}$~\cite[Section 3.1]{shef-kpz}.   
Tightness across scales (Axiom~\refcoord) for $D$ applied with $h(b \cdot ) - h_b(0) \eqD h$ in place of $h$ and $  r /b$ in place of $r$ therefore implies tightness across scales for $D^{(b)}$. 
 
Hence we can apply Theorem~\ref{thm-weak-uniqueness} with $\wt D = D^{(b)}$ to get that for each $b >0$, there is a deterministic constant $\frk k_b >0$ such that whenever $h$ is a whole-plane GFF plus a continuous function, a.s.\ 
\eqbn
D_h^{(b)}  = \frk k_b D_h . 
\eqen
It remains to show that $\frk k_b = 1$. 

For $b_1,b_2 > 0$, we have $D^{(b_1b_2)} =  ( D^{(b_1)} )^{(b_2)}$, which implies that a.s.\ $D_h^{(b_1b_2)} = \frk k_{b_2} D_h^{(b_1)} = \frk k_{b_1} \frk k_{b_2} D_h$. Therefore, 
\eqb \label{eqn-weak-constants-mult}
\frk k_{b_1b_2} = \frk k_{b_1} \frk k_{b_2} . 
\eqe
It is also easy to see that $\frk k_b$ is a Lebesgue measurable function of $b$.  
Indeed, by Weyl scaling (Axiom~\ref{item-metric-f}) and since $h(b \cdot ) - h_b(0) \eqD h$, 
\eqb \label{eqn-weak-constants-law} 
\frk k_b e^{-\xi h_b(0)}  D_h(  b  \cdot    ,  b  \cdot    ) 
= e^{-\xi h_b(0)} D_h^{(b)}( b \cdot , b \cdot ) 
=  b^{\xi Q}  D_{h(b \cdot  ) - h_b(0) }(   \cdot,  \cdot) 
\eqD b^{\xi Q} D_h(\cdot,\cdot) .
\eqe 
The function $b\mapsto b^{-\xi Q} e^{-\xi h_b(0)}$ is continuous and $D_h$ is lower semicontinuous. 
Hence the metrics $b^{-\xi Q} e^{-\xi h_b(0)} D_h(b\cdot,b\cdot)$ depend continuously on $b$ with respect to the topology on lower semicontinuous functions.
Therefore, the law of $\frk k_b^{-1} D_h$ depends continuously on $b$ with respect to the topology on lower semicontinuous functions. 
It follows that $\frk k_b$ is continuous, hence Lebesgue measurable.   

The relation~\eqref{eqn-weak-constants-mult} and the measurability of $b\mapsto \frk k_b$ imply that $\frk k_b = b^\alpha$ for some $\alpha\in\BB R$.
By~\eqref{eqn-weak-constants-law}, we have $b^{\alpha -\xi Q}  e^{-\xi h_b(0)} D_h(b\cdot, b\cdot) \eqD D_h(\cdot,\cdot)$ for each $b>0$.
In particular, Axiom~\refcoord$\,$ holds for $D$ with $ \xi Q - \alpha$ in place of $\xi Q$. 
Hence $\alpha = 0$.  
\end{proof}

\begin{proof}[Proof of Theorem~\ref{thm-lfpp-conv}, assuming Theorem~\ref{thm-weak-uniqueness}]
By~\cite[Theorem 1.2]{dg-supercritical-lfpp}, if $h$ is a whole-plane GFF plus a bounded continuous function, then for each $\xi > 0$, the re-scaled LFPP metrics $\frk a_\ep^{-1} D_h^\ep$ are tight with respect to the topology of Definition~\ref{def-lsc}. 
In fact, by~\cite[Theorem 1.7]{pfeffer-supercritical-lqg}, for any sequence of positive $\ep$ values tending to zero there is a weak LQG metric $D$ and a subsequence $\ep_n \rta 0$ such that whenever $h$ is a whole-plane GFF plus a continuous functions, the metrics $\frk a_{\ep_n}^{-1} D_h^{\ep_n}$ converge in probability to $D_h$ with respect to this topology. 
By Theorem~\ref{thm-weak-uniqueness}, if $D$ and $\wt D$ are two weak LQG metrics arising as subsequential limits in this way, then there is a deterministic $C > 0$ such that a.s.\ $\wt D_h = C D_h$ whenever $h$ is a whole-plane GFF plus a continuous function. 

If $h$ is a whole-plane GFF normalized so that $h_1(0) = 0$, then by the definition of $\frk a_\ep$ in~\eqref{eqn-gff-constant}, the median $\frk a_\ep^{-1} D_h^\ep$-distance between the left and right sides of $[0,1]^2$ is 1.
By passing this through to the limit, we get that the constant $C$ above must be equal to 1. 
Therefore, a.s.\ $D_h = \wt D_h$ whenever $h$ is a whole-plane GFF plus a continuous function, so the subsequential limit of $\frk a_\ep^{-1} D_h^\ep$ is unique. 
\end{proof}

\begin{proof}[Proof of Theorem~\ref{thm-strong-uniqueness}, assuming Theorem~\ref{thm-weak-uniqueness}]
The uniqueness of the strong LQG metric follows from Theorem~\ref{thm-weak-uniqueness} and Lemma~\ref{lem-weak-to-strong}. 
The existence follows from the existence of the limit in Theorem~\ref{thm-lfpp-conv}, \cite[Theorem 1.7]{pfeffer-supercritical-lqg} (which says that the limit is a weak LQG metric), and Lemma~\ref{lem-weak-to-strong}. 
\end{proof}

\subsection{Outline}
\label{sec-outline}

As explained in Section~\ref{sec-weak}, to establish our main results we only need to prove Theorem~\ref{thm-weak-uniqueness}.
To this end, let $h$ be a whole-plane GFF and let $D_h$ and $\wt D_h$ be two weak LQG metrics as in Definition~\ref{def-metric}. 
We need to show that there is a deterministic constant $C>0$ such that a.s.\ $\wt D_h = C D_h$. In this subsection, we will give an outline of the proof of this statement. 
Throughout this outline and the rest of the paper, we will frequently use without comment the following fact, which is~\cite[Proposition 1.12]{pfeffer-supercritical-lqg}. 

\begin{lem}[\!\!\cite{pfeffer-supercritical-lqg}] \label{lem-geodesics}
Almost surely, the metric $D_h$ is complete and finite-valued on $\BB C \setminus \{\text{singular points}\}$. 
Moreover, every pair of points in $\BB C \setminus \{\text{singular points}\}$ can be joined by a $D_h$-geodesic (Definition~\ref{def-metric-properties}). 
\end{lem}

\subsubsection{Optimal bi-Lipschitz-constants}
\label{sec-bilip}

By~\cite[Theorem 1.10]{dg-polylog}, the metrics $D_h$ and $\wt D_h$ are a.s.\ bi-Lipschitz equivalent, so in particular a.s.\ they have the same set of singular points. 
We define the optimal upper and lower bi-Lipschitz constants
\allb \label{eqn-bilip-def}
\Clower &:= \inf\left\{ \frac{\wt D_h(u,v)}{D_h(u,v)} \: : \:  u,v\in\BB C \setminus \{\text{singular points}\},\,  u\not= v \right\} \quad \text{and} \notag\\
\Cupper &:= \sup\left\{ \frac{\wt D_h(u,v)}{D_h(u,v)}\: : \: u,v\in\BB C \setminus \{\text{singular points}\},\,  u\not= v \right\}   .
\alle

\begin{lem} \label{lem-bilip-const}
Each of $\Clower$ and $\Cupper$ is a.s.\ equal to a deterministic, positive, finite constant.
\end{lem}
\begin{proof}
By the bi-Lipschitz equivalence of $D_h$ and $\wt D_h$, a.s.\ $\Clower$ and $\Cupper$ are positive and finite.
We know from~\cite[Lemma 3.12]{pfeffer-supercritical-lqg} that a.s.\ for each $z\in\BB C$, we have $\lim_{R\rta\infty} D_h(z,\bdy B_R(z)) = \infty$. 
With this fact in hand, the lemma follows from exactly the same elementary tail triviality argument as in the subcritical case~\cite[Lemma 3.1]{gm-uniqueness}. 
\end{proof}

We henceforth replace $\Clower$ and $\Cupper$ by their a.s.\ values in Lemma~\ref{lem-bilip-const}, so that each of $\Clower$ and $\Cupper$ is a deterministic constant depending only on the laws of $D_h$ and $\wt D_h$ and a.s.\
\eqb \label{eqn-bilip}
\Clower D_h(u,v) \leq \wt D_h(u,v) \leq \Cupper D_h(u,v) ,\quad\forall u,v\in\BB C .
\eqe

\subsubsection{Main idea of the proof}

To prove Theorem~\ref{thm-weak-uniqueness}, it suffices to show that $\Clower = \Cupper$. 
In the rest of this subsection, we will give an outline of the proof of this fact. There are many subtleties in our proof which we will gloss over in this outline in order to focus on the key ideas. So, the statements in the rest of this subsection should not be taken as mathematically precise. 
 
At a very broad level, the basic strategy of our proof is similar to the proof of the uniqueness of the subcritical LQG metric in~\cite{gm-uniqueness}. 
However, the details in Sections~\ref{sec-attained} and~\ref{sec-construction} are substantially different from the analogous parts of~\cite{gm-uniqueness}, and the argument in Section~\ref{sec-counting} is completely different from anything in~\cite{gm-uniqueness}. 

We now give a very rough explanation of the main idea of our proof. Assume by way of contradiction that $\Clower < \Cupper$.
We will show that for any $\Cmid \in (\Clower,\Cupper)$, there are many ``good" pairs of distinct non-singular points $u,v\in \BB C$ such that $\wt D_h(u,v) \leq \Cmid D_h(u,v)$ (Section~\ref{sec-attained}). 
In fact, we will show that the set of such points is large enough that every $D_h$-geodesic $P$ has to get $\wt D_h$-close to each of $u$ and $v$ for many ``good" pairs of points $u,v$ (Sections~\ref{sec-counting} and~\ref{sec-construction}). 
For each of these good pairs of points, we replace a segment of $P$ by the concatenation of a $\wt D_h$-geodesic from a point of $P$ to $u$, a $\wt D_h$-geodesic from $u$ to $v$, and a $\wt D_h$-geodesic from $v$ to a point of $P$. This gives a new path with the same endpoints as $P$.

By our choice of good pairs of points $u,v$, the $\wt D_h$-length of each of the replacement segments is at most a constant slightly larger than $\Cmid$ times its $D_h$-length. 
Furthermore, by the definition of $\Cupper$ the $\wt D_h$-length of each segment of $P$ which was not replaced is at most $\Cupper$ times its $D_h$-length.
Morally, we would like to say that this implies that there exists $\Cmid' \in (\Cmid , \Cupper)$ such that a.s.
\eqb  \label{eqn-smaller-upper}
\wt D_h(z,w) \leq \Cmid' D_h(z,w) ,\quad \forall z,w \in \BB C .
\eqe 
The bound~\eqref{eqn-smaller-upper} contradicts the fact that $\Cupper$ is the optimal upper bi-Lipschitz constant (recall~\eqref{eqn-bilip-def}). 
In actuality, what we will prove is a bit more subtle: assuming that $\Clower < \Cupper$, we will establish for ``many" small values of $r >0$ and each $\delta>0$ an upper bound for
\eqb  \label{eqn-smaller-upper'}
\BB P\left[ \wt D_h(z,w) \leq (\Cupper -\delta) D_h(z,w) ,\:\text{$\forall z,w \in \ol B_r(0)$ satisfying certain conditions} \right]  . 
\eqe
See Proposition~\ref{prop-outline-zero} for a somewhat more precise statement. 
This upper bound will be incompatible with a lower bound for the same probability (Proposition~\ref{prop-outline-pos}), which will lead to our desired contradiction. 

In the rest of this subsection, we give a more detailed, section-by-section outline of the proof. 

\subsubsection{Section~\ref{sec-prelim}: preliminary estimates}

We will fix some notation, then record several basic estimates for the LQG metric which are straightforward consequences of results in the existing literature (mostly~\cite{pfeffer-supercritical-lqg}). 

\subsubsection{Section~\ref{sec-attained}: quantitative estimates for optimal bi-Lipschitz constants}

Let $\Cmed \in (\Clower,\Cupper)$. By the definition~\eqref{eqn-bilip-def} of $\Clower$ and $\Cupper$, it holds with positive probability that there exists non-singular points $u,v\in\BB C$ such that $\wt D_h(u,v) \geq \Cmed D_h(u,v)$. 
The purpose of Section~\ref{sec-attained} is to prove a quantitative version of this statement. The argument of Section~\ref{sec-attained} is similar to the argument of~\cite[Section 3]{gm-uniqueness}, but many of the details are different due to the fact that our metrics do not induce the Euclidean topology. 

The following is a simplified version of the main result of Section~\ref{sec-attained} (see Proposition~\ref{prop-geo-annulus-prob} for a precise statement). 
 
\begin{prop} \label{prop-outline-pos}
There exists $ p \in (0,1)$, depending only on the laws of $D_h$ and $\wt D_h$, such that for each $\Cmed \in (0,\Cupper)$ and each sufficiently small $\ep > 0$ (depending on $\Cmed$ and the laws of $D_h$ and $\wt D_h$), there are at least $\frac34\log_8\ep^{-1}$ values of $r \in [\ep^2 , \ep] \cap \{8^{-k}\}_{k\in\BB N}$ such that
\eqb \label{eqn-outline-pos} 
\BB P\left[\text{$\exists$ a ``regular" pair of points $u,v\in \ol B_r(0)$ s.t.\ $\wt D_h(u,v)\geq \Cmed D_h(u,v)$} \right] \geq p .
\eqe 
\end{prop}

The statement that $u$ and $v$ are ``regular" in~\eqref{eqn-outline-pos} means that these points satisfy several regularity conditions which are stated precisely in Definition~\ref{def-annulus-geo}. 
These conditions include an upper bound on $D_h(u,v)$ (so in particular $u$ and $v$ are non-singular) and a lower bound on $|u-v|$ in terms of $r$. 
We emphasize that the parameter $p$ in Proposition~\ref{prop-outline-pos} does not depend on $\Cmed$. 
This will be crucial for our purposes, see the discussion just after Proposition~\ref{prop-outline-zero}. 

We will prove Proposition~\ref{prop-outline-pos} by contradiction. In particular, we will assume that there are arbitrarily small values of $\ep > 0$ for which there are at least $\frac14 \log_8\ep^{-1}$ values of $r \in [\ep^2 , \ep] \cap \{8^{-k}\}_{k\in\BB N}$ such that
\eqb \label{eqn-outline-bad} 
\BB P\left[\text{$\wt D_h(u,v) < \Cmed D_h(u,v)$, $\forall$ ``regular" pairs of points $ u,v\in \ol B_r(0)$} \right] \geq 1 - p .
\eqe 
If $p$ is small enough (depending only on the laws of $D_h$ and $\wt D_h$), then we can use the assumption~\eqref{eqn-outline-bad} together with the near-independence of the restrictions of the GFF to disjoint concentric annuli (Lemma~\ref{lem-annulus-iterate}) and a union bound to get the following. 
For any bounded open set $U\subset \BB C$, it holds with high probability that $U$ can be covered by balls $B_r(z)$ for $z\in U$ and $r\in [\ep^2,\ep] \cap \{8^{-k}\}_{k\in\BB N}$ such that the event in~\eqref{eqn-outline-bad} occurs. 

We will then work on the high-probability event that we have such a covering of $U$. Consider points $\BB z , \BB w \in U$ such that there exists a $ D_h$-geodesic $ P$ from $\BB z$ to $\BB w$ which is contained in $U$. We will replace several segments of $ P$ between pairs of ``regular" points $u,v$ as in~\eqref{eqn-outline-bad} by $\wt D_h$-geodesics from $u$ to $v$. The $\wt D_h$-length of each of these geodesics is at most $\Cmed D_h(u,v)$. Furthermore, by~\eqref{eqn-bilip-def}, the $D_h$-length of each segment of $ P$ which we did not replace is at most $\Cupper$ times its $ D_h$-length. We thus obtain a path from $\BB z$ to $\BB w$ with $\wt D_h$-length at most $\Cmed'  D_h(u,v)$, where $\Cmed' \in (\Cmed,\Cupper)$ is a constant depending only on $\Cmed$ and the laws of $D_h$ and $\wt D_h$. With high probability, this works for any $ D_h$-geodesic contained in $U$. So, by taking $U$ to be arbitrarily large, we contradict the definition of $\Cupper$. This yields Proposition~\ref{prop-outline-pos}. 

By the symmetry in our hypotheses for $D_h$ and $\wt D_h$, we also get the following analog of Proposition~\ref{prop-outline-pos} with the roles of $D_h$ and $\wt D_h$ interchanged.

\begin{prop} \label{prop-outline-pos'}
There exists $ p \in (0,1)$, depending only on the laws of $D_h$ and $\wt D_h$, such that for each $\Cmid > \Clower$ and each sufficiently small $\ep > 0$ (depending on $\Cmid$ and the laws of $D_h$ and $\wt D_h$), there are at least $\frac34\log_8\ep^{-1}$ values of $r \in [\ep^2 , \ep] \cap \{8^{-k}\}_{k\in\BB N}$ for which
\eqb \label{eqn-outline-pos'} 
\BB P\left[\text{$\exists$ a ``regular" pair of points $ u,v\in \ol B_r(0)$ s.t.\ $\wt D_h(u,v)\leq \Cmid D_h(u,v)$} \right] \geq p .
\eqe 
\end{prop}

\subsubsection{Section~\ref{sec-counting}: the core argument}

The idea of the rest of the proof of Theorem~\ref{thm-weak-uniqueness} is to show that if $\Clower < \Cupper$, then Proposition~\ref{prop-outline-pos'} implies a contradiction to Proposition~\ref{prop-outline-pos}. 

The core part of the proof is given in Section~\ref{sec-counting}, where we will prove Theorem~\ref{thm-weak-uniqueness} conditional on the existence of events and bump functions satisfying certain specified properties. The needed events and bump functions will be constructed in Section~\ref{sec-construction}. 
Section~\ref{sec-counting} plays a role analogous to~\cite[Sections 4 and 6]{gm-uniqueness}, but the proof is completely different. 

We will consider a set of admissible radii $\mcl R\subset (0,1)$, which will eventually be taken to be equal to $\rho^{-1} \mcl R_0$, where $\rho$ is a constant and $\mcl R_0$ is the set of $r\in \{8^{-k}\}_{k\in\BB N}$ for which~\eqref{eqn-outline-pos'} holds. 
We also fix a constant $\BB p \in (0,1)$, which will eventually be chosen to be close to 1, in a manner depending only on the laws of $D_h$ and $\wt D_h$, and we set
\eqbn
\Cmid := \frac{\Clower + \Cupper}{2} , \quad \text{so that} \quad \Cmid \in (\Clower, \Cupper) \quad \text{if} \quad \Clower < \Cupper. 
\eqen 

We will assume that for each $r\in \mcl R$ and each $z\in\BB C$, we have defined an event $\Er_{z,r}$ and a deterministic function $\fr_{z,r}$ satisfying the following properties. 
\begin{itemize}
\item $\Er_{z,r} $ is determined by $h|_{B_{4r}(z) \setminus B_r(z)}$, viewed modulo additive constant, and $\BB P[\Er_{z,r}] \geq \BB p$.
\item $\fr_{z,r}$ is smooth, non-negative, and supported on the annulus $B_{3r}(z) \setminus B_r(z)$. 
\item Assume that $\Er_{z,r}$ occurs and $P'$ is a $D_{h-\fr_{z,r}}$-geodesic between two points of $\BB C\setminus B_{4r}(z)$ which spends ``enough" time in the support of $\fr_{z,r}$. Then there are times $s < t$ such that $P'([s,t]) \subset B_{4r}(z)$ and
\eqb \label{eqn-outline-inc}
\wt D_{h-\fr_{z,r}}(P'(s) , P'(t)) \leq \Cmid (t-s). 
\eqe
\end{itemize}
The precise list of properties that we need is stated in Section~\ref{sec-counting-setup}. 

Roughly speaking, the support of $\fr_{z,r}$ will be a long narrow tube contained in a small neighborhood of $\bdy B_{2r}(0)$. 
On the event $\Er_{z,r}$, there will be many ``good" pairs of non-singular points $u,v$ in the support of $\fr_{z,r}$ such that $\wt D_h(u,v) \leq \Cmid_0 D_h(u,v)$ and the $\wt D_h$-geodesic from $u$ to $v$ is contained in the support of $\fr_{z,r}$, where $\Cmid_0 \in (\Clower,\Cmid)$ is fixed. See Figure~\ref{fig-outline-counting} for an illustration.
We will show that $\Er_{z,r}$ occurs with high probability for $r\in\mcl R$ using Proposition~\ref{prop-outline-pos'} (with $\Cmid_0$ instead of $\Cmid$) and a long-range independence statement for the GFF (Lemma~\ref{lem-spatial-ind}). 
 
The function $\fr_{z,r}$ will be very large on most of its support. So, by Weyl scaling (Axiom~\ref{item-metric-f}), a $D_{h-\fr_{z,r}}$-geodesic which enters the support of $\fr_{z,r}$ will tend to spend a long time in the support of $\fr_{z,r}$. 
This will force the $D_{h-\fr_{z,r}}$-geodesic to get $D_{h-\fr_{z,r}}$-close to each of $u$ and $v$ for one of the aforementioned ``good" pairs of points $u,v$. 
The estimate~\eqref{eqn-outline-inc} will follow from this and the triangle inequality. 
Most of Section~\ref{sec-counting} is devoted to proving an estimate (Proposition~\ref{prop-counting}) which roughly speaking says the following.

\begin{figure}[ht!]
\begin{center}
\includegraphics[width=.9\textwidth]{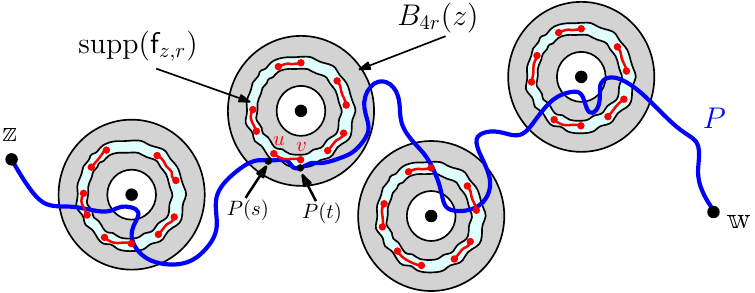} 
\caption{\label{fig-outline-counting} 
Illustration of three ``good" balls (i.e., ones for which $\Er_{z,r}$ occurs) and one ``very good" ball (i.e., one for which $\Er_{z,r}(h+\fr_{z,r})$ occurs) which are hit by the $D_h$-geodesic $P$. Each of the ``good" balls contains several pairs of non-singular points $u,v$ in the support of $\fr_{z,r}$ (light blue) for which $\wt D_h(u,v) \leq \Cmid_0 D_h(u,v)$. These points and the $\wt D_h$-geodesics joining them are shown in red. For the ``very good" ball (the labeled ball in the figure), $P$ gets $D_{h-\fr_{z,r}}$-close to each of $u$ and $v$ for one of the aforementioned pairs of points $u,v$. To prove Proposition~\ref{prop-outline-counting}, we will show that there are lots of ``very good" balls for which $P$ spends a lot of time in the support of $\fr_{z,r}$.  
}
\end{center}
\end{figure}

\begin{prop} \label{prop-outline-counting}
Assume that $\Clower < \Cupper$ and we have defined events $\Er_{z,r}$ and functions $\fr_{z,r}$ satisfying the above properties.
As $\delta\rta 0$, it holds uniformly over all $\BB z,\BB w \in\BB C$ that
\eqb \label{eqn-outline-counting} 
\BB P\left[ \wt D_h\left( \BB z,\BB w  \right) >  (\Cupper - \delta) D_h\left(\BB z,\BB w \right) , \: \text{regularity conditions} \right]  =  O_\delta(\delta^\mu) ,\quad \forall \mu  > 0  .
\eqe
\end{prop}

We think of a ball $B_{4r}(z)$ as ``good" if the event $\Er_{z,r}$ occurs and ``very good" if the event $\Er_{z,r}(h + \fr_{z,r})$, which is defined in the same manner as $\Er_{z,r}$ but with $h+\fr_{z,r}$ instead of $h$, occurs. By definition, if $B_{4r}(z)$ is``good" for $h$, then $B_{4r}(z)$ is ``very good" for $h -\fr_{z,r}$.  

Let $P$ be the $D_h$-geodesic from $\BB z$ to $\BB w$ (which is a.s.\ unique, see Lemma~\ref{lem-geo-unique} below). 
Recall that $\BB P[\Er_{z,r}] \geq \BB p$, which is close to 1, and $\Er_{z,r}$ is determined by $h|_{B_{4r}(z) \setminus B_{r}(z)}$, viewed modulo additive constant.  
From this, it is easy to show using the near-independence of the restrictions of $h$ to disjoint concentric annuli (Lemma~\ref{lem-annulus-iterate}) that $P$ has to hit $B_r(z)$ for lots of ``good" balls $B_{4r}(z)$. 

To prove Proposition~\ref{prop-outline-counting}, it suffices to show that with high probability, there are many ``very good" balls $B_{4r}(z)$ such that the $D_h$-geodesic $P$ from $\BB z$ to $\BB w$ spends ``enough" time in the support of the bump function $\fr_{z,r}$. 
Indeed, the condition~\eqref{eqn-outline-inc} (with $h+\fr_{z,r}$ instead of $h$) will then give us lots of pairs of points $s,t$ such that $\wt D_h(P(s) , P(t)) \leq \Cmid (t-s)$, which in turn will show that $\wt D_h(\BB z,\BB w)$ is bounded away from $\Cupper D_h(\BB z,\BB w)$ (see Proposition~\ref{prop-card}).

In~\cite{gm-uniqueness}, it was shown that $P$ hits many ``very good" balls by using confluence of geodesics (which was proven in~\cite{gm-confluence}) to get an approximate Markov property for $P$. 
In this paper, we will instead show this using a simpler argument based on counting the number of events of a certain type which occur. 
More precisely, for $r\in\mcl R$ and a finite collection of points $Z$ such that the balls $B_{4r}(z)$ for $z\in Z$ are disjoint, we will let $F_{Z,r}$ be (roughly speaking) the event that the following is true.
\begin{itemize}
\item Each ball $B_{4r}(z)$ for $z\in Z$ is ``good". 
\item The $D_h$-geodesic $P$ from $\BB z$ to $\BB w$ hits $B_{r}(z)$ for each $z\in Z$.
\item With $\fr_{Z,r} := \sum_{z\in Z} \fr_{z,r}$, the $D_{h-\fr_{Z,r}}$-geodesic from $\BB z$ to $\BB w$ spends ``enough" time in the support of $\fr_{z,r}$ for each $z\in Z$. 
\end{itemize}
We also let $F_{Z,r}'$ be defined in the same manner as $F_{Z,r}$ but with $h + \fr_{Z,r}$ in place of $h$, i.e., $F_{Z,r}'$ is the event that the following is true.
\begin{itemize}
\item Each $B_{4r}(z)$ for $z\in Z$ is ``very good". 
\item The $D_{h+\fr_{Z,r}}$-geodesic from $\BB z$ to $\BB w$ hits $B_r(z)$ for each $z\in Z$.
\item The $D_h$-geodesic $P$ from $\BB z$ to $\BB w$ spends ``enough" time in the support of $\fr_{z,r}$ for each $z\in Z$. 
\end{itemize}

Using a basic Radon-Nikodym derivative for the GFF, one can show that there is a constant $C > 0$ depending only on the laws of $D_h$ and $\wt D_h$ such that 
\eqb \label{eqn-outline-rn}
C^{-k}  \BB P[F_{Z,r}] \leq \BB P[F_{Z,r}'] \leq C^k \BB P[F_{Z,r}] , \quad \text{whenever $\# Z \leq k$} 
\eqe
(see Lemma~\ref{lem-rn}). We will eventually take $k$ to be a large constant, independent of $r,\BB z,\BB w$, depending on the number $\mu$ in~\eqref{eqn-outline-counting}.
So, the relation~\eqref{eqn-outline-rn} suggests that the number of sets $Z$ such that $\# Z \leq k$ and $F_{Z,r}$ occurs should be comparable to the number of such sets for which $F_{Z,r}'$ occurs. 

Furthermore, one can show that if $\ep$ is small enough, then for each $r \in [\ep^2,\ep]$, the number of sets $Z$ with $\#Z \leq k$ such that $F_{Z,r}$ occurs grows like a positive power of $\ep^{-k}$ (Proposition~\ref{prop-choices}). 
Indeed, as explained above, there are many sets $Z_0$ such that for each $z\in Z_0$, the ball $B_{4r}(z)$ is good and the ball $B_r(z)$ is hit by $P$.
We need to produce many sets $Z$ for which these properties hold and also that $D_{h-\fr_{Z,r}}$-geodesic spends enough time in the support of $\fr_{z,r}$ for each $z \in Z$. 
To do this, we start with a set $Z_0$ as above and iteratively remove the ``bad" points $z\in Z_0$ such that  the $D_{h-\fr_{Z_0,r}}$-geodesic from $\BB z$ to $\BB w$ does not spend very much time in the support of $\fr_{z,r}$. By doing so, we obtain a set $Z \subset Z_0$ such that $F_{Z,r}$ occurs and $\# Z$ is not too much smaller than $\# Z_0$.
See Section~\ref{sec-counting-choices} for details.   

By combining the preceding two paragraphs with an elementary calculation (see the end of Section~\ref{sec-counting-main}), we infer that with high probability there are lots of sets $Z$ with $\# Z \leq k$ such that $F_{Z,r}'$ occurs. In particular, there must be lots of ``very good" balls $B_{4r}(z)$ for which $P$ spends a lot of time in the support of $\fr_{z,r}$. As explained above, this gives Proposition~\ref{prop-outline-counting}. 

Once Proposition~\ref{prop-outline-counting} is established, one can take a union bound over many pairs of points $\BB z , \BB w \in B_r(0)$ to get, roughly speaking, the following (see Lemma~\ref{lem-opt-event-prob} for a precise statement). 
 
\begin{prop} \label{prop-outline-zero}
Assume that $\Clower < \Cupper$. For each sufficiently small $\ep > 0$ (depending only on the laws of $D_h$ and $\wt D_h$), there are at least $\frac34\log_8\ep^{-1}$ values of $r \in [\ep^2 , \ep] \cap \{8^{-k}\}_{k\in\BB N}$ for which
\eqb \label{eqn-outline-zero} 
\lim_{\delta \rta 0} \BB P\left[\text{$\exists$ a ``regular" pair $\BB z , \BB w \in \ol B_r(0)$ s.t.\ $\wt D_h(\BB z,\BB w) \geq  (\Cupper - \delta) D_h(\BB z , \BB w)$} \right] = 0  ,
\eqe 
uniformly over the choices of $\ep$ and $r$. 
\end{prop} 

Proposition~\ref{prop-outline-zero} is incompatible with Proposition~\ref{prop-outline-pos} since the parameter $p$ in Proposition~\ref{prop-outline-pos} does not depend on $\Cmed$. We thus obtain a contradiction to the assumption that $\Clower < \Cupper$, so we conclude that $\Clower = \Cupper$ and hence Theorem~\ref{thm-weak-uniqueness} holds.

\subsubsection{Section~\ref{sec-construction}: constructing events and bump functions}

In Section~\ref{sec-construction}, we will construct the events $\Er_{z,r}$ and the bump functions $\fr_{z,r}$ described just before Proposition~\ref{prop-outline-counting}. 
This part of the argument has some similarity to~\cite[Section 5]{gm-uniqueness}, which gives a roughly similar construction in the subcritical case. But, the details are very different. The main reason for this is as follows. 

Recall that we want to force a $D_{h-\fr_{z,r}}$-geodesic $P'$ to get $D_{h-\fr_{z,r}}$-close to each of $u$ and $v$, where $u,v$ are non-singular points in the support of $\fr_{z,r}$ such that $\wt D_h(u,v) \leq \Cmid_0 D_h(u,v)$. 
We will do this in two steps: first we force $P'$ to get Euclidean-close to each of $u$ and $v$, then we force $P'$ to get $D_{h-\fr_{z,r}}$-close to each of $u$ and $v$. 
In the subcritical phase, the metric $D_h$ is Euclidean-continuous, so the second step is straightforward. However, this is not the case in the supercritical phase, so a substantial amount of work is needed to force $P'$ to get $D_{h-\fr_{z,r}}$-close to each of $u$ and $v$. 
Because of this, we will define the events $\Er_{z,r}$ in a significantly different way as compared to~\cite{gm-uniqueness}. 
We refer to Section~\ref{sec-construction-outline} for a more detailed outline.

\section{Preliminaries}
\label{sec-prelim}

In this subsection, we first establish some standard notational conventions (Section~\ref{sec-notation}). 
We then record several lemmas about a weak LQG metric $D_h$ which are either proven elsewhere (i.e., in~\cite{pfeffer-supercritical-lqg,dg-confluence}) or are straightforward consequences of statements which are proven elsewhere. The reader may wish to skim this section on a first read and refer back to the various lemmas as needed.
 
\subsection{Notational conventions}
\label{sec-notation}

\noindent
We write $\BB N = \{1,2,3,\dots\}$ and $\BB N_0 = \BB N \cup \{0\}$. 
\medskip

\noindent
For $a < b$, we define the discrete interval $[a,b]_{\BB Z}:= [a,b]\cap\BB Z$. 
\medskip

\noindent
If $f  :(0,\infty) \rta \BB R$ and $g : (0,\infty) \rta (0,\infty)$, we say that $f(\ep) = O_\ep(g(\ep))$ (resp.\ $f(\ep) = o_\ep(g(\ep))$) as $\ep\rta 0$ if $f(\ep)/g(\ep)$ remains bounded (resp.\ tends to zero) as $\ep\rta 0$. We similarly define $O(\cdot)$ and $o(\cdot)$ errors as a parameter goes to infinity. 
\medskip

\noindent
Let $\{E^\ep\}_{\ep>0}$ be a one-parameter family of events. We say that $E^\ep$ occurs with
\begin{itemize}
\item \emph{polynomially high probability} as $\ep\rta 0$ if there is a $\mu > 0$ (independent from $\ep$ and possibly from other parameters of interest) such that  $\BB P[E^\ep] \geq 1 - O_\ep(\ep^\mu)$. 
\item \emph{superpolynomially high probability} as $\ep\rta 0$ if $\BB P[E^\ep] \geq 1 - O_\ep(\ep^\mu)$ for every $\mu >0$.  
\end{itemize} 
\medskip

\noindent
For $z\in\BB C$ and $r>0$, we write $B_r(z)$ for the open Euclidean ball of radius $r$ centered at $z$. More generally, for $X\subset \BB C$ we write $B_r(X) = \bigcup_{z\in X} B_r(z)$. We also define the open annulus
\eqb \label{eqn-annulus-def}
\BB A_{r_1,r_2}(z) := B_{r_2}(z) \setminus \ol{B_{r_1}(z)} ,\quad\forall 0 < r_r < r_2 < \infty .
\eqe 
\medskip

\noindent
Topological concepts such as ``open", ``closed", ``boundary", etc., are always defined with respect to the Euclidean topology unless otherwise stated.  
For $X\subset \BB C$, we write $\ol X$ for its Euclidean closure and $\bdy X$ for its Euclidean boundary.
\medskip

\noindent
We will typically use the symbols $r$ and $\BB r$ for Euclidean radii. Many of our estimates for weak LQG metrics are required to be uniform over different values of $r$ (or $\BB r$). The reason why we need to include this condition is that we only have tightness across scales (Axiom~\refcoord) instead of exact scale invariance (Axiom~\ref{item-metric-coord0}), so estimates are not automatically uniform across different Euclidean scales.

\subsection{Some remarks on internal metrics}
\label{sec-closed}

Throughout the rest of this section, we let $h$ be a whole-plane GFF and $D_h$ be a weak LQG metric as in Definition~\ref{def-metric}.

Let $X\subset \BB C$ (not necessarily open or closed) and recall from Definition~\ref{def-metric-properties} that $D_h(\cdot,\cdot;X)$ is the $D_h$-internal metric on $X$, which is a metric on $X$ except that it is allowed to take on infinite values.
It is easy to check (see, e.g.,~\cite[Proposition 2.3.12]{bbi-metric-geometry}) that the $D_h(\cdot,\cdot;X)$-length of any $D_h$-rectifiable path contained in $X$ (and hence also every $D_h(\cdot,\cdot;X)$-rectifiable path) is the same as its $D_h$-length. 

The notion of a $D_h(\cdot,\cdot;X)$-geodesic between points of $X$ is well-defined by Definition~\ref{def-metric-properties}: it is simply a path in $X$ whose $D_h$-length is the same as the $D_h(\cdot,\cdot;X)$-distance between its endpoints, provided this distance is finite. 
Such a geodesic may not exist for every pair of points in $X$.  
However, such geodesics exist for some pairs of points: for example, if $z,w \in X$ and there is a $D_h$-geodesic $P$ from $z$ to $w$ which is contained in $X$, then $P$ is a $D_h(\cdot,\cdot;X)$-geodesic. 

We will most often consider internal metrics on open sets (which appear in the locality assumption Axiom~\ref{item-metric-local} for $D_h$).
But, we will sometimes also have occasion to consider internal metrics on the closures of open sets. 
Recall that for an open set $U\subset \BB C$, $h|_U$ is the random distribution on $U$ obtained by restricting the distributional pairing $f\mapsto (h,f)$ to functions which are supported on $U$. 
Following, e.g.,~\cite[Section 3.3]{ss-contour}, for a closed set $K\subset \BB C$, we define
\eqb \label{eqn-closed-restrict}
\sigma\left( h|_K \right) := \bigcap_{\ep > 0} \sigma\left( h|_{B_\ep(K)} \right)
\eqe
where $B_\ep(K)$ is the Euclidean $\ep$-neighborhood of $K$. 

We say that a random variable is a.s.\ determined by $h|_K$ if it is a.s.\ equal to a random variable which is measurable with respect to $\sigma(h|_K)$. 
Similarly, we say that a random variable is a.s.\ determined by $h|_K$, viewed modulo additive constant, if it is a.s.\ equal to a random variable which is measurable with respect to $ \sigma( (h+c)|_K)$ for any possibly random $c \in \BB R$. 

The metric $D_h(\cdot,\cdot;K)$ is equal to the internal metric of $D_h(\cdot,\cdot;B_\ep(K))$ on $K$ for any $\ep > 0$. 
So, by locality (Axiom~\ref{item-metric-local}) and~\eqref{eqn-closed-restrict}, the metric $D_h(\cdot,\cdot;K)$ is measurable with respect to $\sigma(h|_K)$.

\subsection{Independence for the GFF}
\label{sec-prelim-ind}

The following lemma is a consequence of the fact that the restrictions of the GFF to disjoint concentric annuli, viewed modulo additive constant, are nearly independent. See~\cite[Lemma 3.1]{local-metrics} for a slightly more general statement.

\begin{lem}[\!\!\cite{local-metrics}] \label{lem-annulus-iterate}
Fix $0 < s_1<s_2 < 1$. Let $\{r_k\}_{k\in\BB N}$ be a decreasing sequence of positive numbers such that $r_{k+1} / r_k \leq s_1$ for each $k\in\BB N$ and let $\{E_{r_k} \}_{k\in\BB N}$ be events such that $E_{r_k} \in \sigma\left( (h-h_{r_k}(0)) |_{\BB A_{s_1 r_k , s_2 r_k}(0)  } \right)$ for each $k\in\BB N$. 
For $K\in\BB N$, let $N(K)$ be the number of $k\in [1,K]_{\BB Z}$ for which $E_{r_k}$ occurs. 
\begin{enumerate}
\item \label{item-annulus-iterate} For each $a > 0$ and each $b\in (0,1)$, there exists $p = p(a,b,s_1,s_2) \in (0,1)$ and $c = c(a,b,s_1,s_2) > 0$ (independent of the particular choice of $\{r_k\}$ and $\{E_{r_k}\}$) such that if  
\eqb \label{eqn-annulus-iterate-prob}
\BB P\left[ E_{r_k}  \right] \geq p , \quad \forall k\in\BB N  ,
\eqe 
then 
\eqb \label{eqn-annulus-iterate}
\BB P\left[ N(K)  < b K\right] \leq c e^{-a K} ,\quad\forall K \in \BB N. 
\eqe 
\item \label{item-annulus-iterate-pos} For each $p\in (0,1)$, there exists $a = a(p,s_1,s_2) > 0$, $b = b(p,s_1,s_2) \in (0,1)$, and $c = c(p,s_1,s_2) > 0$  (independent of the particular choice of $\{r_k\}$ and $\{E_{r_k}\}$) such that if~\eqref{eqn-annulus-iterate-prob} holds, then~\eqref{eqn-annulus-iterate} holds. 
\end{enumerate}
\end{lem}

Lemma~\ref{lem-annulus-iterate} still applies if we require that $E_{r_k} \in \sigma\left( (h-h_{r_k}(0)) |_{\ol{\BB A}_{s_1 r_k , s_2 r_k}(0)  } \right)$ (i.e., we consider a closed annulus rather than an open annulus). This is an immediate consequence of the definition of the $\sigma$-algebra generated by the restriction of $h$ to a closed set~\eqref{eqn-closed-restrict}. We will use this fact without comment several times in what follows. 

For the proof of Lemma~\ref{lem-main-extra} below, we will need a minor variant of Lemma~\ref{lem-annulus-iterate} where we do not require that the annuli are concentric.

\begin{lem}  \label{lem-annulus-offcenter}
Fix $0 < s_1< s_2 < 1$ and $s_0 \in (0, \min\{s_1 , 1-s_2\})$. Let $\{r_k\}_{k\in\BB N}$ be a decreasing sequence of positive real numbers and let $\{z_k\}_{k\in\BB N}$ be a sequence of points in $\BB C$ such that 
\eqb \label{eqn-offcenter-pt}
r_{k+1} / r_k \leq s_1 -s_0 \quad \text{and} \quad |z_k| \leq s_0 r_k , \quad \forall k \in \BB N .
\eqe 
Let $\{E_{r_k}(z_k) \}_{k\in\BB N}$ be events such that for each $k\in \BB N$, the event $E_{r_k}(z_k)$ is a.s.\ determined by $h|_{\ol{\BB A}_{s_1 r_k , s_2 r_k}(z_k)}$, viewed modulo additive constant. 
For $K\in\BB N$, let $N(K)$ be the number of $k\in [1,K]_{\BB Z}$ for which $E_{r_k}(z_k)$ occurs. 
\begin{enumerate}
\item \label{item-annulus-offcenter} For each $a > 0$ and each $b\in (0,1)$, there exists $p = p(a,b,s_0,s_1,s_2) \in (0,1)$ and $c = c(a,b,s_0,s_1,s_2) > 0$ (independent of the particular choice of $\{r_k\}$, $\{z_k\}$, and $\{E_{r_k}(z_k)\}$) such that if  
\eqb \label{eqn-annulus-offcenter-prob}
\BB P\left[ E_{r_k}(z_k)  \right] \geq p , \quad \forall k\in\BB N  ,
\eqe 
then 
\eqb \label{eqn-annulus-offcenter}
\BB P\left[ N(K)  < b K\right] \leq c e^{-a K} ,\quad\forall K \in \BB N. 
\eqe 
\item \label{item-annulus-offcenter-pos} For each $p\in (0,1)$, there exists $a = a(p,s_0,s_1,s_2) > 0$, $b = b(p,s_0,s_1,s_2) \in (0,1)$, and $c = c(p,s_0,s_1,s_2) > 0$  (independent of the particular choice of $\{r_k\}$, $\{z_k\}$, and $\{E_{r_k}(z_k)\}$) such that if~\eqref{eqn-annulus-offcenter-prob} holds, then~\eqref{eqn-annulus-offcenter} holds. 
\end{enumerate}
\end{lem}
\begin{proof}
Since $|z_k| \leq s_0 r_k$,  
\eqbn
\BB A_{s_1 r_k , s_2 r_k}(z_k) \subset \BB A_{(s_1-s_0)r_k , (s_2+s_0)r_k}(0) .
\eqen
Hence $E_{r_k}(z_k)$ is a.s.\ determined by $h|_{\ol{\BB A}_{(s_1-s_0)r_k , (s_2+s_0)r_k}(0)}$, viewed modulo additive constant. 
Since $0 < s_1 - s_0 < s_2  + s_0  < 1$ and by~\eqref{eqn-offcenter-pt}, we can apply Lemma~\ref{lem-annulus-iterate} with $s_1-s_0$ in place of $s_1$ and $s_2 + s_0$ in place of $s_2$ to obtain the lemma statement.
\end{proof}

We will also need an estimate which comes from the fact that the restrictions of the GFF to small disjoint Euclidean balls are nearly independent. 
See~\cite[Lemma 2.7]{gm-uniqueness} for a proof. 
 
\begin{lem}[\!\!\cite{gm-uniqueness}]  \label{lem-spatial-ind}
Let $h$ be a whole-plane GFF and fix $s > 0$. 
Let $n\in\BB N$ and let $\mcl Z$ be a collection of $\#\mcl Z = n$ points in $\BB C$ such that $|z-w| \geq 2(1+s)$ for each distinct $z,w\in\mcl Z$. 
For $z\in\mcl Z$, let $E_z$ be an event which is determined by $(h - h_{1+s}(z)) |_{B_1(z)}$. 
For each $p , q \in (0,1)$, there exists $n_* = n_*(s,p,q) \in \BB N$ such that if $\BB P[E_z] \geq p$ for each $z\in\mcl Z$, then
\eqbn
\BB P\left[ \bigcup_{z\in\mcl Z} E_z \right] \geq q ,\quad \forall n \geq n_* .
\eqen 
\end{lem}

\subsection{Basic facts about weak LQG metrics}
\label{sec-prelim-basic}

In this subsection, we will record some facts about our weak LQG metric $D_h$ which are mostly proven elsewhere and which will be used frequently in what follows. Similar results are proven in the subcritical case in~\cite{lqg-metric-estimates,mq-geodesics}.

\begin{remark} \label{remark-c_r}
Many of the estimates in~\cite{pfeffer-supercritical-lqg,dg-confluence} involve ``scaling constants" $\frk c_r$ for $r > 0$. 
It was shown in~\cite[Theorem 1.9]{dg-polylog} that one can take $\frk c_r = r^{\xi Q}$. 
We will use this fact without comment whenever we cite results from~\cite{pfeffer-supercritical-lqg,dg-confluence}. 
\end{remark}

It was shown in~\cite[Lemma 3.1]{pfeffer-supercritical-lqg} that one has the following stronger version of Axiom~\refcoord. 

\begin{lem}[\!\!\cite{pfeffer-supercritical-lqg}] \label{lem-set-tightness}
Let $U\subset \BB C$ be open and let $K_1,K_2\subset U$ be two disjoint, deterministic compact sets (allowed to be singletons).  
The re-scaled internal distances $r^{-1} e^{-\xi h_r(0)} D_h(r K_1,r K_2; r U)$ and their reciporicals as $r$ varies are tight (recall the notation from Definition~\ref{def-metric-properties}). 
\end{lem}

The following proposition, which is~\cite[Proposition 1.8]{pfeffer-supercritical-lqg}, is a more quantitative version of Lemma~\ref{lem-set-tightness} in the case when $K_1,K_2$ are connected and are not singletons. 

\begin{lem}[\!\!\cite{pfeffer-supercritical-lqg}] \label{lem-two-set-dist}
Let $U \subset \BB C$ be an open set (possibly all of $\BB C$) and let $K_1,K_2\subset U$ be two disjoint, deterministic, connected, compact sets which are not singletons.  
For each $r  >0$, it holds with superpolynomially high probability as $R\rta \infty$, at a rate which is uniform in the choice of $r$, that 
\eqbn
 R^{-1}  r^{\xi Q} e^{\xi h_r(0)} \leq D_h(r K_1, r K_2 ; r U) \leq R r^{\xi Q} e^{\xi h_r(0)} .  
\eqen
\end{lem}

Suppose that $A\subset\BB C$ is a deterministic bounded open set which has the topology of a Euclidean annulus and whose inner and outer boundaries are not singletons. 
Recall the notation for $D_h$-distance across and around Euclidean annuli from Definition~\ref{def-around-across}.
It is easy to see from Lemma~\ref{lem-two-set-dist} that with superpolynomially high probability as $R\rta\infty$, uniformly in the choice of $r$,  
\eqbn
 R^{-1}  r^{\xi Q} e^{\xi h_r(0)} \leq D_h(\text{around $A$}) \leq R r^{\xi Q} e^{\xi h_r(0)} ,  
\eqen
and the same is true for $D_h(\text{across $A$})$. 

Recall from Lemma~\ref{lem-geodesics} that a.s.\ any two non-singular points $z,w$ for $D_h$ can be joined by a $D_h$-geodesic, i.e., a path of $D_h$-length $D_h(z,w)$. 
In the subcritical case, it was shown in~\cite[Theorem 1.2]{mq-geodesics} that for a \emph{fixed} choice of $z$ and $w$, a.s.\ this geodesic is unique (see also~\cite[Lemma 4.2]{ddg-metric-survey} for a simplified proof). The same proof also works in the critical and supercritical cases. 
We will need a slightly more general statement than the uniqueness of geodesics between fixed points. 
For two sets $K_1,K_2\subset \BB C$, a \emph{$D_h$-geodesic} from $K_1$ to $K_2$ is a path from a point of $K_1$ to a point of $K_2$ such that
\eqb \label{eqn-set-geo}
\op{len}(P ; D_h) = D_h(K_1,K_2) := \inf_{z\in K_1 , w\in K_2} D_h(z,w) .
\eqe

\begin{lem} \label{lem-geo-unique}
Let $K_1,K_2\subset \BB C$ be deterministic disjoint Euclidean-compact sets.
Almost surely, there is a unique $D_h$-geodesic from $K_1$ to $K_2$.
\end{lem}
\begin{proof}
For existence, choose sequences of points $u_n \in K_1$ and $v_n \in K_2$ such that $\lim_{n\rta\infty} D_h(u_n,v_n) = D_h(K_1,K_2)$. 
Since $K_1$ and $K_2$ are Euclidean-compact, after possibly passing to a subsequence we can find $u\in K_1$ and $v\in K_2$ such that $|u_n - u| \rta 0$ and $|v_n-v| \rta 0$.
By the lower semicontinuity of $D_h$,
\eqbn
 D_h(u,v) \leq \liminf_{n\rta\infty} D_h(u_n,v_n) = D_h(K_1,K_2) .
\eqen
Hence $D_h(u,v) = D_h(K_1,K_2)$ and a $D_h$-geodesic from $u$ to $v$ (which exists by Lemma~\ref{lem-geodesics}) is also a $D_h$-geodesic from $K_1$ to $K_2$.

The uniqueness of the $D_h$-geodesic from $K_1$ to $K_2$ follows from the same argument as in the case when $K_1$ and $K_2$ are singletons, see~\cite[Section 3]{mq-geodesics} or~\cite[Lemma 4.2]{ddg-metric-survey}.
\end{proof}

\subsection{Estimates for distances in disks and annuli}
\label{sec-prelim-dist}

In this subsection, we will prove some basic estimates for $D_h$ which are straightforward consequences of the concentration bounds for LQG distances established in~\cite{pfeffer-supercritical-lqg}. 
We begin with a uniform comparison of distances around and across Euclidean annuli with different center points and radii.

\begin{lem} \label{lem-annulus-union}
Fix $\zeta >0$. 
Let $U\subset \BB C$ be a bounded open set and let $b > a > 0$ and $d > c > 0$. 
For each $\BB r > 0$, it holds with superpolynomially high probability as $\delta_0 \rta 0$ (at a rate which depends on $\zeta,U,a,b,c,d$ and the law of $D_h$, but is uniform in $\BB r$) that 
\eqb \label{eqn-annulus-union} 
D_h\left(\text{around $\BB A_{a \delta \BB r , b \delta \BB r}(z)$}\right) 
\leq \delta^{-\zeta} D_h\left(\text{across $\BB A_{c \delta \BB r , d \delta \BB r}(z)$}\right) ,\quad \forall z \in\BB r U ,\quad \forall \delta \in (0,\delta_0] .
\eqe
\end{lem}
\begin{proof}
Basically, this follows from Lemma~\ref{lem-two-set-dist} and a union bound. A little care is needed to discretize things so that we only have to take a union bound over polynomially many events.

Fix $a_1,a_2 , b_1,b_2 > 0$ and $c_1,c_2 ,d_1,d_2>  0$ such that 
\eqbn
 a < a_2 < a_1 < b_1 < b_2 < b \quad \text{and} \quad  c < c_2 < c_1 < d_1 < d_2 < d .
\eqen
By Lemma~\ref{lem-two-set-dist}, for each $z\in\BB C$ it holds with superpolynomially high probability as $\delta \rta 0$ (at a rate depending only on $\zeta,a_1,b_1,c_1,d_1,$ and the law of $D_h$) that
\allb \label{eqn-annulus-union-start}
D_h\left(\text{around $\BB A_{ a_1  \delta \BB r , b_1 \delta \BB r}(z)$}\right) 
&\leq \delta^{\xi Q - \zeta/2} \BB r^{\xi Q} e^{\xi h_{\delta \BB r}(z)} \quad \text{and} \notag\\
D_h\left(\text{across $\BB A_{ c_1 \delta \BB r , d_1 \delta \BB r}(z)$}\right) 
&\geq \delta^{\xi Q + \zeta/2} \BB r^{\xi Q} e^{\xi h_{\delta \BB r}(z)}  .
\alle
Let $s > 0$ be much smaller than $\min\{ a_1-a_2 , b_2 -b_1  , c_1-c_d , d_2-d_1  \}$. By a union bound, it holds with superpolynomially high probability as $\delta \rta 0$ that the bound~\eqref{eqn-annulus-union-start} holds for all $z \in (s\delta\BB r \BB Z^2) \cap B_{\BB r}(\BB r U)$. 

For each $z\in \BB r U$, there exists $z' \in (s\delta\BB r \BB Z^2) \cap B_{\BB r}(\BB r U)$ such that 
\eqbn
\BB A_{a_1 \delta \BB r , b_1 \delta\BB r}(z') \subset \BB A_{a_2\delta\BB r , b_2\delta\BB r}(z) 
\quad \text{and} \quad
\BB A_{c_1 \delta \BB r , d_1 \delta\BB r}(z') \subset \BB A_{c_2\delta\BB r , d_2\delta\BB r}(z) 
\eqen
For this choice of $z'$,  
\alb
D_h\left(\text{around $\BB A_{a_2 \delta \BB r , b_2 \delta \BB r}(z)$}\right) 
&\leq D_h\left(\text{around $\BB A_{a_1 \delta \BB r , b_1 \delta \BB r}(z')$}\right) 
\quad \text{and} \notag\\
 D_h\left(\text{across $\BB A_{c_2 \delta \BB r , d_2 \delta \BB r}(z)$}\right) 
&\geq D_h\left(\text{across $\BB A_{c_1 \delta \BB r , d_1 \delta \BB r}(z')$}\right)  .
\ale
By~\eqref{eqn-annulus-union-start} with $z'$ in place of $z$, we infer that with superpolynomially high probability as $\delta \rta 0$,
\eqb \label{eqn-annulus-union-onescale}
D_h\left(\text{around $\BB A_{a_2 \delta \BB r , b_2 \delta \BB r}(z)$}\right) 
\leq \delta^{-\zeta} D_h\left(\text{across $\BB A_{c_2 \delta \BB r , c_2 \delta \BB r}(z)$}\right) ,\quad \forall z \in\BB r U. 
\eqe

To upgrade to an estimate which holds for all $\delta \in (0,\delta_0]$ simultaneously, let 
\eqbn
q \in \left( 1 , \left( \min\{  a_2 /a , b/b_2  , c_2/c,d/d_2 \}\right)^{1/100} \right) .
\eqen 
By a union bound over integer powers of $q$, we infer that with superpolynomially high probability as $\delta_0\rta 0$, the estimate~\eqref{eqn-annulus-union-onescale} holds
for all $\delta \in (0,\delta_0] \cap \{q^{-k} : k\in\BB N\}$. By our choice of $q$, for each $\delta \in (0,\delta_0]$, there exists $k \in \BB N$ such that $q^{-k} \in (0,\delta_0]$ and for each $z\in\BB C$,
\eqbn
\BB A_{a_2 q^{-k} \BB r , b_2 q^{-k} \BB r}(z) \subset \BB A_{a  \delta \BB r , b  \delta \BB r}(z) 
\quad \text{and} \quad 
\BB A_{c_2 q^{-k} \BB r , d_2 q^{-k} \BB r}(z) \subset \BB A_{c \delta \BB r , d \delta \BB r}(z) .
\eqen
Hence~\eqref{eqn-annulus-union-onescale} for $\delta$ follows from~\eqref{eqn-annulus-union-onescale} with $q^{-k}$ in place of $\delta$.
\end{proof}

Our next estimate gives a moment bound for the LQG distance from the center point of a closed disk to a point on its boundary, along paths which are contained in the disk.

\begin{lem} \label{lem-ball-bdy-moment}
For each $p \in (0,2Q/\xi)$, there exists $C_p  > 0$, depending only on $p$ and the law of $D_h$, such that
\eqb \label{eqn-ball-bdy-moment}
\BB E\left[ \left(r^{-\xi Q} e^{-\xi h_r(0)} D_h\left(w , 0 ; \ol B_r(0) \right)  \right)^p \right] \leq C_p ,\quad\forall w \in \bdy B_r(0) .
\eqe
\end{lem}
\begin{proof}
Fix $w\in \bdy B_r(0)$. All of our estimates are required to be uniform in the choice of $w$.  The idea of the proof is to string together countably many $D_h$-rectifiable loops centered at points on the segment $[0,w]$, with geometric Euclidean sizes. 

For $\ep \in (0,r)$, define 
\eqbn
w_\ep := \left( 1 - \frac{\ep}{r} \right)w \quad \text{and} \quad A_\ep := \BB A_{\ep/2,\ep}\left( w_\ep \right) 
\eqen
and note that $A_\ep \subset B_r(0)$.  

By Lemma~\ref{lem-two-set-dist}, for each $q > 0$, 
\eqb \label{eqn-square-bdy-annulus}
\BB E\left[ \left(\ep^{-\xi Q} e^{-\xi h_\ep( w_\ep  )} D_h\left(\text{around $A_\ep$} \right) \right)^q \right] \preceq 1 ,\quad \forall \ep > 0 ,
\eqe 
with the implicit constant depending only on $q$ and the law of $D_h$. By H\"older's inequality, for each $p > 0$ and each $q>1$, 
\allb \label{eqn-square-bdy-annulus'}
&\BB E\left[ \left(r^{-\xi Q} e^{-\xi h_r(0 )} D_h\left(\text{around $A_\ep$} \right) \right)^p \right]  \notag\\
&\qquad \qquad \leq \left( \frac{\ep}{r} \right)^{\xi Q p}  
\BB E\left[ \left( \ep^{-\xi Q} e^{-\xi h_\ep(w_\ep  )}  D_h\left(\text{around $A_\ep$} \right) \right)^{\frac{q p}{1-q} } \right]^{1 - 1/q} \notag\\ 
&\qquad\qquad\qquad\qquad\qquad \times \BB E\left[ e^{ q p \xi (h_\ep(w_\ep) - h_r(0))} \right]^{ 1/q} \notag \\
&\qquad \qquad \preceq \left( \frac{\ep}{r} \right)^{\xi Q p}  
\BB E\left[ e^{ q p \xi (h_\ep(w_\ep ) - h_r(0))} \right]^{ 1/q}  ,
\alle
where in the last line we used~\eqref{eqn-square-bdy-annulus}. The random variable $h_\ep(w_\ep ) - h_r(0)$ is centered Gaussian with variance at most $\log(r/\ep)$ plus a universal constant. We therefore infer from~\eqref{eqn-square-bdy-annulus'} that for each $p > 0$ and each $q > 1$, 
\allb \label{eqn-square-bdy-annulus-end}
 \BB E\left[ \left(r^{-\xi Q} e^{-\xi h_r(0)} D_h\left(\text{around $A_\ep$} \right) \right)^p \right]
 \preceq \left( \frac{\ep}{r} \right)^{\xi Q p - q  p^2 \xi^2 /2 }  
\alle
with the implicit constant depending only on $p,q$. 

Let 
\eqbn
w_\ep' := \frac{\ep}{r} w \quad \text{and} \quad A_\ep' := \BB A_{\ep/2,\ep}(w_\ep') , 
\eqen
which is contained in $B_r(0)$ for $\ep \in (0,r/2]$. 
Via a similar argument to the one leading to~\eqref{eqn-square-bdy-annulus-end}, we also have that for each $p > 0$ and each $q > 1$, 
\allb \label{eqn-square-bdy-annulus-end'}
 \BB E\left[ \left(r^{-\xi Q} e^{-\xi h_r(0)} D_h\left(\text{around $A_\ep'$} \right) \right)^p \right]
 \preceq \left( \frac{\ep}{r} \right)^{\xi Q p - q  p^2 \xi^2 /2 }  .
\alle

For $k\in \BB N$, let $\ep_k := 2^{-k} r$. Suppose that $\pi_k$ is a path in $A_{\ep_k}$ which disconnects the inner and outer boundaries and $\pi_k'$ is a path in $A_{\ep_k}'$ which disconnects the inner and outer boundaries of $A_{\ep_k}'$. Then the union of the paths $\pi_k$ and $\pi_k'$ for $k\in\BB N$ is connected and contained in $B_r(0)$ and its closure contains both $0$ and $w$. From this, we see that the union of these paths and $\{0,w\}$ contains a path from $0$ to $w$ which is contained in $\ol B_r(0)$. 
Hence
\eqb \label{eqn-square-bdy-sum}
D_h\left(w , 0 ; \ol B_r(0) \right)
\leq \sum_{k=0}^\infty  D_h\left(\text{around $A_{\ep_k} $} \right)
+\sum_{k=0}^\infty  D_h\left(\text{around $A_{\ep_k}'$} \right) .
\eqe 
 
Assume now that $p \in (0,\min\{1 , 2Q/\xi\})$. Since the function $x\mapsto x^p$ is concave, hence subadditive, we can take $p$th moments of both sides of~\eqref{eqn-square-bdy-sum}, then apply~\eqref{eqn-square-bdy-annulus-end} and~\eqref{eqn-square-bdy-annulus-end'}, to get
\allb
&\BB E\left[ \left(r^{-\xi Q} e^{-\xi h_r(0)} D_h\left(w , 0 ; \ol B_r(0) \right) \right)^p \right] \notag\\
&\qquad\qquad \leq \sum_{k=0}^\infty  \BB E\left[ \left(r^{-\xi Q} e^{-\xi h_r(0)} D_h\left(\text{around $A_{\ep_k}$} \right) \right)^p \right] \notag\\
&\qquad\qquad \qquad +\sum_{k=0}^\infty   \BB E\left[ \left(r^{-\xi Q} e^{-\xi h_r(0)} D_h\left(\text{around $A_{\ep_k}'$} \right) \right)^p \right] \notag\\
&\qquad\qquad\preceq  \sum_{k=0}^\infty   \left( \frac{\ep_k}{r} \right)^{\xi Q p - q  p^2 \xi^2 /2 }  \notag\\
&\qquad\qquad\preceq \sum_{k=0}^\infty 2^{-k (\xi Q p - q  p^2 \xi^2 /2)} .
\alle
Since $p < 2Q/\xi$, if $q > 1$ is sufficiently close to 1, we have $\xi Q p - q p^2\xi^2/2 > 0$. Hence this last sum is finite. This gives~\eqref{eqn-ball-bdy-moment} for $p < 1$. For $p\geq 1$, we obtain~\eqref{eqn-ball-bdy-moment} via the same argument, but with the triangle inequality for the $L^p$ norm used in place of the subadditivity of $p\mapsto x^p$. 
\end{proof}

Using Lemma~\ref{lem-ball-bdy-moment} and Markov's inequality, we obtain the following estimate, which says that with high probability ``most" points on a circle are not too LQG-far from the center point. Note that (unlike for subcritical LQG) we cannot say that this is the case for \emph{all} points on the circle, e.g., because there could be singular points on the circle.

\begin{lem} \label{lem-ball-bdy-union}
For each $R> 1$, 
\eqb \label{eqn-ball-bdy-union}
\BB E\left[ \left| \left\{ w \in \bdy B_r(0) :    D_h\left(w , 0 ; \ol B_r(0) \right)  > R r^{\xi Q} e^{\xi h_r(0)} \right\} \right|      \right] \leq  R^{-2Q/\xi + o_R(1)} r ,
\eqe
where $|\cdot|$ denotes one-dimensional Lebesgue measure and the rate of convergence of the $o_R(1)$ depends only on the law of $D_h$. 
\end{lem}
\begin{proof}
This follows from Lemma~\ref{lem-ball-bdy-moment} and Markov's inequality.
\end{proof}

We will also need a lemma to ensure that all of the $D_h$-geodesics between points in a specified Euclidean-compact set are contained in a larger compact set. 

\newcommand{\compExp}{\mu}

\begin{lem} \label{lem-geo-compact} 
There exists $\compExp > 0$, depending only on the law of $D_h$, such that the following is true. 
Let $K\subset \BB C$ be compact. 
For each $\BB r > 0$, it holds with probability $1-O_R(R^{-\compExp})$ as $R\rta\infty$ (at a rate depending only on $K$ and the law of $D_h$) that each $D_h$-geodesic between two points of $\BB r K$ is contained in $B_{R\BB r}(0)$. 
\end{lem}
\begin{proof}
Fix $\BB r >0$ and for $s >0$, let 
\eqbn
E_s := \left\{ D_h\left(\text{around $\BB A_{s \BB r , 2 s \BB r    }(0)$}\right) < D_h\left(\text{across $\BB A_{2 s \BB r , 3 s \BB r }(0)$} \right) \right\} .
\eqen
Using tightness across scales (Axiom~\refcoord) and a basic absolute continuity argument (see, e.g., the proof of~\cite[Lemma 6.1]{gwynne-ball-bdy}), we can find a $p \in (0,1)$, depending only on the law of $D_h$, such that $\BB P[E_s] \geq p$ for all $s , \BB r > 0$. 

Let $\rho > 0$ be chosen so that $K \subset B_{\rho}(0)$. 
By assertion~\ref{item-annulus-iterate-pos} of Lemma~\ref{lem-annulus-iterate} (applied to logarithmically many radii $r_k \in [\rho \BB r , R \BB r/3]$), we can find $\compExp  >0$ as in the lemma statement such that for with probability $1-O_R(R^{-\compExp})$, there exists $s \in [\rho , R/3]$ such that $E_s$ occurs. 

On the other hand, it is easily seen that if $E_s$ occurs, then no $D_h$-geodesic $P$ between two points of $B_{s\BB r}(0)$ can exit $B_{3s\BB r}(0)$. Indeed, otherwise we could replace a segment of $P$ by a segment of a path in $\BB A_{s\BB r  , 2 s\BB r}(0)$ which disconnects the inner and outer boundaries to get a path with the same endpoints as $P$ but strictly shorter $D_h$-length than $P$. 
\end{proof}

\subsection{Regularity of geodesics}
\label{sec-prelim-geo}

The following lemma is (almost) a re-statement of~\cite[Corollary 3.7]{dg-confluence}. Roughly speaking, the lemma states that every point in an LQG geodesic is surrounded by a loop of small Euclidean diameter whose $D_h$-length is much shorter than the $D_h$-length of the geodesic. A similar lemma also appears in~\cite[Section 2.4]{pfeffer-supercritical-lqg}. 

\begin{lem} \label{lem-hit-ball}
For each $\chi \in (0,1)$, there exists $\geoExp  > 0$, depending only on $\chi$ and the law of $D_h$, such that for each Euclidean-bounded open set $U\subset \BB C$ and each $\BB r > 0$,  
it holds with polynomially high probability as $\ep_0 \rta 0$, uniformly over the choice of $\BB r$, that the following is true for each $\ep \in (0,\ep_0]$. 
Suppose $z\in \BB r U$, $x,y\in\BB C \setminus B_{\ep^{\chi} \BB r}(z)$, and $s>0$ such that there is a $D_h$-geodesic $P$ from $x$ to $y$ with $P(s) \in B_{\ep \BB r}(z)$. Then  
\eqb \label{eqn-hit-ball}
D_h\left( \text{around $\BB A_{ \ep \BB r , \ep^{\chi} \BB r}(z)$} \right) 
\leq \ep^\geoExp s.
\eqe
\end{lem}
\begin{proof}
\cite[Corollary 3.7]{dg-confluence} shows that with polynomially high probability as $\ep_0 \rta 0$, the condition in the lemma statement holds for $\ep  =\ep_0$. 
The statement for all $\ep \in (0,\ep_0]$ follows from the statement for $\ep = \ep_0$ (applied with $\chi$ replaced by $\chi'$ slightly larger than $\chi$) together with a union bound over dyadic values of $\ep$. 
\end{proof}

As explained in~\cite{dg-confluence,pfeffer-supercritical-lqg}, Lemma~\ref{lem-hit-ball} functions as a substitute for the fact that in the supercritical case, $D_h$ is not locally H\"older continuous with respect to the Euclidean metric. It says that the $D_h$-distance around a small Euclidean annulus centered at a point on a $D_h$-geodesic is small. A path of near-minimal length around this annulus can be linked up with various other paths to get upper-bounds for $D_h$-distances in terms of Euclidean distances.

\begin{figure}[ht!]
\begin{center}
\includegraphics[width=.7\textwidth]{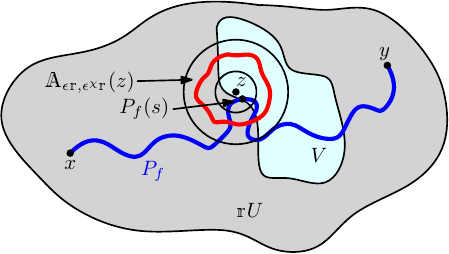} 
\caption{\label{fig-hit-ball} Illustration of the statement of Lemma~\ref{lem-hit-ball-phi} in the case where $s =  \inf\{t > 0 : P_f(t) \in V \} $ (which is the main case that we will use). The path $P_f$ is a $D_{h-f}(\cdot,\cdot; \BB r U)$-geodesic and the set $V$ is the support of $f$. 
The lemma gives us an upper bound for $D_h\left( \text{around $\BB A_{ \ep \BB r , \ep^\chi \BB r}(z)$} \right) $. 
}
\end{center}
\end{figure}

We will need the following generalization of Lemma~\ref{lem-hit-ball}, which follows from exactly the same proof. The lemma statement differs from Lemma~\ref{lem-hit-ball} in that we consider a $D_{h-f}(\cdot,\cdot; \BB r \ol U)$-geodesic, for a possibly random non-negative bump function $f$, instead of a $D_h$-geodesic (recall the discussion of geodesics for internal metrics from Section~\ref{sec-closed}). See Figure~\ref{fig-hit-ball} for an illustration of the lemma statement. 

\begin{lem} \label{lem-hit-ball-phi}
For each $\chi \in (0,1)$, there exists $\geoExp  > 0$ depending only on $\chi$ and the law of $D_h$, such that for each Euclidean-bounded open set $U\subset \BB C$ and each $\BB r > 0$,  
it holds with polynomially high probability as $\ep_0 \rta 0$, uniformly over the choice of $\BB r$, that the following is true for each $\ep \in (0,\ep_0]$. 
Let $V\subset \BB r U$ and let $f : \BB C\rta [0,\infty)$ be a non-negative continuous function which is identically zero outside of $ V$. 
Let $z\in \BB r [ U\setminus B_{\ep^\chi}(\bdy U)]$, $x,y\in (\BB r \ol U) \setminus ( V \cup B_{\ep^\chi \BB r}(z) )$, and $s>0$ such that there is a $D_{h-f}(\cdot,\cdot ; \BB r \ol U)$-geodesic $P_f$ from $x$ to $y$ with $P_f(s) \in B_{\ep \BB r}(z)$. Assume that
\eqb \label{eqn-hit-ball-time}
s \leq \inf\{t > 0 : P_f(t) \in V \} .
\eqe 
Then
\eqb \label{eqn-hit-ball-phi}
D_h\left( \text{around $\BB A_{ \ep \BB r , \ep^\chi \BB r}(z)$} \right) 
\leq \ep^\geoExp s.
\eqe
\end{lem}

The statement of Lemma~\ref{lem-hit-ball-phi} holds with polynomially high probability for all possible choices of $V,f,x,y,z,s,P_f$. In particular, these objects are allowed to be random and/or $\ep$-dependent. 
We also emphasize that the time $s$ in~\eqref{eqn-hit-ball-time} is allowed to be equal to $\inf\{t > 0 : P_f(t) \in V\}$, in which case $P_f(s) \in \bdy V$. 
In fact, this is the main setting in which we will apply Lemma~\ref{lem-hit-ball-phi}. 

In the setting of Lemma~\ref{lem-hit-ball-phi}, since $f$ is non-negative, we have $D_{h-f}(u,v ; \BB r \ol U) \leq D_h(u,v ; \BB r \ol U)$ for all $u,v\in \BB r U$. 
Furthermore, the condition~\eqref{eqn-hit-ball-time} implies that the $D_{h-f}$-length of $P_f|_{[0,s]}$ is the same as its $D_h$-length.   
These two facts allow us to apply the proof of Lemma~\ref{lem-hit-ball} (as given in~\cite[Section 3.2]{dg-confluence}) essentially verbatim to obtain Lemma~\ref{lem-hit-ball-phi}.

Out next lemma tells us that an LQG geodesic cannot trace a deterministic curve. Just like in Lemma~\ref{lem-hit-ball-phi}, we will consider not just a $D_h$-geodesic but a $D_{h-f}(\cdot,\cdot; \BB r \ol U)$-geodesic for a possible random continuous function $f$. 

\newcommand{\lebExp}{\nu}

\begin{lem} \label{lem-geodesic-leb}
For each $M > 0$, there exists $\lebExp > 0$, depending only on $M$ and the law of $D_h$, such that the following is true. 
Let $U\subset\BB C$ be a deterministic open set and let $\eta : [0,T] \rta U \setminus B_{\ep^{1/2}}(\bdy U) $ be a deterministic parametrized curve.
For each $\BB r > 0$, it holds with probability $1-O_\ep(\ep^\lebExp)$ as $\ep\rta 0$ (the implicit constant depends only on $M$ and the law of $D_h$) that the following is true. 
Let $f : \BB C\rta [-M,M]$ be a continuous function and let $P_f$ be a $D_{h-f}(\cdot,\cdot;\BB r \ol U)$-geodesic between two points of $ \BB r [ U \setminus B_{\ep^{1/2}  }(\eta) ]$. 
Then 
\eqb \label{eqn-geodesic-leb}
|\{t \in [0,T] : P_f \cap B_{\ep\BB r}(\BB r \eta(t)) \not=\emptyset \} | \leq \ep^\lebExp T ,
\eqe
where $|\cdot|$ denotes one-dimensional Lebesgue measure.  
\end{lem}

We emphasize that, as in Lemma~\ref{lem-hit-ball-phi}, the function $f$ and the geodesic $P_f$ in Lemma~\ref{lem-geodesic-leb} are allowed to be random and $\ep$-dependent (but $\eta$ is fixed). 

\begin{proof}[Proof of Lemma~\ref{lem-geodesic-leb}]
The idea of the proof is that (by Lemma~\ref{lem-annulus-iterate}) for a ``typical" time $t \in [0,T]$, there is a loop in $\BB A_{\ep \BB r, \ep^{1/2} \BB r}(\BB r \eta(t))$ which disconnects the inner and outer boundaries and whose $D_h$-length is much shorter than the $D_h$-distance from the loop to $B_{\ep\BB r}(\BB r \eta(t))$. The existence of such a loop prevents a $D_{h-f}$-geodesic from hitting $B_{\ep\BB r}(\BB r \eta(t))$. 

For $k\in\BB N$, let 
\eqbn
r_k := 4^k \ep\BB r .
\eqen
For $t\in [0,T]$, define the event 
\eqb \label{eqn-geodesic-leb-event}
 E_k(t) := \left\{ D_h\left(\text{around $\BB A_{2r_k , 3 r_k}(\BB r \eta(t))$} \right) \leq \frac12 e^{-2\xi M} D_h\left(\text{across $\BB A_{r_k , 2r_k}(\BB r \eta(t))$} \right) \right\} .
\eqe
By locality and Weyl scaling (Axioms~\ref{item-metric-local} and~\refcoord), the event $E_k(t)$ is a.s.\ determined by $h|_{\BB A_{r_k,3r_k}(\BB r \eta(t))}$, viewed modulo additive constant.
By adding a bump function to $h$ and using absolute continuity together with tightness across scales (see, e.g., the proof of~\cite[Lemma 6.1]{gwynne-ball-bdy}), we see that there exists $p  > 0$ (depending only on $M$ and the law of $D_h$) such that $\BB P[E_k(t)] \geq p$ for each $k\in\BB N$ and $t\in [0,T]$. Consequently, assertion~\ref{item-annulus-iterate-pos} of Lemma~\ref{lem-annulus-iterate} implies that there exists $\lebExp   > 0$ depending only on $M$ and the law of $D_h$ such that 
\eqb \label{eqn-geodesic-leb-iterate}
\BB P\left[\text{$\exists k \in [1, \log_4 \ep^{-1/2} - 1]_{\BB Z}$ such that $E_k(t)$ occurs} \right] \geq 1 - O_\ep(\ep^{2\lebExp}) ,
\eqe 
with the implicit constant in the $O_\ep(\cdot)$ depending only on $M$ and the law of $D_h$. 

Say that $t\in [0,T]$ is \emph{good} if $E_k(t)$ occurs for some $k\in [1, \log_4 \ep^{-1/2} - 1]_{\BB Z}$, and that $t$ is \emph{bad} otherwise.  
By~\eqref{eqn-geodesic-leb-iterate}, 
\eqbn
\BB E\left[ |\{t\in [0,T] : \text{$t$ is bad} \} | \right]  \leq O_\ep(\ep^{2\lebExp}) T .
\eqen
By Markov's inequality, it holds with probability $1-O_\ep(\ep^\lebExp)$ that 
\eqb
  |\{t\in [0,T] : \text{$t$ is bad} \} |    \leq  \ep^\lebExp T .
\eqe

To prove~\eqref{eqn-geodesic-leb}, it remains to show that if $t$ is good and $f$ is as in the lemma statement, then no $D_{h-f}(\cdot,\cdot;\BB r \ol U)$-geodesic between two points of $\BB r [U \setminus B_{\ep^{1/2}}(\eta) ]$ can hit $B_{\ep\BB r}(\BB r \eta(t))$. 
To see this, let $P_f$ be such a geodesic and choose $ k \in [1, \log_4 \ep^{-1/2} - 1]_{\BB Z}$ such that $E_k(t)$ occurs.
By~\eqref{eqn-geodesic-leb-event}, there is a path $\pi$ in $\BB A_{2r_k, 3r_k}(\BB r \eta(t))$ which disconnects the inner and outer boundaries of this annulus such that
\eqbn
\op{len}\left(\pi ; D_h \right)  < e^{-2\xi M} D_h\left(\text{across $\BB A_{r_k , 2r_k}(\BB r \eta(t))$} \right) .
\eqen
By Weyl scaling (Axiom~\ref{item-metric-f}) and since $f$ takes values in $[-M,M]$,  
\eqb  \label{eqn-geodesic-leb-length}
\op{len}\left(\pi ; D_{h-f} \right)  < D_{h-f}\left(\text{across $\BB A_{r_k , 2r_k}(\BB r \eta(t))$} \right)  .
\eqe 
Since $\ep \BB r \leq r_k \leq \frac12 \ep^{1/2} \BB r$ and the endpoints of $P$ are at Euclidean distance at least $\ep^{1/2} \BB r$ from $\BB r \eta$, we see that if $P_f$ hits $B_{\ep\BB r}(\BB r \eta(t))$ then the following is true. There are times $0 < \tau < \sigma < \op{len}(P ; D_{h-f})$ such that $P(\tau) , P(\sigma) \in \pi$ and $P$ crosses between the inner and outer boundaries of $\BB A_{r_k,2r_k}(\BB r \eta(t))$ between times $\tau$ and $\sigma$. Since $\eta \subset   U \setminus B_{\ep^{1/2}}(\bdy U) $, we have $\pi \subset \BB r U$. By~\eqref{eqn-geodesic-leb-length}, we can obtain a path in $\BB r \ol U$ with the same endpoints as $P_f$ which is $D_{h-f}$-shorter than $P_f$ by replacing $P_f|_{[\tau,\sigma]}$ by a segment of the path $\pi$. This contradicts the fact that $P_f$ is a $D_{h-f}(\cdot,\cdot;\BB r \ol U)$-geodesic, so we conclude that $P_f$ cannot hit $B_{\ep\BB r}(\BB r \eta(t))$, as required.
\end{proof}

\section{Quantifying the optimality of the optimal bi-Lipschitz constants}
\label{sec-attained}

\subsection{Events for the optimal bi-Lipschitz constants}
\label{sec-bilip-events}

Let $h$ be a whole-plane GFF and let $D_h$ and $\wt D_h$ be two weak LQG metrics. 
We define the optimal upper and lower bi-Lipschitz constants $\Clower$ and $\Cupper$ as in Section~\ref{sec-bilip}, so that $\Clower$ and $\Cupper$ are deterministic and a.s.~\eqref{eqn-bilip} holds. Recall from Section~\ref{sec-outline} that we aim to prove by contradiction that $\Clower = \Cupper$. For this purpose, we will need several estimates which have non-trivial content only if $\Clower  < \Cupper$. 

From the optimality of $\Clower$ and $\Cupper$, we know that for every $\Cmed < \Cupper$, 
\eqb \label{eqn-near-optimal}
\BB P\left[ \text{$\exists$ non-singular $ u,v \in \BB C $ such that $\wt D_h(u,v) \geq\Cmed D_h(u,v)$} \right] >  0 .
\eqe
A similar statement holds for every $\Cmid > \Clower$. The goal of this section is to prove various quantitative versions of~\eqref{eqn-near-optimal}, which include regularity conditions on $u$ and $v$ and which are required to hold uniformly over different Euclidean scales. 

Our results will be stated in terms of two events, which are defined in Definitions~\ref{def-opt-event} and~\ref{def-annulus-geo} just below. 
In this subsection, we will prove some basic facts about these events and state the main estimates we need for them (Propositions~\ref{prop-attained-good} and~\ref{prop-attained-good'}). 
Then, in Section~\ref{sec-attained-proof}, we will prove our main estimates. 

\begin{defn} \label{def-opt-event}
For $r > 0$, $\Kopt > 0$, and $\Cmed > 0$, we let $G_r(\Kopt , \Cmed)$ be the event that there exist $z,w\in \ol B_r(0)$ such that 
\eqbn
 \wt D_h\left( B_{\Kopt r}(z) , B_{\Kopt r}(w) \right) \geq \Cmed D_h(z,w)    .
\eqen  
\end{defn}

The event $G_r(\Kopt , \Cmed)$ is a slightly stronger version of the event in~\eqref{eqn-near-optimal}. 
Our other event has a more complicated definition, and includes several regularity conditions on $u$ and $v$. 
See Figure~\ref{fig-annulus-geo} for an illustration. 

\begin{defn} \label{def-annulus-geo}
For $r > 0$, $\Kann \in (3/4,1)$, and $\Cmed> 0$, we let $H_r(\Kann  ,\Cmed)$ be the event that there exist non-singular points $u\in \bdy B_{\Kann r}(0)$ and $v\in\bdy B_r(0)$ such that
\eqb  \label{eqn-annulus-geo-opt} 
\wt D_h(u,v) \geq \Cmed D_h(u,v)  
\eqe 
and a $D_h$-geodesic $P$ from $u$ to $v$ such that the following is true. 
\begin{enumerate}[$(i)$]
\item $P \subset \ol{\BB A}_{\Kann r , r}(0)$.
\item The Euclidean diameter of $P$ is at most $r/100$. 
\item $D_h(u,v) \leq (1-\Kann)^{-1} r^{\xi Q} e^{\xi h_r(0)}$. 
\item Let $\geoExp > 0$ be as in Lemma~\ref{lem-hit-ball-phi} with $\chi = 1/2$. For each $ \delta \in \left(0,(1-\Kann)^2\right]$, \label{item-H0-loop}
\eqb \label{eqn-annulus-geo-loop} 
\max\left\{    D_h\left( u , \bdy B_{\delta r}(u) \right)  ,   D_h\left(\text{around $\BB A_{\delta r , \delta^{1/2} r}(u)$} \right)   \right\} \leq \delta^\geoExp D_h(u,v)
\eqe 
and the same is true with the roles of $u$ and $v$ interchanged. 
\end{enumerate}
\end{defn}

\begin{figure}[ht!]
\begin{center}
\includegraphics[width=.6\textwidth]{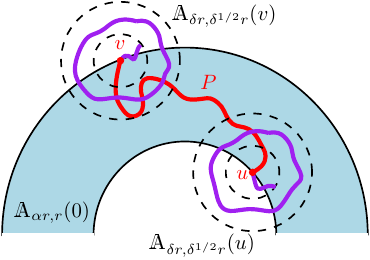} 
\caption{\label{fig-annulus-geo} Illustration of the event $H_r(\Kann,\Cmed)$ of Definition~\ref{def-annulus-geo}. The last condition~\eqref{item-H0-loop} says that for each $\delta>0$, there exist purple paths as in the figure whose $D_h$-lengths are at most $\delta^\geoExp D_h(u,v)$. The figure is not shown to scale --- in actuality we will take $\Kann$ to be close to 1, so the light blue annulus will be quite narrow. 
}
\end{center}
\end{figure}

The main result of this section, which will be proven in Section~\ref{sec-attained-proof}, tells us that (for appropriate values of $\Kopt,\Cmed',\Kann,\Cmed$) if $\BB P[G_{\BB r}(\Kopt , \Cmed')] \geq \Kopt$, then there are lots of ``scales" $r < \BB r$ for which $\BB P[H_r(\Kann  ,\Cmed)]$ is bounded below by a constant which does not depend on $r$ or $\Cmed$. 
 
\begin{prop} \label{prop-attained-good}
There exist $\Kann   \in (3/4 ,1)$ and $p \in (0,1)$, depending only on the laws of $D_h$ and $\wt D_h$, such that for each $\Cmed \in (0,\Cupper)$, there exists $\Cmed'  = \Cmed'(\Cmed)  \in (\Cmed , \Cupper)$ such that for each $\Kopt \in (0,1)$, there exists $\ep_0 = \ep_0( \Kopt,\Cmed )  > 0$ with the following property. If $\BB r > 0$ and $\BB P[ G_{\BB r}(\Kopt, \Cmed') ]  \geq \Kopt$, then the following is true for each $\ep \in (0,\ep_0]$.  
\begin{enumerate}[(A)]  
\item \label{item-attained-good} There are at least $\frac34 \log_8\ep^{-1}$ values of $r \in [\ep^2 \BB r ,\ep \BB r ] \cap \{8^{-k} \BB r : k\in\BB N\}$ for which $\BB P[H_r(\Kann, \Cmed)] \geq p$. 
\end{enumerate}
\end{prop}

We emphasize that in Proposition~\ref{prop-attained-good}, the parameters $\Kann$ and $p$ do \emph{not} depend in $\Cmed$. This will be crucial for our argument in Section~\ref{sec-counting-conclusion}.  

In the remainder of this subsection, we will prove some basic lemmas about the events of Definitions~\ref{def-opt-event} and~\ref{def-annulus-geo}, some of which are consequences of Proposition~\ref{prop-attained-good}. 
In order for Proposition~\ref{prop-attained-good} to have non-trivial content, one needs a lower bound for $\BB P[G_{\BB r}(\Kopt , \Cmed)]$. 
It is straightforward to check that one has such a lower bound if $\BB r = 1$ and $\Kopt$ is small enough.  

\begin{lem} \label{lem-opt-pos}
For each  $\Cmed < \Cupper$, there exists $\Kopt    > 0$, depending on $\Cmed$ and the laws of $D_h$ and $\wt D_h$, such that $\BB P[G_1(\Kopt, \Cmed)] > 0$.
\end{lem}
\begin{proof}
We will prove the contrapositive.
Let $\Cmed > 0$ and assume that 
\eqb  \label{eqn-opt-pos-hyp}
\BB P\left[ G_1(\Kopt , \Cmed) \right]  = 0,\quad \forall \Kopt  > 0 .
\eqe 
We will show that $\Cmed \geq \Cupper$. The assumption~\eqref{eqn-opt-pos-hyp} implies that a.s.\ 
\eqb \label{eqn-opt-pos-balls}
 \wt D_h\left( B_{\Kopt  }(z) , B_{\Kopt  }(w) \right) <  \Cmed D_h(z,w) , \quad \forall z , w \in \ol B_1(0) ,\quad \forall \Kopt > 0 . 
\eqe 
By lower semicontinuity, for each $z, w\in B_1(0)$, 
\eqbn
\wt D_h(z,w) \leq \liminf_{\Kopt \rta 0} \wt D_h\left( B_{\Kopt }(z) , B_{\Kopt }(w) \right)  , 
\eqen 
so~\eqref{eqn-opt-pos-balls} implies that a.s.\
\eqb \label{eqn-opt-pos-unit}
 \wt D_h\left(  z,w \right) \leq  \Cmed D_h(z,w) , \quad \forall z , w \in \ol B_1(0)  . 
\eqe 
By the translation invariance property of $D_h$ (Axiom~\reftranslate) and the translation invariance of the law of $h$, viewed modulo additive constant,~\eqref{eqn-opt-pos-unit} implies that a.s.\
\eqb \label{eqn-opt-pos-all}
 \wt D_h\left(  z,w \right) \leq  \Cmed D_h(z,w) , \quad \forall z , w \in \BB C \: \text{such that}\: |z-w| \leq 1 .
\eqe
For a general pair of non-singular points $z,w\in\BB C$, we can apply~\eqref{eqn-opt-pos-all} to finitely pairs of points along a $D_h$-geodesic from $z$ to $w$ to get that a.s.\ $\wt D_h(z,w) \leq \Cmed D_h(z,w)$ for all $z,w\in\BB C$. By the minimality of $\Cupper$, this shows that $\Cmed \geq \Cupper$, as required.
\end{proof}

By combining Proposition~\ref{prop-attained-good} and Lemma~\ref{lem-opt-pos}, we get the following.
 
\begin{prop} \label{prop-geo-annulus-prob}
There exist $\Kann \in (3/4,1)$ and $p\in(0,1)$, depending only on the laws of $D_h$ and $\wt D_h$, such that for each $\Cmed \in (0,\Cupper)$ and each sufficiently small $\ep > 0$ (depending on $\Cmed$ and the laws of $D_h$ and $\wt D_h$), there are at least $\frac34\log_8\ep^{-1}$ values of $r \in [\ep^2 , \ep] \cap \{8^{-k}\}_{k\in\BB N}$ for which $\BB P[H_r(\Kann , \Cmed)] \geq p $.
\end{prop}
\begin{proof}
Let $\Kann \in (3/4,1)$ and $p\in (0,1)$ (depending only on the laws of $D_h$ and $\wt D_h$) and $\Cmed'  \in (\Cmed , \Cupper)$ (depending only on $\Cmed$ and the laws of $D_h$ and $\wt D_h$) be as in Proposition~\ref{prop-attained-good}. 
By Lemma~\ref{lem-opt-pos} (applied with $\Cmed'$ instead of $\Cmed$), there exists $\Kopt > 0$, depending only on $\Cmed$ and the laws of $D_h$ and $\wt D_h$, such that $\BB P[G_1(\Kopt , \Cmed')] \geq \Kopt$. 
By Proposition~\ref{prop-attained-good} applied with $\BB r = 1$, we now obtain the proposition statement. 
\end{proof}

We will also need an analog of Proposition~\ref{prop-geo-annulus-prob} with the events $G_r(\Kopt , \Cmed)$ in place of the events $H_r(\Kann , \Cmed)$, which strengthens Lemma~\ref{lem-opt-pos}.  
 
\begin{prop} \label{prop-opt-prob}
For each $\Cmed \in (0,\Cupper)$, there exists $\Kopt > 0 $, depending on $\Cmed$ and the laws of $D_h$ and $\wt D_h$, such that for each small enough $\ep > 0$ (depending on $\Cmed$ and the laws of $D_h$ and $\wt D_h$), there are at least $\frac34\log_8\ep^{-1}$ values of $r \in [\ep^2 , \ep] \cap \{8^{-k}\}_{k\in\BB N}$ for which $\BB P[G_r(\Kopt , \Cmed)] \geq \Kopt$.
\end{prop}

We will deduce Proposition~\ref{prop-opt-prob} from Proposition~\ref{prop-geo-annulus-prob} and the following elementary relation between the events $H_r(\cdot,\cdot)$ and $G_r(\cdot,\cdot)$.

\begin{lem} \label{lem-opt-events}
If $\Kann \in (3/4,1)$ and $\zeta \in (0,1)$, there exists $\Kopt  > 0$, depending only on $\Kann,\zeta$, and the laws of $D_h$ and $\wt D_h$, such that the following is true. 
For each $r > 0$ and each $\Cmed  > 0$, if $H_r(\Kann,\Cmed)$ occurs, then $G_r(\Kopt , \Cmed-\zeta)$ occurs.
\end{lem}
\begin{proof}
Assume that $H_r(\Kann , \Cmed)$ occurs and let $u$ and $v$ be as in Definition~\ref{def-annulus-geo} of $H_r(\Kann , \Cmed)$. 
By Definition~\ref{def-opt-event} of $G_r(\Kopt , \Cmed-\zeta)$, it suffices to find $\Kopt > 0$ as in the lemma statement such that
\eqb \label{eqn-opt-events-show}
 \wt D_h\left( B_{\Kopt r}(u) , B_{\Kopt r}(v) \right) \geq ( \Cmed - \zeta)  D_h(u,v)  .
\eqe

To this end, let $\delta >0$ and suppose that $P^\delta$ is a path from $  B_{\delta r}(u)$ to $  B_{\delta r}(v)$; $P_u^\delta$ and $P^\delta_v$ are paths from $u$ and $v$ to $\bdy B_{\delta^{1/2} r}(u)$ and $\bdy B_{\delta^{1/2} r}(v)$, respectively; and $\pi_u^\delta$ and $\pi_v^\delta$ are paths in $\BB A_{\delta r,\delta^{1/2} r}(u)$ and  $\BB A_{\delta r,\delta^{1/2} r}(u)$, respectively, which disconnect the inner and outer boundaries.
Then the union $P^\delta \cup P_u^\delta \cup P_v^\delta \cup \pi_u^\delta \cup \pi_v^\delta$ contains a path from $u$ to $v$. 
From this observation followed by~\eqref{eqn-annulus-geo-loop} of Definition~\ref{def-annulus-geo} and the definition~\eqref{eqn-bilip-def} of $\Cupper$, we get that if $\delta \in (0,(1-\Kann)^4]$ then 
\allb \label{eqn-opt-events-decomp}
\wt D_h\left( u, v \right) 
&\leq  \wt D_h\left( B_{\delta r}(u) , B_{\delta r}(v)  \right) 
+ \sum_{w\in \{u,v\}}  \wt D_h\left( w , \bdy B_{\delta^{1/2} r}(w) \right) \notag\\
&\qquad\qquad\qquad\qquad\qquad + \sum_{w\in\{u,v\}} \wt D_h\left(\text{around $\BB A_{\delta r , \delta^{1/2} r}(w)$} \right) \notag\\ 
&\leq  \wt D_h\left(B_{\delta r}(u) , B_{\delta r}(v)  \right) 
+ \Cupper \sum_{w\in \{u,v\}}  D_h\left( w , \bdy B_{\delta^{1/2} r}(w) \right)  \notag\\
&\qquad\qquad\qquad\qquad\qquad  + \Cupper \sum_{w\in\{u,v\}}   D_h\left(\text{around $\BB A_{\delta r , \delta^{1/2} r}(w)$} \right) \notag\\   
&\leq  \wt D_h\left( B_{\delta r}(u) ,   B_{\delta r}(v)  \right)  +    2 \Cupper \left( \delta^{\geoExp / 2 } +    \delta^\geoExp \right) D_h(u,v) .
\alle

By~\eqref{eqn-annulus-geo-opt} and~\eqref{eqn-opt-events-decomp}, we obtain
\eqb
\wt D_h\left( B_{\delta r}(u) ,   B_{\delta r}(v)  \right) \geq \left[ \Cmed - 2 \Cupper \left( \delta^{\geoExp / 2 } +    \delta^\geoExp \right) \right] D_h(u,v) .
\eqe
We now obtain~\eqref{eqn-opt-events-show} by choosing $\delta \in (0,(1-\Kann)^4]$ to be sufficiently small, depending on $\zeta$ and $\Cupper$, and setting $\Kopt =\delta$. 
\end{proof}

\begin{proof}[Proof of Proposition~\ref{prop-opt-prob}] 
Let $\Kann \in (3/4,1)$ and $p\in (0,1)$ (depending only on the laws of $D_h$ and $\wt D_h$) be as in Proposition~\ref{prop-geo-annulus-prob}.
Also let $\Cmed' := (\Cmed + \Cupper)/2 \in (\Cmed,\Cupper)$. 
By Proposition~\ref{prop-geo-annulus-prob} (applied with $\Cmed'$ instead of $\Cmed$), for each small enough $\ep > 0$, there are at least $\frac34\log_8\ep^{-1}$ values of $r \in [\ep^2 , \ep] \cap \{8^{-k}\}_{k\in\BB N}$ for which $\BB P[H_r(\Kann , \Cmed')] \geq p $. 
By Lemma~\ref{lem-opt-events}, applied with $\Cmed'$ in place of $\Cmed$ and $\zeta = \Cmed' - \Cmed$, we see that there exists $\Kopt > 0$, depending only on $\Cmed$ and the laws of $D_h$ and $\wt D_h$, such that if $H_r(\Kann,\Cmed')$ occurs, then $G_r(\Kopt , \Cmed)$ occurs.
Combining the preceding two sentences gives the proposition statement with $p\wedge\Kopt$ in place of $\Kopt$. 
\end{proof}

Since our assumptions on the metrics $D_h$ and $\wt D_h$ are the same, all of the results above also hold with the roles of $D_h$ and $\wt D_h$ interchanged. For ease of reference, we will record some of these results here.

\begin{defn} \label{def-opt-event'}
For $r > 0$, $\Kopt > 0$, and $\Cmid > 0$, we let $\wt G_r(\Kopt , \Cmid)$ be the event that the event $G_r(\Kopt,1/\Cmid)$ of Definition~\ref{def-opt-event} occurs with the roles of $D_h$ and $\wt D_h$ interchanged.
That is, $\wt G_r(\Kopt , \Cmid)$ is the event that there exists $z,w\in \ol B_r(0)$ such that 
\eqbn
 \wt D_h\left( z, w \right) \leq \Cmid D_h\left( B_{\Kopt r}(z) , B_{\Kopt r}(w)  \right)     .
\eqen  
\end{defn}

\begin{defn} \label{def-annulus-geo'}
For $r > 0$, $\Kann \in (3/4,1)$, and $\Cmid> 0$, we let $\wt H_r(\Kann  ,\Cmid)$ be the event that the event $H_r(\Kann , 1/\Cmid)$ of Definition~\ref{def-annulus-geo} occurs with the roles of $D_h$ and $\wt D_h$ interchanged.
That is, $\wt H_r(\Kann , \Cmid)$ is the event that there exist non-singular points $u\in \bdy B_{\Kann r}(0)$ and $v\in\bdy B_r(0)$ such that
\eqb  \label{eqn-annulus-geo-opt'} 
\wt D_h(u,v) \leq \Cmid D_h(u,v)  
\eqe 
and a $\wt D_h$-geodesic $\wt P$ from $u$ to $v$ such that the following is true. 
\begin{enumerate}[$(i)$]
\item $\wt P \subset \ol{\BB A}_{\Kann r , r}(0)$.
\item The Euclidean diameter of $\wt P$ is at most $r/100$. 
\item $\wt D_h(u,v) \leq (1-\Kann)^{-1} r^{\xi Q} e^{\xi h_r(0)}$. 
\item Let $\geoExp > 0$ be as in Lemma~\ref{lem-hit-ball-phi} with $\chi = 1/2$. For each $ \delta \in \left(0,(1-\Kann)^2\right]$, 
\eqb \label{eqn-annulus-geo-loop-wt} 
\max\left\{   \wt D_h\left( u , \bdy B_{\delta r}(u) \right)  ,   \wt D_h\left(\text{around $\BB A_{\delta r , \delta^{1/2} r}(u)$} \right)   \right\} \leq \delta^\geoExp \wt D_h(u,v)
\eqe 
and the same is true with the roles of $u$ and $v$ interchanged. 
\end{enumerate}
\end{defn}

We have the following analog of Proposition~\ref{prop-attained-good}.

\begin{prop} \label{prop-attained-good'}
There exist $\Kann   \in (3/4 ,1)$ and $p \in (0,1)$, depending only on the laws of $D_h$ and $\wt D_h$, such that for each $\Cmid  > \Clower$, there exists $\Cmid'  = \Cmid'(\Cmid)  \in (\Clower, \Cmid)$ such that for each $\wt\Kopt \in (0,1)$, there exists $\ep_0 = \ep_0( \wt\Kopt,\Cmid )  > 0$ with the following property. If $\BB r > 0$ and $\BB P[ \wt G_{\BB r}(\wt\Kopt, \Cmid') ]  \geq \wt\Kopt$, then the following is true for each $\ep \in (0,\ep_0]$.  
\begin{enumerate}[(A')]  
\item \label{item-attained-good'} There are at least $\frac34 \log_8\ep^{-1}$ values of $r \in [\ep^2 \BB r ,\ep \BB r ] \cap \{8^{-k} \BB r : k\in\BB N\}$ for which $\BB P[\wt H_r(\Kann, \Cmid)] \geq p$. 
\end{enumerate}
\end{prop}

We will also need the following analog of Proposition~\ref{prop-opt-prob}.

\begin{prop} \label{prop-opt-prob'}
For each $\Cmid  > \Clower$, there exists $\wt\Kopt > 0 $, depending on $\Cmid$ and the laws of $D_h$ and $\wt D_h$, such that for each small enough $\ep > 0$ (depending on $\Cmid$ and the laws of $D_h$ and $\wt D_h$), there are at least $\frac34\log_8\ep^{-1}$ values of $r \in [\ep^2 , \ep] \cap \{8^{-k}\}_{k\in\BB N}$ for which $\BB P[\wt G_r(\wt\Kopt , \Cmid)] \geq \wt\Kopt$.
\end{prop}

\subsection{Proof of Proposition~\ref{prop-attained-good}}
\label{sec-attained-proof}

To prove Proposition~\ref{prop-attained-good}, we will prove the contrapositive, as stated in the following proposition.

\begin{prop} \label{prop-attained-bad}
There exists $\Kann   \in (3/4 ,1)$ and $p \in (0,1)$, depending only on the laws of $D_h$ and $\wt D_h$, such that for each $\Cmed \in (0,\Cupper)$, there exists $\Cmed'  = \Cmed'(\Cmed)  \in (\Cmed , \Cupper)$ such that for each $\Kopt \in (0,1)$, there exists $\ep_0 = \ep_0( \Kopt,\Cmed )  > 0$ with the following property. 
If $\BB r > 0$ and there exists $\ep \in (0,\ep_0]$ satisfying the condition~\eqref{item-attained-bad} just below, then $\BB P[ G_{\BB r}(\Kopt, \Cmed') ]  < \Kopt$. 
\begin{enumerate}[(A)] 
\addtocounter{enumi}{1}
\item \label{item-attained-bad} There are at least $\frac14 \log_8\ep^{-1}$ values of $r \in [\ep^2 \BB r ,\ep \BB r ] \cap \{8^{-k} \BB r : k\in\BB N\}$ for which $\BB P[H_r(\Kann, \Cmed)] < p$. 
\end{enumerate}
\end{prop}

Note that the second-to-last last sentence of Proposition~\ref{prop-attained-bad} (i.e., the one just before condition~\eqref{item-attained-bad}) is the contrapositive of the second-to-last sentence of Proposition~\ref{prop-attained-good} (i.e., the one just before condition~\eqref{item-attained-good}).
The proof of Proposition~\ref{prop-attained-bad} is similar to the argument in~\cite[Section 3.2]{gm-uniqueness}, but the definitions of the events involved are necessarily different due to the existence of singular points. 

The basic idea of the proof is as follows. 
If we assume that~\eqref{item-attained-bad} holds for a small enough (universal) choice of $p\in (0,1)$, then we can use Lemma~\ref{lem-annulus-iterate} (independence across concentric annuli) and a union bound to cover space by Euclidean balls of the form $B_{r/2}(z)$ for $r\in [\ep^2 \BB r , \ep \BB r]$ with the following property.
For each $u \in \bdy B_{\Kann r}(z)$ and each $v\in \bdy B_{r}(z)$ which are joined by a geodesic $P$ satisfying the numbered conditions in Definition~\ref{def-annulus-geo}, we have $\wt D_h(u,v) \leq \Cmed D_h(u,v)$.

By considering the times when a $D_h$-geodesic between two fixed points $\BB z , \BB w \in \BB C$ crosses the annulus $\BB A_{\Kann r , r}(z)$ for such a $z$ and $r$, we will be able to show that $\wt D_h(B_{\Kopt}(\BB z) , B_{\Kopt}(\BB w)) \leq \Cmed' D_h(\BB z,\BB w)$ for a suitable constant $\Cmed' \in (\Cmed , \Cupper)$. 
Applying this to an appropriate $\Kopt$-dependent collection of pairs of points $(\BB z, \BB w)$ will show that $\BB P[  G_{\BB r}(\Kopt , \Cmed'  ) ]  < \Kopt$. 
The reason why we need to make $\Kann$ close to 1 is to ensure that the events we consider depend on $h$ in a sufficiently ``local" manner (see the proof of Lemma~\ref{lem-attained-msrble}).

\begin{figure}[ht!]
\begin{center}
\includegraphics[width=.7\textwidth]{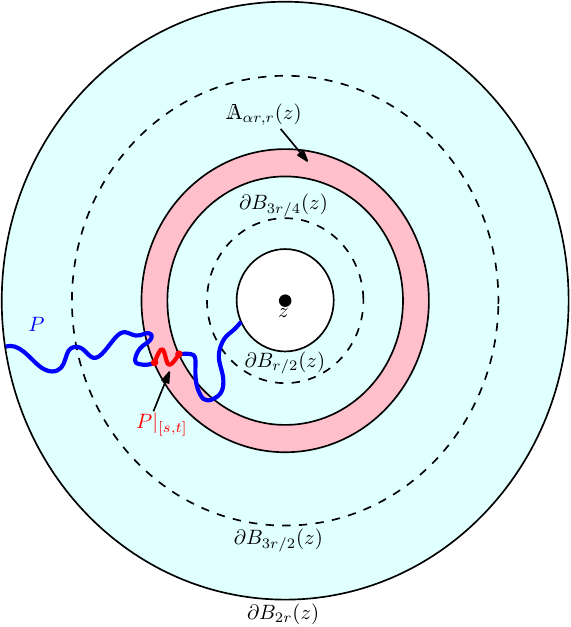} 
\caption{\label{fig-attained-event} Illustration of the definition of $E_r(z)$. We have shown the annuli involved in the definition and an example of a $D_h(\cdot,\cdot;\BB A_{r/2,2r}(z))$-geodesic $P$ between two points of $\bdy\BB A_{r/2,2r}(z)$, which appears in several of the conditions.  
Condition~\ref{item-attained-reg} allows us to compare distances around and across small annuli surrounding points of $\BB A_{3r/4,3r/2}(z)$ which are hit by $P$. Condition~\ref{item-attained-cross} provides an upper bound for the $D_h$-distance around the outer annulus $\BB A_{3r/2,2r}(z)$. Condition~\ref{item-attained-leb} gives an upper bound for the Euclidean diameters of segments of $P$ which are contained in the pink annulus $\BB A_{\Kann r , r}(z)$, such as the red segment in the figure. Condition~\ref{item-attained-dist} gives an upper bound for the $D_h$-distance around $\BB A_{\Kann r , r}(z)$. Finally, condition~\ref{item-attained-geo} will allow us to show that the $\wt D_h$-length of a red segment like $P|_{[s,t]}$ is at most $\Cmed(t-s)$. 
}
\end{center}
\end{figure}

Let us now define the events to which we will apply Lemma~\ref{lem-annulus-iterate}. 
See Figure~\ref{fig-attained-event} for an illustration of the definition. 
We will discuss the purpose of each condition in the event just below. 

For $z \in \BB C$, $r > 0$, and parameters $\Kep \in (0,1/100)$, $\Kann \in (1-\Kep ,1)$, and $ \Karound > 1$, let $E_r(z) = E_r(z; \Kep , \Kann,\Karound,\Cmed)$ be the event that the following is true.
\begin{enumerate}
\item \label{item-attained-reg} \textit{(Regularity along geodesics)} For each $D_h\left(\cdot,\cdot;\ol{\BB A}_{r/2,2r}(z)\right)$-geodesic $P$ between two points of $\bdy\BB A_{r/2,2r}(z)$, each $\delta \in (0,\Kep]$, and each $x \in \BB A_{3r/4,3r/2}(z)$ such that $P \cap B_{\delta r}(x) \not=\emptyset$,  
\eqb \label{eqn-attained-reg}
D_h\left(\text{around $\BB A_{\delta r ,\delta^{1/2} r}(x)$} \right) \leq \delta^\geoExp D_h\left(\text{across $\BB A_{\delta  r , \delta^{1/2} r}(x)$} \right) ,
\eqe
where (as in Definition~\ref{def-annulus-geo}) $\geoExp$ is as in Lemma~\ref{lem-hit-ball-phi} with $\chi=1/2$. 
\item \label{item-attained-cross} \textit{(Distance around $\BB A_{3r/2,2r}(z)$)} We have
\allb \label{eqn-attained-cross} 
&D_h\left( \text{around $\BB A_{3r/2,2r}(z)$}\right) \notag\\
&\leq \min\left\{   (1-\Kann)^{-1} r^{\xi Q} e^{\xi h_r(z)}  , \frac{\Clower}{2\Cupper} \Kep^{-\geoExp} D_h\left(  \BB A_{3r/4 , 3r/2}(z) ,  \bdy \BB A_{r/2,2r}(z) \right)  \right\} .
\alle
\item \label{item-attained-leb} \textit{(Euclidean length of geodesic segments in $\BB A_{\Kann r , r}(z)$)} For each $D_h(\cdot,\cdot;\ol{\BB A}_{r/2,2r}(z)$-geodesic $P$ between two points of $\bdy \BB A_{r/2,2r}(z)$ and any two times $t > s > 0$ such that $P([s,t]) \subset \ol{\BB A}_{\Kann r , r}(z)$, we have 
\eqb \label{eqn-attained-leb}
|P(t) -P(s)| \leq \Kep r .
\eqe
\item \label{item-attained-dist} \textit{(Distance around $\BB A_{\Kann r , r}(z)$)} We have
\eqb \label{eqn-attained-dist} 
D_h\left(\text{around $\BB A_{\Kann r , r}(z)$} \right) \leq \Karound  D_h\left(\text{across $\BB A_{\Kann r , r}(z)$} \right)  .
\eqe
\item \label{item-attained-geo} \textit{(Converse of $H_r(\Kann,\Cmed)$)} Let $u\in \bdy B_{\Kann r}(z)$ and $v\in\bdy B_r(z)$ such that $|u-v| \leq \Kep r$ and
\eqb \label{eqn-attained-geo}
D_h\left(\text{around $\BB A_{\Kep r ,\Kep^{1/2} r}(v)$} \right) \leq  \frac{\Clower}{2 \Cupper}   D_h\left( \BB A_{3r/4,3r/2}(z)  , \bdy\BB A_{r/2,2r}(z) \right) .
\eqe 
Assume that there is a $ D_h$-geodesic $ P'$ from $u$ to $v$ such that the numbered conditions in Definition~\ref{def-annulus-geo} of $H_r(\Kann,\Cmed)$ occur but with $z$ in place of $0$, i.e.,
\begin{enumerate}[$(i)$]
\item $P' \subset \ol{\BB A}_{\Kann r , r}(z)$. \label{item-H-subset}
\item The Euclidean diameter of $P'$ is at most $r/100$. \label{item-H-diam}
\item $D_h(u,v) \leq (1-\Kann)^{-1} r^{\xi Q} e^{\xi h_r(z)}$.  \label{item-H-dist}
\item For each $ \delta \in \left(0,(1-\Kann)^2\right]$, \label{item-H-loop}
\eqb \label{eqn-annulus-geo-loop'} 
\max\left\{    D_h\left( u , \bdy B_{\delta r}(u) \right)  ,   D_h\left(\text{around $\BB A_{\delta r , \delta^{1/2} r}(u)$} \right)   \right\} \leq \delta^\geoExp D_h(u,v)
\eqe 
and the same is true with the roles of $u$ and $v$ interchanged. 
\end{enumerate}
Then $\wt D_h(u,v) \leq \Cmed D_h(u,v)$. 
\end{enumerate}
 
The most important condition in the definition of $E_r(z)$ is condition~\ref{item-attained-geo}. By Definition~\ref{def-annulus-geo} and the translation invariance of the law of $h$, modulo additive constant, if $\BB P[H_r(\Kann,\Cmed)]$ is small, then the probability of condition~\ref{item-attained-geo} is large. The extra condition~\eqref{eqn-attained-geo} on $u$ and $v$ is included in order to prevent $D_h$-geodesics or $\wt D_h$-geodesics between $u$ and $v$ from exiting $\BB A_{r/2,2r}(z)$. This is needed to ensure that $E_r(z)$ is determined by $h|_{\BB A_{r/2,2r}(z)}$, which in turn is needed to apply Lemma~\ref{lem-annulus-iterate}. See Lemma~\ref{lem-attained-msrble}. 

We will eventually consider a $D_h$-geodesic $P$ which enters $B_{r/2}(z)$ and apply condition~\ref{item-attained-geo} to the $D_h$-geodesic $P' = P|_{[s,t]}$ from $u=P(s)$ to $v=P(t)$, where $s$ and $t$ are suitably chosen times such that $P(s) \in \bdy B_{\Kann r}(z)$ and $P(t) \in \bdy B_r(z)$. The first three conditions in the definition of $E_r(z)$ will allow us to do so (see Lemma~\ref{lem-attained-increment}). In particular, condition~\ref{item-attained-reg} will allow us to check~\eqref{eqn-annulus-geo-loop'} for $u = P(s)$ and $v = P(t)$. 
Condition~\ref{item-attained-cross} will be used in conjunction with condition~\ref{item-attained-reg} to check~\eqref{eqn-attained-geo}. 
Condition~\ref{item-attained-leb} will be used to upper-bound the Euclidean diameter of $P|_{[s,t]}$.

Condition~\ref{item-attained-dist} will be used to show that the intervals $[s,t]$ as above for varying choices of $r$ and $z$ such that $E_r(z)$ occurs and $P$ enters $B_{r/2}(z)$ cover a uniformly positive fraction of the time interval on which $P$ is defined. See Lemma~\ref{lem-attained-compare}. 

Let us now explain why we can apply Lemma~\ref{lem-annulus-iterate} to the events $E_r(z)$. 
For the statement, recall the definition of the restriction of the GFF to a closed set from~\eqref{eqn-closed-restrict}. 

\begin{lem} \label{lem-attained-msrble} 
The event $E_r(z)$ is a.s.\ determined by $h|_{\ol{\BB A}_{r/2,2r}(z)}$, viewed modulo additive constant.
\end{lem}
\begin{proof}
It is immediate from Weyl scaling (Axiom~\ref{item-metric-f}) that adding a constant to $h$ does not affect the occurrence of $E_r(z)$. Therefore, $E_r(z)$ is a.s.\ determined by $h$ viewed modulo additive constant. We need to show that $E_r(z)$ is a.s.\ determined by $h|_{\ol{\BB A}_{r/2,r}(z)}$. 

Each of conditions~\ref{item-attained-reg}, \ref{item-attained-cross}, \ref{item-attained-leb}, and~\ref{item-attained-dist} in the definition of $E_r(z)$ depends only on $D_h(\cdot,\cdot;\ol{\BB A}_{r/2,r}(z))$. By locality (Axiom~\ref{item-metric-local}; see also Section~\ref{sec-closed}), we get that each of these four conditions is a.s.\ determined by $h|_{\ol{\BB A}_{r/2,2r}(z)}$. 

We still need to treat condition~\ref{item-attained-geo}.
To this end, we claim that if $u\in \bdy B_{\Kann r}(z)$ and $v\in\bdy B_r(z)$ such that $|u-v| \leq \Kep r$ and~\eqref{eqn-attained-geo} holds (as in condition~\ref{item-attained-geo}), then every $D_h$-geodesic and every $\wt D_h$-geodesic from $u$ to $v$ is contained in $\BB A_{r/2,2r}(z)$.
The claim implies that the set of $D_h(\cdot,\cdot; \BB A_{r/2,2r}(z))$-geodesics from $u$ to $v$ is the same as the set of $D_h$-geodesics from $u$ to $v$, and similarly with $\wt D_h$ in place of $D_h$. 
This, in turn, implies that condition~\ref{item-attained-geo} is equivalent to the analogous condition where we require that $P'$ is a $D_h(\cdot,\cdot;\BB A_{r/2,2r}(z))$-geodesic instead of a $D_h$-geodesic and we replace $ D_h(u,v)$ and $\wt D_h(u,v)$ by $ D_h(u,v; \BB A_{r/2,2r}(z))$ and $\wt D_h(u,v;\BB A_{r/2,2r}(z) ) $, respectively. It then follows from locality (Axiom~\ref{item-metric-local}) that $E_r(z)$ is a.s.\ determined by $h|_{\ol{\BB A}_{r/2,2r}(z)}$, viewed modulo additive constant.

It remains to prove the claim in the preceding paragraph. 
Let $u$ and $v$ be as above and let $P $ be path from $u$ to $v$ which exits $\BB A_{r/2,2r}(z)$. We need to show that $P$ is neither a $D_h$-geodesic nor a $\wt D_h$-geodesic. 
By~\eqref{eqn-attained-geo}, there is a path $\pi \subset \BB A_{\Kep r,  \Kep^{1/2} r}(v)$ such that
\eqb \label{eqn-attained-msrble-D}
\op{len}\left( \pi ; D_h \right) <  \frac{\Clower }{  \Cupper }  D_h\left(  \BB A_{3r/4,3r/2}(z)  , \bdy\BB A_{r/2,2r}(z) \right)   .
\eqe
By the bi-Lipschitz equivalence of $D_h$ and $\wt D_h$, this implies that also
\eqb \label{eqn-attained-msrble-tildeD}
\op{len}\left( \pi ; \wt D_h \right) <  \wt D_h\left(  \BB A_{3r/4,3r/2}(z)  , \bdy\BB A_{r/2,2r}(z) \right)   .
\eqe 
Since $u,v\in B_{\Kep r}(v)$, the path $P$ must hit $\pi$ before the first time it crosses from $\BB A_{3r/4,3r/2}(z)$ to $\bdy \BB A_{r/2,2r}(z)$ and after the last time that it does so. 
Therefore,~\eqref{eqn-attained-msrble-D} implies that we can replace a segment of $P$ with a segment of $\pi$ to get a path with the same endpoints and shorter $D_h$-length. 
Hence $P$ is not a $D_h$-geodesic. Similarly,~\eqref{eqn-attained-msrble-tildeD} implies that $P$ is not a $\wt D_h$-geodesic.
\end{proof}

We now check that $E_r(z)$ occurs with high probability if the parameters are chosen appropriately.

\begin{lem} \label{lem-attained-prob}  
For each $p\in (0,1)$, there exist parameters $\Kep \in (0,1/100)$, $\Kann \in (1-\Kep ,1)$, and $ \Karound > 1$, depending only on $p$ and the laws of $D_h$ and $\wt D_h$, such that the following is true. Let $\Cmed \in (0,\Cupper)$ and $\BB r > 0$ and assume that~\eqref{item-attained-bad} holds for our given choice of $\Kann$ and $p$. Then there are at least $\frac14 \log_8\ep^{-1}$ values of $r \in [\ep^2 \BB r , \ep \BB r] \cap \{8^{-k}\}_{k\in\BB N}$ such that $\BB P[E_r(z)] \geq 1 - 2p $ for each $z\in\BB C$. 
\end{lem}
\begin{proof}
By the translation invariance of the law of $h$, viewed modulo additive constant, and Axiom~\reftranslate, it suffices to prove the lemma in the case when $z = 0$. 

By Lemma~\ref{lem-hit-ball-phi} (applied with $f \equiv 0$), we can find $\Kep \in (0,1/100)$ depending only on $p$ and the laws of $D_h$ and $\wt D_h$ such that for each $r> 0$, the probability of condition~\ref{item-attained-reg} in the definition of $E_r(0)$ is at least $1-p/4$. 
By tightness across scales (Axiom~\refcoord), after possibly shrinking $\Kep$, we can find $\Kann \in (1-\Kep ,1)$ depending only on the laws of $D_h$ and $\wt D_h$ such that the probability of condition~\ref{item-attained-cross} is also at least $1-p/4$.

By Lemma~\ref{lem-geodesic-leb} (applied with $f\equiv 0$ and $\eta$ the unit-speed parametrization of $\bdy B_1(0)$), after possibly shrinking $\Kann$, in a manner depending on $\Kep$, we can arrange that for each $r>0$, it holds with probability at least $1-p/4$ that the following is true. For each $D_h\left( \cdot,\cdot;\ol{\BB A}_{r/2,2r}(0) \right)$-geodesic $P$ from a point of $\bdy B_{r/2}(0)$ to a point of $\bdy B_r(0)$, the one-dimensional Lebesgue measure of the set 
\eqb \label{eqn-attained-leb-set}
\left\{ x \in \bdy B_r(0) : P\cap   B_{100(1-\Kann) r}(x) \not=\emptyset \right\} 
\eqe 
is at most $ \Kep  r  $. If $t > s > 0$ such that $P([s,t]) \subset \ol{\BB A}_{\Kann r ,r}(0)$, then the one-dimensional Lebesgue measure of the set~\eqref{eqn-attained-leb-set} is at least the Euclidean diameter of $P([s,t])$. 
This shows that condition~\ref{item-attained-leb} in the definition of $E_r(0)$ occurs with probability at least $1-p/4$.  

By tightness across scales (Axiom~\refcoord), we can find $\Karound > 1$ (depending on $\Kann$) such that for each $r>0$, condition~\ref{item-attained-dist} in the definition of $E_r(0)$ occurs with probability at least $1-p/4$. 
By~\eqref{item-attained-bad} and the Definition~\ref{def-annulus-geo} of $H_r(\Kann,\Cmed)$, there are at least $\frac14 \log_8\ep^{-1}$ values of $r \in [\ep^2 \BB r , \ep \BB r] \cap \{8^{-k}\}_{k\in\BB N}$ such that condition~\ref{item-attained-geo} in the definition of $E_r(0)$ occurs with probability at least $1-p$. We note that the requirement~\eqref{eqn-attained-geo} does not show up in~\eqref{item-attained-bad}, but including the requirement~\eqref{eqn-attained-geo} makes the condition weaker, so makes the probability of the condition larger. 

Taking a union bound over the five conditions in the definition of $E_r(0)$ now concludes the proof.
\end{proof}

With Lemmas~\ref{lem-attained-msrble} and~\ref{lem-attained-prob} in hand, we can now apply Lemma~\ref{lem-annulus-iterate} to obtain the following. 

\begin{lem} \label{lem-attained-union}
There exist parameters $p_* \in (0,1)$, $\Kep \in (0,1/100)$, $\Kann \in (1-\Kep ,1)$, and $ \Karound > 1$, depending only on the laws of $D_h$ and $\wt D_h$, such that the following is true. Let $\Cmed \in (0,\Cupper)$ and $\BB r > 0$ and assume that~\eqref{item-attained-bad} holds for our given choice of $\Kann$ and with $p = p_*$. 
For each fixed bounded open set $U\subset \BB C$, it holds with probability tending to 1 as $\ep\rta 0$ (at a rate depending only on $U$) that for each $z\in \BB r U$, there exists $r \in [\ep^2 \BB r , \ep\BB r]$ and $w \in B_{r/2}(z)$ such that $E_r(w)$ occurs. 
\end{lem}
\begin{proof} 
By Lemma~\ref{lem-annulus-iterate}, there exists a universal constant $p_* \in (0,1)$ such that the following is true.
Let $\BB r > 0$, let $\ep \in (0,1)$, let $K\geq \frac14\log_8\ep^{-1}$, and let $r_1,\dots,r_K \in [\ep^2 \BB r , \ep \BB r] \cap \{8^{-k}\}_{k\in\BB N}$ be distinct. If $z\in\BB C$ and $F_{r_k}(z)$ for $k=1,\dots,K$ is an event which is a.s.\ determined by $h|_{\ol{\BB A}_{r_j /2,2r_j}(z)}$, viewed modulo additive constant, and has probability at least $1-2p_*$, then
\eqbn
\BB P\left[ \text{$\exists k\in [1,K]_{\BB Z}$ such that $F_{r_k}$ occurs} \right] \geq 1 - O_\ep(\ep^{100}) ,
\eqen
with the implicit constant in the $O_\ep(\cdot)$ universal. 

We now choose $\Kep,\Kann,\Karound$ as in Lemma~\ref{lem-attained-prob} with $p = p_*$.
For $\Cmed \in (0,\Cupper)$ and $\BB r > 0$, we apply the above statement to the radii $r \in [\ep^2 \BB r , \ep \BB r] \cap \{8^{-k}\}_{k\in\BB N}$ from Lemma~\ref{lem-attained-prob}, which are chosen so that $\BB P[E_r(w)] \geq 1 - 2 p_*$ for all $w\in\BB C$. By Lemma~\ref{lem-attained-prob}, if~\eqref{item-attained-bad} holds with $p=p_*$, then there are at least $ \frac14 \log_8 \ep^{-1}$ such radii. Hence, if~\eqref{item-attained-bad} holds, then
\eqb \label{eqn-attained-union-prob}
\BB P\left[ \text{$\exists r \in [\ep^2 \BB r , \ep \BB r]$ such that $E_r(w)$ occurs} \right] \geq 1 - O_\ep(\ep^{100}) ,\quad \forall z\in\BB C  ,
\eqe 
with the implicit constant in the $O_\ep(\cdot)$ universal. 

The lemma statement now follows by applying~\eqref{eqn-attained-union-prob} to each of the $O_\ep(\ep^{-2})$ points $w\in   B_{\BB r}(\BB r U) \cap \left(\frac{\ep \BB r}{100} \BB Z^2 \right)$, then taking a union bound. 
\end{proof}

Henceforth fix $p_* ,\Kep , \Kann,$ and $\Karound$ as in Lemma~\ref{lem-attained-union}. 
Also fix
\eqb \label{eqn-C''-choice}
\Cmed' \in \left( \Cmed  + \frac{\Karound }{ \Karound +1 }(\Cupper - \Cmed)  , \Cupper \right) ,
\eqe 
and note that we can choose $\Cmed'$ in a manner depending only on $\Cmed$ and the laws of $D_h$ and $\wt D_h$ (since $\Karound$ depends only on the laws of $D_h$ and $\wt D_h$). 
 
We will show that for each $\Kopt > 0$, there exists $\ep_0 = \ep_0( \Kopt,\Cmed )  > 0$ such that if $\BB r > 0$, $\ep \in (0,\ep_0]$, and~\eqref{item-attained-bad} holds for the above values of $\BB r , \ep,p_* ,\Kann$, then with probability greater than $1  - \Kopt$,  
\eqb \label{eqn-lower-ratio}
\wt D_h\left( B_{\Kopt \BB r}(\BB z) ,B_{\Kopt \BB r}(\BB w) \right) \leq \Cmed' D_h(\BB z,\BB w) \quad  \forall \BB z,\BB w \in B_{\BB r }(0) .
\eqe
By Definition~\ref{def-opt-event}, the bound~\eqref{eqn-lower-ratio} implies that $\BB P[ G_{\BB r}(\Kopt , \Cmed')^c] > 1 - \Kopt$, which is what we aim to show in Proposition~\ref{prop-attained-bad}. 

By Lemma~\ref{lem-geo-compact}, there is some large bounded open set $U\subset \BB C$ (depending only on $\Kopt$ and the law of $D_h$) such that for each $\BB r > 0$, it holds with probability at least $1-\Kopt/2$ that each $D_h$-geodesic between two points of $\ol B_{\BB r}(0)$ is contained in $\BB r U$. 
For $\ep > 0$, let $F_{\BB r}^\ep$ be the event that this is the case and for each $z\in \BB r U$, there exists $r \in [\ep^2 \BB r , \ep\BB r]$ and $w \in B_{r/2}(z)$ such that $E_r(w)$ occurs. 
By Lemma~\ref{lem-attained-union}, if~\eqref{item-attained-bad} holds then 
\eqb \label{eqn-attained-reg-event}
\BB P[F_{\BB r}^\ep] \geq 1 - \Kopt/2 - o_\ep(1) ,
\eqe 
where the rate of convergence of the $o_\ep(1)$ depends only on $U$, hence only on $\beta$ and the law of $D_h$. 

We henceforth assume that $F_{\BB r}^\ep$ occurs. We will show that if $\ep$ is small enough, then~\eqref{eqn-lower-ratio} holds.  
Let $\BB z,\BB w \in B_{\BB r}(0)  $ and let $P : [0,D_h(\BB z , \BB w )] \rta \BB C$ be a $D_h$-geodesic from $\BB z$ to $\BB w$. 
We assume that 
\eqb \label{eqn-pt-ep-assume}
\ep \leq \frac14 \Kopt  \quad \text{and} \quad |\BB z -\BB w| \geq \Kopt \BB r .
\eqe 
The reason why we can make these assumptions is that $\ep_0$ is allowed to depend on $\Kopt$ and~\eqref{eqn-lower-ratio} holds vacuously if $|\BB z-\BB w| \leq \Kopt \BB r$. 
We will inductively define a sequence of times 
\eqbn
0 = t_0 < s_1 < t_1 < s_2 < t_2 < \dots < s_J < t_J \leq D_h(\BB z,\BB w) .
\eqen
See Figure~\ref{fig-attained-def} for an illustration.

\begin{figure}[ht!]
\begin{center}
\includegraphics[width=.6\textwidth]{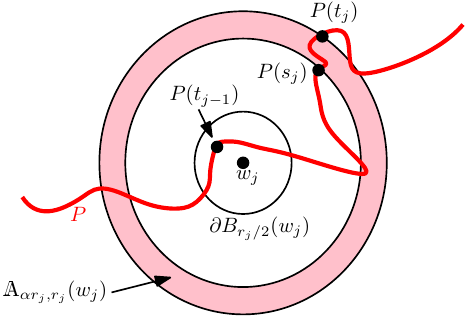}   
\caption{\label{fig-attained-def} Illustration of the definition of the times $s_j$ and $t_j$ and the balls $B_{r_j}(w_j)$.  
}
\end{center}
\end{figure}

Let $t_0  = 0$.
Inductively, assume that $j\in\BB N $ and $t_{j-1}$ has been defined.
By the definition of $F_{\BB r}^\ep$, we have $P(t_{j-1}) \in \BB r U$ and there exists $r_j \in [\ep^2 \BB r , \ep \BB r]$ and $w_j \in B_{r_j/2}(P(t_{j-1}))$ such that $E_{r_j}(w_j)$ occurs. Fix (in some arbitrary manner) a particular choice of $r_j$ and $w_j$ with these properties. 

Let $t_{j}$ be the first time $t\geq t_{j-1}$ for which $P(t) \notin B_{r_j}(w_j)$, or let $t_j = D_h(\BB z,\BB w)$ if no such time exists.  
If $t_j < D_h(\BB z,\BB w)$, we also let $s_j$ be the last time before $t_j$ at which $P$ hits $\bdy B_{\Kann r_j}(w_j)$, so that $s_j \in [t_{j-1} , t_j]$ and $P([s_j , t_j]) \subset \ol{ \BB A}_{\Kann r_j , r_j}(w_j) $.
 
Finally, define
\allb \label{eqn-attained-index}
\ul J &:= \max\left\{ j\in \BB N : |\BB z - P(t_{j-1} )| <  2 \ep \BB r \right\}
\quad \text{and} \notag\\
\ol J &:= \min\left\{j\in\BB N : |\BB w - P(t_{j+1} )|  <  2 \ep \BB r \right\} .
\alle
The reason for the definitions of $\ul J$ and $\ol J$ is that $\BB z , \BB w \notin B_{ r_j}(w_j)$ for $j \in [\ul J , \ol J]_{\BB Z}$ (since $r_j \leq \ep \BB r$ and $P(t_j) \in B_{r_j}(w_j)$). 
Whenever $|\BB w - P(t_{j-1})| \geq  \ep \BB r$, we have $t_j  < D_h(\BB z,\BB w)$ and $|P(t_{j-1}) - P(t_j)| \leq 2 \ep \BB r$. 
Therefore,   
\eqb \label{eqn-endpoint-contain}
P(t_{\ul J}) \in B_{4\ep \BB r}(\BB z) \quad \text{and} \quad P(t_{\ol J}) \in B_{4\ep\BB r}(\BB w) .
\eqe

The most important estimate that we need for the times $s_j$ and $t_j$ is the following lemma.

\begin{lem} \label{lem-attained-increment} 
For each $j\in [\ul J , \ol J]_{\BB Z}$,  
\eqb \label{eqn-attained-increment}
\wt D_h\left( P(s_j) , P(t_j)  \right) \leq \Cmed (t_j - s_j) 
\quad \text{and} \quad
\wt D_h\left( P(t_{j-1}) , P(s_j) \right) \leq \Cupper (s_j - t_{j-1}  )  .
\eqe 
\end{lem}

The second inequality in~\eqref{eqn-attained-increment} is immediate from the definition~\eqref{eqn-bilip-def} of $\Cupper$. 
We will prove the first inequality in~\eqref{eqn-attained-increment} by applying condition~\ref{item-attained-geo} in the definition of $E_{r_j}(w_j)$ with $u = P(s_j)$ and $v = P(t_j)$. 
The following lemma will be used in conjunction with condition~\ref{item-attained-reg} in the definition of $E_{r_j}(w_j)$ to check the requirement~\eqref{eqn-attained-geo} from condition~\ref{item-attained-geo}.

\begin{figure}[ht!]
\begin{center} 
\includegraphics[width=.8\textwidth]{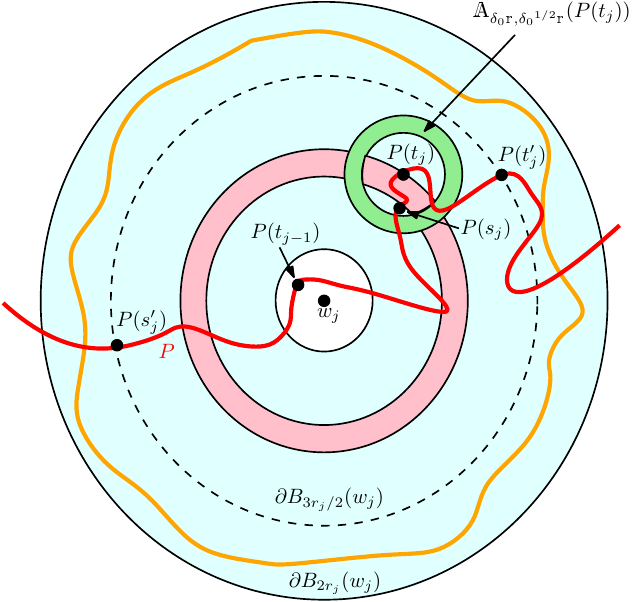}
\caption{\label{fig-attained-loop} Illustration of the proof of Lemma~\ref{lem-attained-loop}. We upper-bound $t_j - s_j$ and $D_h\left(\text{across $\BB A_{\Kep r_j , \Kep^{1/2} r_j }(P(t_j))$} \right) $ in terms of $t_j'-s_j'$, upper-bound $t_j'-s_j'$ in terms of the $D_h$-length of the orange loop, and upper-bound the $D_h$-length of the orange loop using condition~\ref{item-attained-cross} in the definition of $E_{r_j}(w_j)$. Note that the picture is not to scale. For example, in actuality the inner radius of $\BB A_{\Kep r_j , \Kep^{1/2} r_j }(P(t_j))$ is much smaller than its outer radius. 
}
\end{center}
\end{figure}

\begin{lem} \label{lem-attained-loop}
For each $j\in [\ul J ,\ol J ]_{\BB Z}$, we have 
\eqb \label{eqn-attained-time}
t_j - s_j \leq (1-\Kann)^{-1}  r_j^{\xi Q} e^{\xi h_{r_j}(w_j)} 
\eqe
and
\eqb \label{eqn-attained-loop}
D_h\left(\text{across $\BB A_{\Kep r_j , \Kep^{1/2} r_j }(P(t_j))$} \right)  
\leq   \frac{\Clower}{2\Cupper}  \delta^{-\geoExp}  D_h\left(  \BB A_{3r_j /4 , 3r_j /2}(z) ,  \bdy \BB A_{r_j /2,2r_j}(z) \right)  .
\eqe 
\end{lem}
\begin{proof}
See Figure~\ref{fig-attained-loop} for an illustration. 
Let $s_j'$ be the first time that $P$ enters $B_{3r_j/2}(w_j)$ and let $t_j'$ be the last time that $P$ exits $B_{3r_j/2}(w_j)$. 
Then $s_j' < s_j < t_j < t_j'$. 
The definitions~\eqref{eqn-attained-index} of $\ul J$ and $\ol J$ show that the endpoints $\BB z,\BB w$ of $P$ are not in $B_{2r_j}(w_j)$, so $P$ must cross between the inner and outer boundaries of the annulus $\BB A_{3r_j/2,2r_j}(w_j)$ before time $s_j'$ and after time $t_j'$. 
By considering the segment of $P$ between two consecutive times when it hits a path around $\BB A_{3r_j/2,2r_j}(w_j)$ of near-minimal length and using the fact that $P$ is a $D_h$-geodesic, we see that
\eqb \label{eqn-attained-use-around}
t_j' - s_j' \leq D_h\left( \text{around $\bdy \BB A_{3r_j/2,2r_j}(z)$}\right) .
\eqe

By~\eqref{eqn-attained-use-around}, followed by condition~\ref{item-attained-cross} in the definition of $E_{r_j}(w_j)$, we obtain
\eqbn
t_j - s_j \leq t_j' - s_j' \leq D_h\left( \text{around $\bdy \BB A_{3r_j/2,2r_j}(z)$}\right) \leq (1-\Kann)^{-1} r_j^{\xi Q} e^{\xi h_{r_j}(w_j)} ,
\eqen
which is~\eqref{eqn-attained-time}. 

The path $P$ must cross between the inner and outer boundaries of the annulus $\BB A_{\Kep r_j, \Kep^{1/2} r_j}(P(t_j))$ between times $t_j'$ and $s_j'$. 
By~\eqref{eqn-attained-use-around} followed by condition~\ref{item-attained-cross} in the definition of $E_{r_j}(w_j)$,  
\alb
D_h\left(\text{across $\BB A_{\Kep r_j, \Kep^{1/2} r_j}(P(t_j))$} \right) 
&\leq t_j' - s_j' \notag\\ 
&\leq D_h\left( \text{around $\bdy \BB A_{3r_j/2,2r_j}(z)$}\right) \notag\\ 
&\leq   \frac{\Clower}{2\Cupper}  \delta^{-\geoExp}  D_h\left(  \BB A_{3r_j /4 , 3r_j /2}(z) ,  \bdy \BB A_{r_j/2,2r_j}(z) \right)  .
\ale
This gives~\eqref{eqn-attained-loop}. 
\end{proof}

\begin{proof}[Proof of Lemma~\ref{lem-attained-increment}] 
The second inequality in~\eqref{eqn-attained-increment} is immediate from the definition~\eqref{eqn-bilip-def} of $\Cupper$. 
To get the first inequality, we want to apply condition~\ref{item-attained-geo} in the definition of $E_{r_j}(w_j)$ to the points $u = P(s_j) \in \bdy B_{\Kann r_j}(w_j)$ and $v = P(t_j) \in \bdy B_{r_j}(w_j)$. To do this, we need to check the hypotheses of condition~\ref{item-attained-geo} in the definition of $E_{r_j}(w_j)$. 

To this end, let $\sigma_j$ be the last time before $s_j$ at which $P$ enters $\BB A_{r_j/2,2r_j}(w_j)$ and let $\tau_j$ be the first time after $t_j$ at which $P$ exits $\BB A_{r_j/2,2r_j}(w_j)$.  
Then $P|_{[\sigma_j ,\tau_j]}$ is a $D_h(\cdot,\cdot;\ol{\BB A}_{r_j/2,2r_j}(w_j))$-geodesic between two points of $\bdy\BB A_{r_j/2,2r_j}(w_j)$ and $\sigma_j < s_j < t_j < \tau_j$. 
By the definitions of $s_j$ and $t_j$, we have
\eqb \label{eqn-attained-check-subset}
P|_{[s_j,t_j]} \subset \ol{\BB A}_{\Kann r_j , r_j}(w_j) .
\eqe
By~\eqref{eqn-attained-check-subset} and condition~\ref{item-attained-leb} in the definition of $E_{r_j}(w_j)$, 
\eqb \label{eqn-attained-check-eucl}
\left(\text{Euclidean diameter of $P([s_j,t_j])$}\right)    \leq \Kep r_j  \leq \frac{r_j}{100} .
\eqe
 
By condition~\ref{item-attained-reg} in the definition of $E_{r_j}(w_j)$, 
\allb \label{eqn-attained-use-reg}
D_h\left(\text{around $\BB A_{\delta r_j ,\delta^{1/2} r_j}(P(t_j) )$} \right) 
&\leq \delta^\geoExp D_h\left(\text{across $\BB A_{\delta  r_j , \delta^{1/2} r_j}(P(t_j) )$} \right) , \notag\\
&\qquad\qquad \forall \delta \in (0,\Kep]  ;
\alle
and the same is true with $P(s_j)$ in place of $P(t_j)$.  
By definition, $|P(t_j) - P(s_j)| \geq (1-\Kann) r_j$ so for each $\delta \in (0,(1-\Kann)^2]$, the path $P|_{[s_j,t_j]}$ crosses between the inner and outer boundaries of the annuli $\BB A_{\delta r_j ,\delta^{1/2} r_j}(P(s_j))$ and $\BB A_{\delta r_j ,\delta^{1/2} r_j}(P(t_j))$. 
Since $1-\Kann < \Kep$,~\eqref{eqn-attained-use-reg} implies that 
\allb \label{eqn-attained-check-reg0}
D_h\left(\text{around $\BB A_{\delta r_j ,\delta^{1/2} r_j}(P(t_j) )$} \right) 
\leq \delta^\geoExp (t_j - s_j)  
&= \delta^\geoExp D_h(P(s_j) , P(t_j)) , \notag\\
&\quad \forall \delta \in (0,(1-\Kann)^2 ]  ;
\alle
and the same is true with $P(s_j)$ in place of $P(t_j)$ on the left side. 

By~\eqref{eqn-attained-check-reg0}, for each $\zeta > 0$ and each $\delta  \in (0,(1-\Kann)^2]$ we can find a path $\pi_\delta$  in $\BB A_{\delta r_j ,\delta^{1/2} r_j}(P(t_j) )$ which disconnects the inner and outer boundaries and has $D_h$-length at most $( \delta^\geoExp + \zeta) (t_j - s_j)$. If we let $a_\delta$ (resp.\ $b_\delta$) be the first (resp.\ last) time that $P$ hits $\pi_\delta$, then $a_\delta \leq t_j \leq  b_\delta$ and since $P$ is a $D_h$-geodesic we must have $b_\delta - a_\delta \leq \op{len}(\pi_\delta ; D_h)$. 
Furthermore, the segment $P|_{[t_j , b_\delta]}$ hits $\bdy B_{\delta \BB r}(P(t_j))$, so for each $\delta \in (0,(1-\Kann)^2 ] $, 
\allb \label{eqn-attained-check-reg1}
D_h\left( P(t_j) , \bdy B_{\delta \BB r}(P(t_j)) \right) 
\leq b_\delta - t_j
\leq b_\delta - a_\delta
\leq \op{len}(\pi_\delta ; D_h) 
\leq ( \delta^\geoExp + \zeta) (t_j - s_j) .
\alle
Sending $\zeta \rta 0$ and recalling that $P$ is a $D_h$-geodesic gives
\allb \label{eqn-attained-check-reg}
D_h\left( P(t_j) , \bdy B_{\delta \BB r}(P(t_j)) \right)  
\leq  \delta^\geoExp  D_h\left(P(s_j) , P(t_j) \right)  ,\quad \forall \delta \in (0,(1-\Kann)^2 ] .
\alle
We similarly obtain~\eqref{eqn-attained-check-reg} with the roles of $P(s_j)$ and $P(t_j)$ interchanged.

Finally, by Lemma~\ref{lem-attained-loop} and~\eqref{eqn-attained-use-reg} (with $\delta=\Kep$),
\eqb \label{eqn-attained-check-loop}
D_h\left(\text{around $\BB A_{\Kep r_j ,\Kep^{1/2} r_j}(P(t_j) )$} \right)
\leq \frac{\Clower}{2\Cupper} D_h\left(  \BB A_{3r_j /4 , 3r_j /2}(z) ,  \bdy \BB A_{r_j /2,2r_j}(z) \right) .
\eqe

We are now ready to explain why we can apply condition~\ref{item-attained-geo} with $u=P(s_j)$ and $v = P(t_j)$. 
The hypothesis~\eqref{item-H-subset} follows from~\eqref{eqn-attained-check-subset}. 
The condition~\eqref{eqn-attained-geo} and the hypothesis~\eqref{item-H-diam} for the Euclidean diameter of $P|_{[s_j,t_j]}$ follow from~\eqref{eqn-attained-check-eucl}.
The needed upper bound~\eqref{item-H-dist} for $D_h(P(s_j) , P(t_j))$ follows from~\eqref{eqn-attained-time}
The hypothesis~\eqref{item-H-loop} follows from~\eqref{eqn-attained-check-reg0} and~\eqref{eqn-attained-check-reg}. 
The hypothesis~\eqref{eqn-annulus-geo-loop'} follows from~\eqref{eqn-attained-check-loop}.
Hence we can apply condition~\ref{item-attained-geo} in the definition of $E_{r_j}(w_j)$ to $P|_{[s_j,t_j]}$ to get $\wt D_h\left( P(s_j) , P(t_j)  \right) \leq \Cmed (t_j - s_j) $, as required.
\end{proof}

The last lemma we need for the proof of Proposition~\ref{prop-attained-bad} tells us that the time intervals $[s_j,t_j]$ occupy a positive fraction of the total $D_h$-length of the path $P$. 

\begin{lem} \label{lem-attained-compare} 
For each $j\in [\ul J , \ol J]_{\BB Z}$,  
\eqb \label{eqn-attained-compare}
s_j - t_{j-1} \leq  \frac{\Karound}{\Karound+1} (t_j - t_{j-1}) .
\eqe
\end{lem}
\begin{proof}
By the definition of $r_j$ and the definitions of $\ul J$ and $\ol J$ in~\eqref{eqn-attained-index}, 
for $j \in [\ul J , \ol J]_{\BB Z}$ we have $r_j \leq \ep \BB r$ and $|P(t_j) -\BB z| \wedge |P(t_j) - \BB w| \geq 2\ep \BB r$.
Since $P(t_{j-1}) \in B_{r_j/2}(w_j)$ and $P(s_j) \in \bdy B_{\Kann r_j}(w_j)$, we infer that the $D_h$-geodesic $P$ must cross between the inner and outer boundaries of the annulus $\BB A_{\Kann r_j,r_j}(w_j)$ at least once before time $t_{j-1}$ and at least once after time $s_j$. 
By condition~\ref{item-attained-dist} in the definition of $ E_{r_j}(w_j)$, there is a path  in $\BB A_{\Kann r_j,r_j}(w_j)$ disconnecting the inner and outer boundaries of this annulus with $D_h$-length arbitrarily close to $\Karound D_h\left(\bdy B_{\Kann r_j}(w_j) , \bdy B_{  r_j}(w_j) \right)$. 
The geodesic $P$ must hit this path at least once before time $t_{j-1}$ and at least once after time $s_j$. 
Since $P$ is a $D_h$-geodesic and $P(s_j) \in \bdy B_{\Kann r_j}(w_j)$, $P(t_j) \in \bdy B_{ r_j}(w_j)$, it follows that 
\eqbn 
s_j - t_{j-1}  \leq \Karound D_h\left(\bdy B_{\Kann r_j}(w_j) , \bdy B_{  r_j}(w_j) \right) \leq \Karound (t_j - s_j) .
\eqen
Adding $\Karound (s_j - t_{j-1})$ to both sides of this inequality, then dividing by $\Karound+1$, gives~\eqref{eqn-attained-compare}.
\end{proof}

\begin{proof}[Proof of Proposition~\ref{prop-attained-bad}]
Our above estimates show that if the event $F_{\BB r}^\ep$ of~\eqref{eqn-attained-reg-event} occurs, then we have the following string of inequalities:
\allb \label{eqn-attained-sum}
&\wt D_h\left( B_{4 \ep \BB r}(\BB z) , B_{4 \ep \BB r}(\BB w)  \right) \notag\\
&\qquad \leq \sum_{j=\ul J+1}^{\ol J} \left[ \wt D_h\left( P(t_{j-1}) , P(s_j) \right) +   \wt D_h\left( P(s_j) , P(t_j)  \right)  \right] \quad  \text{(by \eqref{eqn-endpoint-contain})} \notag \\
&\qquad \leq \sum_{j=\ul J+1}^{\ol J} \left[ \Cupper( s_j - t_{j-1})  +   \Cmed (t_j - s_j)  \right] \quad  \text{(by Lemma~\ref{lem-attained-increment})} \notag \\
&\qquad = \sum_{j=\ul J+1}^{\ol J} \left[   \Cmed (t_j - t_{j-1} )  + (\Cupper - \Cmed) (s_j - t_{j-1}) \right] \notag \\
&\qquad\leq \left( \Cmed  + \frac{\Karound}{\Karound+1}(\Cupper - \Cmed) \right) \sum_{j=\ul J+1}^{\ol J}  (t_j - t_{j-1} )  \quad \text{(by Lemma~\ref{lem-attained-compare})} \notag \\
&\qquad\leq \left( \Cmed  + \frac{\Karound}{\Karound+1}(\Cupper - \Cmed) \right) D_h(\BB z,\BB w) \quad \text{(since $P$ is a $D_h$-geodesic)} \notag\\
&\qquad\leq \Cmed'  D_h(\BB z,\BB w) \quad \text{(by~\eqref{eqn-C''-choice})}   .
\alle

By~\eqref{eqn-attained-reg-event}, we have $\BB P[F_{\BB r}^\ep] \geq 1 - \Kopt /2 - o_\ep(1)$, with the rate of convergence of the $o_\ep(1)$ uniform in the choice of $\BB r$. 
Hence we can choose $\ep_0 = \ep_0( \Kopt,\Cmed )  > 0$ small enough so that $4\ep_0 \leq \Kopt$ and $\BB P[F_{\BB r}^\ep]  >  1 - \Kopt$ for each $\ep \in (0,\ep_0]$. By~\eqref{eqn-attained-sum} and Definition~\ref{def-opt-event} of $G_{\BB r}(\Kopt , \Cmed')$, we see that for $\ep\in (0,\ep_0]$, the condition~\eqref{item-attained-bad} implies that $\BB P[G_{\BB r}(\Kopt , \Cmed')]  < \Kopt$, as required. 
\end{proof}

\section{The core argument}
\label{sec-counting}

\subsection{Properties of events and bump functions} 
\label{sec-counting-setup}

In this section, we will assume the existence of events and smooth bump functions which satisfy certain conditions. We will then use these objects to prove Theorem~\ref{thm-weak-uniqueness}. The objects will be constructed in Section~\ref{sec-construction} and are illustrated in Figure~\ref{fig-counting-def}.  

To state the conditions which our events and bump functions need to satisfy, we define the optimal upper and lower bi-Lipschitz constants $\Cupper$ and $\Clower$ as in Section~\ref{sec-attained} and we set
\eqb \label{eqn-Cmid-choice}
\Cmid  := \frac{\Clower + \Cupper}{2}  ,
\eqe
which belongs to $(\Clower,\Cupper)$ if $\Clower < \Cupper$. 

We will consider a set of admissible radii $\mcl R \subset (0,1)$ which is required to satisfy
\eqb \label{eqn-admissible-radii}
r' / r \geq 8,\quad \forall r,r' \in \mcl R \quad \text{such that} \quad r' > r .
\eqe 
The reason for restricting attention to a set of radii as in~\eqref{eqn-admissible-radii} is that in Section~\ref{sec-construction}, we will need to use Proposition~\ref{prop-attained-good'} in order to construct our events. 

We also fix a number $\BB p \in (0,1)$, which we will choose later in a manner depending only on $D_h$ and $\wt D_h$ (the parameter $\BB p$ is chosen in Lemma~\ref{lem-main-extra} below). 

Finally, we fix numbers $\Cmax,\Cacross,\Caround,\Crn,\Ctime,\Cinc ,  \Ctube  >0$, which we require to satisfy the relations 
\eqb  \label{eqn-parameter-relation}
\Caround > \Cacross \quad \text{and} \quad \Cacross - 4 e^{-\xi \Cmax} \Ctube > \frac{ 2\Caround }{ \Cacross } \Ctime . 
\eqe
We henceforth refer to these numbers as the \emph{parameters}. Most constants in our proofs will be allowed to depend on the parameters. 
The parameters will be chosen in Section~\ref{sec-construction}, in a manner depending only on $\BB p$ and the laws of $D_h$ and $\wt D_h$ (see also Proposition~\ref{prop-objects-exist}). 

Throughout this section, we will assume that for each $r \in\mcl R$ and each $z\in\BB C$, we have defined the following objects.
\begin{itemize}
\item An event $\Er_{z,r} = \Er_{z,r}(h) $ such that $\Er_{z,r}$ is a.s.\ determined by $h |_{\ol{\BB A}_{r,4r}(z)}$, viewed modulo additive constant (recall~\eqref{eqn-closed-restrict}), $\BB P[\Er_{z,r}] \geq \BB p$, and $\Er_{z,r}$ satisfies the three hypotheses listed just below. 
\item Deterministic open sets $\Ur_{z,r} , \Vr_{z,r} \subset \BB A_{r,3r}(z)$, each of which has the topology of an open Euclidean annulus and disconnects the inner and outer boundaries of $\BB A_{r,3r}(z)$, such that $\ol \Ur_{z,r} \subset \Vr_{z,r}$ and $\ol \Vr_{z,r} \subset \BB A_{r,3r}(z)$. 
\item A deterministic smooth function $\fr_{z,r} : \BB C\rta [0,\Cmax ]  $ such that $\fr_{z,r} \equiv \Cmax $ on $\Ur_{z,r}$ and $\fr_{z,r} \equiv 0$ on $\BB C\setminus \Vr_{z,r}$. 
\end{itemize} 
To state the needed hypotheses for the event $\Er_{z,r}$, we make the following definition.  

\begin{defn} \label{def-excursion}
Let $P : [0,T] \rta \BB C$ be a path and let $O,V \subset \BB C$ be open sets with $\ol V\subset O$. 
A \emph{$(O,V)$-excursion} of $P$ is a 4-tuple of times $(\tau' , \tau , \sigma,\sigma')$ such that 
\eqbn
P(\tau') , P(\sigma') \in \bdy O, \quad P((\tau',\sigma')) \subset O ,
\eqen
$\tau$ is the first time after $\tau'$ that $P$ enters $\ol V$, and $\sigma$ is the last time before $\sigma'$ at which $P$ exits $\ol V$. 
\end{defn}

An $(O,V)$ excursion is illustrated in Figure~\ref{fig-counting-def}. 
We assume that on the event $\Er_{z,r}$, the following is true. 
\begin{enumerate}[A.]
\item \label{item-Ehyp-dist} We have 
\alb
D_h(\Vr_{z,r} , \bdy \BB A_{r,3r}(z) ) &\geq \Cacross r^{\xi Q} e^{\xi h_r(z)} ,\notag\\ 
 D_h(\text{around $\BB A_{3r,4r}(z)$} ) &\leq \Caround r^{\xi Q} e^{\xi h_r(z)} ,\quad \text{and} \notag\\
  D_h(\text{around $\Ur_{z,r}$} ) &\leq \Ctube r^{\xi Q} e^{\xi h_r(z)}  .
\ale
\item \label{item-Ehyp-rn} The Radon-Nikodym derivative of the law of $h + \fr_{z,r}$ w.r.t.\ the law of $h$, with both distributions viewed modulo additive constant, is bounded above by $\Crn$ and below by $1/\Crn$. 
\item \label{item-Ehyp-inc} Let $P' : [0,T] \rta \BB C$ be a $D_{h - \fr_{z,r} }$-geodesic between two points which are not in $B_{4r}(z)$, parametrized by its $D_{h-\fr_{z,r} }$-length. 
Assume that (in the terminology of Definition~\ref{def-excursion}), there is a $(B_{4r}(z) , \Vr_{z,r})$-excursion $(\tau',\tau,\sigma,\sigma')$ for $P'$ such that
\eqb \label{eqn-Ehyp-exc}
  D_h\left( P'(\tau) , P'(\sigma)   ; B_{4r}(z)  \right) \geq \Ctime r^{\xi Q} e^{\xi h_r(z)}  . 
\eqe 
Then there are times $\tau \leq s < t \leq \sigma$ such that 
\eqb  \label{eqn-Ehyp-inc}
t-s \geq \Cinc r^{\xi Q} e^{\xi h_r(z)}   \quad \text{and} \quad 
\wt D_{h-\fr_{z,r} }\left(P'(s) , P'(t)  ; B_{4r}(z)  \right) \leq \Cmid (t-s) .
\eqe   
\end{enumerate}

Constructing objects which satisfy the above conditions (especially hypothesis~\ref{item-Ehyp-inc}) will require a lot of work. The proof of the following proposition will occupy all of Section~\ref{sec-construction}. 

\begin{prop} \label{prop-objects-exist}
Assume that $\Clower < \Cupper$. For each $\BB p \in (0,1)$, there exist $\Cmid'  \in (\Clower , \Cmid)$ and a set of radii $\mcl R $ as in~\eqref{eqn-admissible-radii}, depending only on $\BB p$ and the laws of $D_h$ and $\wt D_h$, with the following properties.
\begin{itemize}
\item There is a choice of parameters depending only on $\BB p$ and the laws of $D_h$ and $\wt D_h$, such that for each $r\in\mcl R$ and each $z\in\BB C$, there exist an event $\Er_{z,r}$, open sets $\Ur_{z,r} , \Vr_{z,r}$, and a function $\fr_{z,r}$ satisfying the above hypotheses. 
\item For each $\wt \Kopt > 0$, there exists $\ep_0  > 0$, depending only on $\BB p$, $\wt\Kopt$, and the laws of $D_h$ and $\wt D_h$, such that the following holds for each $\ep \in (0,\ep_0]$. If $\BB r > 0$ and that the event of Definition~\ref{def-opt-event'} satisfies $\BB P[\wt G_{\BB r}(\wt\Kopt , \Cmid')] \geq \wt\Kopt$, then the cardinality of $\mcl R \cap  [\ep^2 \BB r , \ep \BB r]$ is at least $ \frac58 \log_8 \ep^{-1}$. 
\end{itemize}
\end{prop} 
 
\begin{figure}[ht!]
\begin{center}
\includegraphics[width=.8\textwidth]{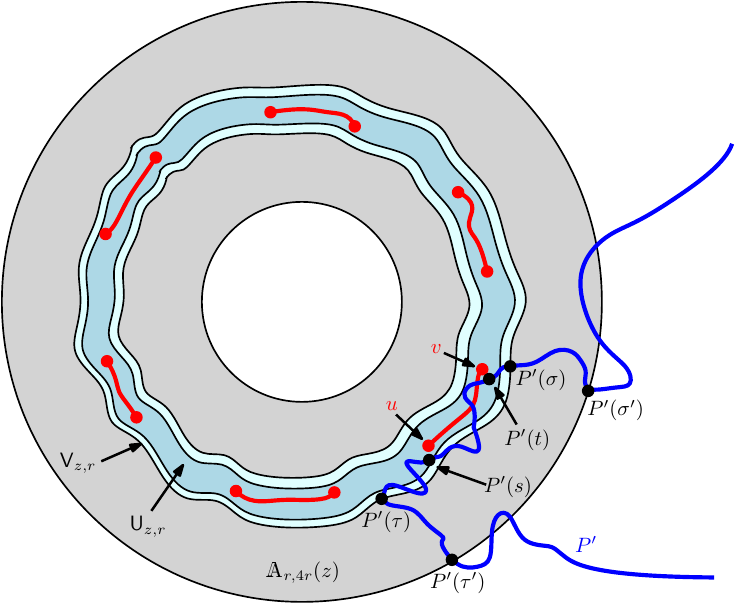} 
\caption{\label{fig-counting-def} 
Illustration of the objects defined in Section~\ref{sec-counting-setup}. The bump function $\fr_{z,r}$ is supported on $\Vr_{z,r}$ and identically equal to $\Cmax$ on $\Ur_{z,r}$. 
The figure shows a $D_{h-\fr_{z,r}}$-geodesic $P'$ (blue) and a $(B_{4r}(z) , \Vr_{z,r})$-excursion $( \tau' ,\tau   , \sigma ,\sigma' )$ for $P'$.  
On the event $\Er_{z,r}$, there are many ``good" pairs of points $u,v\in \Ur_{z,r}$ such that $\wt D_h(u,v) \leq \Cmid D_h(u,v)$ and there is a $\wt D_h$-geodesic from $u$ to $v$ which is contained in $\Ur_{z,r}$ (several such geodesics are shown in red). We obtain hypothesis~\ref{item-Ehyp-inc} for $\Er_{z,r}$ by forcing $P'$ to get close to $u$ and $v$ for one such ``good" pair of points. 
}
\end{center}
\end{figure}

The proof of Proposition~\ref{prop-objects-exist} in Section~\ref{sec-construction} will be via an intricate explicit construction. 
To give the reader some intuition, we will now explain roughly what is involved in this construction, without any quantitative estimates. 
The reader may want to look at Figure~\ref{fig-counting-def} while reading the explanation. 

The set $\Ur_{z,r}$ where $\fr_{z,r}$ attains its maximal possible value will be a long narrow ``tube" which disconnects the inner and outer boundaries of $\BB A_{r,3r}(z)$ and is contained in a small Euclidean neighborhood of $\bdy B_{2r}(z)$. The set $\Vr_{z,r}$ where $\fr_{z,r}$ is supported will be a slightly larger tube containing $\Ur_{z,r}$. 
The event $\Er_{z,r}$ corresponds, roughly speaking, to the event that there are many ``good" pairs of non-singular points $u,v \in \Ur_{z,r}$ with the following properties (plus a long list of regularity conditions): 
\begin{itemize}
\item $\wt D_h(u,v) \leq \Cmid_0 D_h(u,v)$, where $\Cmid_0 \in (\Clower,\Cmid)$ is fixed. 
\item $|u-v|$ is bounded below by a constant times $r$. 
\item There is a $\wt D_h$-geodesic from $u$ to $v$ which is contained in $\Ur_{z,r}$.  
\end{itemize} 

Hypotheses~\ref{item-Ehyp-dist} and~\ref{item-Ehyp-rn} for $\Er_{z,r}$ will be immediate consequences of the regularity conditions in the definition of $\Er_{z,r}$. Hypothesis~\ref{item-Ehyp-inc} will be obtained as follows. Suppose that $P'$ is a $D_{h-\fr_{z,r}}$-geodesic as in hypothesis~\ref{item-Ehyp-inc}. Since the bump function $\fr_{z,r}$ is very large on $\Ur_{z,r}$, we infer that if $x,y\in \Vr_{z,r}$, then the $D_{h-\fr_{z,r}}$-length of any path between $x$ and $y$ which spends a lot of time outside of $\Ur_{z,r}$ is much greater than the $D_{h-\fr_{z,r}}$-length of a path between $x$ and $y$ which spends most of its time in $\Ur_{z,r}$. By applying this with $x = P'(\tau)$ and $y = P'(\sigma)$, we find that $P'|_{[\tau,\sigma]}$ has to spend most of its time in $\Ur_{z,r}$. 

This will allow us to find a ``good" pair of points $u,v\in \Ur_{z,r}$ as above such that $P'|_{[\tau,\sigma]}$ gets very $D_{h-\fr_{z,r}}$-close to each of $u$ and $v$. Since the $\wt D_h$-geodesic between $u$ and $v$ is contained in $\Ur_{z,r}$ and $\fr_{z,r}$ attains its maximal possible value on $\Ur_{z,r}$, subtracting $\fr_{z,r}$ from $h$ reduces $\wt D_h(u,v)$ by at least as much as $D_h(u,v)$. Consequently, one has $\wt D_{h-\fr_{z,r}}(u,v) \leq \Cmid_0 D_{h-\fr_{z,r}}(u,v)$. We will then obtain~\eqref{eqn-Ehyp-inc} by choosing $s$ and $t$ such that $P'(s)$ and $P'(t)$ are close to $u$ and $v$, respectively, and applying the triangle inequality. 

In order to produce lots of ``good" pairs of points $u,v\in \Ur_{z,r}$, we will apply Proposition~\ref{prop-attained-good'} together with a local independence argument based on Lemma~\ref{lem-spatial-ind} (to upgrade from a single pair of points with positive probability to many pairs of points with high probability). 
This application of Proposition~\ref{prop-attained-good'} is the reason why we need to assume that $\BB P[\wt G_{\BB r}(\wt\Kopt , \Cmid')] \geq \wt\Kopt$ in the second part of Proposition~\ref{prop-objects-exist}; and why we need to restrict to a set of admissible radii $\mcl R  $, instead of defining our events for every $r >0$.

\subsection{Estimate for ratios of $D_h$ and $\wt D_h$ distances} 
\label{sec-counting-main}
 
We now state the main estimate which we will prove using the events $\Er_{z,r}$. 
In particular, we will show that the probability of a certain ``bad" event, which we now define, is small. 
For $\BB r > 0$, $\ep > 0$, and disjoint compact sets $K_1,K_2 \subset B_{2\BB r}(0)$, let $  \mcl G_{\BB r}^\ep =  \mcl G_{\BB r}^\ep( K_1,K_2 )$ be the event that the following is true.
\begin{enumerate}
\item $\wt D_h(K_1,K_2) \geq \Cupper D_h(K_1,K_2) - \frac12 \ep^{2\xi(Q+3)} \BB r^{\xi Q} e^{\xi h_{\BB r}(0)}$. \label{item-main-bad} 
\item \label{item-main-exponent} For each $z\in B_{3\BB r}(0)$ and each $r\in [\ep^2 \BB r , \ep \BB r] \cap \mcl R$, we have 
\eqbn
r^{\xi Q} e^{\xi h_r(z)} \in \left[\ep^{2\xi(Q+3)} \BB r^{\xi Q} e^{\xi h_{\BB r}(0)} , \ep^{ \xi(Q-3)} \BB r^{\xi Q} e^{\xi h_{\BB r}(0)} \right] . 
\eqen
\item For each $z \in B_{3\BB r}(0) $, there exists $r \in \mcl R \cap [\ep^2 \BB r ,\ep \BB r]$ and $w \in \left( \frac{r}{100} \BB Z^2 \right) \cap B_{r/25}(z)$ such that $\Er_{w,r}$ occurs. \label{item-main-cover}
\end{enumerate}
The most important condition in the definition of $\mcl G_{\BB r}^\ep$ is condition~\ref{item-main-bad}. We want to show that if $\Clower < \Cupper$, then this condition is extremely unlikely. The motivation for this is that it will eventually be used in Section~\ref{sec-counting-conclusion} to derive a contradiction to Proposition~\ref{prop-geo-annulus-prob}. Indeed, Proposition~\ref{prop-geo-annulus-prob} gives a \emph{lower} bound for the probability that there exist points $u,v \in \ol B_{\BB r}(0)$ satisfying certain conditions such that $\wt D_h(u,v)$ is ``close" to $\Cupper D_h(u,v)$. We will show that this lower bound is incompatible with our upper bound for the probability of condition~\ref{item-main-bad} in the definition of $\mcl G_{\BB r}^\ep$.  

Conditions~\ref{item-main-exponent} and~\ref{item-main-cover} in the definition of $\mcl G_{\BB r}^\ep$ are global regularity conditions. We will show in Lemma~\ref{lem-main-extra} below that Proposition~\ref{prop-objects-exist} implies that these two conditions occur with high probability. This, in turn, means that an upper bound for $\BB P[\mcl G_{\BB r}^\ep]$ implies an upper bound for the probability of condition~\ref{item-main-bad}. 
The next three subsections are devoted to the proof of the following proposition. 

\newcommand{\Ceucl}{\eta}

\begin{prop} \label{prop-counting}
Assume that $\Clower < \Cupper$ and we have constructed a set of admissible radii $\mcl R$ as in~\eqref{eqn-admissible-radii} and events $\Er_{z,r}$, sets $\Ur_{z,r}$ and $\Vr_{z,r}$, and bump functions $\fr_{z,r}$ for $z\in\BB C$ and $r\in\mcl R$ which satisfy the conditions of Section~\ref{sec-counting-setup}. Let $\Ceucl  \in (0,1)$ and $\BB r>0$. Also let $K_1,K_2 \subset B_{2\BB r}(0)$ be disjoint compact sets such that $\op{dist}(K_1,K_2) \geq \Ceucl \BB r$ and $\op{dist}(K_1,\bdy B_{\BB r}(0)) \geq \Ceucl \BB r$, where $\op{dist}$ denotes Euclidean distance.\footnote{\label{footnote-eucl} The reason why we require that  $\op{dist}(K_1,\bdy B_{\BB r}(0)) \geq \Ceucl \BB r $ in Proposition~\ref{prop-counting} is as follows.
Our events involve the circle average $h_{\BB r}(0)$. We only want to add to or subtract from $h$ functions of the form $\fr_{z,r}$ whose supports are disjoint from $\bdy B_{\BB r}(0)$, so that adding or subtracting $\fr_{z,r}$ does not change $h_{\BB r}(0)$.  
The condition that $\op{dist}(K_1,\bdy B_{\BB r}(0)) \geq \Ceucl \BB r$ ensures that there is a segment of the $D_h$-geodesic from $K_1$ to $K_2$ of Euclidean length at least $\Ceucl \BB r$ which is disjoint from $\bdy B_{\BB r}(0)$. We will eventually choose to subtract functions $\fr_{z,r}$ whose supports are close to such a segment, see the proof of Proposition~\ref{prop-choices} at the end of Section~\ref{sec-counting-choices}.
}
 Then
\eqb \label{eqn-counting}
\BB P\left[ \mcl G_{\BB r}^\ep(K_1,K_2)  \right] = O_\ep(\ep^\mu) ,\quad \forall \mu > 0 
\eqe 
with the implicit constant in the $O_\ep(\cdot)$ depending only on $\mu , \Ceucl$, and the parameters (not on $\BB r, K_1,K_2$).
\end{prop}

It is crucial for our purposes that the implicit constant in the $O_\ep(\cdot)$ in~\eqref{eqn-counting} does not depend on $\BB r , K_1,K_2$. This is because we will eventually take $K_1$ and $K_2$ to be Euclidean balls whose radii are a power of $\ep$ times $\BB r$ (see Lemma~\ref{lem-end-union}). 
Proposition~\ref{prop-objects-exist} is not needed for the proof of Proposition~\ref{prop-counting}. Rather, all we need is the statement that $\Er_{z,r}, \Ur_{z,r} , \Vr_{z,r} $, and $\fr_{z,r}$ exist and satisfy the required properties for each $r\in\mcl R$ (we do not care how large $\mcl R$ is). 
Proposition~\ref{prop-objects-exist} is just needed to check that the auxiliary condition~\ref{item-main-cover} in the definition $\mcl G_{\BB r}^\ep$ occurs with high probability. 

We will now explain how to prove Proposition~\ref{prop-counting} conditional on two propositions (Propositions~\ref{prop-choices} and~\ref{prop-card}) whose proofs will occupy most of this section. The proof will be based on counting the number of events of a certain type which occur. Let us now define these events. 
 
Assume that $\Clower < \Cupper$. Also fix $\BB r > 0$ and disjoint compact sets $K_1,K_2 \subset B_{2\BB r}(0)$. 
For $r\in \mcl R$ (which we will eventually take to be much smaller than $\BB r$), let $\mcl Z_r = \mcl Z_r^{\BB r}(K_1,K_2)$ be the set of non-empty subsets $Z \subset \frac{r}{100} \BB Z^2$ such that\footnote{\label{footnote-circle}
The reason why we require that $B_{4r}(z) \cap  \bdy B_{\BB r}(0)  = \emptyset$ in~\eqref{eqn-pt-set} is to ensure that adding or subtracting the function $\fr_{z,r}$ for $z\in Z$ (which is supported on $B_{4r}(z)$) does not change the circle average $h_{\BB r}(0)$ (c.f.\ Footnote~\ref{footnote-eucl}). This fact is used in the proof of Lemma~\ref{lem-event-good} below. 
}
\allb \label{eqn-pt-set}
B_{4r}(z) \cap B_{4r}(z') = \emptyset \quad \text{and} \quad   
&B_{4r}(z) \cap \left(K_1\cup K_2 \cup \bdy B_{\BB r}(0)\right) = \emptyset ,\notag\\
&\qquad\qquad \text{$\forall$ distinct $z,z' \in Z$}  .
\alle
For a set $Z \in \mcl Z_r$, we define 
\eqbn
\fr_{Z,r} = \sum_{z \in Z } \fr_{z,r} .
\eqen 

By Lemma~\ref{lem-geo-unique}, a.s.\ there is a unique $D_h$-geodesic from $K_1$ to $K_2$. 
Since the laws of $h$ and $h-\fr_{Z,r}$ are mutually absolutely continuous~\cite[Proposition 3.4]{ig1}, for each $r\in\mcl R$ and each $Z\in \mcl Z_r$, a.s.\ there is a unique $D_{h-\fr_{Z,r}}$-geodesic from $K_1$ to $K_2$. Hence the following definition makes sense. 
For $Z \in \mcl Z_r$ and $q > 0$ we define $F_{Z,r}^{q,\BB r}  = F_{Z,r}^{q,\BB r}(h ; K_1,K_2) $ to be the event that the following is true.
\begin{enumerate} 
\item $\wt D_h(K_1,K_2) \geq \Cupper D_h(K_1,K_2) - q \BB r^{\xi Q} e^{\xi h_{\BB r}(0)}$. \label{item-ball-set-bad}
\item The event $\Er_{z,r}(h)$ occurs for each $z \in Z$. \label{item-ball-set-good}
\item \label{item-ball-set-compare} We have
\eqbn
r^{\xi Q}  e^{\xi h_r(z)} \in \left[q \BB r^{\xi Q} e^{\xi h_{\BB r}(0)}  , 2q \BB r^{\xi Q} e^{\xi h_{\BB r}(0)}\right] , \quad \forall z \in Z .
\eqen
\item For each $z \in Z$, the $D_h$-geodesic from $K_1$ to $K_2$ hits $B_r(z)$.  \label{item-ball-set-hit}
\item \label{item-ball-set-stable} For each $z \in Z$, the $D_{h-\fr_{Z,r}}$-geodesic $P_Z$ from $K_1$ to $K_2$ has a $(B_{4r}(z) , \Vr_{z,r})$-excursion $(\tau_z',\tau_z,\sigma_z , \sigma_z')$ such that
\eqbn 
  D_h\left( P_Z(\tau_z) , P_Z(\sigma_z)   ; B_{4r}(z) \right) \geq \Ctime r^{\xi Q} e^{\xi h_r(z)}  . 
\eqen 
\end{enumerate}
See Figure~\ref{fig-counting-F} for an illustration of the definition. 
Condition~\ref{item-ball-set-bad} for $F_{Z,r}^{q,\BB r}$ is closely related to the main condition~\ref{item-main-bad} in the definition of $\mcl G_{\BB r}^\ep$. 
The purpose of conditions~\ref{item-ball-set-good} and~\ref{item-ball-set-hit} is to allow us to apply our hypotheses for $\Er_{z,r}$ to study $D_h$-distances on the event $F_{Z,r}^{q,\BB r}$. Condition~\ref{item-ball-set-compare} provides up-to-constants comparisons of the ``LQG sizes" of different balls $B_r(z)$ for $z\in Z$. 
Finally, condition~\ref{item-ball-set-stable} will enable us to apply hypothesis~\ref{item-Ehyp-inc} for $\Er_{z,r}$ to each $z\in Z$.

\begin{figure}[ht!]
\begin{center}
\includegraphics[width=.85\textwidth]{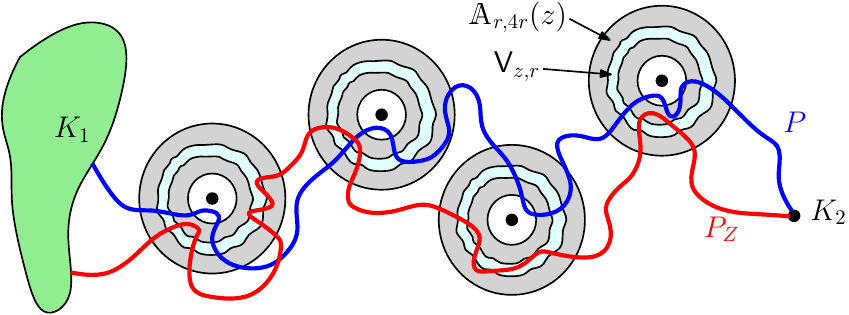} 
\caption{\label{fig-counting-F} 
Illustration of the definition of $F_{Z,r}^{q,\BB r}$. Here, we have shown $K_1$ as a non-singleton set and $K_2$ as a point, but $K_1$ and $K_2$ can be any disjoint compact sets. The set $Z$ consists of the four center points of the annuli in the figure. For each of these points, we have shown the set $\Vr_{z,r}$ (i.e., the support of $\fr_{z,r}$) in light blue and the annulus $\BB A_{r,4r}(z)$ in grey. 
On $F_{Z,r}^{q,\BB r}$, the $D_h$-geodesic from $K_1$ to $K_2$ (blue) hits each of the balls $B_r(z)$ for $z\in Z$. Moreover, the $D_{h-\fr_{Z,r}}$-geodesic from $K_1$ to $K_2$ (red) has a ``large" $(B_{4r}(z) , \Vr_{z,r})$-excursion for each $z\in Z$. 
}
\end{center}
\end{figure}

Proposition~\ref{prop-counting} will turn out to be a straightforward consequence of three estimates for the events $F_{z,r}^{q,\BB r}$, which we now state.
Our first estimate follows from a standard formula for the Radon-Nikodym derivative between the laws of $h$ and $h + \fr_{Z,r}$.

\begin{lem} \label{lem-rn}
For $r\in\mcl R$, $Z\in \mcl Z_r$, and $q > 0$, let $F_{Z,r}^{q,\BB r}(h + \fr_{Z,r})$ be the event $F_{Z,r}^{q,\BB r}(h)$ defined with $h+\fr_{Z,r}$ in place of $h$. 
For each $Z\subset\mcl Z_r$,  
\eqb \label{eqn-rn}
\Crn^{-\#Z} \BB P\left[ F_{Z,r}^{q,\BB r}(h ) \right] \leq \BB P\left[ F_{Z,r}^{q,\BB r}(h + \fr_{Z,r}) \right] \leq \Crn^{ \#Z} \BB P\left[ F_{Z,r}^{q,\BB r}(h) \right] .
\eqe 
\end{lem}
\begin{proof}
By Weyl scaling (Axiom~\ref{item-metric-f}) and the fact that $\Er_{z,r}(h)$ is a.s.\ determined by $h$, viewed modulo additive constant, we get that the event $F_{Z,r}^{q,\BB r}(h)$ is a.s.\ determined by $h$, viewed modulo additive constant. 
By a standard calculation for the GFF (see, e.g., the proof of~\cite[Proposition 3.4]{ig1}), the Radon-Nikodym derivative of the law of $h + \fr_{Z,r}$ with respect to the law of $h$, with both distributions viewed modulo additive constant, is equal to
\eqbn
\exp\left(  (h , \fr_{Z,r})_\nabla - \frac12 (\fr_{Z,r} ,\fr_{Z,r})_\nabla \right) 
\eqen
where $(f,g)_\nabla = \int_{\BB C} \nabla f(z) \cdot \nabla g(z) \, d^2 z$ denotes the Dirichlet inner product. 
Recall that each $\fr_{z,r}$ for $z\in Z$ is supported on the annulus $\BB A_{r,4r}(z)$. 
Since $Z \in \mcl Z_r$, the definition~\eqref{eqn-pt-set} shows that the balls $B_{4r}(z)$ for $z\in Z$ are disjoint. 
Hence, the random variables $(h , \fr_{Z,r})_\nabla$ are independent, so the above Radon-Nikodym derivative factors as the product
\eqb \label{eqn-rn-prod}
\prod_{z\in Z} \exp\left(   (h , \fr_{z,r})_\nabla - \frac12 (\fr_{z,r} ,\fr_{z,r})_\nabla \right) .
\eqe 
By condition~\ref{item-ball-set-good} in the definition of $F_{Z,r}^{q,\BB r}(h  )$, on this event $\Er_{z,r}(h  )$ occurs for each $z\in Z$. Consequently, hypothesis~\ref{item-Ehyp-rn} for $\Er_{z,r}(h)$ shows that on $F_{z,r}^{q,\BB r}(h )$, each of the factors in the product~\eqref{eqn-rn-prod} is bounded above by $\Crn$ and below by $\Crn^{-1}$. 
This implies~\eqref{eqn-rn}. 
\end{proof}


Our next estimate tells us that on $\mcl G_{\BB r}^\ep$, there are many choices of $Z$ for which $F_{Z,r}^{q,\BB r}(h)$ occurs.

\newcommand{\Cexp}{c_1}

\begin{prop} \label{prop-choices}
There exists $\Cexp > 0$, depending only on the parameters and $\Ceucl$, such that for each $k\in\BB N$, there exists $\ep_*  > 0$, depending only on $k$, the parameters, and $\Ceucl$, such that the following is true for each $\BB r > 0 $ and each $\ep \in (0,\ep_*]$. 
Assume that $\op{dist}(K_1,K_2) \geq \Ceucl \BB r$ and $\op{dist}(K_1,\bdy B_{\BB r}(0) ) \geq \Ceucl \BB r$.
If $\mcl G_{\BB r}^\ep(K_1,K_2) $ occurs, then there exists a random $r \in [\ep^2 \BB r , \ep \BB r]$ and a random $q \in \left[\frac12 \ep^{2\xi (Q+3)}   , \ep^{ \xi(Q-3 )} \right] \cap \{2^{-\el}\}_{\el\in\BB N} $ such that
\eqb \label{eqn-choices}
\#\left\{Z \in \mcl Z_r : \text{$\# Z \leq k$ and $F_{Z,r}^{q,\BB r}(h)$ occurs} \right\} \geq \ep^{-\Cexp k}  .
\eqe 
\end{prop}

Proposition~\ref{prop-choices} will be proven in Section~\ref{sec-counting-choices}. 
Our final estimate gives an unconditional upper bound for the number of $Z$ for which $F_{Z,r}^{q,\BB r}(h + \fr_{Z,r})$ occurs.

\newcommand{\Ccard}{C_2}

\begin{prop} \label{prop-card}
There is a constant $\Ccard > 0$, depending only on the parameters, such that the following is true. 
For each $r\in\mcl R$, each $q > 0$, and each $k\in \BB N$, a.s.\ 
\eqb \label{eqn-card}
\#\left\{ Z\in \mcl Z_r \: : \: \text{$\# Z \leq k$ and $F_{Z,r}^{q,\BB r}(h + \fr_{Z,r})$ occurs} \right\} \leq \Ccard^k .
\eqe
\end{prop}

We will give the proof of Proposition~\ref{prop-card} in Section~\ref{sec-counting-card}. 
The proofs of Propositions~\ref{prop-choices} and~\ref{prop-card} are both via elementary deterministic arguments based on the hypotheses for $\Er_{z,r}$ and the definition of $\Fr_{Z,r}^{q,\BB r}$. See the beginnings of Sections~\ref{sec-counting-choices} and~\ref{sec-counting-card} for overviews of the proofs.

Let us now explain how to deduce Proposition~\ref{prop-counting} from the above three estimates. 

\begin{proof}[Proof of Proposition~\ref{prop-counting}] 
Throughout the proof, all implicit constants are required to depend only on $\xi$ and the parameters. 
Fix $\BB r > 0$ and disjoint compact sets $K_1,K_2\subset B_{2\BB r}(0)$ such that $\op{dist}(K_1,K_2) \geq \Ceucl \BB r$ and $\op{dist}(K_1,\bdy B_{\BB r}(0) ) \geq \Ceucl \BB r$.
For $\ep > 0$, let
\eqbn
\mathbf R_\ep := \mcl R \cap [\ep^2 \BB r ,\ep \BB r] \quad \text{and} \quad \mathbf Q_\ep := \left[\frac12 \ep^{2\xi (Q+3)}   , \ep^{ \xi(Q-3)} \right] \cap \{2^{-\el}\}_{\el\in\BB N} .
\eqen 
The cardinality of $\mathbf R_\ep \times \mathbf Q_\ep$ is at most a $\xi$-dependent constant times $(\log \ep^{-1})^2$. 
By interchanging the order of summation and expectation, then applying Proposition~\ref{prop-card} and Lemma~\ref{lem-rn}, we get that for each $k\in\BB N$, 
\allb \label{eqn-counting-sum}
(\log\ep^{-1})^2 
&\succeq \sum_{r\in \mathbf R_\ep} \sum_{q \in \mathbf Q_\ep} \sum_{\substack{Z \in \mcl Z_r \\ \# Z  \leq k} } \BB E\left[ \frac{\BB 1_{F_{Z,r}^{q,\BB r}(h + \fr_{Z,r})} }{ \# \{Z' \in \mcl Z_r : \# Z' \leq k ,\, F_{Z',r}^{q,\BB r}(h + \fr_{Z',r}) \: \text{occurs} \} } \right]  \notag\\
&\succeq \Ccard^{-k} \sum_{r\in \mathsf R_\ep} \sum_{q \in \mathsf Q_\ep} \sum_{\substack{Z \in \mcl Z_r \\ \# Z  \leq k} }  \BB P\left[ F_{Z,r}^{q,\BB r}(h + \fr_{Z,r}) \right]  \notag\qquad \text{(Proposition~\ref{prop-card})} \notag \\
&\succeq \Ccard^{-k} \Crn^{-k} \sum_{r\in \mathsf R_\ep} \sum_{q \in \mathsf Q_\ep} \sum_{\substack{Z \in \mcl Z_r \\ \# Z  \leq k} }  \BB P\left[ F_{Z,r}^{q,\BB r}(h ) \right]  \qquad \text{(Lemma~\ref{lem-rn})}  \notag \\ 
&= \Ccard^{-k} \Crn^{-k} \BB E\left[ \sum_{r\in \mathsf R_\ep} \sum_{q \in \mathsf Q_\ep} \#\left\{Z \in \mcl Z_r : \# Z \leq k ,\:  F_{Z,r}^{q,\BB r}(h ) \: \text{occurs} \right\} \right] .
\alle
By Proposition~\ref{prop-choices}, for each small enough $\ep > 0$ (how small depends on $k$) on the event $\mcl G_{\BB r}^\ep(K_1,K_2)$ the double sum inside the expectation in the last line of~\eqref{eqn-counting-sum} is at least $\ep^{-\Cexp k}$. Hence for each small enough $\ep > 0$ (depending on $k$), 
\allb
(\log\ep^{-1})^2  \succeq \Ccard^{-k} \Crn^{-k} \ep^{-\Cexp k} \BB P\left[ \mcl G_{\BB r}^\ep(K_1,K_2) \right] .
\alle
Re-arranging this inequality and choosing $k$ to be slightly larger than $\mu  / \Cexp$ yields~\eqref{eqn-counting}.
\end{proof}

\subsection{Proof of Proposition~\ref{prop-choices}}
\label{sec-counting-choices}

Fix $\BB r > 0$ and compact sets $K_1,K_2\subset B_{\BB r}(0)$ such that $\op{dist}(K_1,K_2) \geq \Ceucl \BB r$ and $\op{dist}(K_1,\bdy B_{\BB r}(0) ) \geq \Ceucl \BB r$. 
It is straightforward to show from the definition of $\mcl G_{\BB r}^\ep$ that if $\mcl G_{\BB r}^\ep$ occurs, then there are many 3-tuples $(Z,r,q)$ with $r\in \mcl R \cap [\ep \BB r ,\ep^2 \BB r]$, $q \in [\ep^{2\xi (Q+3)} / 2 , \ep^{ \xi(Q-3)}] \cap \{2^{-\el}\}_{\el\in\BB N}$, and $Z\in\mcl Z_r$ for which all of the conditions in the definition of $F_{Z,r}^{q,\BB r}$ occur except possibly condition~\ref{item-ball-set-stable}, i.e., the event of the following definition occurs.

\begin{defn} \label{def-simpler-F}
For $r\in\mcl R$, $Z\in \mcl Z_r$, and $q>0$, we define $\ol F_{Z,r}^{q,\BB r}(h) = \ol F_{Z,r}^{q,\BB r}(h ; K_1,K_2)$ to be the event that all of the conditions in the definition of $F_{Z,r}^{q,\BB r}(h)$ occur except possibly condition~\ref{item-ball-set-stable}, i.e., $\ol F_{Z,r}^{q,\BB r}(h)$ is the event that the following is true. 
\begin{enumerate} 
\item $\wt D_h(K_1,K_2) \geq \Cupper D_h(K_1,K_2) -q \BB r^{\xi Q} e^{\xi h_{\BB r}(0)}$. \label{item-ball-set-bad'}
\item The event $\Er_{z,r}$ occurs for each $z \in Z$. \label{item-ball-set-good'}
\item \label{item-ball-set-compare'} We have
\eqbn
r^{\xi Q}  e^{\xi h_r(z)} \in \left[q \BB r^{\xi Q} e^{\xi h_{\BB r}(0)}    ,2q \BB r^{\xi Q} e^{\xi h_{\BB r}(0)} \right] ,\quad \forall z \in  Z .
\eqen 
\item For each $z \in Z$, the $D_h$-geodesic from $K_1$ to $K_2$ hits $B_r(z)$.  \label{item-ball-set-hit'} 
\end{enumerate}
\end{defn}

Recall that condition~\ref{item-ball-set-stable} asserts that for each $z\in Z$, the $D_{h-\fr_{Z,r}}$-geodesic $P_Z$ from $K_1$ to $K_2$ has a $(B_{4r}(z), \Vr_{z,r})$-excursion $(\tau_z',\tau_z,\sigma_z,\sigma_z')$ such that $D_h\left( P_Z(\tau_z) , P_Z(\sigma_z)   ; B_{4r}(z) \right) \geq \Ctime r^{\xi Q} e^{\xi h_r(z)} $.
The difficulty with checking condition~\ref{item-ball-set-stable} is that the $D_{h-\fr_{Z,r}}$-geodesic from $K_1$ to $K_2$ could potentially spend a very small amount of time in $\Vr_{z,r}$ for some of the points $z \in Z$, or possibly even avoid some of the sets $\Vr_{z,r}$ altogether. To deal with this, we will show that if $Z\in\mcl Z_r$ and $\ol F_{Z,r}^{q,\BB r}$ occurs, then there is a subset $Z' \subset Z$ such that $\# Z' $ is at least a constant times $\# Z$ and $F_{Z',r}^{q,\BB r}$ occurs (Lemma~\ref{lem-stable-subset}). 

The idea for constructing $Z'$ is as follows. In Lemma~\ref{lem-new-dist-upper} we show that $D_{h-\fr_{Z,r}}(K_1,K_2)$ is smaller than $D_h(K_1,K_2)$ minus a constant times $q \BB r^{\xi Q} e^{\xi h_{\BB r}(0)} \# Z$. Intuitively, subtracting $\fr_{Z,r}$ substantially reduces the distance from $K_1$ to $K_2$. Since $\fr_{Z,r}$ is supported on $\bigcup_{z\in Z} \Vr_{z,r}$, this implies that the $D_{h-\fr_{Z,r}}$-geodesic $P_Z$ from $K_1$ to $K_2$ has to spend at least a constant times $q \BB r^{\xi Q} e^{\xi h_{\BB r}(0)} \# Z$ units of time in $\bigcup_{z\in Z} \Vr_{z,r}$ (otherwise, its length would have to be larger than $D_{h-\fr_{Z,r}}(K_1,K_2)$). We then iteratively remove the ``bad" points $z\in Z$ for which there does \emph{not} exist a $(B_{4r}(z) , \Vr_{z,r})$-excursion $(\tau_z',\tau_z,\sigma_z , \sigma_z')$ for $P_Z$ such that
\eqbn 
 D_h\left( P_Z(\tau_z) , P_Z(\sigma_z)   \right) \geq \Ctime r^{\xi Q} e^{\xi h_r(z)}  . 
\eqen  

For each of the above ``bad" points $z\in Z$, the intersection of $P_Z$ with $\Vr_{z,r}$ is in some sense small. 
Since the function $\fr_{z,r}$ is supported on $V_{z,r}$, removing the ``bad" points from $Z$ does not increase $D_{h-\fr_{Z,r}}(K_1,K_2)$ by very much.
Consequently, at each stage of the iterative procedure it will still be the case that $D_{h-\fr_{Z,r}}(K_1,K_2)$ is substantially smaller than $D_h(K_1,K_2)$. As above, this implies that $P_Z$ spends a substantial amount of time in $\bigcup_{z\in Z} \Vr_{z,r}$. We show in Lemma~\ref{lem-new-dist-spent} that the amount of time that $P_Z$ spends in each $\Vr_{z,r}$ is at most a constant times $q \BB r^{\xi Q} e^{\xi h_{\BB r}(0)}  $. This allows us to show that the iterative procedure has to terminate before we have removed too many points from $Z$. 
 
To begin the proof, we establish an upper bound for $D_{h-\fr_{Z,r}}(K_1,K_2)$ in terms of $D_h(K_1,K_2)$ on the event $\ol F_{Z,r}^{q,\BB r}(h)$. 
The reason why this bound holds is that the $D_h$-geodesic from $K_1$ to $K_2$ has to cross the regions $\Ur_{z,r}$ for $z\in Z$. Since $\fr_{Z,r}$ is very large on $\Ur_{z,r}$ and by hypothesis~\ref{item-Ehyp-dist} for $\Er_{z,r}$, the $D_{h-\fr_{Z,r}}$-distances around the regions $\Ur_{z,r}$ for $z\in Z$ is small. This allows us to find $\# Z$ ``shortcuts" along the $D_h$-geodesic with small $D_{h-\fr_{Z,r}}$-length.

\begin{figure}[ht!]
\begin{center}
\includegraphics[width=.75\textwidth]{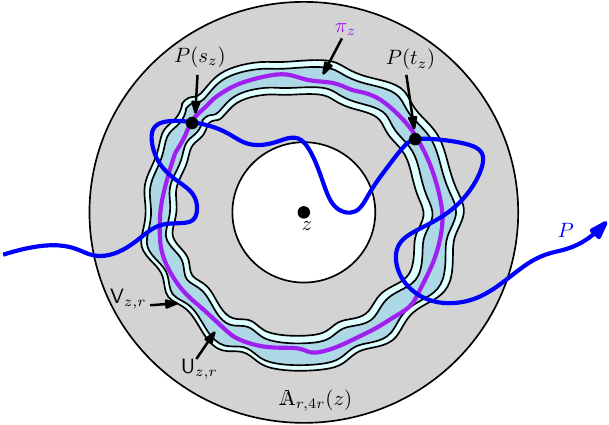} 
\caption{\label{fig-new-dist-upper} 
Illustration of the proof of Lemma~\ref{lem-new-dist-upper}. Since $\fr_{z,r}$ is very large on $\Ur_{z,r}$, the $D_{h-\fr_{Z,r}}$-length of the purple path $\pi_z$ is very short. By replacing the segment $P|_{[s_z,t_z]}$ by a segment of $\pi_z$ for each $z\in Z$, we obtain a new path from $K_1$ to $K_2$ whose $D_{h-\fr_{Z,r}}$-length is substantially smaller than $D_h(K_1,K_2)$. 
}
\end{center}
\end{figure}

\newcommand{\Cnew}{C_3}

\begin{lem} \label{lem-new-dist-upper}
There is a constant $\Cnew > 2\Caround \Ctime / \Cacross$, depending only on the parameters, such that the following is true. 
Let $r\in\mcl R$, $Z\subset \mcl Z_r$, and $q >0$ and assume that $\ol F_{Z,r}^{q,\BB r}(h)$ occurs.
Then
\eqb \label{eqn-new-dist-upper}
D_{h-\fr_{Z,r}}(K_1,K_2) 
\leq D_h(K_1,K_2) -   \Cnew   q \BB r^{\xi Q} e^{\xi h_{\BB r}(0)} \# Z .
\eqe
\end{lem} 
\begin{proof}
See Figure~\ref{fig-new-dist-upper} for an illustration. 
By condition~\ref{item-ball-set-good'} in the definition of $\ol F_{Z,r}^{q,\BB r}(h)$, the event $\Er_{z,r}(h)$ occurs for each $z \in Z$. 
So, by hypothesis~\ref{item-Ehyp-dist} for $\Er_{z,r}$ and condition~\ref{item-ball-set-compare'} in the definition of $\ol F_{Z,r}^{q,\BB r}(h)$, we can find for each $z\in Z$ a path $\pi_z$ in $\Ur_{z,r}$ which disconnects the inner and outer boundaries of $\Ur_{z,r}$ such that
\eqb \label{eqn-new-dist-around0}
\op{len}\left( \pi_z ; D_h \right) \leq 2 D_h\left(\text{around $\Ur_{z,r}$} \right) \leq 4 \Ctube  q \BB r^{\xi Q} e^{\xi h_{\BB r}(0)} .
\eqe

By condition~\ref{item-ball-set-hit'} in the definition of $\ol F_{Z,r}^{q,\BB r}(h)$, the $D_h$-geodesic $P$ from $K_1$ to $K_2$ hits $B_r(z)$ for each $z\in Z$. Furthermore, $ B_{4r}(z) \cap (K_1\cup K_2) = \emptyset$ for each $z\in Z$ (recall~\eqref{eqn-pt-set}) and $\pi_z$ disconnects the inner and outer boundaries of $\BB A_{r,4r}(z)$ for each $z\in Z$. It follows that for each $z\in Z$, we can find times $s_z < t_z$ such that $P(s_z) , P(t_z) \in \pi_z$, the path $P|_{[s_z,t_z]}$ hits $B_r(z)$, and $P((s_z,t_z))$ lies in the open region which is disconnected from $\infty$ by $\pi_z$. 
Since the balls $B_{4r}(z)$ for $z\in Z$ are disjoint (again by~\eqref{eqn-pt-set}), the time intervals $[s_z,t_z]$ for $z\in Z$ are disjoint. 

The path $P$ must cross from $\Vr_{z,r}$ to $\bdy B_r(z)$ between times $s_z$ and $t_z$, so by hypothesis~\ref{item-Ehyp-dist} for $\Er_{z,r}$ and condition~\ref{item-ball-set-compare'} in the definition of $\ol F_{Z,r}^{q,\BB r}(h)$,  
\eqb \label{eqn-new-dist-across}
t_z - s_z \geq D_h\left(\Vr_{z,r} , \bdy B_r(z) \right) \geq \Cacross q \BB r^{\xi Q} e^{\xi h_{\BB r}(0)} .
\eqe 

Let $P'$ be the path obtained from $P$ by excising each segment $P|_{[s_z,t_z]}$ and replacing it by a segment of $\pi_z$ with the same endpoints. 
Since $\fr_{Z,r}$ is non-negative, Weyl scaling (Axiom~\ref{item-metric-f}) shows that 
\allb \label{eqn-new-dist-between}
\op{len}\left( P' \setminus \bigcup_{z\in Z} \pi_z ;  D_{h-\fr_{Z,r}}  \right) 
&\leq \op{len}\left( P' \setminus \bigcup_{z\in Z} \pi_z ; D_h \right) \notag\\
&= \op{len}\left( P ; D_h \right) - \sum_{z\in Z} (t_z - s_z) \notag\\
&\leq  D_h( K_1,K_2)  - \Cacross q \BB r^{\xi Q} e^{\xi h_{\BB r}(0)} \# Z \quad \text{(by~\eqref{eqn-new-dist-across})} . 
\alle
Furthermore, since $\fr_{Z,r}$ is identically equal to $\Cmax$ on each of the sets $\Ur_{z,r}$ for $z\in Z$ (which contains $\pi_z$) we get from~\eqref{eqn-new-dist-around0} that 
\allb \label{eqn-new-dist-pi}
\op{len}\left(  \pi_z ;   D_{h-\fr_{Z,r}} \right) \leq 4 e^{-\xi \Cmax} \Ctube q \BB r^{\xi Q} e^{\xi h_{\BB r}(0)}  .
\alle
Combining~\eqref{eqn-new-dist-between} and~\eqref{eqn-new-dist-pi} shows that
\eqbn
D_{h-\fr_{Z,r}}(K_1,K_2) 
\leq \op{len}\left(P' ; D_{h-\fr_{Z,r}}  \right)
\leq D_h(K_1,K_2)  - \left( \Cacross - 4 e^{-\xi \Cmax} \Ctube \right) q \BB r^{\xi Q} e^{\xi h_{\BB r}(0)} \# Z  .
\eqen
This gives~\eqref{eqn-new-dist-upper} with $\Cnew =  \Cacross - 4 e^{-\xi \Cmax} \Ctube$. We note that $\Cnew > 2\Caround \Ctime / \Cacross$ due to~\eqref{eqn-parameter-relation}. 
\end{proof}

We next establish an inequality in the opposite direction from the one in Lemma~\ref{lem-new-dist-upper}, i.e., an upper bound for $D_h(K_1,K_2)$ in terms of $D_{h-\fr_{Z,r}}(K_1,K_2)$. This latter estimate holds unconditionally (i.e., we do not need to truncate on any event).

\begin{figure}[ht!]
\begin{center}
\includegraphics[width=.9\textwidth]{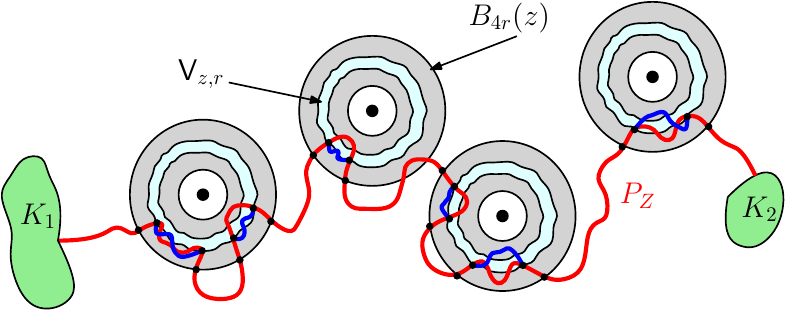} 
\caption{\label{fig-new-dist-lower0} 
Illustration of the proof of Lemma~\ref{lem-new-dist-lower0}. 
The set $Z$ consists of the four center points of the annuli in the figure.
For each $z\in Z$, we have indicated each of the points $P_Z(\tau'), P_Z(\tau) , P_Z(\sigma) , P_Z(\sigma')$ for the $(B_{4r}(z) , \Vr_{z,r})$-excursions $(\tau',\tau,\sigma,\sigma') \in \mcl T_{z,r}(P_Z)$ with a black dot. The proof proceeds by replacing each of the segments $P_Z|_{[\tau,\sigma]}$ by a $D_h$-geodesic with the same endpoints (shown in blue). 
}
\end{center}
\end{figure}

\begin{lem} \label{lem-new-dist-lower0}   
Let $r\in\mcl R$ and $Z \in  \mcl Z_r$.  
Let $P_Z$ be the $D_{h-\fr_{Z,r}}$-geodesic from $K_1$ to $K_2$. 
For $z\in Z$, let $ \mcl T_{z,r}(P_Z)$ be the set of $(B_{4r}(z) , \Vr_{z,r})$-excursions of $P_Z$ (Definition~\ref{def-excursion}). 
Then
\eqb \label{eqn-new-dist-lower0}
     D_h(K_1,K_2) \leq  D_{h-\fr_{Z,r}}(K_1,K_2) +    \sum_{z\in Z} \sum_{(\tau',\tau,\sigma,\sigma') \in \mcl T_{z,r}(P_Z)} D_h\left( P_Z(\tau) , P_Z(\sigma)  \right)      .
\eqe 
\end{lem}
\begin{proof}
See Figure~\ref{fig-new-dist-lower0} for an illustration. 
By the definition~\eqref{eqn-pt-set} of $\mcl Z_r$, we have $  B_{4r}(z) \cap (K_1\cup K_2) = \emptyset$ for each $z\in Z$. 
From this and Definition~\ref{def-excursion}, we see that for each $z\in Z$, the set $P_Z^{-1}(\Vr_{z,r})$ is contained in the union of the excursion intervals $[\tau,\sigma]$ for $ (\tau',\tau,\sigma,\sigma') \in \bigcup_{z\in Z} \mcl T_{z,r}(P_Z)$. 
Furthermore, since the balls $B_{4r}(z)$ for $z\in Z$ are disjoint, it follows that the excursion intervals $[\tau,\sigma]$ for $ (\tau',\tau,\sigma,\sigma') \in \bigcup_{z\in Z} \mcl T_{z,r}(P_Z)$ are disjoint. 
Since $P_Z$ is continuous, there are only finitely many such intervals. 

Let $P_Z'$ be the path from $K_1$ to $K_2$ obtained from $P_Z$ by replacing each of the segments $P_Z|_{[\sigma,\tau]}$ for $ (\tau',\tau,\sigma,\sigma') \in \bigcup_{z\in Z} \mcl T_{z,r}(P_Z)$ by a $D_h$-geodesic from $P_Z(\tau)$ to $P_Z(\sigma)$. 
The function $\fr_{Z,r}$ is supported on $\bigcup_{z\in Z} \Vr_{z,r}$ and the path $P_Z$ does not hit $\bigcup_{z\in Z} \Vr_{z,r}$ except during the above excursion intervals $[\sigma,\tau]$. 
Hence the $D_h$-length of each of the segments of $P_Z$ which are not replaced when we construct $P_Z'$ is the same as its $D_{h-\fr_{Z,r}}$-length. 
From this, we see that the $D_h$-length of $P_Z'$ is at most $\op{len}(P_Z ; D_{h-\fr_{Z,r}})$ plus the sum of the $D_h$-lengths of the replacement segments. 
In other words, $\op{len}(P_Z' ; D_h)$ is at most the right side of~\eqref{eqn-new-dist-lower0}. 
\end{proof}

If we assume that $\bigcap_{z\in Z} \Er_{z,r}$ occurs, then we can replace the second sum on the right side of~\eqref{eqn-new-dist-lower0} by a maximum.  

\begin{lem} \label{lem-new-dist-lower}   
Let $r\in\mcl R$ and $Z \in  \mcl Z_r$.  
Assume that $\bigcap_{z\in Z} \Er_{z,r}$ occurs and let $P_Z$ be the $D_{h-\fr_{Z,r}}$-geodesic from $K_1$ to $K_2$. 
For $z\in Z$, let $  \mcl T_{z,r}(P_Z)$ be as in Lemma~\ref{lem-new-dist-lower0}. 
Then
\eqb \label{eqn-new-dist-lower}
D_h(K_1,K_2) \leq  D_{h-\fr_{Z,r}}(K_1,K_2) +   \frac{\Caround}{\Cacross} \sum_{z\in Z} \max_{(\tau',\tau,\sigma,\sigma') \in \mcl T_{z,r}(P_{Z })} D_h\left( P_Z(\tau) , P_Z(\sigma)  \right)      .
\eqe 
\end{lem}

For the proof of Lemma~\ref{lem-new-dist-lower}, we will need an upper bound for the amount of time that $P_Z$ can spend in $\Vr_{Z,r}$. 
This upper bound is a straightforward consequence of the upper bound for $D_h(\text{around $\BB A_{3r,4r}(z)$})$ from hypothesis~\ref{item-Ehyp-dist} for $\Er_{z,r}$.

\begin{lem} \label{lem-new-dist-around}  
Let $r\in\mcl R$, let $Z\subset \mcl Z_r$, and assume that $\bigcap_{z\in Z} \Er_{z,r}$ occurs. 
Let $P_Z$ be the $D_{h-\fr_{Z,r}}$-geodesic from $K_1$ to $K_2$. 
For $z\in Z$ such that $P_Z \cap V_{z,r} \not=\emptyset$, let $S_z$ (resp.\ $T_z$) be the first time that $P_Z$ enters $\ol \Vr_{z,r}$ (resp.\ the last time that $P_Z$ exits $  \Vr_{z,r}$).
Then
\eqb \label{eqn-new-dist-around}
T_z - S_z   \leq \Caround  r^{\xi Q} e^{\xi h_r(z)}   .
\eqe
\end{lem}
\begin{proof} 
By hypothesis~\ref{item-Ehyp-dist} for $\Er_{z,r} $, for each $\zeta > 0$ there is a path $\pi_z$ in $\BB A_{3r,4r}(z)$ which disconnects the inner and outer boundaries of $\BB A_{3r,4r}(z)$ such that
\eqb \label{eqn-old-geo-around}
\op{len}\left(  \pi_z  ;    D_h  \right) 
\leq (\Caround  +\zeta)  r^{\xi Q} e^{\xi h_r(z)} .
\eqe 
Since $\fr_{Z,r}$ is non-negative, the $D_{h-\fr_{Z,r}}$-length of $\pi_z$ is at most its $D_h$-length. 

Since $  B_{4r}(z) \cap (K_1\cup K_2) = \emptyset$ (recall~\eqref{eqn-pt-set}), the path $P_Z$ must hit $\pi_z$ before time $S_z$ and again after time $T_z$. Since $P_Z$ is a $D_{h-\fr_{Z,r}}$-geodesic, the $D_{h-\fr_{Z,r}}$-length of the segment of $P_Z$ between any two times when it hits $\pi_z$ is at most the $D_{h-\fr_{Z,r}}$-length of $\pi_z$ (otherwise, concatenating two segments of $P_Z$ with a segment of $\pi_z$ would produce a path with the same endpoints as $P_Z$ which is $D_{h-\fr_{Z,r}}$-shorter than $P_Z$).
Therefore,~\eqref{eqn-old-geo-around} gives
\eqb 
T_z - S_z 
\leq \op{len}\left(  \pi_z  ;    D_{h-\fr_{Z,r}}  \right) 
\leq \op{len}\left(  \pi_z  ;    D_h  \right) 
\leq (\Caround  +\zeta)  r^{\xi Q} e^{\xi h_r(z)}  .
\eqe 
Sending $\zeta \rta 0$ now concludes the proof.
\end{proof}

\begin{proof}[Proof of Lemma~\ref{lem-new-dist-lower}]
In light of Lemma~\ref{lem-new-dist-lower0}, it suffices to show that for each $z\in Z$, the number of $(B_{4r}(z) , \Vr_{z,r})$-excursions satisfies
\eqb \label{eqn-dist-lower-show}
\# \mcl T_{z,r}(P_Z) \leq \frac{\Caround}{\Cacross} .
\eqe
To obtain~\eqref{eqn-dist-lower-show}, we first note that for each $(\tau',\tau,\sigma,\sigma') \in \mcl T_{z,r}(P_Z)$, the path $P_Z$ crosses between $\bdy B_{3r}(z)$ and $\Vr_{z,r}$ during each of the time intervals $[\tau',\tau]$ and $[\sigma,\sigma']$. Since $\fr_{Z,r}$ vanishes in $B_{3r}(z) \setminus \Vr_{z,r}$ and by hypothesis~\ref{item-Ehyp-dist} for $\Er_{z,r}$, 
\eqb \label{eqn-dist-lower-across}
\min\{\tau - \tau' , \sigma' - \sigma\} 
\geq D_{h-\fr_{Z,r}}(\bdy B_{3r}(z) , \Vr_{z,r})
\geq D_h(\bdy B_{3r}(z) , \Vr_{z,r})
\geq \Cacross r^{\xi Q} e^{\xi h_r(z)} .
\eqe

Let $S_z$ and $T_z$ be the first time that $P_Z$ enters $V_{z,r}$ and the last time that $P_Z$ exits $V_{z,r}$, as in Lemma~\ref{lem-new-dist-around}. 
If $(\tau_0',\tau_0,\sigma_0,\sigma_0') \in \mcl T_{z,r}(P_Z)$ and $(\tau_1',\tau_1,\sigma_1,\sigma_1') \in \mcl T_{z,r}(P_Z)$ are the first and last excursions in chronological order, then $S_z = \tau_0$ and $T_z = \sigma_1$. 
Hence, for each excursion $(\tau',\tau,\sigma,\sigma') \in \mcl T_{z,r}(P_Z)$ which is not the first (resp.\ last) excursion in chronological order, the time interval $[\tau',\tau]$ (resp.\ $[\sigma,\sigma']$) is contained in $[S_z,T_z]$. 
Furthermore, these time intervals for different excursions are disjoint. By summing the estimate~\eqref{eqn-dist-lower-across} over all elements of $\mcl T_{z,r}(P_Z)$, we get that if $\#\mcl T_{z,r}(P_Z) \geq 2$, then
\eqb \label{eqn-dist-lower-sum}
T_z - S_z  \geq \Cacross r^{\xi Q} e^{\xi h_r(z)}  \#\mcl T_{z,r}(P_Z) .
\eqe
Combining~\eqref{eqn-dist-lower-sum} and~\eqref{eqn-new-dist-around} gives~\eqref{eqn-dist-lower-show} in the case when $\#\mcl T_{z,r}(P_Z) \geq 2$. If $\#\mcl T_{z,r}(P_Z) \leq 1$, then~\eqref{eqn-dist-lower-show} holds vacuously since $\Caround / \Cacross \geq 1$.  
\end{proof}

For the proof of Proposition~\ref{prop-choices}, we will need a slightly different upper bound for the amount of time that the $D_{h-\fr_{Z,r}}$-geodesic can spend in $\Vr_{z,r}$ as compared to the one in Lemma~\ref{lem-new-dist-around}.

\newcommand{\Cspent}{C_4}

\begin{lem} \label{lem-new-dist-spent} 
There is a constant $\Cspent > 0$, depending only on the parameters, such that the following is true. 
Let $r\in\mcl R$, $Z\subset \mcl Z_r$, and $q>0$ and assume that $\ol F_{Z,r}^{q,\BB r}(h)$ occurs.
Let $P_Z$ be the $D_{h-\fr_{Z,r}}$-geodesic from $K_1$ to $K_2$. 
For each $z \in Z$,  
\eqb \label{eqn-new-dist-spent}
\max\left\{ \sup_{u,v\in P_Z \cap \Vr_{z,r}} D_h\left(u,v  \right) , \op{len}\left( P_Z \cap \Vr_{z,r} ; D_h \right) \right\}   \leq  \Cspent q \BB r^{\xi Q} e^{\xi h_{\BB r}(0)} .
\eqe
\end{lem}
\begin{proof} 
By condition~\ref{item-ball-set-good'} in the definition of $\ol F_{Z,r}^{q,\BB r}(h)$, the event $\bigcap_{z\in Z} \Er_{z,r} $ occurs.
The bound~\eqref{eqn-new-dist-spent} holds vacuously if $P_Z\cap \Vr_{z,r} = \emptyset$, so assume that $P_Z\cap \Vr_{z,r} \not=\emptyset$. 
For $z\in Z$, let $S_z$ (resp.\ $T_z$) be the first time that $P_Z$ enters $\ol \Vr_{z,r}$ (resp.\ the last time that $P_Z$ exits $\ol \Vr_{z,r}$), as in Lemma~\ref{lem-new-dist-around}. By Lemma~\ref{lem-new-dist-around} followed by condition~\ref{item-ball-set-compare'} in the definition of $\ol F_{Z,r}(h)$, 
\eqbn
T_z - S_z \leq  \Caround  r^{\xi Q} e^{\xi h_r(z)}  \leq 2\Caround q \BB r^{\xi Q} e^{\xi h_{\BB r}(0)} 
\eqen
 Furthermore, $P_Z^{-1}(\Vr_{Z,r}) \subset [S_z,T_z]$, so
\alb
\max\left\{ \sup_{u,v\in P_Z\cap \Vr_{z,r}}  D_{h-\fr_{Z,r}}(u,v)  , \op{len}\left(P_Z\cap \Vr_{z,r} ;  D_{h-\fr_{Z,r}} \right) \right\}
&\leq T_z - S_z \notag\\
&\leq 2\Caround q \BB r^{\xi Q} e^{\xi h_{\BB r}(0)}  .
\ale 
Since $\fr_{Z,r} \leq \Cmax$, the bound~\eqref{eqn-new-dist-upper} combined with Weyl scaling (Axiom~\ref{item-metric-f}) gives~\eqref{eqn-new-dist-spent} with $\Cspent = 2 e^{\xi \Cmax} \Caround$. 
\end{proof}

The following lemma is the main input in the proof of Proposition~\ref{prop-choices}. It allows us to produce configurations $Z$ for which $F_{Z,r}^{q,\BB r}(h)$, instead of just $\ol F_{Z,r}^{q,\BB r}(h)$, occurs.

\newcommand{\Csub}{c_5}

\begin{lem} \label{lem-stable-subset}
There is a constant $\Csub > 0$, depending only on the parameters, such that the following is true. 
Let $r\in\mcl R$, $Z\in\mcl Z_r$, and $q>0$ and assume that $\ol F_{Z,r}^{q,\BB r}(h)$ occurs. 
There exists $Z'\subset Z$ such that $F_{Z',r}^{q,\BB r}(h)$ occurs and $\#Z' \geq \Csub \# Z  $. 
\end{lem}
\begin{proof}
\noindent\textit{Step 1: iteratively removing ``bad" points.}
It is immediate from Definition~\ref{def-simpler-F} that if $\ol F_{Z,r}^{q,\BB r}(h)$ occurs and $Z'\subset Z$ is non-empty, then $Z' \in \mcl Z_r$ and $\ol F_{Z' , r}^{q,\BB r}(h)$ occurs.
So, we need to produce a set $Z'\subset Z$ such that $\#Z'$ is at least a constant times $\# Z$ and condition~\ref{item-ball-set-stable} in the definition of $F_{Z' , r}^{q,\BB r}(h)$ occurs. Since $D_h(u,v ; B_{4r}(z)) \geq D_h(u,v)$ for all $u,v\in\BB C$, it suffices to find $Z'\subset Z$ such that if $P_{Z'}$ is the $D_{h-\fr_{Z',r}}$-geodesic from $K_1$ to $K_2$ and $\mcl T_{z,r}(P_{Z'})$ denotes the set of $(B_{4r}(z) , \Vr_{z,r})$-excursions for $P_{Z'}$, then 
\eqb  \label{eqn-stable-subset-show}
\max_{(\tau',\tau,\sigma,\sigma') \in \mcl T_{z,r}(P_{Z'})}  D_h\left( P_{Z'}(\tau) , P_{Z'}(\sigma)   \right) \geq \Ctime r^{\xi Q} e^{\xi h_r(z)}  . 
\eqe 

We will construct $Z'$ by iteratively removing the ``bad" points $z\in Z'$ such that the condition of~\eqref{eqn-stable-subset-show} does not hold. To this end, let $Z_0 := Z$. Inductively, suppose that $m\in \BB N_0$ and $Z_m \subset Z$ has been defined. Let $P_{Z_m}$ be the $D_{h-\fr_{Z_m,r}}$-geodesic from $K_1$ to $K_2$ and let $Z_{m+1}$ be the set of $z\in Z_m$ such that 
\eqb  \label{eqn-stable-subset-iterate}
\max_{(\tau',\tau,\sigma,\sigma') \in \mcl T_{z,r}(P_{Z_m})}  D_h\left( P_{Z_m}(\tau) , P_{Z_m}(\sigma)   \right)  \geq \Ctime r^{\xi Q} e^{\xi h_r(z)}   . 
\eqe 

If $Z_{m+1} = Z_m$, then~\eqref{eqn-stable-subset-show} holds with $Z' = Z_m$, so the event $F_{Z_m,r}^{q,\BB r}(h)$ occurs. So, to prove the lemma it suffices to show that the above procedure stabilizes before $\# Z_m$ gets too much smaller than $\# Z$. More precisely, we will show that there exists $\Csub > 0$ as in the lemma statement such that 
\eqb \label{eqn-stable-subset-lower}
\# Z_m \geq \Csub   \# Z ,\quad \forall m \in \BB N .
\eqe 
Since $Z_{m+1} \subset Z_m$ for each $m\in\BB N_0$ and $Z_0$ is finite, it follows that there must be some $m \in \BB N$ such that $Z_m = Z_{m+1}$. We know that $F_{Z_m,r}^{q,\BB r}(h)$ occurs for any such $m$, so~\eqref{eqn-stable-subset-lower} implies the lemma statement. 

It remains to prove~\eqref{eqn-stable-subset-lower}. The idea of the proof is as follows. At each step of our iterative procedure, we only remove points $z\in Z_m$ for which $P_{Z_m} \cap \Vr_{z,r}$ is small, in a certain sense.
Using this, we can show that $D_{h-\fr_{Z_{m+1},r}}(K_1,K_2)$ is not too much bigger than $D_{h-\fr_{Z_m,r}}(K_1,K_2)$ (see~\eqref{eqn-dist-increment}). Iterating this leads to an upper bound for $D_{h-\fr_{Z_m,r}}(K_1,K_2)$ in terms of $D_{h-\fr_{Z,r}}(K_1,K_2)$ (see~\eqref{eqn-dist-iterate}). 
We then use the fact that $D_{h-\fr_{Z,r}}(K_1,K_2)$ has to be substantially smaller than $D_h(K_1,K_2)$ (Lemma~\ref{lem-new-dist-upper}) together with our upper bound for the amount of time that $P_{Z_m}$ spends in each of the $\Vr_{z,r}$'s (Lemma~\ref{lem-new-dist-spent}) to obtain~\eqref{eqn-stable-subset-lower}. 
\medskip

\noindent\textit{Step 2: comparison of $D_{h-\fr_{Z_m,r}}(K_1,K_2)$ and $D_h(K_1,K_2)$.}
Let us now proceed with the details. 
Let $m\in \BB N_0$. 
By the definition~\eqref{eqn-stable-subset-iterate} of $Z_{m+1}$ and condition~\ref{item-ball-set-compare'} in the definition of $\ol F_{Z,r}^{q,\BB r}(h)$, 
\eqb \label{eqn-subset-diam}
\max_{(\tau',\tau,\sigma,\sigma') \in \mcl T_{z,r}(P_{Z_m})}  D_h\left( P_{Z_m}(\tau) , P_{Z_m}(\sigma)   \right) 
\leq 2 \Ctime q \BB r^{\xi Q} e^{\xi h_{\BB r}(0)}    ,\quad \forall z \in Z_m \setminus Z_{m+1} .
\eqe

We have $Z_m \setminus Z_{m+1} \in \mcl Z_r$ and $h -  \fr_{Z_m,r} =  h - \fr_{Z_{m+1},r}   -  \fr_{Z_m \setminus Z_{m+1} ,r}$. 
Since we are assuming that $\ol F_{Z,r}^{q,\BB r}(h)$ occurs and $Z_m\setminus Z_{m+1} \subset Z$, condition~\ref{item-ball-set-good'} of Definition~\ref{def-simpler-F} implies that $ \bigcap_{z\in Z_m \setminus Z_{m+1}} \Er_{z,r}$ occurs.
Since $\Er_{z,r}$ depends only on $h |_{\ol{\BB A}_{r,4r}(z)}$ and the support of $\fr_{Z_{m+1},r}$ is disjoint from $\ol{\BB A}_{r,4r}(z)$ for $z\in Z_m\setminus Z_{m+1}$, we get that $ \bigcap_{z\in Z_m \setminus Z_{m+1}} \Er_{z,r}$ also occurs with $h - \fr_{Z_{m+1},r}$ in place of $h$. 
We may therefore apply Lemma~\ref{lem-new-dist-lower} with $h - \fr_{Z_{m+1},r}$ in place of $h$ and $Z_m \setminus Z_{m+1}$ in place of $Z$ to get that
\allb  \label{eqn-dist-increment}
&D_{h-\fr_{Z_{m+1},r}}(K_1,K_2)    \notag\\ 
&\qquad \leq D_{h-\fr_{Z_{m },r}}(K_1,K_2)   \notag\\
&\qquad\qquad+  \frac{\Caround}{\Cacross} \sum_{z\in Z_m \setminus Z_{m+1}} \max_{(\tau',\tau,\sigma,\sigma') \in \mcl T_{z,r}(P_{Z_m})} D_{h-\fr_{Z_{m+1},r}}\left( P_{Z_m}(\tau) , P_{Z_m}(\sigma)  \right)   \notag\\
&\qquad\qquad\qquad\qquad \text{(by Lemma~\ref{lem-new-dist-lower})}\notag \\
&\qquad \leq D_{h-\fr_{Z_{m },r}}(K_1,K_2) \notag\\
&\qquad\qquad +  \frac{\Caround}{\Cacross} \sum_{z\in Z_m \setminus Z_{m+1}} \max_{(\tau',\tau,\sigma,\sigma') \in \mcl T_{z,r}(P_{Z_m})} D_h\left( P_{Z_m}(\tau) , P_{Z_m}(\sigma)  \right)   \notag\\
&\qquad\qquad\qquad\qquad \text{(since $\fr_{Z_{m+1},r} \geq 0$)} \notag \\      
&\qquad \leq D_{h-\fr_{Z_{m },r}}(K_1,K_2)   +    \frac{2 \Caround \Ctime}{\Cacross} q \BB r^{\xi Q} e^{\xi h_{\BB r}(0)}  (\# Z_m - \# Z_{m+1}) \quad \text{(by \eqref{eqn-subset-diam})} .    
\alle   

Iterating the inequality~\eqref{eqn-dist-increment} $m$ times, then applying Lemma~\ref{lem-new-dist-upper} to $Z = Z_0 \in \mcl Z_r$ gives
\allb \label{eqn-dist-iterate}
D_{h-\fr_{Z_m,r}}(K_1,K_2 )  
&\leq  D_{h-\fr_{Z ,r}}(K_1,K_2) + \frac{2 \Caround \Ctime}{\Cacross} q \BB r^{\xi Q} e^{\xi h_{\BB r}(0)}    (\# Z - \# Z_m) \notag\\
&\leq  D_h(K_1,K_2) - \left( \Cnew - \frac{2 \Caround \Ctime}{\Cacross} \right) q \BB r^{\xi Q} e^{\xi h_{\BB r}(0)}  \# Z \notag\\
&\qquad\qquad\qquad\qquad - \frac{2 \Caround \Ctime}{\Cacross} q \BB r^{\xi Q} e^{\xi h_{\BB r}(0)}   \# Z_m \notag\\
&\leq D_h(K_1,K_2) - \left( \Cnew - \frac{2 \Caround \Ctime}{\Cacross} \right) q \BB r^{\xi Q} e^{\xi h_{\BB r}(0)}  \# Z .
\alle
Note that in the last line, we simply dropped a negative term. 
\medskip

\noindent\textit{Step 3: conclusion.}
By Lemma~\ref{lem-new-dist-lower} (with $Z_m$ in place of $Z$), followed by~\eqref{eqn-dist-iterate},
\allb \label{eqn-inc-sum-lower}
\frac{\Caround}{\Cacross} \sum_{z\in Z_m} \max_{(\tau',\tau,\sigma,\sigma') \in \mcl T_{z,r}(P_{Z_m}) } D_h\left( P_{Z_m} (\tau) , P_{Z_m}(\sigma)  \right)   
&\geq  D_h(K_1,K_2)  -   D_{h-\fr_{Z_m,r}}(K_1,K_2 ) \notag\\
&\geq  \left( \Cnew - \frac{2 \Caround \Ctime}{\Cacross} \right) q \BB r^{\xi Q} e^{\xi h_{\BB r}(0)}  \# Z  .
\alle
As explained above, since $Z_m \subset Z$ we know that $\ol F_{Z_m,r}^{q,\BB r}(z)$ occurs. 
Hence we can apply Lemma~\ref{lem-new-dist-spent} (with $Z_m$ in place of $Z$), then sum over all $z\in Z_m$, to get 
\allb \label{eqn-one-inc-upper}
\sum_{z\in Z_m} \max_{(\tau',\tau,\sigma,\sigma') \in \mcl T_{z,r}(P_{Z_m}) } D_h\left( P_{Z_m} (\tau) , P_{Z_m}(\sigma)  \right) 
\leq \Cspent q \BB r^{\xi Q} e^{\xi h_{\BB r}(0)} \# Z_m ,\quad \forall z\in Z_m .
\alle
Combining~\eqref{eqn-inc-sum-lower} and~\eqref{eqn-one-inc-upper} yields
\eqb
\# Z_m \geq \Csub   \# Z \quad \text{with} \quad \Csub = \frac{\Cacross}{\Caround \Cspent}  \left( \Cnew - \frac{2 \Caround \Ctime}{\Cacross} \right) .
\eqe
That is,~\eqref{eqn-stable-subset-lower} holds with this choice of $\Csub$. 
Note that $\Csub > 0$ since $\Cnew > 2\Caround \Ctime / \Cacross$ (Lemma~\ref{lem-new-dist-upper}).  
\end{proof}

\begin{proof}[Proof of Proposition~\ref{prop-choices}]
Fix $\BB r > 0$ and compact sets $K_1,K_2 \in B_{2\BB r}(0)$ with $\op{dist}(K_1,K_2) \geq \Ceucl \BB r$. Assume that $\mcl G_{\BB r}^\ep  = \mcl G_{\BB r}^\ep(K_1,K_2)$ occurs and let $P  $ be the $D_h$-geodesic from $K_1$ to $K_2$. 
We first produce an $r\in \mcl R \cap [\ep^2 \BB r , \ep \BB r] $, a $q>0$, and a large collection of sets $Z \in \mcl Z_r$ for which $\ol F_{Z,r}^{q,\BB r}(h)$ occurs. 
 
To this end, let $T$ be the first exit time of $P$ from $B_{3\BB r}(0)$, or $T = D_h(K_1,K_2)$ if $P\subset B_{3\BB r}(0)$ (the reason why we consider $T$ is that conditions~\ref{item-main-exponent} and~\ref{item-main-cover} in the definition of $\mcl G_{\BB r}^\ep$ are only required to hold on $B_{3\BB r}(0)$). 
By condition~\ref{item-main-cover} in the definition of $\mcl G_{\BB r}^\ep$, for each point $w \in P([0,T])$ there exists  $r \in \mcl R \cap [\ep^2 \BB r ,\ep \BB r]$ and $z \in \left(\frac{r}{100} \BB Z^2 \right) \cap B_{3\BB r}(0)$ such that $\Er_{z,r}$ occurs and $w \in B_{r/25}(z)$. 

Since $\op{dist}(K_1,K_2)  \geq \Ceucl \BB r$ and $\op{dist}(K_1,\bdy B_{3\BB r}(0)) \geq \BB r$, it follows that $P([0,T])$ is a connected set of Euclidean diameter at least $\Ceucl \BB r$. Furthermore, since $\op{dist}(K_1,\bdy B_{\BB r}(0) ) \geq \Ceucl \BB r$, there must be a segment of $P|_{[0,T]}$ of Euclidean diameter at least $\Ceucl \BB r$ which is disjoint from $\bdy B_{\BB r}(0)$.

Hence we can find a constant $x > 0$, depending only on $\Ceucl$, with the following property. 
There are at least $\lfloor x / \ep \rfloor$ pairs $(z_1 , r_1) ,\dots , (z_{\lfloor x /\ep \rfloor} , r_{\lfloor x /\ep \rfloor})$, each consisting of a radius $r_j \in \mcl R \cap [\ep^2 \BB r , \ep \BB r] $ and a point $z_j \in \left(\frac{r}{100} \BB Z^2 \right) \cap B_{3\BB r}(0)$, such that the following is true.
\begin{enumerate}[($i$)]
\item The balls $B_{4r_j}(z_j)$ for $j=1,\dots,\lfloor x / \ep \rfloor$ are disjoint and none of these balls intersects $K_1 \cup K_2 \cup \bdy B_{\BB r}(0) $. \label{item-choice-disjoint}
\item $\Er_{z_j,r_j}$ occurs for each $j=1,\dots,\lfloor x / \ep \rfloor$. \label{item-choice-good}
\item The path $P$ hits $B_{r_j/25}(z_j)$ for each $j=1,\dots,\lfloor x /\ep\rfloor$. \label{item-choice-hit}
\end{enumerate} 

By condition~\ref{item-main-exponent} in the definition of $\mcl G_{\BB r}^\ep$, for each $j \in [1,\lfloor x/\ep \rfloor]_{\BB Z}$ there exists $q \in [\ep^{2\xi (Q+3)} / 2 , \ep^{ \xi(Q-3)}] \cap \{2^{-\el}\}_{\el\in\BB N} $ such that $r_j^{\xi Q} e^{\xi h_{r_j}(z_j)} \in \left[ q \BB r^{\xi Q} e^{\xi h_{\BB r}(0)}   ,2q \BB r^{\xi Q} e^{\xi h_{\BB r}(0)} \right]$. 
The cardinality of the set 
\eqbn
\left( \mcl R \cap [\ep^2 \BB r ,\ep \BB r] \right) \times \left( \left[ \frac12 \ep^{2\xi (Q+3)}  , \ep^{ \xi(Q-3)} \right] \cap \{2^{-\el}\}_{\el\in\BB N} \right)
\eqen
is at most a constant (depending only on $\xi$) times $(\log\ep^{-1})^2$.  
So, there must be some $r \in \mcl R \cap [\ep^2 \BB r ,\ep \BB r]$ and $q \in [\ep^{2\xi (Q+3)} / 2 , \ep^{ \xi(Q-3)}] \cap \{2^{-\el}\}_{\el\in\BB N} $ such that 
\allb \label{eqn-pidgeonhole}
&\# \mcl J \succeq \frac{1}{\ep (\log\ep^{-1})^2}  \quad \text{where} \notag\\ 
&\qquad \mcl J := \left\{ j \in [1,\lfloor x\ep^{-1} \rfloor]_{\BB Z} \: : \: r_j = r , \: r_j^{\xi Q} e^{\xi h_{r_j}(z_j)}  \in \left[ q \BB r^{\xi Q} e^{\xi h_{\BB r}(0)}   ,2q \BB r^{\xi Q} e^{\xi h_{\BB r}(0)} \right] \right\} ,
\alle
with the implicit constant depending only on $x$ (hence only on $\Ceucl$). Henceforth fix such an $r$ and $q$ and let $\mcl J$ be as in~\eqref{eqn-pidgeonhole}. Also define
\eqb \label{eqn-index-to-pt} 
\mcl S := \left\{z_j : j \in \mcl J \right\} ,\quad \text{so that} \quad \# \mcl S \succeq \frac{1}{\ep (\log\ep^{-1})^2} .
\eqe 

If $Z \subset \mcl S$, then property~\eqref{item-choice-hit} above implies that $Z \in \mcl Z_r$, where $\mcl Z_r$ is defined as in~\eqref{eqn-pt-set}. 
Furthermore, since $q \geq \ep^{2\xi (Q+3)}/2$, condition~\ref{item-main-bad} in the definition of $\mcl G_{\BB r}^\ep $ implies that $\wt D_h(\op{dist}(K_1,K_2)) \geq \Cupper D_h(\op{dist}(K_1,K_2)) - q\BB r^{\xi Q} e^{\xi h_{\BB r}(0)} $. From this together with properties~\eqref{item-choice-good} and~\eqref{item-choice-hit} above and our choice of $\mcl J$ in~\eqref{eqn-pidgeonhole}, we see that the event $\ol F_{Z,r}^{q,\BB r}(h)$ of Definition~\ref{def-simpler-F} occurs.  

By Lemma~\ref{lem-stable-subset}, for each $Z\subset \mcl S$ there exists $Z'\subset Z$ such that $F_{Z',r}^{q,\BB r} (h)$ occurs and $\# Z' \geq \Csub \# Z$. Fix (in some arbitrary manner) a choice of $Z'$ for each $Z$, so that $Z\mapsto Z'$ is a function from subsets of $\mcl S$ to subsets of $\mcl S$ for which $F_{Z',r}^{q,\BB r}(h)$ occurs.  We will now lower-bound the cardinality of the set
\eqb \label{eqn-set-to-count}
\left\{Z' : \# Z = k \right\} .
\eqe 

To this end, consider a set $\wt Z \subset \mcl S$ for which $F_{\wt Z,r}^{q,\BB r}(h)$ occurs and $\# \wt Z \in [\Csub k , k]$ (i.e., $\wt Z$ is a possible choice of the set $Z'$ when $\# Z  =k$).  
Since $Z' \subset Z$ for each $Z\subset \mcl S$, the number of $Z \subset \mcl S$ such that $\# Z = k$ and $Z' = \wt Z$ is at most the number of possibilities for the set $Z \setminus \wt Z$ (subject to $\# Z = k$ and $Z' = \wt Z$), which is at most
\eqbn
\binom{\# \mcl S}{k - \# \wt Z} \leq \binom{\# \mcl S}{\lfloor (1-\Csub) k \rfloor } .
\eqen
On the other hand, for each $k \in\BB N$, the number of sets $Z \subset \mcl S$ such that $\# Z = k$ is $\binom{\#\mcl S}{k}$. 

The cardinality of the set~\eqref{eqn-set-to-count} is least the number of $Z \subset \mcl S$ with $\# Z = k$, divided by the maximal cardinality of the pre-image of a set $\wt Z$ under $Z\mapsto Z'$. Hence, by combining the two counting formulas from the previous paragraph, we get that the cardinality of the set in~\eqref{eqn-set-to-count}, and hence the number of sets $\wt Z \subset \mcl S$ for which $F_{\wt Z,r}^{q,\BB r}(h)$ occurs and $\# \wt Z \in [\Csub k , k]$, is at least
\eqbn
\binom{\#\mcl S}{k}   \binom{\# \mcl S}{\lfloor (1-\Csub) k \rfloor }^{-1} \succeq (\#\mcl S)^{\Csub k} \succeq  \ep^{-\Csub k} (\log\ep^{-1})^{-2\Csub k}
\eqen
with the implicit constant depending only on the parameters and $k$ (in the last inequality we used~\eqref{eqn-index-to-pt}). This gives~\eqref{eqn-choices} for $\Cexp$ slightly smaller than $\Csub$. 
\end{proof}

\subsection{Proof of Proposition~\ref{prop-card}}
\label{sec-counting-card}

The proof of Proposition~\ref{prop-card} is based on counting the number of points $z\in \frac{r}{100} \BB Z^2$ which could possibly be an element of some $Z\in\mcl Z_r$ for which $F_{Z,r}^{q,\BB r}(h+\fr_{Z,r})$ occurs. To this end, we make the following definition. 

\begin{defn} \label{def-good-count}
For $r\in \mcl R$ and $q > 0$, we say that $z \in \frac{r}{100} \BB Z^2$ is \emph{$r,q$-good} if the following conditions are satisfied. 
\begin{enumerate}[($i$)]
\item The event $\Er_{z,r}(h + \fr_{z,r})$ occurs. \label{item-good-count-good}
\item $r^{\xi Q}  e^{\xi h_r(z)} \in \left[ q  \BB r^{\xi Q} e^{\xi h_{\BB r}(0)} ,2q  \BB r^{\xi Q} e^{\xi h_{\BB r}(0)} \right]$. \label{item-good-count-compare}
\item Let $P$ be the $D_h$-geodesic from $K_1$ to $K_2$. There is a $(B_{4r}(z) , \Vr_{z,r})$-excursion $(\tau_z',\tau_z ,\sigma_z ,\sigma_z')$ for $P$ such that \label{item-good-count-stable}
\eqb  \label{eqn-good-count-stable}
  D_{h+\fr_{z,r}}\left( P (\tau_z) , P (\sigma_z) ; B_{4r}(z)    \right) \geq \Ctime r^{\xi Q} e^{\xi h_r(z)}  . 
\eqe 
\end{enumerate}
\end{defn}

\begin{lem} \label{lem-event-good}
Let $r\in \mcl R$ and $q > 0$. 
If $Z\in\mcl Z_r$ and $F_{Z,r}^{q,\BB r}(h+\fr_{Z,r})$ occurs, then every $z\in Z$ is $r,q$-good.
\end{lem}
\begin{proof}
Let $z \in Z$ and assume that $F_{Z,r}^{q,\BB r}(h + \fr_{Z,r})$ occurs. 
By condition~\ref{item-ball-set-good} in the definition of $F_{Z,r}^{q,\BB r}(h+\fr_{Z,r})$, the event $\Er_{z,r}(h + \fr_{Z,r})$ occurs. 
Since $\Er_{z,r}(h + \fr_{z,r})$ depends only on $(h+\fr_{z,r})|_{\BB A_{r,4r}(z)}$ and $\fr_{Z,r}  -\fr_{z,r} \equiv 0$ outside of $B_{4r}(z)$, it follows that $\Er_{z,r}(h + \fr_{Z,r}) = \Er_{z,r}(h+\fr_{z,r})$. 
This gives condition~\eqref{item-good-count-good} in Definition~\ref{def-good-count}.

Condition~\eqref{item-good-count-compare} in Definition~\ref{def-good-count} follows from condition~\ref{item-ball-set-compare} in the definition of $F_{Z,r}^{q,\BB r}(h+\fr_{Z,r})$ and the fact that the support of $\fr_{Z,r}$ is disjoint from $\bdy B_{\BB r}(0)$ and from $\bdy B_r(z)$ for each $z\in Z$ (recall~\eqref{eqn-pt-set}). 
By condition~\ref{item-ball-set-stable} in the definition of $F_{Z,r}^{q,\BB r}(h+\fr_{Z,r})$, we get that $z$ satisfies condition~\eqref{item-good-count-stable} of Definition~\ref{def-good-count} but with $D_{h+\fr_{Z,r}  }$ instead of $D_{h+\fr_{z,r}}$ in~\eqref{eqn-good-count-stable}. Since the support of $\fr_{Z,r} - \fr_{z,r}$ is disjoint from  $B_{4r}(z)$, the internal distances of $D_{h+\fr_{Z,r}}$ and $D_{h+\fr_{z,r}}$ on $B_{4r}(z)$ are identical. Hence condition~\eqref{item-good-count-stable} holds. 
\end{proof}

In light of Lemma~\ref{lem-event-good}, we seek to upper-bound the number of $r,q$-good points $z\in \frac{r}{100} \BB Z^2$. 
When doing so, we can assume without loss of generality that $F_{Z_0, r}^{q,\BB r}(h + \fr_{Z_0 ,r})$ occurs for some $Z_0 \in\mcl Z_r$ with $\#Z_0 \leq k$ (otherwise, the proposition statement is vacuous). 
The main input in the proof of Proposition~\ref{prop-card} is the following lemma.

\newcommand{\Ccount}{C_6}

\begin{lem} \label{lem-old-geo-count}
There is a constant $\Ccount > 0$, depending only on the parameters and the laws of $D_h$ and $\wt D_h$, such that the following is true. 
Let $r\in \mcl  R$ and let $Z_0 , Z_1 \in \mcl Z_r$. 
Assume that the event $F_{Z_0 ,r}^{q,\BB r}(h+\fr_{Z_0 ,r})$ occurs, each $z \in Z_1$ is $r,q$-good, and each ball $B_{4r}(z)$ for $z \in Z_1$ is disjoint from $\bigcup_{z' \in Z_0}  B_{4r}(z')$ (equivalently, $Z_0 \cup Z_1 \in \mcl Z_r$). 
Then
\eqbn
\# Z_1 \leq  \Ccount  \# Z_0 .
\eqen
\end{lem}

We now explain the idea of the proof of Lemma~\ref{lem-old-geo-count}. 
By condition~\ref{item-ball-set-bad} in the definition of $F_{Z_0 ,r}^{q,\BB r}(h + \fr_{Z_0,r})$, on this event,
\eqb  \label{eqn-use-ball-set-bad}
\wt D_{h + \fr_{Z_0,r}}(K_1,K_2) \geq \Cupper D_{h + \fr_{Z_0,r}}(K_1,K_2) -   q \BB r^{\xi Q} e^{\xi h_{\BB r}(0)}  . 
\eqe
We will show that if $\# Z_1$ is too much larger than $\# Z_0$, then~\eqref{eqn-use-ball-set-bad} cannot hold. The reason for this is as follows.
Let $P$ be the $D_h$-geodesic from $K_1$ to $K_2$.
By condition~\eqref{item-good-count-stable} in Definition~\ref{def-good-count}, each $z\in Z_1$ satisfies the condition of hypothesis~\ref{item-Ehyp-inc} for the event $\Er_{z,r}(h + \fr_{z,r})$. Hypothesis~\ref{item-Ehyp-inc} therefore gives us a pair of times $s_z  , t_z \in P^{-1}( B_{4r}(z) )$ such that $t_z - s_z \geq \Cinc q \BB r^{\xi Q} e^{\xi h_{\BB r}(0)}$ and 
\eqb  \label{eqn-old-geo-count-shortcut}
\wt D_h(P(s_z) , P(t_z) ; B_{4r}(z)) \leq \Cmid (t_z - s_z) = \Cmid D_h(P(s_z) , P(t_z)) . 
\eqe 
Since $\fr_{Z_0,r}$ vanishes on $B_{4r}(z)$ for each $z\in Z_1$ and $\fr_{Z_0,r}$ is non-negative, the relation~\eqref{eqn-old-geo-count-shortcut} implies that also 
\eqbn
\wt D_{h+\fr_{Z_0,r}}(P(s_z) , P(t_z) ; B_{4r}(z)) \leq  \Cmid D_{h+\fr_{Z_0,r}}(P(s_z) , P(t_z)) .  
\eqen
In other words, we have at least $\# Z_1$ ``shortcuts" along $P$ where the $\wt D_{h+\fr_{Z_0,r}}$-distance is at most $\Cmid$ times the $D_{h+\fr_{Z_0,r}}$-distance. 
By following $P$ and taking these shortcuts, we obtain a path from $K_1$ to $K_2$ whose $\wt D_{h+\fr_{Z_0,r}}$-length is at most $\Cupper$ times the $D_{h+\fr_{Z_0,r}}$-length of $P$ minus a positive constant times $q \BB r^{\xi Q} e^{\xi h_{\BB r}(0)} \# Z_1$ (see~\eqref{eqn-subtract-inc}). We then use Lemma~\ref{lem-old-geo-length} just below to upper-bound the $D_{h+\fr_{Z_0,r}}$-length of $P$ in terms of $\# Z_0$. This leads to an upper bound for $\wt D_{h + \fr_{Z_0,r}}(K_1,K_2)$ which is inconsistent with~\eqref{eqn-use-ball-set-bad} unless $\# Z_1$ is bounded above by a constant times $\# Z_0$. 

We need the following lemma for the proof of Lemma~\ref{lem-old-geo-count}. 

\begin{lem} \label{lem-old-geo-length}
Let $\Cspent > 0$ be as in Lemma~\ref{lem-new-dist-spent}.  
Let $r\in \mcl  R$, $Z \in \mcl Z_r$, and $q > 0$ and assume that $F_{Z,r}^{q,\BB r}(h + \fr_{Z,r})$ occurs.  
Then the $D_h$-geodesic $P$ from $K_1$ to $K_2$ satisfies
\eqb \label{eqn-old-geo-length}
\op{len}\left(   P  ;   D_{h+\fr_{Z,r}} \right) \leq D_h(K_1,K_2 )   +    \Cspent q \BB r^{\xi Q} e^{\xi h_{\BB r}(0)}   \# Z .
\eqe 
\end{lem}  
\begin{proof} 
The function $\fr_{Z,r}$ is supported on $\bigcup_{z\in Z} \Vr_{z,r}$. By Weyl scaling (Axiom~\ref{item-metric-f}), 
\eqb \label{eqn-old-geo-btwn}
\op{len}\left(  P \setminus \bigcup_{z\in Z} \Vr_{z,r} ;     D_{h+\fr_{Z,r}}   \right) 
= \op{len} \left( P \setminus \bigcup_{z\in Z} \Vr_{z,r} ;   D_h  \right) 
\leq D_h(K_1,K_2).
\eqe
By Lemma~\ref{lem-new-dist-spent}, applied with $h + \fr_{Z,r}$ in place of $h$, 
\eqb \label{eqn-old-geo-in}
\op{len} \left( P\cap \Vr_{z,r} ;     D_{h+\fr_{Z,r}}  \right) \leq \Cspent q  \BB r^{\xi Q} e^{\xi h_{\BB r}(0)} , \quad \forall z \in Z .   
\eqe
Combining~\eqref{eqn-old-geo-btwn} and~\eqref{eqn-old-geo-in} yields~\eqref{eqn-old-geo-length}.
\end{proof}

\begin{proof}[Proof of Lemma~\ref{lem-old-geo-count}]
Let $P$ be the $D_h$-geodesic from $K_1$ to $K_2$.
By conditions~\eqref{item-good-count-good} and~\eqref{item-good-count-stable} in Definition~\ref{def-good-count} together with hypothesis~\ref{item-Ehyp-inc} for the event $\Er_{z,r}(h + \fr_{z,r})$, for each $z\in Z_1$, there are times $0 < s_z < t_z < D_h(K_1,K_2)$ such that $P([s_z,t_z]) \subset B_{4r}(z)$,
\eqb \label{eqn-use-Ehyp-inc0}
t_z - s_z  \geq \Cinc r^{\xi Q} e^{\xi h_r(z)} \geq \Cinc q \BB r^{\xi Q} e^{\xi h_{\BB r}(0)} ,\quad \text{and} \quad
\wt D_h\left(P(s_z) , P(t_z) ; B_{4r}(z) \right) \leq \Cmid (t_z - s_z) .
\eqe
Note that to get $r^{\xi Q} e^{\xi h_r(z)} \geq  q \BB r^{\xi Q} e^{\xi h_{\BB r}(0)}$, we used condition~\eqref{item-good-count-compare} from Definition~\ref{def-good-count} and to get that $P([s_z,t_z]) \subset B_{4r}(z)$, we used Definition~\ref{def-excursion}.  

If $z \in Z_1$, then by hypothesis $B_{4r}(z)$ is disjoint from $\bigcup_{z' \in Z_0}  B_{4r}(z')$. Hence $B_{4r}(z)$ and $P([s_z,t_z])$ are disjoint from the support of $\fr_{Z_0,r}$. We can therefore deduce from~\eqref{eqn-use-Ehyp-inc0} and Weyl scaling (Axiom~\ref{item-metric-f}) that for each $z\in Z_1$, 
\allb \label{eqn-use-Ehyp-inc}
&\op{len}\left( P|_{[s_z,t_z]} ; D_{h+\fr_{Z_0 ,r}}    \right)  =t_z - s_z  \geq \Cinc q \BB r^{\xi Q} e^{\xi h_{\BB r}(0)} \quad \text{and} \quad \notag\\
&\wt D_{h + \fr_{Z_0 ,r}}\left(P(s_z) , P(t_z) ; B_{4r}(z) \right) \leq   \Cmid (t_z - s_z) \leq \Cmid  D_{h+\fr_{Z_0 ,r}} \left( P(s_z) , P(t_z) \right) .
\alle

Let $N = \# Z_1$ and let $z_1,\dots,z_N$ be the elements of $Z_1$, ordered so that 
\eqbn
s_{z_1} < t_{z_1} < s_{z_2} < t_{z_2} < \dots < s_{z_N} < t_{z_N} .
\eqen 
Such an ordering is possible since $P([s_z,t_z]) \subset B_{4r}(z)$, so these path increments are disjoint. 
For notational simplicity, we also define $t_{z_0} = 0$ and $s_{z_{N+1}} = D_h(K_1,K_2)$, so that $P(t_{z_0}) \in K_1$ and $P(t_{z_{N+1}}) \in K_2$. 

By the bi-Lipschitz equivalence of $D_h$ and $\wt D_h$~\eqref{eqn-bilip} and Weyl scaling,  
\eqb \label{eqn-good-count-between}
\wt D_{h+\fr_{Z_0,r}}(P(t_{z_n}) , P(s_{z_{n+1}})) \leq \Cupper D_{h+\fr_{Z_0,r}}(P(t_{z_n}) , P(s_{z_{n+1}}))  ,\quad\forall n \in [0,N]_{\BB Z}.
\eqe
We now have the following estimate:
\allb \label{eqn-subtract-inc}
&\wt D_{h + \fr_{Z_0,r}}(K_1,K_2) \notag\\
&\qquad \leq \sum_{n=1}^N \wt D_{h + \fr_{Z_0,r}}(P(s_{z_n}) , P(t_{z_n})) + \sum_{n=0}^N \wt D_{h + \fr_{Z_0,r}}(P(t_{z_n}) , P(s_{z_{n+1}})) \notag\\
&\qquad\qquad\qquad \text{(triangle inequality)} \notag \\
&\qquad \leq \Cmid \sum_{n=1}^N  D_{h + \fr_{Z_0,r}}(P(s_{z_n}) , P(t_{z_n})) + \Cupper \sum_{n=0}^N  D_{h + \fr_{Z_0,r}}(P(t_{z_n}) , P(s_{z_{n+1}})) \notag\\
&\qquad\qquad\qquad \text{(by~\eqref{eqn-use-Ehyp-inc} and~\eqref{eqn-good-count-between})} \notag  \\
&\qquad =   \Cupper \left[ \sum_{n=1}^N  D_{h + \fr_{Z_0,r}}(P(s_{z_n}) , P(t_{z_n}))  +     \sum_{n=0}^N  D_{h + \fr_{Z_0,r}}(P(t_{z_n}) , P(s_{z_{n+1}})) \right] \notag\\
&\qquad\qquad\qquad  -   (\Cupper - \Cmid) \sum_{n=1}^N  D_{h + \fr_{Z_0,r}}(P(s_{z_n}) , P(t_{z_n}))   \notag  \\
&\qquad\leq   \Cupper  \op{len}\left( P ;  D_{h+\fr_{Z_0,r}}  \right)    -   (\Cupper - \Cmid)  \Cinc q \BB r^{\xi Q} e^{\xi h_{\BB r}(0)} \# Z_1 \quad \text{(by~\eqref{eqn-use-Ehyp-inc})}  \notag  \\
&\qquad\leq   \Cupper  D_h(K_1,K_2) +   \Cupper \Cspent q \BB r^{\xi Q} e^{\xi h_{\BB r}(0)}  \# Z_0   -   (\Cupper - \Cmid)  \Cinc q \BB r^{\xi Q} e^{\xi h_{\BB r}(0)} \# Z_1 \notag\\
&\qquad\qquad\qquad \text{(by Lemma~\ref{lem-old-geo-length})}  \notag  \\
&\qquad\leq   \Cupper  D_{h + \fr_{Z_0,r}}(K_1,K_2) +   \Cupper \Cspent q \BB r^{\xi Q} e^{\xi h_{\BB r}(0)}  \# Z_0   -   (\Cupper - \Cmid)  \Cinc q \BB r^{\xi Q} e^{\xi h_{\BB r}(0)} \# Z_1 \notag\\
&\qquad\qquad\qquad \text{(since $\fr_{Z_0,r} \geq 0$)} .    
\alle

By combining~\eqref{eqn-use-ball-set-bad} and~\eqref{eqn-subtract-inc}, we obtain
\alb
&(\Cupper - \Cmid)  \Cinc q \# Z_1  - \Cupper \Cspent q  \BB r^{\xi Q} e^{\xi h_{\BB r}(0)} \# Z_0  \leq  q \BB r^{\xi Q} e^{\xi h_{\BB r}(0)} \leq q \BB r^{\xi Q} e^{\xi h_{\BB r}(0)} \# Z_0  \notag \\
&\qquad \text{which implies} \quad \# Z_1 \leq \Ccount \# Z \quad \text{where} \quad \Ccount :=  \frac{ 1 +  \Cupper \Cspent }{ (\Cupper - \Cmid) \Cinc} .
\ale
\end{proof}

\begin{proof}[Proof of Proposition~\ref{prop-card}]
We can assume that there exists some $Z_0 \in \mcl Z_r$ with $\# Z_0 \leq k$ such that $F_{Z_0,r}^{q,\BB r}(h + \fr_{Z_0 ,r})$ occurs (otherwise,~\eqref{eqn-card} holds vacuously).  
Let $Z_1 \in \mcl Z_r$ be a set such that each $z \in Z_1$ is $r,q$-good (Definition~\ref{def-good-count}) and each $B_{4r}(z)$ for $z\in Z_1$ is disjoint from $\bigcup_{z'\in Z_0} B_{4r}(z')$. 
We assume that $\# Z_1$ is maximal among all subsets of $\mcl Z_r$ with this property. 
By Lemma~\ref{lem-old-geo-count}, we have $\# Z_1 \leq \Ccount k$. 

Now let $Z\in \mcl Z_r$ such that $F_{Z,r}^{q,\BB r}(h + \fr_{Z,r})$ occurs. 
We claim that for each $z\in Z$, the ball $B_{4r}(z)$ intersects $B_{4r}(z')$ for some $z' \in Z_0 \cup Z_1$. 
Indeed, by Lemma~\ref{lem-event-good}, each $z \in Z$ is $r,q$-good. 
So, if there is a $z \in Z$ such that $B_{4r}(z)$ is disjoint from $B_{4r}(z')$ for each $z' \in Z_0 \cup Z_1$, then $Z_1 \cup \{z\}$ satisfies the conditions in the definition of $Z_1$. This contradicts the maximality of $\# Z_1$. 

Each $z\in Z$ belongs to $\frac{r}{100} \BB Z^2$. Hence, for each $z' \in Z_0\cup Z_1$, the number of $z\in Z$ for which $B_{4r}(z) \cap B_{4r}(z') \not=\emptyset$ is at most some universal constant $R$. By the preceding paragraph, any $Z\in \mcl Z_r$ such that $F_{Z,r}^{q,\BB r}(h + \fr_{Z,r})$ occurs can be obtained by the following procedure. For each $z' \in Z_0 \cup Z_1$, we either choose a point $z\in \frac{r}{100} \BB Z^2$ such that $B_{4r}(z) \cap B_{4r}(z') \not=\emptyset$; or we choose no point (so we have at most $R+1$ choices for each $z' \in Z_0\cup Z_1$). Then, we take $Z$ to be the set of points that we have chosen. Therefore, 
\allb
\#\left\{ Z\in \mcl Z_r \: : \: \text{$\# Z \leq k$ and $F_{Z,r}^{q,\BB r}(h + \fr_{Z,r})$ occurs} \right\} 
&\leq (R+1)^{\# Z_0 + \# Z_1} \notag\\
&\leq (R+1)^{(\Ccount + 1) k} .
\alle  
This gives~\eqref{eqn-card} with $\Ccard = (R+1)^{\Ccount+1}$. 
\end{proof}

\subsection{Proof of uniqueness assuming Proposition~\ref{prop-objects-exist}}
\label{sec-counting-conclusion}

In this subsection, we will prove Theorem~\ref{thm-weak-uniqueness}, which asserts the uniqueness of weak LQG metrics, assuming Proposition~\ref{prop-objects-exist}. 
As explained in Section~\ref{sec-bilip}, it suffices to show that the optimal bi-Lipschitz constants satisfy $\Clower = \Cupper$. 
To accomplish this, we will assume by way of contradiction that $\Clower < \Cupper$. We also assume the conclusion of Proposition~\ref{prop-objects-exist} (whose proof has been postponed). 
Throughout this subsection, we fix $\BB p\in (0,1)$ (which will be chosen in Lemma~\ref{lem-main-extra} below) and we let $\Cmid' \in (\Clower , \Cupper)$ and $\mcl R\subset  (0,1)$ be as in Proposition~\ref{prop-objects-exist} for this choice of $\BB p$. We also assume that the parameters of Section~\ref{sec-counting-setup} have been chosen as in Proposition~\ref{prop-objects-exist} for our given choice of $\BB p$.  

We first check that the auxiliary conditions in the definition of the event $\mcl G_{\BB r}^\ep(K_1,K_2)$ of Section~\ref{sec-counting-main} occur with high probability when $\ep$ is small, which together with Proposition~\ref{prop-counting} leads to an upper bound for the probability of the main condition
\eqbn
\wt D_h(K_1,K_2) \geq \Cupper D_h(K_1,K_2) - \frac12 \ep^{2\xi(Q+3)} \BB r^{\xi Q} e^{\xi h_{\BB r}(0)} .
\eqen
We note that the auxiliary conditions do not depend on $K_1$ and $K_2$.

\begin{lem} \label{lem-main-extra}
There is a universal choice of the parameter $\BB p \in (0,1)$ such that the following is true. 
Let $\wt\Kopt > 0$ and let $\BB r > 0$ such that $\BB P[\wt G_{\BB r}(\wt\Kopt , \Cmid')] \geq \wt\Kopt$. 
It holds with probability tending to 1 as $\ep\rta 0$ (at a rate depending only on $\wt\Kopt$ and the laws of $D_h$ and $\wt D_h$, not on $\BB r$) that conditions~\ref{item-main-exponent} and~\ref{item-main-cover} in the definition of $\mcl G_{\BB r}^\ep$ occur, i.e.,
\begin{enumerate} 
\addtocounter{enumi}{1}
\item \label{item-main-exponent'} For each $z\in B_{3\BB r}(0)$ and each $r\in [\ep^2 \BB r , \ep \BB r] \cap \mcl R$, we have 
\eqbn
r^{\xi Q} e^{\xi h_r(z)} \in \left[\ep^{2\xi(Q+3)} \BB r^{\xi Q} e^{\xi h_{\BB r}(0)} , \ep^{ \xi(Q-3)} \BB r^{\xi Q} e^{\xi h_{\BB r}(0)} \right] . 
\eqen
\item For each $z \in B_{3\BB r}(0) $, there exist $r \in \mcl R \cap [\ep^2 \BB r ,\ep \BB r]$ and $w \in \left( \frac{r}{100} \BB Z^2 \right) \cap B_{r/25}(z)$ such that $\Er_{w,r}$ occurs. \label{item-main-cover'}
\end{enumerate}
\end{lem}
\begin{proof}
By a standard estimate for the circle average process of the GFF (see, e.g.,~\cite[Proposition 2.4]{lqg-tbm2}), it holds with polynomially high probability as $r\rta 0$ that $|h_r(z) | \leq 3\log r^{-1}$ for all $z \in B_3(0)$. By the scale invariance of the law of $h$, modulo additive constant, we get that with polynomially high probability as $r\rta 0$ (at a universal rate) we have $|h_r(z)  - h_{\BB r}(0)| \leq 3\log(\BB r / r)$ for all $z\in B_{3\BB r}(0)$. 
By a union bound over logarithmically many values of $r\in \mcl R\cap [\ep^2 \BB r, \ep\BB r]$, we get that with probability tending to 1 as $\ep\rta 0$, 
\allb \label{eqn-circle-avg3}
&|h_r(z) - h_{\BB r}(0)| \leq 3\log(\BB r/r) \in [3\log \ep^{-2} , 3\log \ep^{-1} ] , \notag\\
&\qquad\qquad \forall r \in \mcl R\cap [\ep^2 \BB r , \ep\BB r] ,\quad \forall z \in B_{3\BB r}(0) .
\alle
The bound~\eqref{eqn-circle-avg3} immediately implies condition~\ref{item-main-exponent'} in the lemma statement. 

We now turn our attention to condition~\ref{item-main-cover'}.   
By the properties of the events $\Er_{z,r}$, we know that $\Er_{z,r} $ is a.s.\ determined by $ h |_{\ol{\BB A}_{r,4r}(z)} $, viewed modulo additive constant, and $\BB P[\Er_{z,r}] \geq \BB p$ for each $z\in\BB C$ and $r\in\mcl R$. Furthermore, by Proposition~\ref{prop-objects-exist} our hypothesis that $\BB P[\wt G_{\BB r}(\wt\Kopt , \Cmid')] \geq \wt\Kopt$ implies that for each small enough $\ep > 0$ (how small depends only on $\wt\Kopt$ and the laws of $D_h$ and $\wt D_h$), 
\eqbn
\#\left( \mcl R \cap [\ep^2 \BB r , \ep \BB r] \right) \geq \frac58 \log_8\ep^{-1} .
\eqen
We may therefore apply Lemma~\ref{lem-annulus-offcenter} with the radii $r_k \in \mcl R\cap [\ep^2 \BB r , \ep\BB r]$, the points $z_k \in \frac{r_k}{100} \BB Z^2   $ chosen so that $|z-z_k| \leq r_k/50$, and the events $E_{r_k}(z_k) = \Er_{z_k , r_k}$. 
From Lemma~\ref{lem-annulus-offcenter}, we obtain that if $\BB p$ is chosen to be sufficiently close to 1, in a universal manner, then for each $z \in \BB C$, it holds with probability at least $1-O_\ep(\ep^5)$ (at a rate depending only on the laws of $D_h$ and $\wt D_h$) that there exist $r\in \mcl R\cap [\ep^2 \BB r , \ep\BB r]$ and $w \in \left(\frac{r}{100} \BB Z^2 \right) \cap B_{r/50}(z)$ such that $\Er_{w,r}$ occurs. 

By a union bound, it holds with probability tending to 1 as $\ep\rta 0$ (at a rate depending only on $\wt\Kopt$ and the laws of $D_h$ and $\wt D_h$) that for each $z\in \left( \frac{\ep^2 \BB r}{100} \BB Z^2 \right) \cap B_{3\BB r}(0)$, there exist $r\in \mcl R\cap [\ep^2 \BB r , \ep\BB r]$ and $w \in \left(\frac{r}{100} \BB Z^2 \right) \cap B_{r/50}(z)$ such that $\Er_{w,r}$ occurs.
Henceforth assume that this is the case. 
For a general choice of $z\in B_{3\BB r}(0)$, we choose $z' \in \left( \frac{\ep^2 \BB r}{100} \BB Z^2 \right) \cap B_{3\BB r}(0)$ such that $|z-z'| \leq \ep^2 \BB r/50$, then we choose $r\in \mcl R\cap [\ep^2 \BB r , \ep\BB r]$ and $w \in  \left(\frac{r}{100} \BB Z^2 \right) \cap B_{r/50}(z')$ such that $\Er_{w,r}$ occurs.
Then $|w-z'| \leq (\ep^2 \BB r + r) / 50 \leq r/ 25$. 
Hence condition~\ref{item-main-cover'} in the lemma statement holds with probability tending to 1 as $\ep\rta 0$. 
\end{proof}

We henceforth assume that the parameter $\BB p$ is chosen as in Lemma~\ref{lem-main-extra}. 
By combining Proposition~\ref{prop-counting} with Lemma~\ref{lem-main-extra}, we obtain the following.

\begin{lem} \label{lem-end-union}
Let $\wt\Kopt > 0$ and let $\BB r > 0$ such that $\BB P[\wt G_{\BB r}(\wt\Kopt , \Cmid')] \geq \wt\Kopt$. 
Also let $\nu > 0$ and $\Kopt > 0$. It holds with probability tending to 1 as $\delta \rta 0$ (at a rate depending only on $\nu ,\wt\Kopt, \Kopt$ and the laws of $D_h$ and $\wt D_h$) that
\allb
&\wt D_h\left( B_{\delta^\nu \BB r  }(  z) , B_{\delta^\nu \BB r}(  w) \right)  
 \leq  \Cupper  D_h\left( B_{\delta^\nu \BB r  }(  z) , B_{\delta^\nu \BB r}(  w) \right) - \delta \BB r^{\xi Q} e^{\xi h_{\BB r}(0)} , \notag\\
&\qquad \forall z,w\in \left(\frac{\delta^\nu \BB r}{100} \BB Z^2  \right) \cap B_{\BB r}(0) \quad \text{such that} \quad |z-w| \geq  \Kopt \BB r \notag\\
&\qquad\qquad  \text{and} \quad \op{dist}(z,\bdy B_{\BB r}(0)) \geq \Kopt \BB r .
\alle
\end{lem}
\begin{proof} 
Fix $\nu' > 0$ to be chosen later, in a manner depending only on $\nu$ and $\xi$. 
By Proposition~\ref{prop-counting} (applied with $\Ceucl  =  \Kopt/2$) and a union bound, it holds with superpolynomially high probability as $\ep\rta 0$ that the event $\mcl G_{\BB r}^\ep\left( \ol B_{\ep^{\nu'} \BB r}(z) ,  \ol B_{\ep^{\nu'} \BB r }(w) \right) $ does not occur for any pair of points $z,w\in \left(\frac{\ep^{\nu'}  \BB r}{100} \BB Z^2  \right) \cap B_{\BB r}(0)$ with $|z-w| \geq  \Kopt \BB r$ and $\op{dist}(z,\bdy B_{\BB r}(0)) \geq \Kopt \BB r$. By combining this with Lemma~\ref{lem-main-extra} and recalling the definition of $\mcl G_{\BB r}^\ep$ (in particular, condition~\ref{item-main-bad}), we get that with probability tending to 1 as $\ep\rta 0$,  
\allb \label{eqn-end-union0}
& \wt D_h\left( B_{ \ep^{\nu'} \BB r  }(  z) , B_{ \ep^{\nu'} \BB r}(  w) \right)  
 \leq \Cupper D_h\left( B_{\ep^{\nu'} \BB r   }(  z) , B_{\ep^{\nu'} \BB r}(  w) \right)  
 -   \ep^{2\xi(Q+3)} \BB r^{\xi Q} e^{\xi h_{\BB r}(0)} ,\quad \notag\\
&\qquad \forall z,w \in \left(\frac{\ep^{\nu'}  \BB r}{100} \BB Z^2  \right) \cap B_{\BB r}(0)  \quad \text{such that} \quad |z-w| \geq  \Kopt \BB r \notag\\
&\qquad\qquad \text{and} \quad \op{dist}(z,\bdy B_{\BB r}(0)) \geq \Kopt \BB r .
\alle 
We now conclude the proof by applying the above estimate with $\ep = \ep(\delta) > 0$ chosen so that $ \ep^{2\xi(Q+3) } = \delta$ and with $\nu' = \nu / (2\xi(Q+3))$. 
\end{proof}

Recall the definition of the event $H_{\BB r}(\Kann , \Cmed)$ from Definition~\ref{def-annulus-geo}, which says that there is a point $u\in \bdy B_{\Kann \BB r}(0)$ and a point $v\in \bdy B_{\BB r}(0)$ satisfying certain conditions such that $\wt D_h(u,v) \leq \Cmed D_h(u,v)$.  
From Lemma~\ref{lem-end-union} and a geometric argument, we obtain the following, which will eventually be used to get a contradiction to Proposition~\ref{prop-geo-annulus-prob}. 

\begin{lem} \label{lem-opt-event-prob}
Let $\wt\Kopt > 0$ and let $\BB r > 0$ such that $\BB P[\wt G_{\BB r}(\wt\Kopt , \Cmid')] \geq \wt\Kopt$. 
For each $\Kann \in (3/4,1)$, we have 
\eqbn
\lim_{\delta \rta 0} \BB P\left[ H_{\BB r}(\Kann , \Cupper - \delta) \right] = 0 
\eqen
at a rate depending only on $\Kann , \wt\Kopt$, and the laws of $D_h$ and $\wt D_h$. 
\end{lem}
\begin{proof}
Let $\nu > 0$ to be chosen later, in a manner depending only on the laws of $D_h$ and $\wt D_h$. 
By Lemma~\ref{lem-end-union} applied with $\Kopt = (1-\Kann)/2$, it holds with probability tending to 1 as $\delta \rta 0$ that
\allb \label{eqn-use-end-union}
&\wt D_h\left( B_{\delta^\nu \BB r  }(  z) , B_{\delta^\nu \BB r}(  w) \right)  
 \leq  \Cupper  D_h\left( B_{\delta^\nu \BB r  }(  z) , B_{\delta^\nu \BB r}(  w) \right) - \delta \BB r^{\xi Q} e^{\xi h_{\BB r}(0)} , \notag\\
&\qquad  \forall z,w\in \left(\frac{\delta^\nu \BB r}{100} \BB Z^2  \right) \cap B_{\BB r}(0) \quad \text{such that} \quad |z-w| \geq  \frac{1-\Kann}{2} \BB r \notag\\
&\qquad\qquad \text{and} \quad \op{dist}(z,\bdy B_{\BB r}(0)) \geq \frac{1-\Kann}{2} \BB r .
\alle
Henceforth assume that that~\eqref{eqn-use-end-union} holds.

Recalling Definition~\ref{def-annulus-geo}, we consider points $u \in \bdy B_{\Kann \BB r}(0)$ and $v\in \bdy B_{\BB r}(0)$ such that 
\begin{itemize}
\item $D_h(u,v) \leq (1-\Kann)^{-1} \BB r^{\xi Q} e^{\xi h_{\BB r}(0)}$; and 
\item For each $ \delta \in \left(0,(1-\Kann)^2\right]$, we have
\eqb \label{eqn-end-geo-loop} 
\max\left\{    D_h\left( u , \bdy B_{\delta \BB r}(u) \right)  ,   D_h\left(\text{around $\BB A_{\delta \BB r , \delta^{1/2} \BB r}(u)$} \right)   \right\} \leq \delta^\geoExp D_h(u,v)
\eqe 
and the same is true with the roles of $u$ and $v$ interchanged. 
\end{itemize}
We will show that if $\nu$ is chosen to be large enough (depending only on the laws of $D_h$ and $\wt D_h$), then for each small enough $\delta >0$ (depending only on $\Kann , \wt\Kopt$, and the laws of $D_h$ and $\wt D_h$), we have
\eqb \label{eqn-opt-event-show}
\wt D_h(u,v) \leq \left( \Cupper  -  \frac{1-\Kann}{4} \delta   \right)  D_h(u,v)  , \quad \text{$\forall u,v$ satisfying the above conditions} .
\eqe
By Definition~\ref{def-annulus-geo}, the relation~\eqref{eqn-opt-event-show} implies that $H_{\BB r}\left(\Kann , \Cupper - \frac{1-\Kann}{4} \delta   \right)$ does not occur. Since $\delta$ can be made arbitrarily small, this implies the lemma statement.

\begin{figure}[ht!]
\begin{center}
\includegraphics[width=.8\textwidth]{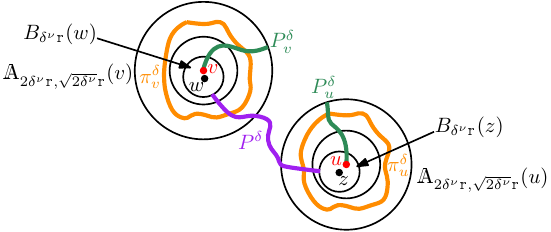} 
\caption{\label{fig-opt-event-prob} 
Illustration of the five paths used to get an upper bound for $\wt D_h(u,v)$ in the proof of Lemma~\ref{lem-opt-event-prob}. The $\wt D_h$-length of $P^\delta$ is bounded above using~\eqref{eqn-use-end-union} and the $\wt D_h$-lengths of the other four paths are bounded above using~\eqref{eqn-end-geo-loop}. 
}
\end{center}
\end{figure}

See Figure~\ref{fig-opt-event-prob} for an illustration of the proof of~\eqref{eqn-opt-event-show}. Let $z \in  \left(\frac{\delta^\nu \BB r}{100} \BB Z^2  \right) \cap B_{\delta^\nu \BB r }(u)$ and $w \in \left(\frac{\delta^\nu \BB r}{100} \BB Z^2  \right) \cap B_{\delta^\nu \BB r }(v)$. If $\delta$ is small enough, then $|z-w| \geq (1-\Kann) \BB r/2$ and $\op{dist}(z,\bdy B_{\BB r}(0)) \geq ( 1-\Kann)\BB r / 2$. By~\eqref{eqn-use-end-union}, there is a path $P^\delta$ from $B_{\delta^\nu \BB r}(z)$ to $B_{\delta^\nu \BB r}(w)$ such that
\allb \label{eqn-opt-path-main}
\op{len}\left( P^\delta ; \wt D_h\right) 
&\leq \Cupper  D_h\left( B_{\delta^\nu \BB r  }(  z) , B_{\delta^\nu \BB r}(  w) \right) - \frac{\delta}{2}  \BB r^{\xi Q} e^{\xi h_{\BB r}(0)} \notag\\
&\leq \Cupper  D_h\left( u , v \right) - \frac{\delta}{2}  \BB r^{\xi Q} e^{\xi h_{\BB r}(0)} \quad \text{(since $u \in B_{\delta^\nu \BB r}(z)$ and $v\in B_{\delta^\nu \BB r}(w)$)} \notag\\
&\leq \left( \Cupper -  \frac{1-\Kann}{2} \delta \right) D_h(u,v) \quad \text{(since $D_h(u,v) \leq (1-\Kann)^{-1} \BB r^{\xi Q} e^{\xi h_{\BB r}(0)}$)} . 
\alle

By~\eqref{eqn-end-geo-loop} (applied with $\sqrt{2\delta^\nu}$ in place of $\delta$), if $\delta$ is small enough (depending on $\Kann$) then there are paths $P_u^\delta$ and $P^\delta_v$ from $u$ and $v$ to $\bdy B_{ \sqrt{2 \delta^\nu} \BB r }(u)$ and $\bdy B_{\sqrt{2 \delta^\nu} \BB r }(v)$, respectively, such that
\eqb \label{eqn-opt-path-circle}
\max\left\{\op{len}(P_u^\delta  ; D_h) ,\op{len}( P_v^\delta ; D_h) \right\} \leq 2^{\geoExp/2} \delta^{\nu \geoExp/2} D_h(u,v) .
\eqe
Furthermore, by~\eqref{eqn-end-geo-loop} applied with $2\delta^\nu$ in place of $\delta$, there are paths $\pi_u^\delta$ and $\pi_v^\delta$ in $\BB A_{2\delta^\nu \BB r ,  \sqrt{2\delta^\nu} \BB r}(u)$ and  $\BB A_{2\delta^\nu \BB r ,  \sqrt{2\delta^\nu} \BB r}(u)$, respectively, which disconnect the inner and outer boundaries and satisfy
\eqb \label{eqn-opt-path-around}
\max\left\{\op{len}(\pi_u^\delta  ; D_h) , \op{len}(\pi_v^\delta ; D_h) \right\} \leq 2^{\geoExp } \delta^{\nu \geoExp } D_h(u,v) .
\eqe

Since $\max\{|z-u| , |w-v|\} \leq \delta^\nu \BB r$, the union $P^\delta \cup P_u^\delta \cup P_v^\delta \cup \pi_u^\delta \cup \pi_v^\delta$ contains a path from $u$ to $v$. 
Therefore, combining~\eqref{eqn-opt-path-main}, \eqref{eqn-opt-path-circle}, and~\eqref{eqn-opt-path-around}, then using the bi-Lipschitz equivalence of $D_h$ and $\wt D_h$~\eqref{eqn-bilip} gives 
\alb
\wt D_h(u,v) 
&\leq \left( \Cupper  -  \frac{1-\Kann}{2} \delta  \right)  D_h(u,v)  + \sum_{x\in \{u,v\}} \left( \op{len}(P_x^\delta  ;\wt D_h) + \op{len}(\pi_x^\delta  ; \wt D_h) \right) \notag\\
&\leq \left( \Cupper  -  \frac{1-\Kann}{2} \delta +  2^{\geoExp/2 + 1} \Cupper \delta^{\nu \geoExp/2} + 2^{\geoExp +1 } \Cupper \delta^{\nu \geoExp }        \right)  D_h(u,v)  .
\ale
If $\nu > 2/\geoExp$ and $\delta$ is small enough, then this implies~\eqref{eqn-opt-event-show}. 
\end{proof}

\begin{proof}[Proof of Theorem~\ref{thm-weak-uniqueness}]
By Proposition~\ref{prop-geo-annulus-prob}, there exist $\Kann \in (3/4,1)$ and $p\in (0,1)$, depending only on the laws of $D_h$ and $\wt D_h$, such that for each $\delta > 0$ and each small enough $\ep > 0$ (depending only on $\delta$ and the laws of $D_h$ and $\wt D_h$), there are at least $\frac34\log_8\ep^{-1}$ values of $r\in [\ep^2,\ep] \cap \{8^{-k}\}_{k\in\BB N}$ such that 
\eqb \label{eqn-use-geo-annulus-prob}
\BB P[H_r(\Kann , \Cupper - \delta )] \geq p 
\eqe

Let $\Cmid'$ be as in Proposition~\ref{prop-objects-exist}, so that $\Cmid'$ depends only on the laws of $D_h$ and $\wt D_h$. By Proposition~\ref{prop-opt-prob'} (applied with $\Cmid'$ in place of $\Cmid$), there exist $\wt\Kopt > 0$ and $\ep_0 > 0$ (depending only on the laws of $D_h$ and $\wt D_h$) such that for each $\ep \in (0,\ep_0]$, there are at least $\frac34\log_8\ep^{-1}$ values of $r\in [\ep^2 , \ep] \cap \{8^{-k}\}_{k\in\BB N}$ for which $\BB P[\wt G_r(\wt\Kopt , \Cmid')] \geq \wt\Kopt$.  
By combining this with Lemma~\ref{lem-opt-event-prob}, we get that if $\Kann$ and $p$ are as in~\eqref{eqn-use-geo-annulus-prob}, then there exists $\delta > 0$, depending only on $\Kann , p$, and the laws of $D_h$ and $\wt D_h$, such that for each $\ep\in (0,\ep_0]$, there are at least $\frac34 \log_8\ep^{-1}$ values of $r \in [\ep^2 , \ep] \cap \{8^{-k}\}_{k\in\BB N}$ for which
\eqb \label{eqn-use-opt-event-prob}
 \BB P\left[ H_r(\Kann , \Cupper - \delta) \right]  \leq \frac{p}{2 }. 
\eqe  

The total number of radii $r\in [\ep^2 ,\ep]\cap \{8^{-k}\}_{k\in\BB N}$ is at most $\log_8\ep^{-1}$, so there cannot be at least $\frac34 \log_8\ep^{-1}$ values of  $r\in [\ep^2 ,\ep]\cap \{8^{-k}\}_{k\in\BB N}$ for which~\eqref{eqn-use-geo-annulus-prob} holds and at least $\frac34 \log_8\ep^{-1}$ values of  $r\in [\ep^2 ,\ep]\cap \{8^{-k}\}_{k\in\BB N}$ for which~\eqref{eqn-use-opt-event-prob} holds. We thus have a contradiction, so we conclude that $\Clower = \Cupper$. 
\end{proof}

\section{Constructing events and bump functions}
\label{sec-construction}

\subsection{Setup and outline}
\label{sec-construction-outline}
 
The goal of this section is to prove Proposition~\ref{prop-objects-exist}. 
Extending~\eqref{eqn-Cmid-choice}, we define
\eqb \label{eqn-Cmid0-choice}
\Cmid := \frac{\Clower + \Cupper}{2} \quad\text{and} \quad \Cmid_0 := \frac{\Clower + \Cmid}{2} ,
\eqe  
so that if $\Clower < \Cupper$, then $\Clower < \Cmid_0 < \Cmid < \Cupper$. 

Throughout this section, we fix $\BB p \in (0,1)$ as in Proposition~\ref{prop-objects-exist}. 
Note that $\BB p$ is allowed to be arbitrarily close to $1$.
We seek to construct a set of radii $\mcl R \subset (0,1)$ and, for each $z\in \BB C$ and $r\in\mcl R$, open sets $\Ur_{z,r} \subset \Vr_{z,r} \subset \BB A_{r,4r}(z)$, a smooth bump function $\fr_{z,r}$ supported on $\Vr_{z,r}$, and an event $\Er_{z,r}$ with $\BB P[\Er_{z,r}] \geq \BB p$ which satisfy the conditions in Section~\ref{sec-counting-setup}.
 
For simplicity, for most of this section we will take $z = 0$ and remove $z$ from the notation, so we will call our objects $\Ur_r , \Vr_r , \fr_r , \Er_r$. 
At the very end of the proof, we will define objects for a general choice of $z$ by translating space. 

Let $\Kann \in (3/4,1)$ and $p_0 = p \in (0,1)$ be as in Proposition~\ref{prop-attained-good'}, so that $\Kann$ and $p_0$ depend only on the laws of $D_h$ and $\wt D_h$. 
We define our initial set of ``good" radii 
\eqb \label{eqn-initial-radii}
\mcl R_0 := \left\{ r  \in \{8^{-k}\}_{k\in \BB N} : \BB P[\wt H_r(\Kann , \Cmid_0 )] \geq p_0 \right\} .
\eqe 
By Proposition~\ref{prop-attained-good'}, there exists $\Cmid' > 0$, depending only on the laws of $D_h$ and $\wt D_h$, such that if $\BB r > 0$ and $\wt\Kopt > 0$ such that $\BB P[ \wt G_{\BB r}(\wt\Kopt, \Cmid') ]  \geq \wt\Kopt$, then for each small enough $\ep  > 0$ (how small is independent of $\BB r$),  
\eqbn
\#\left( \mcl R_0 \cap [\ep^2 \BB r , \ep \BB r ]\right) \geq \frac34 \log_8\ep^{-1} .
\eqen
We will eventually establish Proposition~\ref{prop-objects-exist} with the set of admissible radii given by $\mcl R = \rho^{-1} \mcl R_0$, where $\rho \in (0,1)$ is a constant depending only on $\BB p$ and the laws of $D_h$ and $\wt D_h$. 

Recall the basic idea of the construction as explained just after Proposition~\ref{prop-objects-exist}. We will take $\Ur_r$ to be a narrow ``tube" with the topology of a Euclidean annulus which is contained in a small neighborhood of $\bdy B_{2r}(0)$, and $\Vr_r$ to be a small Euclidean neighborhood of $\Ur_r$. 
We will then take $\Er_r$ to be the event that there are many ``good" pairs of points $u,v\in \Ur_r$ such that $\wt D_h(u,v) \leq \Cmid_0 D_h(u,v)$, plus a long list of regularity conditions. The idea for checking hypothesis~\ref{item-Ehyp-inc} for $\Er_r$ is that by Weyl scaling (Axiom~\ref{item-metric-f}), the $D_{h-\fr_r}$-lengths of paths contained in $\Ur_r$ tend to be much shorter than the $D_{h-\fr_r}$-lengths of paths outside of $\Vr_r$. We will use this fact to force a $D_{h-\fr_r}$-geodesic $P_r$ to get $D_{h-\fr_r}$-close to each of $u$ and $v$ for one of our good pairs of points $u,v$. We will then apply the triangle inequality to find times $s,t$ such that $\wt D_{h-\fr_r}(P_r(s) , P_r(t)) \leq \Cmid (t-s)$. Note that the application of the triangle inequality here is the reason why we need to require that $\wt D_h(u,v) \leq \Cmid_0 D_h(u,v)$ for $\Cmid_0 < \Cmid$. 

The broad ideas of this section are similar to those of~\cite[Section 5]{gm-uniqueness}, which performs a similar construction in the subcritical case. However, the details are quite different from~\cite[Section 5]{gm-uniqueness}, for three reasons. First, the conditions which we need our event to satisfy are slightly different from the ones needed in~\cite{gm-uniqueness} since our argument in Section~\ref{sec-counting} is completely different from the argument of~\cite[Section 4]{gm-uniqueness}. Second, we make some minor simplifications to various steps of the construction as compared to~\cite{gm-uniqueness}. Third, and most importantly, we want to treat the supercritical case so there are a number of additional difficulties arising from the fact that the metric does not induce the Euclidean topology. These difficulties necessitate additional conditions on the events and additional arguments as compared to the subcritical case. Especially, many of the conditions in the definition of $\Er_r$ and all of arguments of Section~\ref{sec-endpt-close} can be avoided in the subcritical case.
We will now give a more detailed outline of our construction.

In \textbf{Section~\ref{sec-endpt-ball}}, we will consider an event for a single ``good" pair of points $u,v$ and show that for $r\in\mcl R_0$, the probability of this event is bounded below by a constant $\pr$ depending only on the laws of $D_h$ and $\wt D_h$. See Lemma~\ref{lem-endpt-ball} for a precise statement and Figure~\ref{fig-endpt-ball} for an illustration of the event. 

The event we consider is closely related to the event $\wt H_r(\Kann , \Cmid_0)$ of Definition~\ref{def-annulus-geo'}. We require that there is a point $u\in \bdy B_{\Kann r}(0)$ and a point $v \in \bdy B_r(0)$ such that $\wt D_h(u,v) \leq \Cmid_0 D_h(u,v)$ and  a $\wt D_h$-geodesic $\wt P$ from $u$ to $v$ which is contained in a specified deterministic half-annulus $\Hr_r \subset \BB A_{\Kann r , r}(0)$. We also impose two additional constraints on $u$ and $v$ which will be important later:
\begin{enumerate}[(i)]
\item We require that $u$ is contained in a certain small \emph{deterministic} ball $B_{\sr_r}(\ur_r)$ centered at a point $\ur_r \subset \bdy B_{\Kann r}(0)$ and $v$ is contained in a small deterministic ball $B_{\sr_r}(\vr_r)$ centered at a point $\vr_r \in \bdy B_r(0)$, where $\sr_r$ is deterministic number which is comparable to a small constant times $r$.
The reason for this condition is that we will eventually define our set $\Ur_r$ so that it has a ``bottleneck" at several translated and scaled copies of the balls $B_{\sr_r}(\ur_r)$ and $B_{\sr_r}(\vr_r)$ (i.e., removing these balls disconnects $\Ur_r$; see Figure~\ref{fig-U-def}), and we need $\Ur_r$ to be deterministic. We will show that this condition happens with positive probability by considering finitely many possible choices for the balls $B_{\sr_r}(\ur_r)$ and $B_{\sr_r}(\vr_r)$ and using a pigeonhole argument. 
\item We require that the internal distance $D_h\left(u , x ; \ol B_{\sr_r}(\ur_r)\right)$ is small for ``most" points $x\in \bdy B_{\sr_r}(\ur_r)$, and we impose a similar condition for $v$. The purpose of this condition is to upper-bound the $D_{h-\fr_r}$-distance from a $D_{h-\fr_r}$-geodesic to $u$, once we have forced it to get Euclidean-close to $u$. The condition will be shown to occur with high probability using Lemma~\ref{lem-ball-bdy-union}. \label{item-outline-leb}
\end{enumerate}

In \textbf{Section~\ref{sec-block-event}}, we will define $\Fr_{z,r}$ for $z\in \BB C$ and $r\in\mcl R_0$ to be the event of Section~\ref{sec-endpt-ball}, but translated so that we are working with annuli centered at $z$ rather than 0. We will then show that $\Fr_{z,r}$ is locally determined by $h$ (Lemma~\ref{lem-endpt-event-msrble}). 

In \textbf{Section~\ref{sec-tube-def}}, we will introduce several parameters to be chosen later, including the parameter $\rho \in (0,1)$ mentioned above. We will then define the open sets $\Ur_r$ and $\Vr_r$ and the bump function $\fr_r$ for $r\in \rho^{-1} \mcl R_0$ in terms of these parameters. More precisely:
\begin{itemize}
\item The set $\Ur_r$ will be the union of a large finite number of disjoint sets of the form $\Hr_{\rho r} \cup B_{\sr_{\rho r} }(\ur_{\rho r}) \cup B_{\sr_{\rho r}}(\vr_{\rho r} ) + z$ for $z \in \bdy B_{2r}(0)$ (i.e., the sets appearing in the definition of $\Fr_{z,\rho r}$), together with long narrow ``tubes" linking these sets together into an annular region. See Figure~\ref{fig-U-def} for an illustration.
\item The set $\Vr_r$ will be a small Euclidean neighborhood of $\Ur_r$.
\item The function $\fr_r$ will attain its maximal value at each point of $\Ur_r$ and will be supported on $\Vr_r$. 
\end{itemize}
The reason for our definition of $\Ur_r$ is as follows. Since $r\in \rho^{-1} \mcl R_0$, for each of the sets  $\Hr_{\rho r} \cup B_{\sr_{\rho r} }(\ur_{\rho r}) \cup B_{\sr_{\rho r}}(\vr_{\rho r} ) + z$ in the definition of $\Ur_r$, there is a positive chance that the event $\Fr_{z,\rho r}$ of Section~\ref{sec-block-event} occurs. 
Hence, by the long-range independence properties of the GFF (Lemma~\ref{lem-spatial-ind}), it is very likely that $\Fr_{z,\rho r}$ occurs for many of the points $z$. 
This gives us the desired large collection of ``good" pairs of points $u,v \in \Ur_r$. See Lemma~\ref{lem-E-event}. 

In \textbf{Section~\ref{sec-E}}, we will define the event $\Er_r$. The event $\Er_r$ includes the condition that $\Fr_{z,\rho r}$ occurs for many of the points $z\in \bdy B_{2r}(0)$ involved in the definition of $\Ur_r$ (condition~\ref{item-E-event}), plus a large number of additional high-probability regularity conditions. 
Then, in \textbf{Section~\ref{sec-E-prob}}, we will show that we can choose the parameters of Section~\ref{sec-tube-def} in such a way that $\Er_r$ occurs with probability at least $\BB p$ (Proposition~\ref{prop-E-prob}). We will also show that $\Er_r$ satisfies hypotheses~\ref{item-Ehyp-dist} and~\ref{item-Ehyp-rn} of Section~\ref{sec-counting-setup} (Proposition~\ref{prop-E-hyp0}). 
In \textbf{Section~\ref{sec-objects-exist}}, we will explain how to conclude the proof of Proposition~\ref{prop-objects-exist} assuming that our objects also satisfy hypothesis~\ref{item-Ehyp-inc} of Section~\ref{sec-counting-setup}. 

The rest of the section is then devoted to checking that our objects satisfy hypothesis~\ref{item-Ehyp-inc} of Section~\ref{sec-counting-setup} (Proposition~\ref{prop-shortcut}). Recalling the statement of hypothesis~\ref{item-Ehyp-inc}, we will assume that $\Er_r$ occurs and consider a $D_{h-\fr_r}$-geodesic $P_r$ between two points of $\BB C\setminus B_{4r}(0)$. We will further assume that $P_r$ has a $(B_{4r}(0) , \Vr_r)$-excursion  $(\tau',\tau,\sigma,\sigma')$ such that $  D_h\left( P_r(\tau) , P_r(\sigma)   ; B_{4r}(0)  \right)$ is bounded below by an appropriate constant times $r^{\xi Q} e^{\xi h_r(0)}$ (recall Definition~\ref{def-excursion}). We aim to find times $s < t$ for $P_r$ such that $t-s$ is not too small and $\wt D_{h-\fr_{z,r} }\left(P_r(s) , P_r(t)  ; B_{4r}(0)  \right) \leq \Cmid (t-s) $. 

In \textbf{Section~\ref{sec-excursion}}, we will show that the \emph{Euclidean} distance between the points $P_r(\tau) , P_r(\sigma) \in \bdy \Vr_r$ is bounded below by a constant times $r$ (Lemma~\ref{lem-excursion-sp}) and that $P_r|_{[\tau,\sigma]}$ is contained in a small Euclidean neighborhood of $\Vr_r$ (Lemma~\ref{lem-excursion-tube}). These statements are proven using the regularity conditions in the definition of $\Er_r$. In particular, the lower bound for $|P_r(\tau) - P_r(\sigma)|$ comes from the regularity of $D_h$-distances along a geodesic (Lemma~\ref{lem-hit-ball-phi}). The statement that $P_r|_{[\tau,\sigma]}$ is contained in a small Euclidean neighborhood of $\Vr_r$ is proven as follows. Since $\fr_r$ is very large on $\Ur_r$, we know that $D_{h-\fr_r}$-distances inside $\Ur_r$ are very small, which leads to a very small upper bound for $\sigma -\tau = D_{h-\fr_r}(P_r(\tau) , P_r(\sigma))$ (Lemma~\ref{lem-excursion-length}). Since $\fr_r$ is supported on $\Vr_r$, the $D_{h-\fr_r}$-length of any segment of $P_r$ which is disjoint from $\Vr_r$ is the same as its $D_h$-length, which will be larger than our upper bound for $\sigma-\tau$ unless the Euclidean diameter of the segment is very small. 

In \textbf{Section~\ref{sec-excursion-hit}}, we will use the results of Section~\ref{sec-excursion} and the definition of $\Ur_r$ to show that the following is true. 
There is a point $z \in \bdy B_{2r}(0)$ as in the definition of $\Ur_r$ such that $\Fr_{z,\rho r}$ occurs and $P_r$ gets Euclidean-close to each of the ``good" points $u$ and $v$ in the definition of $\Fr_{z,\rho r}$ (Lemma~\ref{lem-excursion-hit}). 
The reason why this is true is that, by the results of Section~\ref{sec-excursion}, $P_r([\tau,\sigma])$ is contained in a small neighborhood of $\Ur_r$ and has Euclidean diameter of order $r$, and the definition of $\Ur_r$ implies that removing small neighborhoods of the points $u$ and $v$ disconnects $\Ur_r$ (see Figure~\ref{fig-U-def}). 

Showing that $P_r$ gets Euclidean-close to $u$ and $v$ is not enough for our purposes since $D_{h-\fr_r}$ is not Euclidean-continuous, so it is possible for two points to be Euclidean-close but not $D_{h-\fr_r}$-close. Therefore, further arguments are needed to show that $P_r$ gets $D_{h-\fr_r}$-close to each of $u$ and $v$. We remark that this is one of the main reasons why the argument in this section is more difficult than the analogous argument in the subcritical case~\cite[Section 5]{gm-uniqueness}. 

In \textbf{Section~\ref{sec-endpt-close}}, we will show that there are times $s$ and $t$ for $P_r$ such that $D_{h-\fr_r}(P_r(t) , u)$ and $D_{h-\fr_r}(P_r(s) , v)$ are each much smaller than $D_{h-\fr_r}(u,v)$ (Lemma~\ref{lem-endpt-close}). 
The key tool which allows us to do this is the condition in the definition of  $\Fr_{z,\rho r}$ which says that $D_h\left(u , x ; \ol B_{\sr_{\rho r}}(\ur_{\rho r}) + z\right)$ is small for ``most" points of $\bdy B_{\sr_{\rho r} }(\ur_{\rho r}) + z$ (recall point~\eqref{item-outline-leb} in the summary of Section~\ref{sec-endpt-ball}). 
However, this condition is not sufficient for our purposes since it is possible that the ``Euclidean size" of $P_r \cap (B_{\sr_{\rho r} }(\ur_{\rho r}) + z)$ is small, and hence $P_r$ manages not to hit a geodesic from $u$ to $x$ for any of the ``good" points $x \in \bdy B_{\sr_{\rho r}}(\ur_{\rho r}) + z $ such that $D_h\left(u , x ; B_{\sr_{\rho r}}(\ur_{\rho r}) + z\right)$ is small. 
To avoid this difficulty, we will need to carry out a careful analysis of, roughly speaking, the ``excursions" that $P_r$ makes in and out of the ball $B_{\sr_{\rho r}}(\ur_{\rho r}) + z$. 

In \textbf{Section~\ref{sec-shortcut}}, we will conclude the proof that $\Er_r$ satisfies hypothesis~\ref{item-Ehyp-inc} using the result of Section~\ref{sec-endpt-close} and the triangle inequality. 

\subsection{Existence of a shortcut with positive probability}
\label{sec-endpt-ball}

Throughout the rest of this section, we let
\eqb \label{eqn-small-const}
\llambda  \in \left( 0 , 10^{-100} \min \left\{ \Clower, 1/\Cupper, (\Clower/\Cupper)^2 \right\}  \right)  
\eqe
be a small constant to be chosen later, in a manner depending only on the laws of $D_h$ and $\wt D_h$ (not on $\BB p$).
We will frequently use $\llambda$ in the definitions of events and other objects when we need a small constant whose particular value is unimportant.   
 
In this subsection, we will prove that for each $r\in\mcl R_0$, it holds with positive probability (uniformly in $r\in\mcl R_0$) that there is a ``good" pair of non-singular points $u,v\in  \ol B_r(0)$ such that $\wt D_h(u,v) \leq \Cmid_0  D_h(u,v)$ and certain regularity conditions hold. In later subsections, we will use the long-range independence of the GFF to say that with high probability, there are many such pairs of points contained in our open set $\Ur_r$. 
To state our result, we need the following definition.

\begin{defn}
Let $z\in\BB C$ and $b > a > 0$. 
A \emph{horizontal or vertical half-annulus} $H\subset \BB A_{a,b}(z)$ is the intersection of $\BB A_{a,b}(z)$ with one of the four half-planes
\alb
&\left\{w \in\BB C : \re w > \re z\right\} ,\quad 
 \left\{w \in\BB C : \re w < \re z\right\} ,\notag\\
&\left\{w \in\BB C : \im w > \im z\right\} ,\quad \text{or} \quad 
 \left\{w \in\BB C : \im w < \im z\right\}  .
\ale
\end{defn}

\begin{figure}[ht!]
\begin{center}
\includegraphics[width=.5\textwidth]{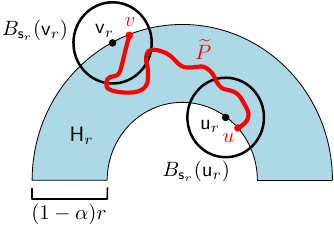} 
\caption{\label{fig-endpt-ball} Illustration of the objects involved in Lemma~\ref{lem-endpt-ball}. 
}
\end{center}
\end{figure}

\begin{lem} \label{lem-endpt-ball}
Let $\Kann$ and $\mcl R_0$ be as in~\eqref{eqn-initial-radii}.
There exists $\Aendpt  \in \left(0, \llambda (1-\Kann)^2 \right]$, $\Aloc  > 3$, and $\pr \in (0,1)$ (depending only on $\llambda$ and the laws of $D_h$ and $\wt D_h$) such that for each $r\in\mcl R_0$, there exists a deterministic horizontal or vertical half-annulus $\Hr_r  \subset \BB A_{\Kann r , r}(0)$, a deterministic radius $\sr_r \in [\Aendpt  r , \Aendpt^{1/2} r] \cap \{4^{-k} r\}_{k\in\BB N}$, and deterministic points 
\allb \label{eqn-endpt-centers}
\ur_r &\in \bdy \Hr_r \cap \left\{ \Kann r e^{i \llambda \Aendpt k} : k \in [1,2\pi \llambda^{-1} \Aendpt^{-1} ]_{\BB Z} \right\} \quad \text{and} \notag\\
\vr_r &\in \bdy \Hr_r \cap \left\{  r e^{i \llambda \Aendpt k} : k \in [1,2\pi\llambda^{-1} \Aendpt^{-1} ]_{\BB Z} \right\}
\alle
such that with probability at least $\pr$, the following is true.
There exist non-singular points $ u \in \bdy B_{\Kann r}(0) \cap B_{\sr_r/2}(\ur_r) $ and $v \in \bdy B_r(0)\cap B_{\sr_r / 2}(\vr_r)$ with the following properties.  
\begin{enumerate}
\item $\wt D_h(u,v) \leq \Cmid_0 D_h(u,v)$. \label{item-endpt-ball-dist} 
\item There is a $\wt D_h$-geodesic $\wt P$ from $u$ to $v$ which is contained in $\ol{\Hr}_r$. \label{item-endpt-ball-annulus}  
\item The one-dimensional Lebesgue measure of the set \label{item-endpt-ball-leb} 
\eqbn
\left\{x  \in \bdy B_{ \sr_r}(\ur_r) \, : \,   D_h\left( x , u  ; \ol B_{ \sr_r}(\ur_r) \right)  >  \llambda \wt D_h(u,v) \right\} 
\eqen
is at most $(\llambda/2) \sr_r$. Moreover, the same is true with $v$ and $\vr_r$ in place of $u$ and $\ur_r$.
\item There exists $t \in [3r , \Aloc r]$ such that \label{item-endpt-ball-around}
\eqbn
 D_h\left(\text{around $\BB A_{t,2t}(0)$} \right) \leq  \llambda  D_h\left(\text{across $\BB A_{2t,3t}(0)$} \right) .
\eqen 
\end{enumerate}  
\end{lem}

See Figure~\ref{fig-endpt-ball} for an illustration of the statement of Lemma~\ref{lem-endpt-ball}. 
Most of this subsection is devoted to the proof of Lemma~\ref{lem-endpt-ball}. Before discussing the proof, we will first discuss the motivation for the various conditions in the lemma statement. 

In Section~\ref{sec-tube-def}, we will consider a small but fixed constant $\rho  \in (0,1)$. In order to build the set $\Ur_r = \Ur_{0,r}$ appearing in Section~\ref{sec-counting}, we will use long narrow tubes to ``link up" several sets of the form $\Hr_{\rho r} \cup B_{\sr_{\rho r}}(\ur_{\rho r}) \cup B_{\sr_{\rho r}}(\vr_{\rho r}) + z$, for varying choices of $z\in \bdy B_{2r}(0)$. We need $\Ur_r$ to be deterministic, which is why we need to make a deterministic choice of the half-annulus $\Hr_r$, the radius $\sr_r$, and the points $\ur_r$ and $\vr_r$ in Lemma~\ref{lem-endpt-ball}. Furthermore, we want there to be only finitely many possibilities for the set $r^{-1} \Ur_r$, which allows us to get certain estimates for $\Ur_r$ trivially by taking a maximum over the possibilities. This is why we require that $\Hr_r$ is a vertical or horizontal half-annulus and why we require that the points $\ur_r$ and $\vr_r$ belong to the finite sets in~\eqref{eqn-endpt-centers}.
  
Our set $\Ur_r$ will have ``bottlenecks" at the balls $B_{\sr_{\rho r}}(\ur_{\rho r}) + z$ and $  B_{\sr_{\rho r}}(\vr_{\rho r}) + z$, so that any path which travels more than a constant-order Euclidean distance inside the set $\Ur_r$ will have to enter many of these balls. 
The requirement that $u \in B_{\sr_{\rho r}/2}(\ur_{\rho r})$ and $v\in B_{\sr_{\rho r}/2}(\vr_{\rho r})$ is needed to force a path which spends a lot of time in $\Ur_r$ to get close to $u$ and $v$. The requirement that $\wt P \subset \ol{\Hr}_r$ in condition~\ref{item-endpt-ball-annulus} is needed to ensure that subtracting from $h$ a large bump function which attains its maximal value at each point of $\Ur_r$ decreases $\wt D_h(u,v)$ by at least as much as $D_h(u,v)$, so the condition $\wt D_h(u,v) \leq \Cmid_0 D_h(u,v)$ is preserved.

Condition~\ref{item-endpt-ball-leb} in Lemma~\ref{lem-endpt-ball} is needed to upper-bound the LQG distance from a path to each of $u$ and $v$, once we know that it gets Euclidean-close to $u$ and $v$ (this is done in Section~\ref{sec-endpt-close}). 
The reason why our distance bound is in terms of $\wt D_h(u,v)$ is that we eventually want to show that the $\wt D_{h-\fr_r}$-distance from a $D_{h-\fr_r}$-geodesic to each of $u$ and $v$ is at most a small constant times $\wt D_{h-\fr_r}(u,v)$. We will then use condition~\ref{item-endpt-ball-dist} in Lemma~\ref{lem-endpt-ball} and the triangle inequality to deduce hypothesis~\ref{item-Ehyp-inc}. Note that condition~\ref{item-endpt-ball-leb} includes a bound on $D_h$-distances, but this immediately implies a bound for $\wt D_h$-distances due to the bi-Lipschitz equivalence of $D_h$ and $\wt D_h$~\eqref{eqn-bilip}.

The only purpose of condition~\ref{item-endpt-ball-around} is to ensure that the event in the lemma statement depends locally on $h$ (see Lemma~\ref{lem-endpt-event-msrble}).
This local dependence is not automatically true since a $D_h$-geodesic from $u$ to $v$ could get very Euclidean-far away from $u$ and $v$.

We now turn our attention to the proof of Lemma~\ref{lem-endpt-ball}. 
To this end, let us first record what we get from the Definition~\ref{def-annulus-geo'} of $\wt H_r(\Kann,\Cmid_0)$ and the Definition~\eqref{eqn-initial-radii} of $\mcl R_0$. 
 
\begin{lem} \label{lem-endpt-geodesic}
For each $r\in \mcl R_0$, there is a deterministic horizontal or vertical half-annulus $\Hr_r \subset \BB A_{\Kann r , r}(0)$ such that with probability at least $p_0/4$, there exist non-singular points $u \in \bdy B_{\Kann r}(0)$ and $v \in \bdy B_r(0)$ with the following properties.  
\begin{enumerate}
\item $\wt D_h(u,v) \leq \Cmid_0 D_h(u,v)$. \label{item-endpt-geodesic-dist} 
\item There is a $\wt D_h$-geodesic $\wt P$ from $u$ to $v$ which is contained in $\ol{\Hr}_r$.  \label{item-endpt-geodesic-contained} 
\item With $\geoExp = \geoExp(1/2)$ as in Lemma~\ref{lem-hit-ball-phi}, for each $\delta \in (0,(1-\Kann)^2]$,   \label{item-endpt-geodesic-initial}
\eqbn
\max\left\{ \wt D_h\left( u , \bdy B_{\delta r}(u) \right) , \wt D_h\left( v , \bdy B_{\delta r}(v) \right) \right\}  \leq \delta^{\geoExp} \wt D_h(u,v) .
\eqen
\end{enumerate}  
\end{lem}
\begin{proof}
By Definition~\ref{def-annulus-geo'} of $\wt H_r(\Kann,\Cmid_0)$ and the definition~\eqref{eqn-initial-radii} of $\mcl R_0$, for each $r\in\mcl R_0$ it holds with probability at least $p_0$ that there exist $u \in \bdy B_{\Kann r}(0)$ and $v \in \bdy B_r(0)$ such that conditions~\ref{item-endpt-geodesic-dist} and~\ref{item-endpt-geodesic-initial} in the lemma statement hold and there is a $\wt D_h$-geodesic $\wt P$ from $u$ to $v$ which is contained in $\ol{\BB A}_{\Kann r , r}(0)$ and has Euclidean diameter at most $r/100$. 
Since $\wt P \subset \ol{\BB A}_{\Kann r , r}(0)$ and $\wt P$ has Euclidean diameter at most $r/100$, trivial geometric considerations show that $\wt P$ must be contained in the closure of one of the four horizontal or vertical half-annuli of $\BB A_{\Kann r , r}(0)$. Hence we can choose one such half-annulus $\Hr_r$ in a deterministic manner such that with probability at least $p_0/4$, conditions~\ref{item-endpt-geodesic-dist} and~\ref{item-endpt-geodesic-initial} in the lemma statement hold and $\wt P \subset \ol{\Hr}_r$, i.e., condition~\ref{item-endpt-geodesic-contained} holds. 
\end{proof}

Lemma~\ref{lem-endpt-geodesic} gives us a pair of points $u,v$ satisfying conditions~\ref{item-endpt-ball-dist} and~\ref{item-endpt-ball-annulus} in Lemma~\ref{lem-endpt-ball}.
We still need to check conditions~\ref{item-endpt-ball-leb} and~\ref{item-endpt-ball-around}.
Condition~\ref{item-endpt-ball-leb} will require the most work. 
To get this condition, we want to apply Lemma~\ref{lem-ball-bdy-union}. 
However, the points $u$ and $v$ are random, so we cannot just apply the lemma directly. Instead, we will apply Lemma~\ref{lem-ball-bdy-union} in conjunction with Lemma~\ref{lem-annulus-iterate} (independence across concentric annuli) and a union bound to cover space by balls where an event occurs which is closely related to the one in Lemma~\ref{lem-ball-bdy-union}. 
Then, we will use a geometric argument based on condition~\ref{item-endpt-geodesic-initial} of Lemma~\ref{lem-endpt-geodesic} to transfer from an estimate for balls containing $u$ and $v$ to an estimate for $u$ and $v$ themselves.
 
Let us now define the event to which we will apply Lemma~\ref{lem-annulus-iterate}. 
For $z \in \BB C$, $s>0$, and $R > 0$, let $G_s(z;R)$ be the event that the following is true.
\begin{enumerate}
\item The one-dimensional Lebesgue measure of the set of $x \in \bdy B_s(z)$ for which 
\eqbn
\wt D_h\left( x , \bdy B_{s/2}(z) ; \ol{\BB A}_{s/2,s}(z)  \right)  > R s^{\xi Q} e^{\xi h_s(z)} 
\eqen
is at most $(\llambda/2) s$.  \label{item-endpt-leb}
\item $\wt D_h(\text{around $\BB A_{s/2 ,  s}(z)$} ) \leq R s^{\xi Q} e^{\xi h_s(z)}$.  \label{item-endpt-around}
\item $\wt D_h\left(\text{across $\BB A_{s/2 , s}(z)$} \right) \geq (1/R)s^{\xi Q} e^{\xi h_s(z)}$. \label{item-endpt-across}
\end{enumerate}
Since the event $G_s(z;R)$ involves only internal distances in $\ol{\BB A}_{s/2,s}(z)$, the locality property (Axiom~\ref{item-metric-local}; see also Section~\ref{sec-closed}) implies that $G_s(z;R)$ is a.s.\ determined by $h|_{\ol{\BB A}_{s/2,s}(z)}$. 
Furthermore, by Weyl scaling (Axioms~\ref{item-metric-f}), the occurrence of $G_s(z;R)$ is unaffected by adding a constant to $h$. 
Therefore,
\eqb \label{eqn-endpt-event-msrble}
G_{s}(z ; R) \in \sigma\left( (h-h_{2s}(z))|_{\ol{\BB A}_{s/2 , s}(z)} \right) .
\eqe
We can also arrange that the probability of $G_s(z;R)$ is close to 1 by making $R$ large.

\begin{lem} \label{lem-endpt-event-prob}
For each $p \in (0,1)$, there exists $R  > 0$, depending only on $p,\llambda$ and the law of $\wt D_h$, such that for each $z\in\BB C$ and each $s > 0$, we have $\BB P[G_s(z; R)] \geq p$.
\end{lem}
\begin{proof}
By Lemma~\ref{lem-ball-bdy-union} (and the fact that a path from $x \in \bdy B_s(z)$ to $z$ must hit $\bdy B_{s/2}(z)$), if $R$ is chosen to be sufficiently large, depending only on $p$ and the law of $\wt D_h$, then the first condition in the definition of $G_s(z;R)$ has probability at least $1-p/3$. By tightness across scales (Axiom~\refcoord), after possibly increasing $R$ we can arrange that the other two conditions in the definition of $G_s(z;R)$ also have probability at least $p$. 
\end{proof}

Let us now apply Lemma~\ref{lem-annulus-iterate} to get the following.

\begin{lem} \label{lem-endpt-event-union}
There exists $R > 0$, depending only on $\llambda$ and the law of $\wt D_h$, such that for each $r > 0$, it holds with polynomially high probability as $\ep\rta 0$ (at a rate depending only on $\llambda$ and the law of $\wt D_h$) such that the following is true. For each point
\eqb \label{eqn-endpt-union-set}
z \in \left\{ \Kann r e^{i \llambda \ep k} : k \in [1,2\pi\llambda^{-1} \ep^{-1} ]_{\BB Z} \right\} \cup \left\{  r e^{i \llambda \ep k} : k \in [1,2\pi\llambda^{-1} \ep^{-1} ]_{\BB Z} \right\}
\eqe 
we have
\eqb
\#\left\{ k\in \left[ \frac12 \log_4 \ep^{-1}  , \log_4 \ep^{-1} \right]_{\BB Z}  : G_{4^{-k}  r}(z;R) \: \text{occurs} \right\} \geq \frac38 \log_4 \ep^{-1} .
\eqe
\end{lem}
\begin{proof}
By~\eqref{eqn-endpt-event-msrble} and Lemma~\ref{lem-endpt-event-prob} (applied with $p$ sufficiently close to 1), we can apply Lemma~\ref{lem-annulus-iterate} (independence across concentric annuli) to get the following. There exists $R > 0$ as in the lemma statement such that for each $z\in\BB C$ and each $r > 0$, 
\eqbn
\BB P\left[ \#\left\{ k\in \left[ \frac12 \log_4 \ep^{-1}  , \log_4 \ep^{-1} \right]_{\BB Z}  : G_{4^{-k}   r}(z;R) \: \text{occurs} \right\} \geq \frac38 \log_4 \ep^{-1} \right] \geq 1 - O_\ep(\ep^2) .
\eqen
The lemma follows from this and a union bound over the $O_\ep(\ep^{-1})$ points in the set~\eqref{eqn-endpt-union-set}. 
\end{proof}

The following lemma is the main step in the proof of Lemma~\ref{lem-endpt-ball}.

\begin{lem} \label{lem-endpt-ball0}
There exist $\Aendpt  \in (0,\llambda (1-\Kann)^2]$ and $\pr \in (0,1)$ (depending only on $\llambda$ and the laws of $D_h$ and $\wt D_h$) such that for each $r\in\mcl R_0$, there exist a deterministic vertical or horizontal half-annulus $\Hr_r  \subset \BB A_{\Kann r , r}(0)$, a deterministic radius $\sr_r \in [\Aendpt  r , \Aendpt^{1/2} r]   \cap \{4^{-k} r\}_{k\in\BB N}$, and deterministic points $\ur_r , \vr_r \in \bdy\Hr_r$ as in~\eqref{eqn-endpt-centers} such that with probability at least $2 \pr$, the following is true.
There exist non-singular points $u \in \bdy B_{\Kann r}(0) \cap B_{\sr_r}(\ur_r)$ and $v \in \bdy B_r(0) \cap B_{\sr_r}(\vr_r)$ such that conditions~\ref{item-endpt-ball-dist}, \ref{item-endpt-ball-annulus}, and~\ref{item-endpt-ball-leb} from Lemma~\ref{lem-endpt-ball} hold. 
\end{lem}
 \begin{proof}
\noindent\textit{Step 1: setup.}
Let $\Kann$ and $p_0$ be as in the definition of $\mcl R_0$ from~\eqref{eqn-initial-radii}. Let the half-annulus $\Hr_r$ for $r\in\mcl R_0$ be as in Lemma~\ref{lem-endpt-geodesic} and let $R > 0$ be as in Lemma~\ref{lem-endpt-event-union}.
Also let $\Aendpt > 0$ be small enough so that the event of Lemma~\ref{lem-endpt-event-union} with $\Aendpt$ in place of $\ep$ occurs with probability at least $1-p_0/8$. We can arrange that $\Aendpt$ is small enough so that 
\eqb \label{eqn-Aendpt-choice} 
\Aendpt \leq \llambda (1- \Kann)^2 \quad \text{and} \quad (2R^2 + 1) (2\Aendpt)^{\geoExp} \leq  \llambda^2 ,
\eqe 
where $\geoExp$ is as in Lemma~\ref{lem-endpt-geodesic}. 
Then with probability at least $p_0/8$, the event of Lemma~\ref{lem-endpt-geodesic} and the event of Lemma~\ref{lem-endpt-event-union} with $\ep = \Aendpt$ both occur. Henceforth assume that these two events occur.

Let $\wt P$ be the $\wt D_h$-geodesic from $u$ to $v$ which is contained in $\ol{\Hr}_r$, as in Lemma~\ref{lem-endpt-geodesic}.
By the conditions in Lemma~\ref{lem-endpt-geodesic}, the conditions~\ref{item-endpt-ball-dist} and~\ref{item-endpt-ball-annulus} in the statement of Lemma~\ref{lem-endpt-ball} hold for this choice of $u,v$, and $\wt P$. 
It remains to deal with condition~\ref{item-endpt-ball-leb}. 
\medskip

\noindent\textit{Step 2: reducing to a statement for a random radius and pair of points.}
We can choose random points
\alb
z_1 &\in \bdy \Hr_r \cap \left\{ \Kann r e^{i \llambda \Aendpt k} : k \in [1,2\pi \llambda^{-1} \Aendpt^{-1} ]_{\BB Z} \right\} \quad \text{and} \notag\\
z_2 &\in \bdy \Hr_r \cap \left\{  r e^{i \llambda \Aendpt k} : k \in [1,2\pi\llambda^{-1} \Aendpt^{-1} ]_{\BB Z} \right\}
\ale
such that 
\eqb \label{eqn-endpt-z}
|u-z_1| \leq \Aendpt r/50 \quad \text{and} \quad |v-z_2| \leq \Aendpt r/50 . 
\eqe
The event of Lemma~\ref{lem-endpt-event-union} (with $\ep = \Aendpt$) implies that for each $i\in \{1,2\}$, there are at least $\frac38 \log_4\Aendpt^{-1}$ values of $k\in \left[ \frac12 \log_4 \Aendpt^{-1}  , \log_4 \Aendpt^{-1} \right]_{\BB Z}$ such that $ G_{4^{-k}  r}(z_i ;R)$ occurs. Since the number of choices for $k$ is at most $\frac12 \log_4 \Aendpt^{-1}$, there must be some (random) $k_* \in \left[ \frac12 \log_4 \Aendpt^{-1}  , \log_4 \Aendpt^{-1} \right]_{\BB Z}$ such that $ G_{4^{-k_*}  r}(z_1 ;R) \cap G_{4^{-k_*} r}(z_2 ; R)$ occurs. 
We pick one such value of $k_*$ in a measurable manner and set
\eqb \label{eqn-endpt-s}
s := 4^{-k_*} r,\quad \text{so that} \quad s \in [\Aendpt  r , \Aendpt^{1/2} r]   \cap \{4^{-k} r\}_{k\in\BB N} .
\eqe 

We claim that condition~\ref{item-endpt-ball-leb} in Lemma~\ref{lem-endpt-ball} holds with $s$ in place of $\sr_r$ and $z_1,z_2$ in place of $\ur_r , \vr_r$. 
Once the claim has been proven, we have that with probability at least $p_0/8$, the conditions in the lemma statement hold with the random variables $s,z_1,z_2$ in place of the deterministic parameters $\sr_r , \ur_r , \vr_r$. The number of possible choices for $s$ is at most $\frac12 \log_4 \Aendpt^{-1}$ and the number of possible choices for each of $z_1,z_2$ is at most a constant (depending only on $\llambda$ and the laws of $D_h$ and $\wt D_h$) times $\Aendpt^{-1}$. Therefore, our claim implies that there is some constant $\pr > 0$ (which depends only on $p_0$ and $\Aendpt$, hence only on the laws of $D_h$ and $\wt D_h$) and a \emph{deterministic} choice of parameters $\sr_r , \ur_r$, and $\vr_r$ such that with probability at least $2\pr$, the conditions of the lemma statement hold for $\sr_r , \ur_r$, and $\vr_r$. 
\medskip

\noindent\textit{Step 3: estimates for distances in $B_s(z_1)$ and $B_s(z_2)$.}
It remains to prove the claim in the preceding paragraph. 
By our choices of $z_1 , z_2 $~\eqref{eqn-endpt-z} and $s$~\eqref{eqn-endpt-s},  
\eqb \label{eqn-endpt-ball-include}
u \in B_{s/2}(z_1)  \subset B_s(z_1) \subset B_{2\Aendpt^{1/2} r}(u) \quad \text{and} \quad v\in B_{s/2}(z_2) \subset B_s(z_2) \subset B_{2\Aendpt^{1/2} r}(v). 
\eqe
From this, condition~\ref{item-endpt-geodesic-initial} from Lemma~\ref{lem-endpt-geodesic} (with $\delta = 2\Aendpt^{1/2}$), and the definition of $G_s(z_i ; R)$, we obtain 
\allb \label{eqn-endpt-len}
(2\Aendpt^{1/2})^\geoExp  \wt D_h(u,v)
&\geq  \max\left\{ \wt D_h\left( u , \bdy B_{2\Aendpt^{1/2} r}(u) \right) , \wt D_h\left( v , \bdy B_{2\Aendpt^{1/2} r}(v) \right) \right\}         \quad \text{(by Lemma~\ref{lem-endpt-geodesic})} \notag \\ 
&\geq \max\left\{ \wt D_h\left( u , \bdy B_{s}(z_1) \right) , \wt D_h\left( v , \bdy B_{s}(z_2) \right) \right\}   \quad \text{(by~\eqref{eqn-endpt-ball-include})}\notag \\
&\geq \max_{i\in \{1,2\}} \wt D_h\left(\text{across $\BB A_{s/2 , s}(z_i)$} \right) \notag\\
&\qquad\qquad\qquad \text{(since $u \in B_{s/2}(z_1)$ and $v\in B_{s/2}(z_2)$)} \notag \\
&\geq \frac{1}{R} \max_{i\in\{1,2\}} s^{\xi Q} e^{\xi h_s(z_i)}   \quad \text{(by condition~\ref{item-endpt-across} for $G_s(z_i ;R)$)} . 
\alle
We now apply~\eqref{eqn-endpt-len} to upper-bound the quantities $s^{\xi Q} e^{\xi h_s(z_i)}$ appearing in conditions~\ref{item-endpt-leb} and~\ref{item-endpt-around} in the definition of $G_s(z_i ; R)$. Upon doing so, we obtain the following observations for $i=1,2$.
\begin{enumerate}[(i)]
\item The one-dimensional Lebesgue measure of the set of $x \in \bdy B_{ s}(z_i)$ for which 
\eqbn
 \wt D_h\left( x , \bdy B_{s/2}(z_i) ; \ol B_{ s}(z_i)  \right)  > R^2 (2\Aendpt^{1/2})^\geoExp  \wt D_h(u,v)
\eqen
is at most $(\llambda/2) s$.  \label{item-use-endpt-leb}
\item We have \label{item-use-endpt-around} 
\eqb \label{eqn-use-endpt-around}
\wt D_h\left(\text{around $\BB A_{s/2 , s}(z_i)$} \right)
\leq R^2 (2\Aendpt^{1/2})^\geoExp  \wt D_h(u,v) .
\eqe
\end{enumerate}
\medskip

\begin{figure}[ht!]
\begin{center}
\includegraphics[width=.5\textwidth]{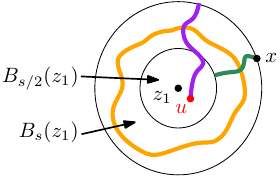} 
\caption{\label{fig-endpt-ball0} Illustration of the proof of condition~\ref{item-endpt-ball-leb} in Lemma~\ref{lem-endpt-ball} with $(s,z_1)$ in place of $(\sr_r , \ur_r)$. 
The concatenation of the purple, orange, and green paths in the figure contains a path from $u$ to $x$. The $\wt D_h$-length of the purple path can be bounded above in terms of $\wt D_h(u,v)$ by condition~\ref{item-endpt-geodesic-initial} from Lemma~\ref{lem-endpt-geodesic}. The $\wt D_h$-length of the orange path can be bounded above in terms of $\wt D_h(u,v)$ using~\eqref{eqn-use-endpt-around}, which in turn is proven using conditions~\ref{item-endpt-around} and~\ref{item-endpt-across} in the definition of $G_s(z_1 ; R)$. 
For most points $x\in \bdy B_s(z_1)$, the $\wt D_h$-length of the green path can be bounded above in terms of $\wt D_h(u,v)$ by condition~\ref{item-endpt-leb} in the definition of $G_s(z_1;R)$. 
}
\end{center}
\end{figure}

\noindent\textit{Step 4: checking condition~\ref{item-endpt-ball-leb}.} 
If $x \in \bdy B_{ s}(z_1)$, then the union of any path from $x$ to $\bdy B_{s/2}(z_1)$, any path in $\BB A_{s/2, s}(z_1)$ which disconnects the inner and outer boundaries of $\BB A_{s/2,s}(z_i)$, and any path from $u$ to $\bdy B_s(z_1)$ must contain a path from $u$ to $x$ (see Figure~\ref{fig-endpt-ball0}). 
By~\eqref{eqn-use-endpt-around} and the second inequality in~\eqref{eqn-endpt-len}, we therefore have 
\allb \label{eqn-endpt-conc}
\wt D_h\left( x , u  ; \ol B_{ s}(z_1) \right)
&\leq \wt D_h\left( x , \bdy B_{s/2}(z_1) ; \ol B_{  s}(z_1)  \right) +  \wt D_h\left(\text{around $\BB A_{s/2 , s}(z_1)$} \right) \notag\\
&\qquad \qquad \qquad +  \wt D_h\left( u , \bdy B_s(z_1) \right) \notag\\ 
&\leq \wt D_h\left( x , \bdy B_{s/2}(z_1) ; \ol B_{ s}(z_1)  \right) +  \left( R^2   + 1 \right) (2\Aendpt^{1/2})^\geoExp  \wt D_h(u,v)  . 
\alle

By combining~\eqref{eqn-endpt-conc} with observation~\eqref{item-use-endpt-leb} above, we get that for all $x\in \bdy B_s(z_1)$ except on a set of one-dimensional Lebesgue measure at most $(\llambda/2) s $, 
\eqb \label{eqn-endpt-bdy-A}
\wt D_h\left( x , u  ; \ol B_{ s}(z_1) \right) 
 \leq (2R^2 + 1) (2 \Aendpt)^{\geoExp }  \wt D_h(u,v)  .
\eqe
By~\eqref{eqn-endpt-bdy-A} and our choice of $\Aendpt$ in~\eqref{eqn-Aendpt-choice}, we get that for all $x\in \bdy B_s(z_1)$ except on a set of one-dimensional Lebesgue measure at most $(\llambda/2) s$,  
\eqb \label{eqn-endpt-bdy-wt}
\wt D_h\left( x , u  ; \ol B_{ s}(z_1) \right)  \leq       \llambda^2 \wt D_h(u,v)      .
\eqe
Since $\llambda < \Clower  $, the estimate~\eqref{eqn-endpt-bdy-wt} together with the bi-Lipschitz equivalence of $D_h$ and $\wt D_h$ implies that
\eqb
  D_h\left( x , u  ; \ol B_{ s}(z_1) \right)  \leq       \llambda \wt D_h(u,v)  . 
\eqe
This gives condition~\ref{item-endpt-ball-leb} in Lemma~\ref{lem-endpt-ball} with $z_1$ in place of $\ur_r$ and $s$ in place of $\sr_r$. The analagous bound with $z_2$ in place of $\vr_r$ and $s$ in place of $\sr_r$ is proven similarly. 
\end{proof}

\begin{proof}[Proof of Lemma~\ref{lem-endpt-ball}]
Let $\pr$ be as in Lemma~\ref{lem-endpt-ball0}. In light of Lemma~\ref{lem-endpt-ball0}, it suffices to find $\Aloc > 3$ such that with probability at least $1-\pr$, condition~\ref{item-endpt-ball-around} in the lemma statement holds, i.e., there exists $t \in [3r , \Aloc r]$ such that 
\eqb \label{eqn-endpt-ball-around}
  D_h\left(\text{around $\BB A_{t,2t}(0)$} \right) \leq  \llambda D_h\left(\text{across $\BB A_{2t,3t}(0)$} \right) .
\eqe 
One can easily check using a ``subtracting a bump function" argument and Weyl scaling (Axiom~\ref{item-metric-f}) that there exists $q  \in (0,1)$ (depending only on $\llambda$ and the law of $D_h$) such that for each fixed $t > 0$, the probability of the event in~\eqref{eqn-endpt-ball-around} is at least $q$. See~\cite[Lemma 6.1]{gwynne-ball-bdy} for similar argument. We can then apply assertion~\ref{item-annulus-iterate-pos} of Lemma~\ref{lem-annulus-iterate} to a collection of logarithmically many evenly spaced radii $t_k \in [3r , \Aloc r]$ to find that the probability that there does not exist $t\in [3r , \Aloc r]$ such that~\eqref{eqn-endpt-ball-around} holds decays like a negative power of $\Aloc$ as $\Aloc \rta\infty$, at a rate which depends only on the laws of $D_h$ and $\wt D_h$.  
We can therefore choose $\Aloc$ large enough so that this probability is at most $\pr$, as required.
\end{proof}

\subsection{Building block event}
\label{sec-block-event}

We will use Lemma~\ref{lem-endpt-ball} to define an event which will be the ``building block" for the event $\Er_r = \Er_{0,r}$. 
Let the parameters $  \Aloc , \pr > 0$, the half-annulus $\Hr_r \subset \BB A_{\Kann r , r}(0)$, the radius $\sr_r \in [\Aendpt r , \Aendpt^{1/2} r] \cap \{4^{-k} r\}_{k\in\BB N}$, and the points
\alb
\ur_r &\in \bdy \Hr_r \cap \left\{ \Kann r e^{i \llambda \Aendpt k} : k \in [1,2\pi \llambda^{-1} \Aendpt^{-1} ]_{\BB Z} \right\}   \quad \text{and} \notag\\
\vr_r &\in \bdy \Hr_r \cap \left\{  r e^{i \llambda \Aendpt k} : k \in [1,2\pi\llambda^{-1} \Aendpt^{-1} ]_{\BB Z} \right\}
\ale
be as in Lemma~\ref{lem-endpt-ball}. 

For $z\in \BB C$, let 
\alb
\Hr_{z,r} 
&:= \Hr_r + z \subset \BB A_{\Kann r , r}(z) ,\notag\\ 
\ur_{z,r} &:= \ur_r + z \in \bdy\Hr_{z,r} \cap \bdy B_{\Kann r}(z) ,\quad \text{and}\notag\\
\vr_{z,r} &:= \vr_r + z \in \bdy\Hr_{z,r} \cap \bdy B_r(z)  .
\ale
We also let $\Fr_{z,r}$ be the event of Lemma~\ref{lem-endpt-ball} with the translated field $h(\cdot - z)$ in place of $h$.  
That is, $\Fr_{z,r}$ is the event that there exist non-singular points $u \in \bdy B_{\Kann r}(z) \cap B_{\sr_r/2}(\ur_{z,r})$ and $v \in \bdy B_r(z) \cap B_{\sr_r/2}(\vr_{z,r})$ with the following properties.  
\begin{enumerate}
\item $\wt D_h(u,v) \leq \Cmid_0  D_h(u,v)$. \label{item-endpt-ball-dist'} 
\item There is a $\wt D_h$-geodesic $\wt P$ from $u$ to $v$ which is contained in $\ol{\Hr}_{z,r}$. \label{item-endpt-ball-annulus'}  
\item The one-dimensional Lebesgue measure of the set \label{item-endpt-ball-leb'} 
\eqbn
\left\{x \in \bdy B_{  \sr_r}(\ur_{z,r} ) \, : \,   D_h\left( x , u  ; \ol B_{ \sr_r}(\ur_{z,r}) \right)  > \llambda \wt D_h(u,v) \right\} 
\eqen
is at most $(\llambda/2) \sr_r $ and the same is true with $v$ and $\vr_{z,r}$ in place of $u$ and $\ur_{z,r}$.
\item There exists $t \in [3r , \Aloc r]$ such that \label{item-endpt-ball-around'}
\eqbn
 D_h\left(\text{around $\BB A_{t,2t}(z)$} \right) \leq  \llambda D_h\left(\text{across $\BB A_{2t,3t}(z)$} \right) .
\eqen 
\end{enumerate}  

By Lemma~\ref{lem-endpt-ball}, the translation invariance of the law of $h$, viewed modulo additive constant, and the translation invariance of $D_h$ and $\wt D_h$ (Axiom~\reftranslate), we have 
\eqb \label{eqn-endpt-prob}
\BB P[\Fr_{z,r}] \geq \pr ,\quad \forall z\in\BB C , \quad \forall r\in\mcl R_0 . 
\eqe 
The other property of $\Fr_{z,r}$ which we need is that it depends locally on $h$. 

\begin{lem} \label{lem-endpt-event-msrble} 
The event $\Fr_{z,r}$ is a.s.\ determined by the restriction of $h$ to $B_{3\Aloc r}(z)$, viewed modulo additive constant.
\end{lem}
\begin{proof}
It is clear from Weyl scaling (Axiom~\ref{item-metric-f}) that adding a constant to $h$ does not affect the occurrence of $\Fr_{z,r}$, so $\Fr_{z,r}$ is a.s.\ determined by $h$, viewed modulo additive constant. It therefore suffices to show that $\Fr_{z,r}$ is a.s.\ determined by $h|_{B_{3\Aloc r}(z)}$. 

To this end, we first observe that by locality (Axiom~\ref{item-metric-local}), the condition~\ref{item-endpt-ball-around'} in the definition of $\Fr_{z,r}$ is a.s.\ determined by $h|_{B_{3\Aloc r}(z)}$. We claim that if this condition holds, then
\eqb \label{eqn-endpt-msrble-claim}
  D_h(x,y) =  D_h\left(x,y ; B_{3\Aloc r}(z) \right) ,\quad \forall x ,y \in B_{3r}(z) ; 
\eqe 
and the same is true with $\wt D_h$ in place of $ D_h$. 

Indeed, it is clear that~\eqref{eqn-endpt-msrble-claim} holds if $x=y$ or if either $x$ or $y$ is a singular point. Hence we can assume that $x\not=y$ and that $x$ and $y$ are not singular points. To prove~\eqref{eqn-endpt-msrble-claim}, it suffices to show that each $  D_h$-geodesic from $x$ to $y$ is contained in $B_{3\Aloc r}(z)$. To see this, let $P$ be a path from $x$ to $y$ which exits $B_{3\Aloc r}(z)$. Let $t \in [3r, \Aloc r]$ be as in condition~\ref{item-endpt-ball-around'} in the definition of $\Fr_{z,r}$. We can find a path $\pi \subset \BB A_{t,2t}(z)$ which disconnects the inner and outer boundaries of $\BB A_{t,2t}(z)$ such that
\eqbn
\op{len}\left( \pi ;   D_h \right) <  D_h\left(\text{across $\BB A_{2t,3t}(z)$} \right) .
\eqen
Since $x,y\in B_{3r}(z)$ and $P$ exists $B_{3t}(z)$, the path $P$ must hit $\pi$, then cross between the inner and outer boundaries of $\BB A_{2t,3t}(z)$, then subsequently hit $\pi$ again. This means that there are two points of $P\cap \pi$ such that $D_h$-length of the segment of $P$ between the two points is at least $ D_h\left(\text{across $\BB A_{2t,3t}(z)$} \right)$. The $ D_h$-distance between these two points is at most the $  D_h$-length of $\pi$, which by our choice of $\pi$ is strictly less than $  D_h\left(\text{across $\BB A_{2t,3t}(z)$} \right)$. Hence $P$ cannot be a $ D_h$-geodesic. We therefore obtain~\eqref{eqn-endpt-msrble-claim} for $ D_h$. 

To prove~\eqref{eqn-endpt-msrble-claim} with $\wt D_h$ in place of $ D_h$, we observe that if $t$ is as in condition~\ref{item-endpt-ball-around'} in the definition of $\Fr_{z,r}$, then
\alb 
\wt D_h\left(\text{around $\BB A_{t,2t}(z)$} \right) 
\leq \Cupper   D_h\left(\text{around $\BB A_{t,2t}(z)$} \right) 
&\leq  \llambda \Cupper D_h\left(\text{across $\BB A_{2t,3t}(z)$} \right) \notag\\
&\leq \llambda (\Cupper/\Clower) \wt D_h\left(\text{across $\BB A_{2t,3t}(z)$} \right) .
\ale
We have $\llambda (\Cupper / \Clower) < 1$, so we can now prove~\eqref{eqn-endpt-msrble-claim} with $\wt D_h$ in place of $  D_h$ via exactly the same argument given above. 

Due to~\eqref{eqn-endpt-msrble-claim}, the definition of $\Fr_{z,r}$ is unaffected if we require that $\wt P$ is a $\wt D_h(\cdot,\cdot;B_{3\Aloc r}(z))$-geodesic instead of a $\wt D_h$-geodesic and we replace $D_h$-distances and $\wt D_h$-distances by $ D_h(\cdot,\cdot;B_{3\Aloc r}(z))$-distances and $\wt D_h(\cdot,\cdot;B_{3\Aloc r}(z))$-distances throughout. It then follows from locality (Axiom~\ref{item-metric-local}) that $\Fr_{z,r}$ is a.s.\ determined by $h|_{B_{3\Aloc r}(z)}$, as required.
\end{proof}

\subsection{Definitions of $\Ur_r$, $\Vr_r$, and $\fr_r$}
\label{sec-tube-def}

\newcommand{\Kr}{{\mathsf K}}
\newcommand{\Lr}{{\mathsf L}}
\newcommand{\Zr}{{\hyperref[eqn-test-pts]{\mathsf Z}}}

\newcommand{\Aacross}{{\hyperref[item-E-across]{\mathsf a_1}}}
\newcommand{\Aaround}{{\hyperref[item-E-around]{\mathsf A_2}}}
\newcommand{\Aep}{{\hyperref[item-E-reg]{\mathsf a_3}}}
\newcommand{\Asp}{{\hyperref[item-E-event]{\mathsf a_4}}} 
\newcommand{\Anar}{{\hyperref[item-E-narrow]{\mathsf a_5}}}  
\newcommand{\Aset}{{\hyperref[item-E-internal]{\mathsf a_6}}}
\newcommand{\Ainternal}{{\hyperref[item-E-internal]{\mathsf A_7}}}   
\newcommand{\Amax}{{\hyperref[eqn-f-def]{\mathsf A_8}}}      
\newcommand{\Asup}{{\hyperref[item-E-sup]{\mathsf a_9}}}    
\newcommand{\Arn}{{\hyperref[item-E-rn]{\mathsf A_{10}}}}  

The definitions of $\Er_r , \Ur_r,\Vr_r$, and $\fr_r$ will depend on parameters
\eqb \label{eqn-A-parameters}
1 > \Aacross  > \frac{1}{\Aaround} > \Aep > \Asp > \Anar > \Aset > \frac{1}{\Ainternal}  > \frac{1}{\Amax} > \Asup > \frac{1}{\Arn}  ,
\eqe 
which will be chosen in Section~\ref{sec-E} in a manner depending only on $\BB p,\llambda$, and the laws of $D_h$ and $\wt D_h$. 
The parameters are listed in~\eqref{eqn-A-parameters} in the order in which they will be chosen. 
Each parameter will be allowed to depend on the earlier parameters as well as the number $\llambda$ from~\eqref{eqn-small-const} (which is allowed to depend only on the laws of $D_h$ and $\wt D_h$, not on $\BB p$).
Each parameter will also be allowed to depend on the numbers $\Kann , \Aendpt, \Aloc, \pr$ appearing in Lemma~\ref{lem-endpt-ball} (which have already been fixed, in a manner depending only on $\llambda$ and the laws of $D_h$ and $\wt D_h$).  

Also let $\rho \in (0,1)$ be a small parameter which will also be chosen in Section~\ref{sec-E} in a manner depending only on $\llambda$ and the laws of $D_h$ and $\wt D_h$.
We will have
\eqb \label{eqn-rho-parameter}
\Asp > \rho > \Anar ,
\eqe
and $\rho$ will be allowed to depend on $\llambda,\Aacross,\Aaround,\Aep,\Asp$ and the numbers appearing in Lemma~\ref{lem-endpt-ball}.

\begin{figure}[ht!]
\begin{center}
\includegraphics[width=.7\textwidth]{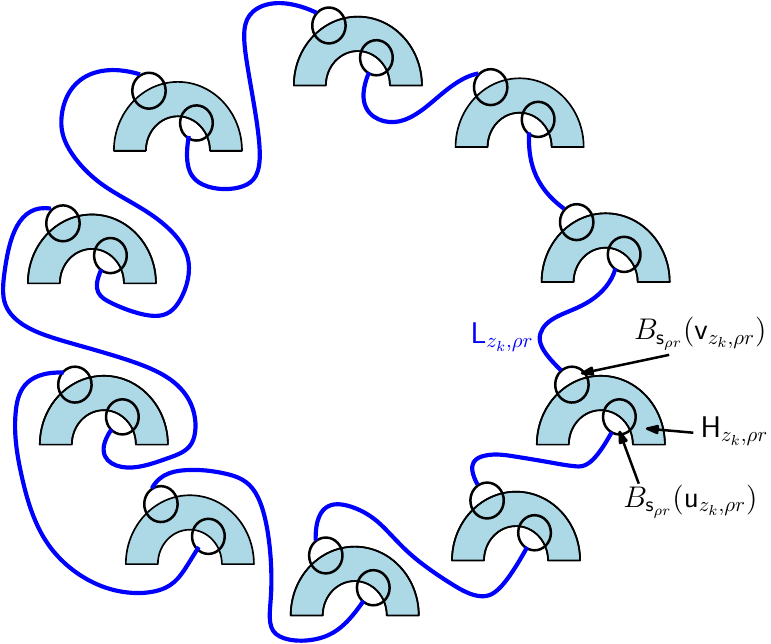} 
\caption{\label{fig-U-def} The figure shows the sets $\Hr_{z,\rho r}$, $ B_{\sr_{\rho r}}(\ur_{z,\rho r})$, $B_{\sr_{\rho r}}(\vr_{z,\rho r})$, and $\Lr_{z,\rho r}$ for $z\in \Zr_r$.  
We define $\Ur_r$ to be the union of $\Hr_{z,\rho r}$, $ B_{\sr_{\rho r}}(\ur_{z,\rho r})$, $B_{\sr_{\rho r}}(\vr_{z,\rho r})$ and $B_{\llambda \Aendpt \rho r}(\Lr_{z,\rho r})$ for $z\in\Zr_r$. We define $\Vr_r :=  B_{\Asup r}(\Ur_r) $. 
The bump function $\fr_r$ is supported on $\Vr_r$ and attains its maximal value $\Amax$ at every point of $\Ur_r$. 
}
\end{center}
\end{figure}

In the rest of this subsection, we will give the definition of the open sets $\Ur_r$ and $\Vr_r$ and the bump function $\fr_r$ in terms of $\rho$ and the parameters from~\eqref{eqn-A-parameters}. See Figure~\ref{fig-U-def} for an illustration. 
For $r\in \rho^{-1} \mcl R_0$, let 
\eqb \label{eqn-test-pts-count}
\Kr_\rho := \left\lceil  \frac{\llambda}{\Aloc \rho} \right\rceil ,
\eqe  
where $\Aloc$ is as in Lemma~\ref{lem-endpt-ball}. 
We define the set of ``test points" 
\eqb \label{eqn-test-pts}
\Zr_r = \Zr_r(\rho) := \left\{ 2 r  \exp\left( 2 \pi i k / \Kr_\rho  \right)  : k \in  \left[1,    \Kr_\rho  \right]_{\BB Z} \right\}  \subset \bdy B_{2r}(0) .
\eqe 
The event $\Er_r$ will include the condition that the event $\Fr_{z,\rho r}$ of Section~\ref{sec-block-event} occurs for ``many" of the points $z\in \Zr_r$. 

Recall the half-annuli $\Hr_{z, \rho r}$ and the balls $B_{\sr_{\rho r} }(\ur_{z , \rho r} )$ and $B_{\sr_{\rho r}}(\vr_{z , \rho r})$ from the definition of $\Fr_{z,\rho r}$.
We emphasize that by Lemma~\ref{lem-endpt-ball}, the number of possible choices for the half-annulus $(\rho r)^{-1} [\Hr_{z, \rho r}  - z]$ and the balls $(\rho r)^{-1} [ B_{\sr_{\rho r} }(\ur_{z , \rho r} ) -z]$ and $(\rho r)^{-1}[ B_{\sr_{\rho r}}(\vr_{z , \rho r}) - z]$ is at most a constant depending only on $\llambda$ and the laws of $D_h$ and $\wt D_h$. 

We will now construct a ``tube" which links up the sets $\Hr_{z, \rho r}  \cup B_{\sr_{\rho r} }(\ur_{z , \rho r} ) \cup B_{\sr_{\rho r}}(\vr_{z , \rho r})$ for $z\in \Zr_r$. 
For $k \in [1,\Kr_\rho]_{\BB Z}$, let $z_k :=  2 r  \exp\left( 2 \pi i k / \Kr_r  \right)$ be the $k$th element of $\Zr_r$. We also set $z_{\Kr_\rho + 1}  := z_{1}$. 
We choose for each $k\in [1,\Kr_\rho]_{\BB Z}$ a smooth simple path $\Lr_{z_k , \rho r}$ from the point of $B_{\sr_{\rho r}}(\vr_{z_k , \rho r})$ which is furthest from $\Hr_{z_k ,\rho r}$ to the point of $ B_{\sr_{\rho r} }(\ur_{z_{k+1} , \rho r} ) $ which is furthest from $\Hr_{z_{k+1} , \rho r} $. We can arrange that these paths have the following properties. 
\begin{enumerate}[(i)] 
\item Each $\Lr_{z_k,\rho r}$ is contained in the $10\rho r$-neighborhood of $\bdy B_{2r}(0)$. 
\item The Euclidean distance from $\Lr_{z_k,\rho r}$ to each of the half-annuli $\Hr_{z_k,\rho r}$ and $\Hr_{z_{k+1} ,\rho r}$ is at least $\sr_{\rho r} / 2$. 
\item The Euclidean distance from $\Lr_{z_k,\rho r}$ to each of the following sets is at least $(1-\Kann) \rho r / 4$: 
\begin{itemize}
\item The sets $\Hr_{w,\rho r}$ for $w\in \Zr_r\setminus \{z_k,z_{k+1}\}$;
\item The sets $\Lr_{w,\rho r}$ for $w \in \Zr_r\setminus \{ z_k\}$;
\item The sets $B_{\sr_{\rho r}}(\vr_{w,\rho r})$ for $w \in \Zr_r\setminus \{ z_k\}$;
\item The sets $B_{\sr_{\rho r}}(\ur_{w,\rho r})$ for $w \in \Zr_r\setminus \{ z_{k+1}\}$. 
\end{itemize} 
\item The number of possibilities for the path $(\rho r)^{-1}(\Lr_{z_k , \rho r} - z_k)$ is at most a constant depending only on $\rho$, $\llambda$, and the laws of $D_h$ and $\wt D_h$. 
\end{enumerate}
 
With $\Aendpt$ as in Lemma~\ref{lem-endpt-ball}, we define
\eqb \label{eqn-U-def}
\Ur_r = \Ur_r(\rho) := \bigcup_{z \in \Zr_r(\rho)} \left[ \Hr_{z,\rho r} \cup B_{\sr_{\rho r}}(\ur_{z,\rho r}) \cup B_{\sr_{\rho r}}(\vr_{z,\rho r}) \cup B_{\llambda \Aendpt \rho r }(\Lr_{z,\rho r}  )   \right]  
\eqe 
and 
\eqb \label{eqn-V-def}
\Vr_r = \Vr_r(\Ur_r , \Asup) := B_{\Asup r}(\Ur_r) .
\eqe
We emphasize that $\Vr_r$ is determined by $\Ur_r$ and $\Asup$ and (once $\Asup$ is fixed) the number of possible choices for the set $r^{-1} \Ur_r$ is at most a finite constant depending only on $\rho,\llambda$, and the laws of $D_h$ and $\wt D_h$. We cannot take $r^{-1} \Ur_r$ to be independent from $r$ since the radius $\sr_{\rho r}$ and the half-annulus $\Hr_{\rho r}$ from Lemma~\ref{lem-endpt-ball} are allowed to depend on $\rho r$. This is a consequence of the fact that we only have tightness across scales, not exact scale invariance. However, a constant upper bound for the number of possibilities for $r^{-1} \Ur_r$ will be enough for our purposes.

Let 
\eqb \label{eqn-f-def}
\fr_r : \BB C\rta [0,\Amax]
\eqe
 be a smooth bump function which is identically equal to $\Amax$ on $\Ur_r$ and which is supported on $\Vr_r $. 
We can choose $\fr_r$ in such a way that $\fr_r(r\cdot)$ depends only on $r^{-1} \Ur_r$, which means that the number of possible choices for $\fr_r(r\cdot)$ is at most a finite constant depending only on $\Aendpt,\rho,\llambda$, and the laws of $D_h$ and $\wt D_h$.

\subsection{Definition of $\Er_r$}
\label{sec-E}

We will now define the event $\Er_r = \Er_{0,r}$ appearing in Section~\ref{sec-counting-setup}. 
Recall the parameters from~\eqref{eqn-A-parameters} and~\eqref{eqn-rho-parameter}.  
For $r\in \rho^{-1} \mcl R_0$, let $\Er_r$ be the event that the following is true. We will discuss the purpose of each condition just after the definition. 
\begin{enumerate}
\item \label{item-E-across} \textit{(Bound for distance across)} We have  
\eqbn
\min\left\{ D_h\left(\text{across $\BB A_{ r , 1.5 r}(0)$}\right) , \: D_h\left(\text{across $\BB A_{2.5 r , 3r}(0)$}\right) \right\} 
\geq  \Aacross  r^{\xi Q} e^{\xi h_r(0)}  .
\eqen
\item  \label{item-E-around} \textit{(Bound for distance around)} We have 
\eqbn
 D_h\left(\text{around $\BB A_{3r , 4r }(0)$}\right)\leq   \Aaround  r^{\xi Q} e^{\xi h_r(0)}   .
\eqen 
\item \label{item-E-reg} \textit{(Regularity along geodesics)} The event of Lemma~\ref{lem-hit-ball-phi} occurs with $U = \BB A_{1,4}(0)$, $\chi =1/2$, and $\ep_0 = \Aep$. That is, for each $\ep \in (0,\Aep]$, the following is true. Let $V\subset \BB A_{r,4r}(0)$ and let $f : \BB C\rta [0,\infty)$ be a non-negative continuous function which is identically zero outside of $ V$. 
Let $z\in  \BB A_{r + \ep^{1/2} ,4r - \ep^{1/2} }(0)$, $x,y\in \ol{\BB A}_{r,4r}(0) \setminus ( V \cup B_{\ep^{1/2} r}(z))$, and $s>0$ such that there is a $D_{h-f}(\cdot,\cdot ; \ol{\BB A}_{r,4r}(0))$-geodesic $P_f$ from $x$ to $y$ with $P_f(s) \in B_{\ep r}(z)$. Assume that $s \leq \inf\{t > 0 : P_f(t) \in V \}$. 
Then with $\geoExp = \geoExp(1/2)>0$ as in Lemma~\ref{lem-hit-ball-phi},
\eqb \label{eqn-E-reg}
D_h\left( \text{around $\BB A_{ \ep r , \ep^{1/2} r}(z)$} \right) 
\leq \ep^\geoExp s.
\eqe
\item \label{item-E-event} \textit{(Existence of shortcuts)} Let $\Zr_r$ be the set of test points as in~\eqref{eqn-test-pts}. 
For each connected circular arc $I \subset \bdy B_{2r}(0)$ with Euclidean length at least $\Asp r / 2$, there exists $z \in I \cap \Zr_r $ such that the event $\Fr_{z,\rho r}$ of Section~\ref{sec-block-event} occurs.
\item \label{item-E-compare} \textit{(Comparison of distances in small annuli)} For each $z\in \BB A_{1.5 r, 3 r}(0)$ and each $\delta \in (0,\Anar]$, 
\eqb \label{eqn-E-compare}
D_h\left(\text{around $\BB A_{\delta r / 4 , \delta r / 2}(z)$}\right) \leq \delta^{-1/4} D_h\left(\text{across $\BB A_{2 \delta r, 3\delta r}(z)$}\right) .
\eqe
\item \label{item-E-narrow} \textit{(Reverse H\"older continuity)} For each $z,w\in \BB A_{1.5 r, 3 r}(0)$ with $|z - w| \leq \llambda^{-1} \Anar r$,  
\eqbn
D_h\left(z,w  ; \BB A_{r,4r}(0) \right) \geq \left( \frac{|z-w|}{r}  \right)^{\xi (Q+3)}  r^{\xi Q} e^{\xi h_r(0)} .
\eqen
\item \label{item-E-internal} \textit{(Internal distance in $\Ur_r$)} We have
\eqb \label{eqn-E-internal-around}
D_h\left(\text{around $\Ur_r$} \right) \leq \Ainternal r^{\xi Q} e^{\xi h_r(0)} .
\eqe
More strongly, there is a path $\Pi  \subset \Ur_r$ which disconnects the inner and outer boundaries of $\Ur_r$ and has $D_h$-length at most $\Ainternal r^{\xi Q} e^{\xi h_r(0)}$ such that each point of the outer boundary\footnote{The set $\Ur_r$ has the topology of a Euclidean annulus, so its boundary has two connected components, one of which disconnects the other from $\infty$. The outer boundary is the outer of these two components.} of $\Ur_r$ lies at Euclidean distance at most $\Aset r$ from $\Pi$. 
\item \label{item-E-sup} \textit{(Intersections of geodesics with a small neighborhood of the boundary)} Let $f : \BB C\rta [0 , \Amax]$ be a continuous function and let $P_f$ be a $D_{h - f}(\cdot,\cdot ; \ol{\BB A}_{r,4r}(0))$-geodesic between two points of $\bdy B_{4r}(0)$. The one-dimensional Lebesgue measure of the set of $x\in\bdy \Ur_r$ such that $P_f \cap B_{2\Asup r}(x) \not=\emptyset$ is at most $\llambda \Aendpt \rho r $. Moreover, the same is true with $\bdy \Ur_r$ replaced by each of the circles $\bdy B_{\sr_{\rho r}}(\ur_{z,\rho r})$ and $\bdy B_{\sr_{\rho r}}(\vr_{z,\rho r})$ for $z\in\Zr_r$. 
\item \label{item-E-rn} \textit{(Radon-Nikodym derivative bound)} The Dirichlet inner product of $h$ with $\fr_r$ satisfies
\eqb
|(h,\fr_r)_\nabla| \leq \Arn .
\eqe
\end{enumerate} 

We will eventually show that $\Er_r$ satisfies the hypotheses for $\Er_{0,r}$ listed in Section~\ref{sec-counting-setup}. 
Before beginning the proof of this fact, we discuss the various conditions in the definition of $\Er_r$.  

Conditions~\ref{item-E-across} and~\ref{item-E-around}  occur with high probability due to tightness across scales (Axiom~\refcoord). These conditions are needed to ensure that hypothesis~\ref{item-Ehyp-dist} from Section~\ref{sec-counting-setup} is satisfied. 
Condition~\ref{item-E-around} is also useful for upper-bounding the amount of time that a $D_h$-geodesic or a $D_{h-\fr_r}$-geodesic between points outside of $B_{4r}(0)$ can spend in $\Vr_r$. Indeed, if $\pi$ is a path in $\BB A_{3r,4r}(0)$ which disconnects the inner and outer boundaries of near-minimal $D_h$-length (equivalently, near-minimal $D_{h-\fr_r}$-length since $\Vr_r\cap \BB A_{3r,4r}(0) = \emptyset$), then any such geodesic must hit $\pi$ both before and after hitting $\Vr_r$. The length of the geodesic segment between these hitting times is at most the length of $\pi$. See Lemma~\ref{lem-E-loop} for an application of this argument. 

Condition~\ref{item-E-reg} holds with high probability due to Lemma~\ref{lem-hit-ball-phi}. This condition will eventually be applied with $V = \Vr_r$ and $f  = \fr_r$. We allow a general choice of $V$ and $f$ in the condition statement since we will choose the parameter $\Aep$ in condition~\ref{item-E-reg} before we choose the parameters $\rho ,  \Amax , \Asup$ involved in the definitions of $\Vr_r$ and  $\fr_r$. 
The condition will be used in two places: to lower-bound the Euclidean distance between two points on a $D_{h-\fr_r}$-geodesic in terms of their $D_h$-distance (Lemma~\ref{lem-E-sp}); and to link up a point on a $D_{h-\fr_r}$-geodesic which is close to $\bdy\Ur_r$ with a path in $\Ur_r$ (Lemma~\ref{lem-excursion-length}). 

Condition~\ref{item-E-event} is in some sense the most important condition in the definition of $\Er_r$. Due to the definition of $\Fr_{z,\rho r}$ from Section~\ref{sec-block-event}, this condition provides a large collection of ``good" pairs of points $u,v\in \Ur_r$ such that $\wt D_h(u,v) \leq \Cmid_0 D_h(u,v)$. 
The fact that we consider the event $\Fr_{z,\rho r}$ in this condition is the reason why we need to require that $r\in\rho^{-1} \mcl R_0$. 
We will need to make $\rho$ small in order to make the set of test points $z \in \Zr_r$ of~\eqref{eqn-test-pts} large, so that we can apply a long-range independence result for the GFF (Lemma~\ref{lem-spatial-ind}) to say that condition~\ref{item-E-event} occurs with high probability. See Lemma~\ref{lem-E-event}. 

Condition~\ref{item-E-compare} has high probability due to Lemma~\ref{lem-annulus-union}, and will be used in Section~\ref{sec-endpt-close}. More precisely, we will consider a segment of a $D_{h-\fr_r}$-geodesic which is contained in a small Euclidean neighborhood of the ball $B_{\sr_{\rho r}}(\ur_{z,\rho r})$ in the definition of $\Fr_{z,\rho r}$.
We will use the paths around annuli provided by condition~\ref{item-E-compare} to ``link up" this geodesic segment to a short path from $u$ to the boundary of this ball, as provided by condition~\ref{item-endpt-ball-leb'} in the definition of $\Fr_{z,\rho r}$ (see Lemma~\ref{lem-arc-loop}). 

Condition~\ref{item-E-narrow} has high probability due to the local reverse H\"older continuity of $D_h$ w.r.t.\ the Euclidean metric~\cite[Proposition 3.8]{pfeffer-supercritical-lqg}. This condition will be used in several places, e.g., to force a $D_{h-\fr_r}$-geodesic between two points of $\bdy \Vr_r$ to stay in a small Euclidean neighborhood of $\Vr_r$ (Lemma~\ref{lem-excursion-tube}). See also the summary of Section~\ref{sec-excursion} in Section~\ref{sec-construction-outline}. The requirement that $|z - w| \leq \llambda^{-1} \Anar r$ is needed to make the condition occur with high probability (c.f.~\cite[Proposition 3.8]{pfeffer-supercritical-lqg}).

Condition~\ref{item-E-internal} has high probability due to a straightforward argument based on tightness across scales and the fact that there are only finitely many possibilities for $r^{-1} \Ur_r$ (see Lemma~\ref{lem-E-internal}). This condition will be used to check the condition on $D_h(\text{around $\Ur_r$})$ in hypothesis~\ref{item-Ehyp-dist} for $\Er_r$. The reason why we need to require that each point of the outer boundary of $\Ur_r$ is close to the path $\Pi$ is as follows. In the proof of Lemma~\ref{lem-excursion-length}, we will consider a $D_{h-\fr_r}$-geodesic $P_r$ and times $\tau  <\sigma$ at which it hits $\bdy \Vr_r$. We will upper-bound $\sigma - \tau = D_{h-\fr_r}(P_r(\tau) , P_r(\sigma))$ by concatenating a segment of $\Pi$ with segments of small loops surrounding $P_r(\tau)$ and $P_r(\sigma)$ which are provided by condition~\ref{item-E-reg}. The condition on $\Pi$ is needed to ensure that these small loops actually intersect $\Pi$. 

Recall that $\fr_r : \BB C\rta [0,\Amax]$. Condition~\ref{item-E-sup} has high probability due to Lemma~\ref{lem-geodesic-leb}. We will eventually apply this condition with $f = \fr_r$ in order to say that a $D_{h-\fr_r}$-geodesic cannot spend much time in the region $\Vr_r\setminus \Ur_r$ where $\fr_r$ takes values strictly between $0$ and $\Amax$ (see Lemmas~\ref{lem-dc-set} and~\ref{lem-macro-arcs}). The reason why we allow a general choice of $f$ in the condition statement is that $\Vr_r = B_{\Asup r}(\Ur_r)$, and hence also $\fr_r$, depends on the parameter $\Asup$, which needs to be made small enough to make the probability of condition~\ref{item-E-sup} close to 1. 

The purpose of condition~\ref{item-E-rn} is to check the Radon-Nikodym derivative hypothesis~\ref{item-Ehyp-rn} from Section~\ref{sec-counting-setup}, see Proposition~\ref{prop-E-hyp0}. This condition occurs with high probability due to the scale invariance of the law of $h$, modulo additive constant, and the fact that there are only finitely many possibilities for $\fr_r(r\cdot)$ (Lemma~\ref{lem-E-rn}).

\subsection{Properties of $\Er_r$}
\label{sec-E-prob}

We first check that $\Er_r$ satisfies an appropriate measurability condition.

\begin{lem} \label{lem-E-msrble}
The event $\Er_r$ is a.s.\ determined by $h|_{\ol{\BB A}_{r , 4r}(0)}$, viewed modulo additive constant.
\end{lem}
\begin{proof}
By Weyl scaling (Axiom~\ref{item-metric-f}) that the occurrence of $\Er_r$ is unaffected by adding a constant to $h$, so $\Er_r$ is a.s.\ determined by $h$ viewed modulo additive constant. It is immediate from locality (Axiom~\ref{item-metric-local}; see also Section~\ref{sec-closed}) that each condition in the definition of $\Er_r$ except possibly condition~\ref{item-E-event} is a.s.\ determined by $h|_{\ol{\BB A}_{r,4r}(0)}$. Lemma~\ref{lem-endpt-event-msrble} implies that condition~\ref{item-E-event} is a.s.\ determined by $h|_{\ol{\BB A}_{r,4r}(0)}$ as well. 
\end{proof}

Most of the rest of this subsection is devoted to proving the following. 

\begin{prop} \label{prop-E-prob}
For each $\BB p\in (0,1)$, we can choose the parameters in~\eqref{eqn-A-parameters} and~\eqref{eqn-rho-parameter} in such a way that 
\eqb
\BB P\left[ \Er_r \right] \geq \BB p,\quad \forall r \in \rho^{-1} \mcl R_0 .
\eqe
\end{prop}

To prove Proposition~\ref{prop-E-prob}, we will treat the conditions in the definition of $\Er_r$ in order. For each condition, we will choose the parameters involved in the condition, in a manner depending only on $\BB p,\llambda$, and the laws of $D_h$ and $\wt D_h$, in such a way that the condition occurs with high probability. 
For some of the conditions, we will impose extra constraints on the parameters beyond just the numerical ordering in~\eqref{eqn-A-parameters} and~\eqref{eqn-rho-parameter}. 
These constraints will be stated and referenced as needed in the later part of the proof.

\begin{lem} \label{lem-E-reg}
There exists $\Aacross > 1/ \Aaround > \Aep > 0$ depending only on $\BB p,\llambda$, and the laws of $D_h$ and $\wt D_h$ such that for each $r>0$, the probability of each of conditions~\ref{item-E-across}, \ref{item-E-around}, and~\ref{item-E-reg} in the definition of $\Er_r$ is at least $1- (1-\BB p)/10$.
\end{lem}
\begin{proof}
By tightness across scales (Axiom~\refcoord), we can choose $\Aacross , \Aaround > 0$ such that the probabilities of conditions~\ref{item-E-across} and~\ref{item-E-around} are each at least $1-(1-\BB p)/10$. By Lemma~\ref{lem-hit-ball-phi}, we can choose $\Aep  > 0$ such that the probability of condition~\ref{item-E-reg} is at least $1-(1-\BB p)/10$. 
\end{proof}

We henceforth fix $\Aacross,\Aaround,\Aep $ as in Lemma~\ref{lem-E-reg}. 
Our next task is to make an appropriate choice of the parameter $\Asp$ appearing in condition~\ref{item-E-event}.

\begin{lem} \label{lem-E-sp}
Let $r >0$ and assume that conditions~\ref{item-E-across}, \ref{item-E-around}, and~\ref{item-E-reg} in the definition of $\Er_r$ occur.
Let $V\subset \BB A_{r,3r}(0)$ and let $f : \BB C\rta [0,\infty)$ be a non-negative continuous function which is identically zero outside of $ V$.
Also let $P_f$ be a $D_{h-f}(\cdot,\cdot ; \ol{\BB A}_{r,4r}(0))$-geodesic between two points of $\bdy B_{4r}(0)$ and define the times
\eqb \label{eqn-E-sp-times}
\tau := \inf\{t > 0 : P_f(t) \in V\} \quad \text{and} \quad \sigma := \sup\{t > 0 : P_f(t) \in V\} . 
\eqe
There exists $\Asp > 0$ depending only on $\BB p,\llambda$, and the laws of $D_h$ and $\wt D_h$ such that the following is true.  
If
\eqb \label{eqn-E-sp-assume}
D_h\left( P_f(\tau) , P_f(\sigma) ; B_{4r}(0) \right) \geq \frac{\Aacross^2}{4\Aaround} r^{\xi Q} e^{\xi h_r(0)} ,
\eqe 
then
\eqb \label{eqn-E-sp}
|P_f(\tau) - P_f(\sigma)| \geq   \Asp r. 
\eqe
\end{lem}

The motivation for our choice of $\Asp$ comes from hypothesis~\ref{item-Ehyp-inc} for $\Er_r$ from Section~\ref{sec-counting-setup}.
We will eventually apply Lemma~\ref{lem-E-sp} with $V = \Vr_r$, $f = \fr_r$, and $P_f$ equal to a $(B_{4r}(0) , \Vr_r)$-excursion of a $D_{h-\fr_r}$-geodesic between two points of $\BB C\setminus B_{4r}(0)$ (recall Definition~\ref{def-excursion}). 
The assumption~\eqref{eqn-E-sp-assume} is closely related to the condition~\eqref{eqn-Ehyp-exc} from hypothesis~\ref{item-Ehyp-inc}. 
The lower bound for $|P_f(\tau) - P_f(\sigma)|$ from~\eqref{eqn-E-sp} will eventually be combined with condition~\ref{item-E-event} in the definition of $\Er_r$ to ensure that there is a $z\in\Zr_r$ such that $\Fr_{z,r}$ occurs and our $D_{h-\fr_r}$-geodesic gets Euclidean-close to each of the points $u,v$ appearing in the definition of $\Fr_{z,r}$ (see Section~\ref{sec-excursion-hit}). 

For the proof of Lemma~\ref{lem-E-sp}, we need the following lemma. 

\begin{lem} \label{lem-E-loop} 
Assume we are in the setting of Lemma~\ref{lem-E-sp} and let $V,f,P_f,\tau$, and $\sigma$ be as in that lemma. 
For each $\ep  \in (0,\Aep]$, one has
\allb \label{eqn-E-loop}
&\max\left\{D_h\left(\text{around $\BB A_{\ep r , \ep^{1/2} r}(P_f(\tau))$} \right) , D_h\left(\text{around $\BB A_{\ep r , \ep^{1/2} r}(P_f(\sigma))$} \right) \right\} \notag\\
&\qquad\qquad\qquad\qquad \leq 2 \Aaround \ep^\geoExp  r^{\xi Q} e^{\xi h_r(0)}  .
\alle
\end{lem}
\begin{proof}
Let $\tau_0$ (resp.\ $\sigma_0$) be the last time before $\tau$ (resp.\ the first time after $\sigma$) at which $P_f$ hits $\bdy B_{3r}(0)$. 
By condition~\ref{item-E-around} in the definition of $\Er_r$, there is a path $\Pi \subset \BB A_{3r,4r}(0)$ with $D_h$-length at most $2\Aaround r^{\xi Q} e^{\xi h_r(0)}$ which disconnects the inner and outer boundaries of $\BB A_{3r,4r}(0)$.  
Since $f$ is supported on $\BB A_{r,3r}(0)$, the $D_{h-f}$-length of $\Pi$ is the same as its $D_h$-length. 
The path $P_f$ must hit $\Pi$ before time $\tau_0$ and after time $\sigma_0$. Since $P_f$ is a $D_{h-f}(\cdot,\cdot ; \ol{\BB A}_{r,4r}(0))$-geodesic, we infer that 
\eqb \label{eqn-E-sp-time}
\sigma_0 - \tau_0 \leq \op{len}\left(\Pi ; D_{h-f}\right) \leq 2\Aaround r^{\xi Q} e^{\xi h_r(0)} .
\eqe 
Indeed, otherwise we could replace a segment of $P_f$ by a segment of $\Pi$ to get a path in $\ol{\BB A}_{r,4r}(0)$ with the same endpoints as $P_f$ but shorter $D_{h-f}$-length. 

By condition~\ref{item-E-reg} in the definition of $\Er_r$ applied to the $D_{h-f }(\cdot,\cdot ;\ol{\BB A}_{r,4r}(0))$-geodesic $P_f|_{[\tau_0 , \sigma_0]}$ and with $z = P_f(\tau)$ and $s = \tau - \tau_0$, for each $\ep \in (0,\Aep]$, 
\eqb 
D_h\left(\text{around $\BB A_{\ep r , \ep^{1/2} r}(P_f(\tau))$} \right)
\leq  \ep^\geoExp (\tau-\tau_0) 
\leq  \ep^\geoExp (\sigma_0-\tau_0) 
\leq 2\ep^\geoExp \Aaround r^{\xi Q} e^{\xi h_r(0)}  ,
\eqe 
where the last inequality is by~\eqref{eqn-E-sp-time}.
The analogous bound with $\sigma$ in place of $\tau$ follows from the same argument applied with $P_f$ replaced by its time reversal. 
\end{proof}

\begin{figure}[ht!]
\begin{center}
\includegraphics[width=.5\textwidth]{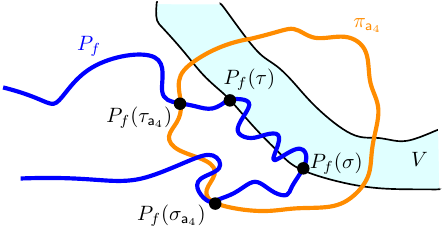} 
\caption{\label{fig-E-sp} Illustration of the proof of Lemma~\ref{lem-E-sp}. If $|P_f(\tau) - P_f(\sigma)| <   \Asp r$, then the union of the orange loop $\pi_\Asp$ and the segments $P_f|_{[\tau_\Asp,\tau]}$ and $P_f|_{[\sigma , \sigma_\Asp]}$ contains a path from $P_f(\tau)$ to $P_f(\sigma)$ of $D_{h-f}$-length less than $\frac{\Aacross^2}{4\Aaround}r^{\xi Q} e^{\xi h_r(0)}$. This yields the contrapositive of the lemma statement. 
}
\end{center}
\end{figure}

\begin{proof}[Proof of Lemma~\ref{lem-E-sp}] 
See Figure~\ref{fig-E-sp} for an illustration. 
By Lemma~\ref{lem-E-loop}, for each $\ep \in (0,\Aep]$ there is a path $\pi_\ep \subset \BB A_{\ep r , \ep^{1/2} r}(P_f(\tau))$ such that
\eqb \label{eqn-E-sp-loop}
\op{len}\left( \pi_\ep ; D_h \right) \leq 4\ep^\geoExp \Aaround r^{\xi Q} e^{\xi h_r(0)}   . 
\eqe  
Let $\Asp   \in (0,\Aep]$ be chosen so that 
\eqb
4 \Asp^\geoExp \Aaround   < \frac{\Aacross^2}{16\Aaround}  . 
\eqe
By~\eqref{eqn-E-sp-loop} and since $f$ is non-negative,
\eqb \label{eqn-E-sp-loop'}
\op{len}\left( \pi_\Asp ; D_{h-f} \right)  \leq \op{len}\left( \pi_\Asp ; D_h \right)  < \frac{\Aacross^2}{16\Aaround} r^{\xi Q} e^{\xi h_r(0)}  .
\eqe
We will prove the contrapositive of the lemma statement with this choice of $\Asp$, i.e., we will show that if $|P_f(\tau) - P_f(\sigma)| < \Asp r$, then $D_h\left( P_f(\tau) , P_f(\sigma) ; B_{4r}(0) \right) < \frac{\Aacross^2}{4\Aaround} r^{\xi Q} e^{\xi h_r(0)}$. 

If $|P_f(\tau) - P_f(\sigma)| < \Asp r$, then $P_f(\sigma) \in B_{\Asp r}(P_f(\tau))$. Since the endpoints of $P_f$ lie in $\bdy B_{4r}(0)$, which is disjoint from $B_{\Asp^{1/2} r}(P_f(\tau))$, it follows that $P_f$ hits $\pi_\Asp$ before time $\tau$ and after time $\sigma$.
Let $\tau_\Asp$ (resp.\ $\sigma_\Asp$) be the last time before time $\tau $ (resp.\ the first time after time $\sigma $) at which $P_f$ hits $\pi_\Asp$. 
Since $P_f$ is a $D_{h-f}(\cdot,\cdot ; \ol{\BB A}_{r,4r}(0))$-geodesic, 
\eqbn
\sigma_\Asp -\tau_\Asp \leq \op{len}\left( \pi_\Asp ; D_{h-f} \right)  < \frac{\Aacross^2}{16\Aaround} r^{\xi Q} e^{\xi h_r(0)}  .
\eqen

By the definitions~\eqref{eqn-E-sp-times} of $\tau$ and $\sigma$, the path segments $P_f|_{[\tau_\Asp,\tau]}$ and $P_f|_{[\sigma,\sigma_\Asp]}$ are disjoint from the support of $f$. So, the $D_{h-f}$-lengths of these segments are the same as their $D_h$-lengths. 
Consequently,
\allb \label{eqn-E-sp-segments}
\op{len}\left( P_f|_{[\tau_\Asp,\tau]} ; D_h \right) +\op{len}\left( P_f|_{[\sigma,\sigma_\Asp]} ; D_h \right)
&\leq \op{len}\left( P_f|_{[\tau_\Asp , \sigma_\Asp]} ; D_{h-f} \right)  \notag\\
&= \sigma_\Asp - \tau_\Asp  
 < \frac{\Aacross^2}{16\Aaround} r^{\xi Q} e^{\xi h_r(0)}  .
\alle 

The union of $P_f([\tau_\Asp,\tau])$, $P_f([\sigma , \sigma_\Asp])$, and $\pi_\Asp$ contains a path from $P_f(\tau)$ to $P_f(\sigma)$. 
Since $V\subset B_{3r}(0)$, this path is contained in $B_{4r}(0)$. We therefore infer from~\eqref{eqn-E-sp-loop'} and~\eqref{eqn-E-sp-segments} that
\eqbn
D_h\left( P_f(\tau) , P_f(\sigma)  ; B_{4r}(0) \right) \leq \frac{3\Aacross^2}{16\Aaround} r^{\xi Q} e^{\xi h_r(0)}  < \frac{\Aacross^2}{4\Aaround}r^{\xi Q} e^{\xi h_r(0)}
\eqen
as required.
\end{proof}

Henceforth fix $\Asp$ as in Lemma~\ref{lem-E-sp}. 
We will now choose $\rho$ so that condition~\ref{item-E-event} in the definition of $\Er_r$ occurs with high probability.

\begin{lem} \label{lem-E-event} 
There exists $\rho \in (0, \llambda \Asp)$, depending only on $\BB p,\llambda$, and the laws of $D_h$ and $\wt D_h$, such that
\eqb \label{eqn-Arho}
\rho^\geoExp \Aaround \leq \llambda \Aacross 
\eqe 
and the following is true.  For each $r \in \rho^{-1} \mcl R_0$, it holds with probability at least $1-(1-\BB p)/10$ that condition~\ref{item-E-event} in the definition of $\Er_r$ occurs.  
\end{lem}
\begin{proof}
By the definition of $\Kr_\rho$ in~\eqref{eqn-test-pts-count} and the definition of $\Zr_r(\rho)$ in~\eqref{eqn-test-pts}, there is a constant $c > 0$ depending only on $ \Aloc , \Asp,$ and $\llambda$ (hence only on $\BB p,\llambda,$ and the laws of $D_h$ and $\wt D_h$) such that for each $\rho\in (0,\llambda/\Aloc)$ and each $r\in \rho^{-1} \mcl R_0$, the set $\Zr_r = \Zr_r(\rho)$ satisfies the following properties. 
\begin{enumerate}[(i)]
\item We have $|z-w| \geq 50 \Aloc \rho r$ for each distinct $z,w \in \Zr_r(\rho)$ (note that $\llambda $ is much smaller than $1/50$, see~\eqref{eqn-small-const}). 
\item Each connected circular arc $J \subset \bdy B_{2r}(0)$ with Euclidean length at least $\Asp r / 4$ contains at least $\lfloor c \rho^{-1} \rfloor$ points of $\Zr_r(\rho)$. 
\end{enumerate} 
Furthermore, there is a constant $C > 0$ depending only on $\Asp$ and a deterministic collection $\mcl J$ of arcs $J\subset \bdy B_{2r}(0)$ such that $\# \mcl J \leq C$, each $J\in\mcl J$ has Euclidean length $\Asp r / 4$, and each arc $I\subset \bdy B_{2r}(0)$ with Euclidean length at least $\Asp r / 2$ contains some $J\in\mcl J$. 

By~\eqref{eqn-endpt-prob}, for each $r\in \rho^{-1} \mcl R_0$ and each $z\in \Zr_r(\rho)$, we have $\BB P[\Fr_{z,\rho r}] \geq \pr$. 
By Lemma~\ref{lem-endpt-event-msrble}, each $\Fr_{z,\rho r}$ is a.s.\ determined by $h|_{B_{3\rho r}(z)}$, viewed modulo additive constant.  
Therefore, we can apply Lemma~\ref{lem-spatial-ind} with $h$ replaced by the the re-scaled field $h(r\cdot)$, which agrees in law with $h$ modulo additive constant, and $\mcl Z = r^{-1}(J \cap \Zr_r)$ to get the following. If $\rho$ is chosen to be sufficiently small (depending on $\pr$ and $C$, hence only on $\BB p,\llambda,$ and the laws of $D_h$ and $\wt D_h$), then 
\eqbn
\BB P\left[ \bigcup_{z\in \Zr_r \cap J} \Fr_{z,\rho r} \right] \geq 1 - \frac{1- \BB p }{10C}  , \quad \forall J \in \mcl J .
\eqen
By a union bound over all $J\in\mcl J$, we get that with probability at least $1-(1-\BB p)/10$, each $J \in \mcl J$ contains a point $z \in \Zr_r(\rho)$ such that $\Fr_{z,\rho r}$ occurs. By the defining property of $\mcl J$, this concludes the proof.
\end{proof}

We next deal with conditions~\ref{item-E-compare} and~\ref{item-E-narrow} in the definition of $\Er_r$, which amounts to citing some already-proven lemmas.

\begin{lem} \label{lem-E-short}
There exists $\Anar \in (0, \llambda (1-\Kann) \Aendpt \rho ]$ (where $\Aendpt$ is as in Lemma~\ref{lem-endpt-ball}), depending only on $\BB p,\llambda,$ and the laws of $D_h$ and $\wt D_h$, such that for each $r  > 0$, the probability of each of conditions~\ref{item-E-compare} and~\ref{item-E-narrow} in the definition of $\Er_r$ is at least $1-(1-\BB p)/10$. 
\end{lem}
\begin{proof}
The existence of $\Anar\in (0, \llambda \Aendpt \rho ]$ such that condition~~\ref{item-E-compare} in the definition of $\Er_r$ each occur with probability at least $1-(1-\BB p)/10$ follows from Lemma~\ref{lem-annulus-union}. 
By the local reverse H\"older continuity of $D_h$ w.r.t.\ the Euclidean metric~\cite[Proposition 3.8]{pfeffer-supercritical-lqg}, after possibly shrinking $\Anar$ we can arrange that condition~\ref{item-E-narrow} also occurs with probability at least $1-(1-\BB p)/10 $. 
\end{proof}

We henceforth fix $\Anar$ as in Lemma~\ref{lem-E-short}. 
We also let $\Aset \in (0, \min\{\llambda \Aep , \Anar\} )$ be chosen (in a manner depending only on $\BB p\llambda,$ and the laws of $D_h$ and $\wt D_h$) so that
\eqb \label{eqn-Aset}
(2\Aset)^{\geoExp } \Aaround  \leq \llambda \Anar^{ \xi(Q+3) }   .
\eqe
The particular choice of $\Aset$ from~\eqref{eqn-Aset} will be important in the proof of Lemma~\ref{lem-excursion-length} below.

\begin{lem} \label{lem-E-internal}
There exists $\Ainternal > 1/\Aset$, depending only on $\BB p,\llambda,$ and the laws of $D_h$ and $\wt D_h$, such that for each $r \in \rho^{-1} \mcl R_0$, the probability of condition~\ref{item-E-internal} in the definition of $\Er_r$ is at least $1-(1-\BB p)/10$. 
\end{lem}
\begin{proof} 
The set $\Ur_r$ has the topology of a Euclidean annulus and its boundary consists of two piecewise smooth Jordan loops. 
Write $\bdy^{\op{out}} \Ur_r$ for the outer boundary of $\Ur_r$, i.e., the outer of the two loops. 
If $r \in \rho^{-1}\mcl R_0$ is fixed, then as $\ep\rta 0$ the Euclidean Hausdorff distance between the following two sets tends to zero: $\bdy^{\op{out}} \Ur_r$ and $\bdy B_{\ep r}(\bdy^{\op{out}} \Ur_r) \cap \Ur_r$ (i.e., the intersection with $\Ur_r$ of the boundary of the Euclidean $\ep$-neighborhood of $\bdy^{\op{out}} \Ur_r$). 

Since we have already chosen $\rho$ in a manner depending only on $\BB p,\llambda,$ and the laws of $D_h$ and $\wt D_h$, the number of possible choices for $r^{-1} \Ur_r$ is at most a constant depending only on $\BB p,\llambda,$ and the laws of $D_h$ and $\wt D_h$.
By combining this with the preceding paragraph, we find that there exists $\ep  > 0$, depending only on $\BB p,\llambda,$ and the laws of $D_h$ and $\wt D_h$, such that for each $r \in \rho^{-1}\mcl R_0$, the Euclidean Hausdorff distance between $\bdy^{\op{out}} \Ur_r$ and $\bdy B_{\ep r}(\bdy^{\op{out}} \Ur_r) \cap \Ur_r$ is at most $\Aset r$.

By tightness across scales (in the form of Lemma~\ref{lem-set-tightness}) and the fact that there are only finitely many possibilities for $r^{-1} \Ur_r$, there exists $\Ainternal > 0$ such that for each $r\in\rho^{-1}\mcl R_0$, it holds with probability at least $1-(1-\BB p)/10$ that the following is true. 
There is a path $\Pi  \subset   B_{\ep r}(\bdy^{\op{out}} \Ur_r) \cap \Ur_r$ which disconnects $\bdy^{\op{out}} \Ur_r$ from $\bdy B_{\ep r}(\bdy^{\op{out}} \Ur_r) \cap \Ur_r$ and has $D_h$-length at most $\Ainternal r^{\xi Q} e^{\xi h_r(0)}$. 

The path $\Pi$ disconnects the inner and outer boundaries of $\Ur_r$, so the existence of $\Pi$ immediately implies~\eqref{eqn-E-internal-around}. Furthermore, by our choice of $\ep$, each point $x\in  \bdy^{\op{out}} \Ur_r $ lies at Euclidean distance at most $\Aset r$ from a point of $\bdy B_{\ep r}(\bdy^{\op{out}} \Ur_r) \cap \Ur_r$. Since $\Pi$ disconnects $\bdy^{\op{out}} \Ur_r$ from $\bdy B_{\ep r}(\bdy^{\op{out}} \Ur_r) \cap \Ur_r$, the line segment from $x$ to this point of $\bdy B_{\ep r}(\bdy^{\op{out}} \Ur_r) \cap \Ur_r$ intersects $\Pi$. 
Consequently, the Euclidean distance from $x$ to $\Pi$ is at most $\Aset r$.  
\end{proof}

We henceforth fix $\Ainternal$ as in Lemma~\ref{lem-E-internal} and define  
\eqb \label{eqn-Amax}
\Amax :=  \frac{1}{\xi} \max\left\{ \log \frac{  \Ainternal}{ \llambda \Anar^{ \xi(Q+3) } } , \log \frac{  \Ainternal}{ \llambda \Aacross} \right\} .
\eqe 
Recall from~\eqref{eqn-f-def} that $\Amax$ is the maximal value attained by $\fr_r$. We now treat the remaining two conditions in the definition of $\fr_r$. 

\begin{lem} \label{lem-E-rn}
There exists $\Asup \in (0,  \llambda /\Amax)$ and $\Arn > 1/\Asup$, depending only on $\BB p,\llambda,$ and the laws of $D_h$ and $\wt D_h$, such that for each $r\in\rho^{-1}\mcl R_0$, the probability of each of conditions~\ref{item-E-sup} and~\ref{item-E-rn} in the definition of $\Er_r$ is at least $1-(1-\BB p)/10$. 
\end{lem}
\begin{proof}
Since we have already chosen $\rho$ in a manner depending only on $\BB p,\llambda,$ and the laws of $D_h$ and $\wt D_h$, the number of possible choices for $r^{-1} \Ur_r$ is at most a constant depending only on $\BB p,\llambda,$ and the laws of $D_h$ and $\wt D_h$.
The set $\Ur_r$ has the topology of a Euclidean annulus and its boundary consists of two piecewise smooth Jordan loops. By the preceding sentence, the Euclidean length of each of the two boundary loops of $\Ur_r$ is at most a constant (depending only on $\BB p,\llambda,$ and the laws of $D_h$ and $\wt D_h$) times $r$. 
We can therefore apply Lemma~\ref{lem-geodesic-leb} with $M  =\Amax$ and the curve $\eta$ given by each of the two boundary loops of $\Ur_r$, parametrized by its Euclidean length. This shows that there exists $\Asup \in (0,  \llambda /\Amax) $ depending only on $\BB p,\llambda,$ and the laws of $D_h$ and $\wt D_h$ such that the event of condition~\ref{item-E-sup} in the definition of $\Er_r$ for the set $\bdy \Ur_r$ occurs with probability at least $1-(1-\BB p)/20$. 

By a union bound over at most a universal constant times $(\llambda \Aendpt \rho)^{-1}$ points $z\in \Zr_r$, after possibly decreasing $\Asup$ we can also arrange that with probability at least $1-(1-\BB p)/20$, the event of condition~\ref{item-E-sup} occurs for each of the circles $\bdy B_{\sr_{\rho r}}(\ur_{z,\rho r})$ and $\bdy B_{\sr_{\rho r}}(\vr_{z,\rho r})$ for $z\in\Zr_r$. Combining this with the preceding paragraph shows that condition~\ref{item-E-sup} has probability at least $1-(1-\BB p)/10$. 

The number of possible choices for the function $\fr_r(r\cdot)$ is at most a constant depending only on $\BB p,\llambda,$ and the laws of $D_h$ and $\wt D_h$.
By the conformal invariance of the Dirichlet inner product and the scale invariance of the law of $h$, viewed modulo additive constant,  
\eqbn
(h , \fr_r)_\nabla = (h(r\cdot) , \fr_r(r\cdot))_\nabla \eqD (h , \fr_r(r\cdot))_\nabla .
\eqen
Therefore, we can find $\Arn > 1/\Asup$ depending only on $\BB p,\llambda,$ and the laws of $D_h$ and $\wt D_h$ such that the probability of condition~\ref{item-E-rn} is at least $1-(1-\BB p)/10$.
\end{proof}

\begin{proof}[Proof of Proposition~\ref{prop-E-prob}]
Combine Lemmas~\ref{lem-E-reg}, \ref{lem-E-event}, \ref{lem-E-short}, \ref{lem-E-internal}, and~\ref{lem-E-rn}.
\end{proof}

We can also easily check the first two of the three hypotheses for $\Er_r$ from Section~\ref{sec-counting-setup}. 
 
\begin{prop} \label{prop-E-hyp0}
Let $r\in \rho^{-1} \mcl R_0$. On the event $\Er_r$, hypotheses~\ref{item-Ehyp-dist} and~\ref{item-Ehyp-rn} in Section~\ref{sec-counting-setup} hold for $\Er_{0,r} = \Er_r$ with
\eqb \label{eqn-Ehyp0-parameters}
\Cacross = \Aacross, \quad \Caround = \Aaround,\quad \Ctube = \Ainternal  ,
\eqe 
and an appropriate choice of $\Crn > 0$ depending only on the parameters from~\eqref{eqn-A-parameters} and~\eqref{eqn-rho-parameter} (hence only on $\BB p,\llambda,$ and the laws of $D_h$ and $\wt D_h$). 
That is, on $\Er_r$, the following is true. 
\begin{enumerate}[A.]
\item \label{item-Ehyp-dist0} We have 
\alb
D_h(\Vr_r , \bdy \BB A_{r,3r}(0) ) &\geq \Aacross r^{\xi Q} e^{\xi h_r(0)} ,\notag\\ 
 D_h(\text{around $\BB A_{3r,4r}(0)$} ) &\leq \Aaround r^{\xi Q} e^{\xi h_r(0)} ,\quad \text{and} \notag\\
  D_h(\text{around $\Ur_r$} ) &\leq \Ainternal r^{\xi Q} e^{\xi h_r(0)}  .
\ale
\item \label{item-Ehyp-rn0} There is a constant $\Crn > 0$, depending only on  the parameters from~\eqref{eqn-A-parameters} and~\eqref{eqn-rho-parameter}, such that the Radon-Nikodym derivative of the law of $h + \fr_r$ w.r.t.\ the law of $h$, with both distributions viewed modulo additive constant, is bounded above by $\Crn$ and below by $\Crn^{-1}$. 
\end{enumerate}
\end{prop}
\begin{proof}
We have $\Vr_r \subset \BB A_{1.5 r , 2.5 r}(0)$, so hypothesis~\ref{item-Ehyp-dist0} follows immediately from conditions~\ref{item-E-across}, \ref{item-E-around}, and~\ref{item-E-internal} in the definition of $\Er_r$. By a standard calculation for the GFF (see, e.g., the proof of~\cite[Proposition 3.4]{ig1}), the Radon-Nikodym derivative of the law of $h + \fr_r$ with respect to the law of $h$, with both distributions viewed modulo additive constant, is equal to
\eqbn
\exp\left(  (h , \fr_r)_\nabla - \frac12 (\fr_r ,\fr_r)_\nabla \right)   
\eqen
where $(\cdot,\cdot)_\nabla$ is the Dirichlet inner product. 
Since the number of possibilities for $\fr_r(r\cdot)$ is at most a constant depending only on $\BB p,\llambda,$ and the laws of $D_h$ and $\wt D_h$, we infer that $(\fr_r ,\fr_r)_\nabla$ is bounded above by a constant $C$ depending only on $\BB p,\llambda,$ and the laws of $D_h$ and $\wt D_h$ (c.f.\ the proof of Lemma~\ref{lem-E-rn}). 
By combining this with condition~\ref{item-E-rn} in the definition of $\Er_r$, we get that on $\Er_r$, we have the Radon-Nikodym derivative bounds
\eqbn
\exp\left( - \Arn - \frac12 C \right) \leq  \exp\left(  (h , \fr_r)_\nabla - \frac12 (\fr_r ,\fr_r)_\nabla \right) \leq \exp\left(\Arn \right) .
\eqen
This gives hypothesis~\ref{item-Ehyp-rn0} with $\Crn = \exp(\Arn + C/2)$. 
\end{proof}

Most of the rest of this section is devoted to checking hypothesis~\ref{item-Ehyp-inc} of Section~\ref{sec-counting-setup} for the events $\Er_r$.

\begin{prop} \label{prop-shortcut}
Fix $\Cmid  > \Cmid_0$. 
If $\llambda$ is chosen to be small enough (in a manner depending only on the laws of $D_h$ and $\wt D_h$) and the parameters from~\eqref{eqn-A-parameters} and~\eqref{eqn-rho-parameter} are chosen appropriately, subject to the constraints stated in the discussion around~\eqref{eqn-A-parameters} and~\eqref{eqn-rho-parameter}, then hypothesis~\ref{item-Ehyp-inc} holds for the events $\Er_r$ with
\eqb \label{eqn-Cinc-choice}
\Ctime := \frac{\Aacross^2}{4\Aaround} \quad \text{and} \quad \Cinc := \Anar^{\xi(Q+3)}  e^{-\xi \Amax} .
\eqe
That is, let $r\in\rho^{-1} \mcl R_0$ and assume that $\Er_r$ occurs. 
Let $P_r  $ be a $D_{h-\fr_r }$-geodesic between two points of $\BB C\setminus B_{4r}(0)$, parametrized by its $D_{h-\fr_r}$-length. Assume that there is a $(B_{4r}(0) , \Vr_r)$-excursion $(\tau' , \tau , \sigma, \sigma')$ for $P_r$ (Definition~\ref{def-excursion}) such that 
\eqb \label{eqn-shortcut-hyp}
D_h\left( P_r(\tau) , P_r(\sigma) ; B_{4r}(0) \right) \geq \Ctime r^{\xi Q} e^{\xi h_r(0)} . 
\eqe
There exist times $\tau \leq s < t \leq \sigma $ such that 
\eqb \label{eqn-shortcut}
t-s \geq  \Cinc  r^{\xi Q} e^{\xi h_r(0)}  \quad \text{and} \quad
\wt D_{h - \fr_r  }\left( P_r(s) , P_r(t) ; \BB A_{r,4r}(0)  \right) \leq \Cmid  (t-s) .
\eqe
\end{prop}

The proof of Proposition~\ref{prop-shortcut} will occupy Sections~\ref{sec-excursion} through~\ref{sec-shortcut}.

\subsection{Proof of Proposition~\ref{prop-objects-exist} assuming Proposition~\ref{prop-shortcut}}
\label{sec-objects-exist}

In this subsection, we will assume Proposition~\ref{prop-shortcut} and deduce Proposition~\ref{prop-objects-exist}. As explained in Section~\ref{sec-counting}, this gives us a proof of our main results modulo Proposition~\ref{prop-shortcut}. 

Assume that the parameters from~\eqref{eqn-A-parameters} and~\eqref{eqn-rho-parameter} are chosen so that the conclusions of Propositions~\ref{prop-E-prob} and~\ref{prop-shortcut} are satisfied. 
Let $\mcl R_0$ be as in~\eqref{eqn-initial-radii} and let $\mcl R := \rho^{-1} \mcl R_0$. 
Since $\mcl R_0 \subset \{8^{-k}\}_{k\in\BB N}$, we have $r'/r \geq 8$ whenever $r,r' \in \mcl R$ with $r' > r$, so~\eqref{eqn-admissible-radii} holds.  

The event $\Er_r$ is defined for each $r\in\mcl R$. 
By Lemma~\ref{lem-E-msrble}, the event $\Er_r$ is a.s.\ determined by $h|_{\ol{\BB A}_{r,4r}(0)}$, viewed modulo additive constant.
By Proposition~\ref{prop-E-prob}, $\BB P[\Er_r] \geq \BB p$ for each $r\in\mcl R$.
By the definitions in Section~\ref{sec-tube-def}, the sets $\Ur_r$ and $\Vr_r$ and the functions $\fr_r$ satisfy the requirements for $\Ur_{0,r}$, $\Vr_{0,r}$, and $\fr_{0,r}$ from Section~\ref{sec-counting-setup}, with the maximal value of $\fr_r$ given by $\Cmax = \Amax$. 
By Propositions~\ref{prop-E-hyp0} and~\ref{prop-shortcut}, the event $\Er_r$ satisfies hypotheses~\ref{item-Ehyp-dist}, \ref{item-Ehyp-rn}, and~\ref{item-Ehyp-inc} from Section~\ref{sec-counting-setup} for $z=0$, with the parameters $\Cacross,\Caround,\Ctube,\Crn,\Ctime,\Cinc$ depending on the parameters from~\eqref{eqn-A-parameters} and~\eqref{eqn-rho-parameter}. 

To check the needed parameter relation~\eqref{eqn-parameter-relation}, we observe that Proposition~\ref{prop-E-hyp0} gives $\Cacross = \Aacross$, $\Caround = \Aaround$, and $\Ctube = \Ainternal$. 
By~\eqref{eqn-A-parameters}, we immediately get $\Caround \geq \Cacross$. 
Furthermore, by~\eqref{eqn-Cinc-choice},  
\eqb  \label{eqn-parameter-relation0}
 \frac{ 2\Caround }{ \Cacross } \Ctime 
 = \frac{2\Aaround}{\Aacross} \times  \frac{\Aacross^2}{4\Aaround} 
 = \frac{\Aacross}{2}  .
\eqe
Moreover, by~\eqref{eqn-Amax},
\eqb \label{eqn-parameter-relation1}
 \Cacross - 4 e^{-\xi \Cmax} \Ctube  
 = \Aacross - 4 e^{-\xi \Amax} \Ainternal
 \geq \Aacross - 4 \llambda \Aacross > \frac{\Aacross}{2} .
\eqe
Combining~\eqref{eqn-parameter-relation0} and~\eqref{eqn-parameter-relation1} gives the second inequality in~\eqref{eqn-parameter-relation}. 
 
For $r\in \mcl R$ and $z\in\BB C$, we define $\Er_{z,r}$ to be the event $\Er_r$ of Section~\ref{sec-E} with the translated field $h(\cdot - z) - h_1(-z) \eqD h$ in place of $h$. 
We also define $\Ur_{z,r} := \Ur_r + z$, $\Vr_{z,r} := \Vr_r + z$, and $\fr_{z,r}(\cdot) := \fr_r(\cdot-z)$. 
By the translation invariance property of weak LQG metrics (Axiom~\reftranslate), the objects $\Er_{z,r}, \Ur_{z,r}, \Vr_{z,r}$, and $\fr_{z,r}$ satisfy the hypotheses of Section~\ref{sec-counting-setup}. 

It remains to prove the asserted lower bound for $\#\left( \mcl R \cap  [\ep^2 \BB r , \ep \BB r] \right) $ under the assumption that $\BB P[\wt G_{\BB r}(\wt\Kopt , \Cmid')] \geq \wt\Kopt$. 
By Proposition~\ref{prop-attained-good'} (applied with $\Cmid_0$ instead of $\Cmid$), the definition~\eqref{eqn-initial-radii}, of $\mcl R_0$, and our choice of $\Kann$ and $p_0$ immediately preceding~\eqref{eqn-initial-radii}, there exists $\Cmid' \in (\Clower,\Cupper)$ depending only on $\Cmid_0$ and the laws of $D_h$ and $\wt D_h$ such that the following is true. For each $\wt\Kopt > 0$ there exists $\ep_1  > 0$, depending only on $\BB p , \wt\Kopt$, and the laws of $D_h$ and $\wt D_h$, such that for each $\ep \in (0,\ep_1]$ and each $\BB r > 0$ such that $\BB P[\wt G_{\BB r}(\wt\Kopt , \Cmid')] \geq \wt\Kopt$, the cardinality of $\mcl R_0 \cap  [\ep^2 \BB r , \ep \BB r]$ is at least $\frac34 \log_8\ep^{-1}$.
This implies that if $\ep \in (0,\ep_1]$,
\alb
 \#\left( \mcl R \cap  [\ep^2 \BB r , \ep \BB r] \right) 
&= \#\left( \mcl R_0 \cap [\rho \ep^2 \BB r , \rho \ep \BB r] \right)  \quad \text{(since $\mcl R = \rho^{-1} \mcl R_0$)} \notag\\
&\geq \#\left( \mcl R_0 \cap [(\rho \ep)^2 \BB r , \rho \ep \BB r] \right)  - \#\left( \mcl R_0 \cap [(\rho\ep)^2 \BB r , \rho \ep^2 \BB r] \right) \notag\\
&\geq \#\left( \mcl R_0 \cap [(\rho \ep)^2 \BB r , \rho \ep \BB r] \right)  - \log_8 \rho^{-1} \quad \text{(since $\mcl R_0 \subset \{8^{-k}\}_{k\in\BB N}$)} \notag\\
&\geq \frac34 \log_8\ep^{-1}  - \log_8 \rho^{-1} \quad \quad \text{(since $ \rho \ep \leq \ep_1$)} \notag\\
&\geq \frac58 \log_8 \ep^{-1}  \quad \quad \text{(for small enough $\ep>0$, depending on $\rho$)} .
\ale
Thus, Proposition~\ref{prop-objects-exist} has been proven.
\qed

\subsection{Initial estimates for a geodesic excursion}
\label{sec-excursion}

To prove our main results, it remains to prove Proposition~\ref{prop-shortcut}. 
In the rest of this section, we will assume that we are in the setting of Proposition~\ref{prop-shortcut}, i.e., we assume that $\Er_r$ occurs, $P_r$ is a $D_{h-\fr_r}$-geodesic between two points of $\BB C\setminus B_{4r}(0)$, and $(\tau',\tau,\sigma,\sigma')$ is a $(B_{4r}(0) , \Vr_r)$-excursion satisfying~\eqref{eqn-shortcut-hyp}. 
It follows from Definition~\ref{def-excursion} that
\allb \label{eqn-excursion-def}
&P_r(\tau') , P_r(\sigma') \in \bdy B_{4r}(0),  \quad 
P_r(\tau) , P_r(\sigma) \in \bdy \Vr_r, \quad
P_r((\tau',\sigma')) \subset B_{4r}(0) , \notag\\
&\qquad \text{and} \quad P_r((\tau' , \tau)) \cup P_r((\sigma,\sigma')) \subset B_{4r}(0) \setminus \ol\Vr_r  .
\alle
We will prove~\eqref{eqn-shortcut} via a purely deterministic argument. 
We first check the following lemma, which will enable us to apply conditions~\ref{item-E-reg} and~\ref{item-E-sup} in the definition of $\Er_r$ to $P_r|_{[\tau',\sigma']}$. 

\begin{lem} \label{lem-excursion-annulus}
The path $P_r|_{[\tau' , \sigma']}$ is contained in $\ol{\BB A}_{r,4r}(0)$ and is a $D_{h-\fr_r}(\cdot,\cdot ;\ol{\BB A}_{r,4r}(0))$-geodesic between two points of $\bdy B_{4r}(0)$. 
\end{lem} 
\begin{proof}
We have $P_r|_{(\tau',\sigma')}\subset   B_{4r}(0)$ and $P_r(\tau') , P_r(\sigma') \in \bdy B_{4r}(0)$ by~\eqref{eqn-excursion-def}. We claim that $P_r$ does not enter $B_r(0)$. Assume the claim for the moment. Then $P_r|_{(\tau',\sigma')} \subset \ol{\BB A}_{r,4r}(0)$. Since $P_r$ is a $D_{h-\fr_r}$-geodesic, the $D_{h-\fr_r}$-length of $P_r|_{[\tau',\sigma']}$ is the same as the $D_{h-\fr_r  }$-distance between its endpoints. We conclude that $P_r|_{(\tau',\sigma')}$ is a path in $\ol{\BB A}_{r,4r}(0)$ whose $D_{h-\fr_r}$-length is the same as the $D_{h-\fr_r}$-distance between its endpoints, which is at most the $D_{h-\fr_r}(\cdot,\cdot ;\ol{\BB A}_{r,4r}(0))$-distance between its endpoints. 
Hence, $P_r|_{[\tau',\sigma']}$ is a $D_{h-\fr_r}(\cdot,\cdot ;\ol{\BB A}_{r,4r}(0))$-geodesic. 

It remains to show that $P_r$ does not enter $B_r(0)$. 
Assume by way of contradiction that $P_r\cap B_r(0) \not=\emptyset$. 
By condition~\ref{item-E-internal} (internal distance in $\Ur_r$) in the definition of $\Er_r$, there is a path $\Pi$ in $\Ur_r$ which disconnects the inner and outer boundaries of $\Ur_r$ such that
\eqbn
\op{len}\left( \Pi ; D_h \right) \leq 2 \Ainternal r^{\xi Q} e^{\xi h_r(0)} .
\eqen 
Let $\tau_0$ (resp.\ $\sigma_0$) be the first (resp.\ last) time that $P_r$ hits $\Pi$.

Since $P_r$ is a $D_{h-\fr_r}$-geodesic and $\fr_r \equiv \Amax$ on $\Ur_r$,
\eqb  \label{eqn-exc-ann-around}
\sigma_0 - \tau_0
= D_{h-\fr_r}(P_r(\tau_0) , P_r(\sigma_0)) 
\leq \op{len}\left( \Pi ; D_{h-\fr_r} \right) 
\leq 2  e^{-\xi \Amax} \Ainternal r^{\xi Q} e^{\xi h_r(0)} .
\eqe  
On the other hand, since $\Ur_r\subset \BB A_{1.5 r, 2.5 r}(0)$ and we are assuming that $P_r$ hits $B_r(0)$, it follows that $P_r$ must cross between the inner and outer boundaries of the annulus $\BB A_{r,1.5 r}(0)$ between time $\tau_0$ and time $\sigma_0$. Since $\fr_r\equiv 0$ on $\BB A_{r,1.5 r}(0)$ and by condition~\ref{item-E-across} (lower bound for distance across) in the definition of $\Er_r$, 
\eqb \label{eqn-exc-ann-across}
\sigma_0 - \tau_0
= \op{len}\left( P_r|_{[  \tau_0 , \sigma_0]} ; D_{h-\fr_r} \right) 
\geq D_h\left(\text{across $\BB A_{r,1.5 r}(0)$}\right)
\geq \Aacross r^{\xi Q} e^{\xi h_r(0)} .
\eqe
By our choice of $\Amax$ in~\eqref{eqn-Amax}, the right side of~\eqref{eqn-exc-ann-around} is smaller than the right side of~\eqref{eqn-exc-ann-across}, which supplies the desired contradiction.
\end{proof}

From Lemma~\ref{lem-E-sp}, we now obtain the following.

\begin{lem} \label{lem-excursion-sp}
We have 
\eqbn
|P_r(\sigma) - P_r(\tau)| \geq   \Asp r .
\eqen
\end{lem}
\begin{proof}
Due to Lemma~\ref{lem-excursion-annulus} and~\eqref{eqn-shortcut-hyp}, this follows from Lemma~\ref{lem-E-sp} applied with $V = \Vr_r$, $f=\fr_r$, and $P_f$ equal to the $D_{h-\fr_r}$-geodesic $P_r|_{[\tau',\sigma']}$. 
\end{proof}

By~\eqref{eqn-excursion-def}, we have $P_r^{-1}(\ol \Vr_r)\subset [\tau,\sigma]$. 
We will now establish an upper bound for the length of this time interval.

\begin{figure}[ht!]
\begin{center}
\includegraphics[width=.8\textwidth]{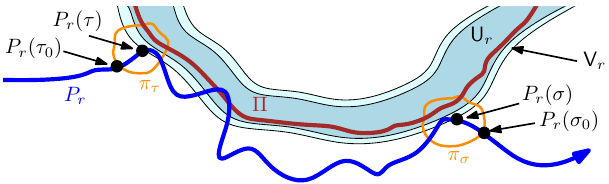} 
\caption{\label{fig-excursion-length} Illustration of the proof of Lemma~\ref{lem-excursion-length}. We obtain a path from a point of $P_r([\tau',\tau])$ to a point of $P_r([\sigma,\sigma'])$ whose $D_{h-\fr_r}$-length is at most the right side of~\eqref{eqn-excursion-length} by concatenating segments of $\pi_\tau, \Pi$, and $\pi_\sigma$. 
This implies an upper bound for $\sigma-\tau$ since $P_r$ is a $D_{h-\fr}$-geodesic. 
}
\end{center}
\end{figure}

\begin{lem} \label{lem-excursion-length}
We have
\eqb \label{eqn-excursion-length}
\sigma - \tau \leq  \frac12 \Anar^{\xi(Q+3)}  r^{\xi Q} e^{\xi h_r(0)} .
\eqe
\end{lem} 
\begin{proof}
See Figure~\ref{fig-excursion-length} for an illustration. 
Let $\Aset \in (0, \llambda \Aep]$ be as in~\eqref{eqn-Aset}. 
By Lemma~\ref{lem-excursion-annulus}, we can apply Lemma~\ref{lem-E-loop} (with $\ep = 2\Aset$) to the $D_h(\cdot,\cdot;\ol{\BB A}_{r,4r}(0))$-geodesic $P_r|_{[\tau',\sigma']}$ to get that there are paths $\pi_\tau \subset \BB A_{ 2\Aset r , (2\Aset)^{1/2} r}(P_r(\tau))$ and $\pi_\sigma \subset \BB A_{ 2\Aset r , (2\Aset)^{1/2} r}(P_r(\sigma))$ which disconnect the inner and outer boundaries of their respective annuli such that
\eqb \label{eqn-excursion-loop}
\max\left\{ \op{len}\left(\pi_\tau ; D_h \right) , \op{len}\left(\pi_\sigma ; D_h \right) \right\} 
\leq (2\Aset)^\geoExp \Aaround r^{\xi Q} e^{\xi h_r(0)} 
\leq  \llambda \Anar^{\xi(Q+3)}   r^{\xi Q} e^{\xi h_r(0)} ,
\eqe 
where the last inequality is by~\eqref{eqn-Aset}. 
Let $\tau_0$ be the last time before $\tau$ that $P_r$ hits $\pi_\tau$ and let $\sigma_0$ be the first time after $\sigma$ that $P_r$ hits $\pi_\sigma$. 
Then $\tau_0 \in [\tau',\tau]$ and $\sigma_0 \in [\sigma,\sigma']$. 

By condition~\ref{item-E-internal} (internal distance in $\Ur_r$) in the definition of $\Er_r$, there is a path $\Pi  \subset \Ur_r$ which disconnects the inner and outer boundaries of $\Ur_r$, has $D_h$-length at most $\Ainternal r^{\xi Q} e^{\xi h_r(0)}$, and such that each point of the outer boundary of $\Ur_r$ lies at Euclidean distance at most $\Aset r$ from $\Pi$. 
We have $P_r(\tau) \in \bdy \Vr_r = \bdy B_{\Asup r}(\Ur_r)$ and $P([\tau',\tau])$ is contained in the unbounded connected component of $\BB C\setminus \Ur_r$. Hence $P_r(\tau)$ lies at Euclidean distance at most $\Asup r$ from the outer boundary of $\Ur_r$. 
Therefore, the Euclidean distance from $P_r(\tau)$ to $\Pi$ is at most $(\Asup + \Aset) r \leq 2\Aset r$, where we use that $\Asup \leq \Aset$ by definition.

Since $\pi_\tau \subset \BB A_{ 2\Aset r , (2\Aset)^{1/2} r}(P_r(\tau))$ and $\pi_\tau$ disconnects the inner and outer boundaries of\\ $\BB A_{ 2\Aset r , (2\Aset)^{1/2} r}(P_r(\tau))$, it follows from the preceding paragraph that $\pi_\tau$ intersects $\Pi$. 
Similarly, $\pi_\sigma$ intersects $\Pi$.   
Hence the union of the loops $\Pi$, $\pi_\tau$, and $\pi_\sigma$ contains a path from $P_r(\tau_0)$ to $P_r(\sigma_0)$. 
Therefore,
\allb \label{eqn-loop-decomp}
\sigma - \tau
 \leq \sigma_0 - \tau_0   
 &= D_{h-\fr_r}\left(P_r(\tau_0) , P_r(\sigma_0) \right)   \notag\\
 &\leq \op{len}\left(\pi_\tau ; D_{h-\fr_r} \right) +  \op{len}\left(\pi_\sigma ; D_{h-\fr_r} \right)  + \op{len}\left( \Pi ; D_{h-\fr_r} \right) 
\alle

Let us now bound the right side of~\eqref{eqn-loop-decomp}.
Since $\fr_r$ is non-negative, the $D_{h-\fr_r}$-length of each of $\pi_\tau$ and $\pi_\sigma$ is at most the right side of~\eqref{eqn-excursion-loop}. 
Since $\fr_r \equiv \Amax$ on $\Ur_r$, 
\eqb 
\op{len}\left( \Pi ; D_{h-\fr_r} \right) 
= e^{-\xi \Amax} \op{len}\left( \Pi ; D_h \right) 
\leq e^{-\xi  \Amax}  \Ainternal r^{\xi Q} e^{\xi h_r(0)} 
\leq \llambda \Anar^{\xi(Q+3)} r^{\xi Q} e^{\xi h_r(0)} ,
\eqe
where the last inequality uses the definition~\eqref{eqn-Amax} of $\Amax$. Plugging these estimates into~\eqref{eqn-loop-decomp} gives
\eqb
\sigma - \tau \leq  3\llambda \Anar^{\xi(Q+3)} r^{\xi Q} e^{\xi h_r(0)}  ,
\eqe 
which is stronger than~\eqref{eqn-excursion-length}. 
\end{proof}

Combining Lemma~\ref{lem-excursion-length} with condition~\ref{item-E-narrow} (reverse H\"older continuity) in the definition of $\Er_r$ allows us to show that any segment of $P_r|_{[\tau,\sigma]}$ which is disjoint from $\Vr_r$ must have small Euclidean diameter.

\begin{lem} \label{lem-excursion-tube}
Each segment of $P_r|_{[\tau,\sigma]}$ which is disjoint from $\Vr_r$ has Euclidean diameter at most $\Anar r$. In particular, 
\eqbn
P_r([\tau,\sigma]) \subset B_{\Anar r  }(\Vr_r)  .
\eqen
\end{lem}
\begin{proof}
Suppose by way of contradiction that there is a segment $P_r|_{[t,s]}$ for times $\tau \leq t < s \leq \sigma$ which is disjoint from $\Vr_r$ and has Euclidean diameter larger than $\Anar r$. 
By~\eqref{eqn-excursion-def}, $P_r([\tau,\sigma])$ intersects $\ol{\Vr}_r$. Hence, by possibly replacing $P_r|_{[t,s]}$ by a segment of $P_r$ which travels from $\bdy \Vr_r$ to $\bdy B_{\Anar r}(\Vr_r)$, we can assume without loss of generality that $P_r([t,s])$ is contained in $B_{\Anar r}(\Vr_r)$, which in turn is contained in $\BB A_{1.5 r , 3 r}(0)$  by the definition of $\Vr_r$ (Section~\ref{sec-tube-def}).
By the reverse H\"older continuity condition~\ref{item-E-narrow} in the definition of $\Er_r$, the $D_h$-length of $P_r|_{[t,s]}$ is at least $\Anar^{\xi(Q+3)} r^{\xi Q} e^{\xi h_r(0)}$. 
Since $\fr_r$ is supported on $\Vr_r$, the $D_{h-\fr_r }$-length of $P_r|_{[t,s]}$ is equal to its $D_h$-length, so is also at least $\Anar^{\xi(Q+3)}  r^{\xi Q} e^{\xi h_r(0)}$. 
Since $P_r|_{[\tau,\sigma]}$ is a $D_{h-\fr_r }$-geodesic, we therefore have
\eqb
\sigma-\tau \geq s - t \geq \Anar^{\xi(Q+3)}  r^{\xi Q} e^{\xi h_r(0)} .
\eqe
This contradicts Lemma~\ref{lem-excursion-length}. 
\end{proof}

\subsection{Forcing a geodesic to enter balls centered at $\ur_{z,\rho r}$ and $\vr_{z,\rho r}$}
\label{sec-excursion-hit}

Recall the balls $B_{\sr_{\rho r}}(\ur_{z,\rho r})$ and $B_{\sr_{\rho r}}(\vr_{z,\rho r})$ appearing in the definition of the ``building block" event $\Fr_{z,\rho r}$ from Section~\ref{sec-block-event}. On $\Fr_{z,\rho r}$, there are points $u \in B_{\sr_{\rho r}}(\ur_{z,\rho r})$ and $v\in B_{\sr_{\rho r}}(\vr_{z,\rho r})$ which satisfy $\wt D_h(u,v) \leq \Cmid_0 D_h(u,v)$, plus several other conditions. In order to prove Proposition~\ref{prop-shortcut}, we want to force $P_r$ to get $D_{h-\fr_r}$-close to each of $u$ and $v$ for one of these pairs of points $u,v$, then apply the triangle inequality. To do this, the first step is to force $P_r$ to get close to the balls $ B_{\sr_{\rho r}}(\ur_{z,\rho r})$ and $v\in B_{\sr_{\rho r}}(\vr_{z,\rho r})$ for some $z\in\Zr_r$ such that $\Fr_{z,\rho r}$ occurs. We will carry out this step in this subsection.
Our goal is to prove the following lemma.

\begin{lem} \label{lem-excursion-hit}
Let $\Zr_r \subset \bdy B_{2r}(0)$ be as in~\eqref{eqn-test-pts}. 
There exists $z\in \Zr_r$ such that $\Fr_{z,\rho r}$ occurs and the following is true. 
Let $\sr_{\rho r}$, $\ur_{z,\rho r}$, and $\vr_{z,\rho r}$ be the radius and points as in the definition of $\Fr_{z,\rho r}$. 
There exist times $\tau \leq a < b  \leq \sigma$ which satisfy the following conditions: 
\eqb \label{eqn-excursion-hit}
P_r(a), P_r(b) \in \bdy B_{\sr_{\rho r} + \Asup r}(\ur_{z,\rho r})  , \quad 
 |P_r(b) - P_r(a)|  \geq \sr_{\rho r} / 8   ,\quad \text{and} \quad 
\eqe
\eqb \label{eqn-excursion-hit'}
P_r([a,b]) \subset B_{\sr_{\rho r} +  (\Asup + \Anar) r}(\ur_{z,\rho r}) \setminus \left(\Vr_r \setminus B_{\sr_{\rho r} + \Asup r}(\ur_{z,\rho r}) \right) .
\eqe
Moreover, the same is true with $\vr_{z,\rho r}$ in place of $\ur_{z,\rho r}$. 
\end{lem}

\begin{figure}[ht!]
\begin{center}
\includegraphics[width=1\textwidth]{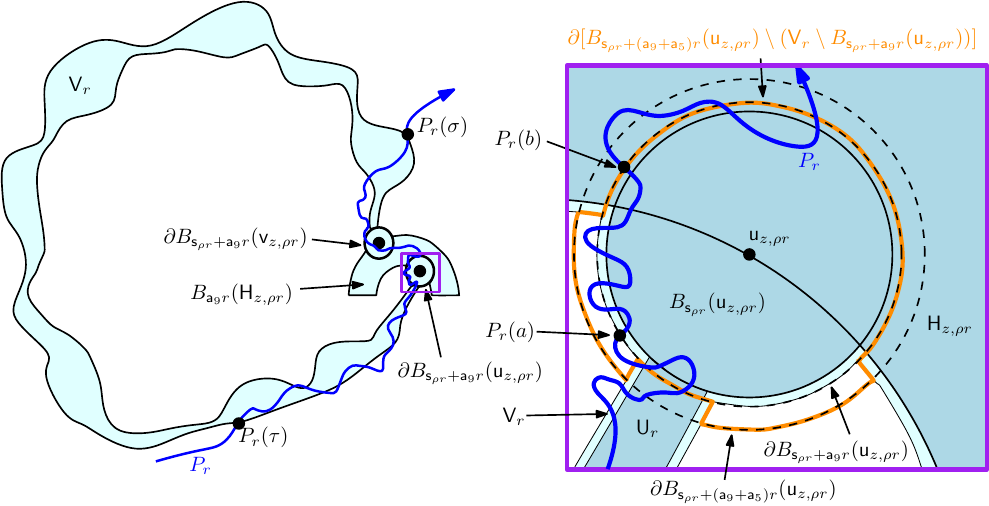} 
\caption{\label{fig-excursion-hit} Illustration of the statement of Lemma~\ref{lem-excursion-hit}. \textbf{Left:} the set $\Vr_r$ (light blue) and the path segment $P_r|_{[\tau,\sigma]}$. For simplicity, we have not drawn the details of $\Vr_r$ except in the $\Asup r$-neighborhood of the set $\Hr_{z,\rho r} \cup B_{\sr_{\rho r}}(\ur_{z,\rho r}) \cup B_{\sr_{\rho r}}(\vr_{z,\rho r})$. The set $\Ur_r$ is not shown. \textbf{Right:} the left panel zoomed in on the purple box. We have shown a subset of $\Ur_r$ (light blue) and a subset of $\Vr_r \setminus \Ur_r$ (lighter blue). By~\eqref{eqn-excursion-hit'}, the path segment $P_r|_{[a,b]}$ is required to stay region outlined in orange. 
}
\end{center}
\end{figure}

See Figure~\ref{fig-excursion-hit} for an illustration of the statement of Lemma~\ref{lem-excursion-hit}. Before discussing the proof, we make some comments on the statement. 
The ball $B_{\sr_{\rho r} + \Asup r}(\ur_{z,\rho r})$ appearing in Lemma~\ref{lem-excursion-hit} is significant because, by the definition of $\Vr_{\rho r}$ in~\eqref{eqn-V-def}, this is the largest Euclidean ball centered at $\ur_{z,\rho r}$ which is contained in $\Vr_{\rho r}$. 
The significance of the ball $B_{\sr_{\rho r} +  (\Asup + \Anar) r}(\ur_{z,\rho r})$ appearing in~\eqref{eqn-excursion-hit'} is that by Lemma~\ref{lem-excursion-tube}, the path $P_r|_{[\tau,\sigma]}$ cannot exit the $\Anar r$-neigbhborhood of $\Vr_r$. 
We note that $\sr_{\rho r} \geq \Aendpt \rho r$ (Lemma~\ref{lem-endpt-ball}), which is much larger than $\Anar r$ (Lemma~\ref{lem-E-short}), which in turn is much larger than $\Asup r$ (recall the discussion surrounding~\eqref{eqn-A-parameters}). So the balls in~\eqref{eqn-excursion-hit} and~\eqref{eqn-excursion-hit'} are only slightly larger than $B_{\sr_{\rho r}}(\ur_{z,\rho r})$. 

Lemma~\ref{lem-excursion-hit} will be a consequence of Lemmas~\ref{lem-excursion-sp} and~\ref{lem-excursion-tube} (which give a lower bound for $|P_r(\tau) - P_r(\sigma)|$ and an upper bound for the Euclidean diameter of any segment of $P_r$ which is disjoint from $\Vr_r$), condition~\ref{item-E-event} in the definition of $\Er_r$ (which gives lots of points $z\in \Zr_r$ for which $\Fr_{z,\rho r}$ occurs), and some basic geometric arguments based on the definition of $\Ur_r$ from Section~\ref{sec-tube-def}. 

We encourage the reader to look at Figure~\ref{fig-excursion-hit-proof} while reading the proof. 
Let us start by explaining why we can apply condition~\ref{item-E-event} in the definition of $\Er_r$. 
We have $P_r(\tau) , P_r(\sigma) \in \bdy \Vr_r$ by~\eqref{eqn-excursion-def} and $|P_r(\sigma) - P_r(\tau)| \geq  \Asp r$ by Lemma~\ref{lem-excursion-sp}. Moreover, by the definition of $\Vr_r$ in Section~\ref{sec-tube-def}, the Euclidean distance from each point of $\Vr_r$ to $\bdy B_{2r}(0)$ is at most $100 \rho r$, which by our choice of $\rho$ in Lemma~\ref{lem-E-event} is at most $100\llambda  \Asp r \leq \Asp r / 100$. 
Therefore, the set $\bdy B_{2r}(0) \setminus [\ol B_{100\rho r}(P_r(\tau)) \cup \ol B_{100\rho r}(P_r(\sigma))]$ consists of two disjoint connected arcs of $\bdy B_{2r}(0)$ which each have Euclidean length at least $ \Asp r /2  $. Let $J$ (resp.\ $J'$) be the one of these two arcs which goes in the counterclockwise (resp.\ clockwise) direction from $\ol B_{100\rho r}(P_r(\tau)) $ to $\ol B_{100\rho r}(P_r(\sigma))$.

\newcommand{\Bhit}{{\mathbf B}}

By condition~\ref{item-E-event} in the definition of $\Er_r$, there exist $z\in J \cap \Zr_r$ and $z' \in J' \cap \Zr_r$ such that $\Fr_{z,\rho r}$ and $\Fr_{z' , \rho r}$ both occur. To lighten notation, we write
\eqbn
\ur := \ur_{z,\rho r} , \quad \vr := \vr_{z,\rho r} , \quad \ur' :=  \ur_{z',\rho r} , \quad \vr' := \vr_{z',\rho r}
\eqen
and
\eqb
\Bhit(w) := B_{\sr_{\rho r} + \Asup r}(w),\quad \forall w \in \{\ur , \vr , \ur' , \vr' \} .
\eqe
The set $\Vr_r\setminus [ \Bhit(\ur) \cup \Bhit(\vr) \cup \Bhit(\vr') \cup \Bhit(\ur')]$ consists of exactly four connected components which each lie at Euclidean distance at least $\sr_{\rho r}/4$ from each other. We call these connected components $V^\tau , V^\sigma , O,O'$. We can choose the labeling so that with $\Hr_{z,\rho r}$ and $\Hr_{z',\rho r}$ the half-annuli as in the definitions of $\Fr_{z,\rho r}$ and $\Fr_{z',\rho r}$, 
\eqb
P_r(\tau) \in \bdy V^\tau ,\quad P_r(\sigma) \in \bdy V^\sigma, \quad O\subset B_{\Asup r}(\Hr_{z,\rho r}) \quad \text{and} \quad O' \subset B_{\Asup r}(\Hr_{z',\rho r}) .
\eqe 
We note that the boundary of each of these connected components intersects exactly two of the boundaries of the balls $\Bhit(w)$ for $w\in \{\ur,\vr,\ur',\vr'\}$.
See Figure~\ref{fig-excursion-hit-proof}, left, for an illustration.

\begin{figure}[ht!]
\begin{center}
\includegraphics[width=1\textwidth]{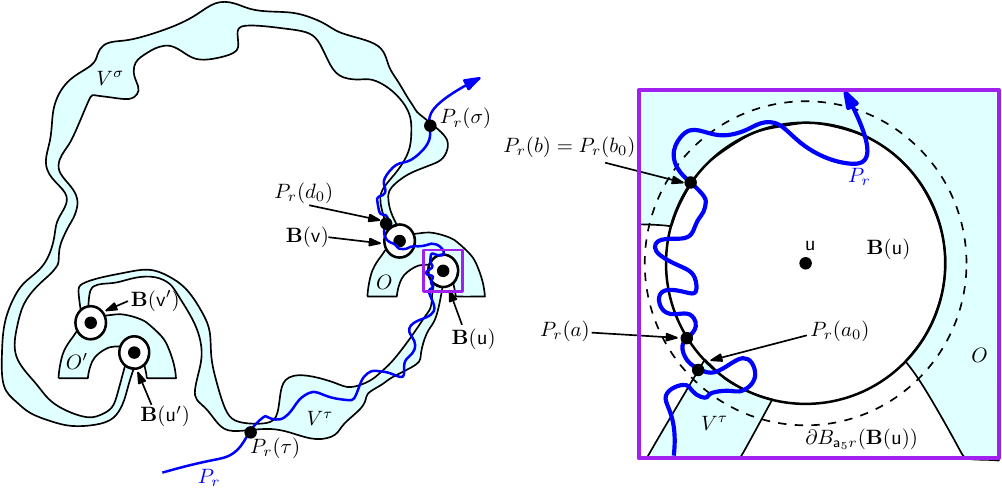} 
\caption{\label{fig-excursion-hit-proof} \textbf{Left:} the connected components $V^\tau,V^\sigma,O,O'$ of $\Vr_r\setminus [ \Bhit(\ur) \cup \Bhit(\vr) \cup \Bhit(\vr') \cup \Bhit(\ur')]$ and the point $P_r(d_0)$ where $P_r$ first enters $V^\sigma$. For simplicity we have drawn $V^\tau$ and $V^\sigma$ as ``blobs" rather than showing the details of how $\Vr_r$ is defined in Section~\ref{sec-tube-def} (c.f.\ Figure~\ref{fig-U-def}).
\textbf{Right:} A zoomed-in view in the purple box from the left figure. Here $b_0$ is the first time that $P_r$ hits $O$, $a_0$ is the last time before $b_0$ at which $P_r$ exits $V^\tau$, $a$ is the first time after $a_0$ at which $P_r$ hits $\bdy \Bhit(\ur)$, and $b$ is the last time before $b_0$ at which $P_r$ exits $\Bhit(\ur)$. In the figure, we have $a\not= a_0$ and $b =b_0$, but any combination of $a = a_0$ or $a\not=a_0$ and/or $b = b_0$ or $b\not=b_0$ is possible. 
}
\end{center}
\end{figure}

Let $d_0$ be the first time that $P_r|_{[\tau ,\sigma ]}$ hits $\ol V^\sigma$ (this time is well-defined since we know that $P_r(\sigma) \in \bdy V^\sigma$).
By Lemma~\ref{lem-excursion-tube}, each segment of $P_r|_{[\tau,\sigma]}$ which is disjoint from $ \Vr_r $ has Euclidean diameter at most $\Anar r$, which is much smaller than $\sr_{\rho r}/4$.
It follows that either $P_r(d_0) \in B_{\Anar r}( \Bhit(\vr) ) \cap \ol V^\sigma $ or $P_r(d_0) \in B_{\Anar r}( \Bhit(\ur') ) \cap \ol V^\sigma $. 
For simplicity, we henceforth assume that 
\eqb \label{eqn-d0-hit}
P_r(d_0) \in  B_{\Anar r}( \Bhit(\vr) ) \cap \ol V^\sigma ;
\eqe
the other case can be treated in an identical manner.

Most of the rest of the proof will focus on what happens near $\Bhit(\ur)$. See Figure~\ref{fig-excursion-hit-proof}, right, for an illustration. 
We first define a time $b_0$ such that $P_r(b_0)$ will be Euclidean-close to the point $P_r(b)$ from Lemma~\ref{lem-excursion-hit}.

\begin{lem} \label{lem-first-ball}
Let $b_0$ be the smallest $t\geq \tau$ for which $P_r(b_0) \in \ol O$. 
Then $b_0 < d_0$ and $P_r(b_0) \in   \bdy O \cap B_{\Anar r}(\Bhit(\ur))$.
\end{lem}
\begin{proof}
The path $P_r|_{[\tau, d_0]}$ travels from $\bdy V^\tau$ to $  B_{\Anar r}(\Bhit(\vr)) \cap \ol V^\sigma$ and does not enter $V^\sigma$. 
The set $\Vr_r\setminus ( V^\sigma  \cup O )$ has two connected components which lie at Euclidean distance at least $(1-\Kann) \rho r / 2 \geq   \Anar r$ (recall our choice of $\Anar$ from Lemma~\ref{lem-E-short}) from each other, one of which contains $ \Bhit(\vr)$ and the other of which contains $V^\tau$. 
By Lemma~\ref{lem-excursion-tube}, $P_r|_{[\tau,d_0]}$ cannot travel Euclidean distance more than $\Anar r$ without hitting $\Vr_r$.
Hence $P_r|_{[\tau,d_0]}$ must hit $O$ before it hits $\ol V^\sigma$.  
Therefore, $b_0  < d_0$  and $P_r(b_0) \in \bdy O$. 
Furthermore, since $\Bhit(\vr)$ and $V^\tau$ are contained in different connected components of $\Vr_r\setminus ( V^\sigma  \cup O )$ and by the definitions of $b_0$ and $d_0$, we have $P_r([\tau,b_0]) \cap (V^\sigma \cup O\cup \Bhit(\vr)) =\emptyset$. 

We need to show that $P_r(b_0) \in  B_{\Anar r}(\Bhit(\ur))$. Indeed, since $P_r|_{[\tau,b_0]}$ cannot hit $V^\sigma \cup O \cup \Bhit(\vr)$ and cannot travel Euclidean distance more than $\Anar r$ outside of $\Vr_r$, it must be the case that 
\eqbn
P_r(b_0) \in B_{\Anar r}\left( V^\tau \cup O' \cup  \Bhit(\ur) \cup \Bhit(\ur') \cup \Bhit(\vr')  \right).
\eqen
The sets $V^\tau$, $O'$, $\Bhit(\ur')$, and $\Bhit(\vr')$ each lie at Euclidean distance larger than $\Anar r$ from $O$, so since $P_r(b_0) \in \bdy O$ we must have $P_r(b_0) \in B_{\Anar r}(\Bhit(\ur))$. 
\end{proof}

Next, we define a time $a_0$ such that $P_r(a_0)$ will be Euclidean-close to the point $P_r(a)$ from Lemma~\ref{lem-excursion-hit}.

\begin{lem}\label{lem-entry-time}
Let $a_0$ be the last time $t$ before $b_0$ for which $P_r(t) \in \ol V^\tau$. 
Then
\eqb \label{eqn-entry-time} 
|P_r(b_0) - P_r(a_0)| \geq \sr_{\rho r} / 4 \quad \text{and} \quad 
P_r([a_0,b_0]) \subset B_{\Anar r}(\Bhit(\ur)) \setminus \left( \Vr_r\setminus \Bhit(\ur) \right).
\eqe
\end{lem}
\begin{proof}
Since $P_r(b_0) \in \bdy O$ and the Euclidean distance from $V^\tau$ to $O$ is at least $\sr_{\rho r}/4$, we immediately obtain that $|P_r(b_0) - P_r(a_0)| \geq \sr_{\rho r} / 4$. It remains to prove the inclusion in~\eqref{eqn-entry-time}.

By definition, the set $P_r([a_0,b_0])$ is disjoint from $  V^\tau \cup O$. 
Furthermore, by Lemma~\ref{lem-excursion-tube}, each segment of $P_r|_{[a_0,b_0]}$ which is not contained in $\Vr_r$ has Euclidean diameter at most $\Anar r$. 
Therefore,
\eqb \label{eqn-excursion-hit-contain'}
P_r([a_0,b_0]) \subset B_{\Anar r}\left( V^\sigma \cup O' \cup   \Bhit(\ur) \cup \Bhit(\vr) \cup \Bhit(\vr') \cup \Bhit(\ur')  \right) .
\eqe
The set on the right side of~\eqref{eqn-excursion-hit-contain'} has two connected components, one of which is equal to $B_{\Anar r}(\Bhit(\ur))$ and the other of which contains the other five sets in the union. Since $P_r(b_0) \in B_{\Anar r}(\Bhit(\ur))$ (Lemma~\ref{lem-first-ball}), we get that $P_r([a_0,b_0]) \subset B_{\Anar r}(\Bhit(\ur))$ and $P_r([a_0,b_0])$ is disjoint from  $ V^\sigma \cup O' \cup   \Bhit(\vr)  \cup \Bhit(\vr') \cup \Bhit(\ur') $. Since we already know that $P_r([a_0,b_0])$ is disjoint from $V^\tau \cup O$, we obtain the inclusion in~\eqref{eqn-entry-time}.
\end{proof}

\begin{proof}[Proof of Lemma~\ref{lem-excursion-hit}]
Let $a$ be the first time $t\geq a_0$ such that $P_r(t) \in \ol\Bhit(\ur)$ and let $b$ be the last time $t\leq b_0$ such that $P_r(t) \in \ol\Bhit(\ur)$.  
Note that we might have $a = a_0$ and/or $b = b_0$ (see Figure~\ref{fig-excursion-hit-proof}, right). 
By~\eqref{eqn-entry-time}, $P_r|_{[a_0,b_0]}$ cannot hit $\Vr_r \setminus \Bhit(\ur)$. By this and Lemma~\ref{lem-excursion-tube}, $P_r|_{[a_0,b_0]}$ cannot travel Euclidean distance more than $\Anar r$ without entering $\Bhit(\ur)$. Consequently, the times $a$ and $b$ are well-defined and
\eqb \label{eqn-excursion-hit-travel}
\max\left\{ |P_r(a) - P_r(a_0)|  , |P_r(b) - P_r(b_0)| \right\} \leq \Anar r .
\eqe
By~\eqref{eqn-entry-time} and~\eqref{eqn-excursion-hit-travel} and the triangle inequality,
\eqb
|P_r(b) - P_r(a)| \geq \sr_{\rho r}/4 - 2\Anar r ,
\eqe
which is at least $\sr_{\rho r}/8$ since $\sr_{\rho r} \geq \Aendpt \rho r \geq \llambda \Anar$ (by our choice of $\sr_{\rho r}$ in Lemma~\ref{lem-endpt-ball} and our choice of $\Anar$ in Lemma~\ref{lem-E-short}). 
By the definitions of $a$ and $b$, we have $P_r(a),P_r(b) \in \bdy \Bhit(\ur)$. 
Since $a,b\in [a_0,b_0]$ and by Lemma~\ref{lem-entry-time}, we also have the inclusion~\eqref{eqn-excursion-hit'}. 

This gives the lemma statement for $\ur = \ur_{z,\rho r}$. The statement with $\vr = \vr_{z,\rho r}$ in place of $\ur$ follows by repeating Lemma~\ref{lem-entry-time} and the argument above with $d_0$ used in place of $b_0$.
\end{proof}

\subsection{Forcing a geodesic to get close to $u$ and $v$}
\label{sec-endpt-close}

\newcommand{\Bin}{\hyperref[eqn-ball-abbrv]{{\mathbf B}^{\Ur}}}
\newcommand{\Bsup}{\hyperref[eqn-ball-abbrv]{{\mathbf B}^{\Vr}}} 
\newcommand{\Bout}{\hyperref[eqn-ball-abbrv]{{\mathbf B}^{\mathrm{out}}}}
\newcommand{\Imid}{\hyperref[eqn-Imid-def]{{\mathbf I}^{\Vr}}}
\newcommand{\Iout}{\hyperref[eqn-Imid-def]{{\mathbf I}^{\mathrm{out}}}}

We henceforth fix $z \in \Zr_r$ and times $a,b\in [\tau,\sigma]$ as in Lemma~\ref{lem-excursion-hit}. 
We also let $u$ and $v$ be as in the definition of $\Fr_{z,\rho r}$, so that $u \in B_{\sr_{\rho r}/2}(\ur_{z,\rho r} )$, $v \in B_{\sr_{\rho r}/2}(\vr_{z,\rho r} )$, and $\wt D_h(u,v) \leq \Cmid_0 D_h(u,v)$. Recall that we are trying to force the path $P_r$ to get $D_{h-\fr_r}$-close to each of $u$ and $v$. 

Lemma~\ref{lem-excursion-hit} tells us that $P_r$ gets Euclidean-close to each of $u$ and $v$, but this is not sufficient for our purposes since in the supercritical case $D_h$ is not continuous with respect to the Euclidean metric. In order to ensure that $P_r$ gets $D_{h-\fr_r}$-close to each of $u$ and $v$, we will need a careful argument involving several of the conditions in the definitions of $\Fr_{z,\rho r}$ and $\Er_r$.   
The main result of this subsection is the following lemma. 

\newcommand{\Cclose}{{C}}

\begin{lem}  \label{lem-endpt-close}
There is a constant $\Cclose > 0$, depending only on $\xi$, such that the following is true. 
Almost surely, there exists $t \in [\tau,\sigma]$ such that 
\eqb \label{eqn-endpt-close-eucl}
P_r(t) \in B_{\sr_{\rho r} + (3\Anar + \Asup) r}(\ur_{z,\rho r}) \quad\text{and} 
\eqe
\eqb \label{eqn-endpt-close} 
D_{h-\fr_r}\left( P_r(t) , u ; \BB A_{r,4r}(0) \right) \leq \Cclose \llambda e^{-\xi \Amax} \wt D_h(u,v)  .
\eqe
Moreover, the same is true with $v$ and $\vr_{z,\rho r}$ in place of $u$ and $\ur_{z,\rho r}$. 
\end{lem}

We will eventually choose $\llambda$ to be much smaller than $1/\Cclose$, so that the right side of~\eqref{eqn-endpt-close} is much smaller than $e^{-\xi \Amax} \wt D_h(u,v)$. 
We will only prove Lemma~\ref{lem-endpt-close} for $u$; the statement with $v$ in place of $u$ is proven in an identical manner. 

\subsubsection{Setup}

\begin{figure}[ht!]
\begin{center}
\includegraphics[width=.6\textwidth]{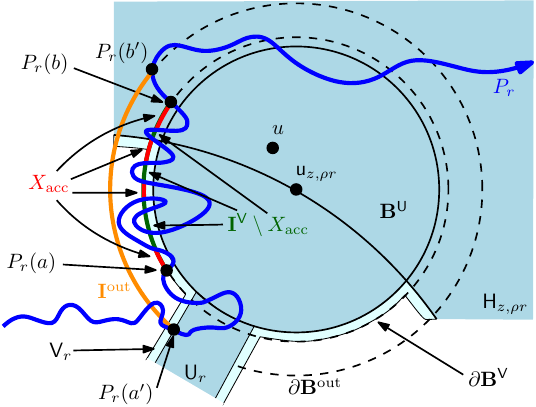} 
\caption{\label{fig-endpt-close-setup} Illustration of several of the objects involved in Section~\ref{sec-endpt-close}. 
The arc $\Imid \subset \bdy \Bsup$  is the union of the red set $X_{\mathrm{acc}}$ consisting of points which are accessible from $\Iout$ in $\ol\Bout\setminus (\Bsup \cup P_r([a',b'])$ and the green set $\Imid \setminus X_{\op{acc}}$. Note that a connected component of $\Imid \setminus X_{\op{acc}}$ can contain points of $P_r([a',b'])$ in its interior (relative to $\Imid$).
}
\end{center}
\end{figure}

Before proceeding with the proof of Lemma~\ref{lem-endpt-close}, we introduce some notation. See Figure~\ref{fig-endpt-close-setup} for an illustration. 
We define the Euclidean balls
\allb \label{eqn-ball-abbrv}
\Bin   := B_{\sr_{\rho r}  }(\ur_{z,\rho r}) ,\quad 
\Bsup := B_{\sr_{\rho r} + \Asup r}(\ur_{z,\rho r}),  
\quad \text{and} \quad \Bout := B_{\sr_{\rho r} + (3\Anar + \Asup) r}(\ur_{z,\rho r}) .
\alle
The reason why we care about $\Bin$ and $\Bsup$ is that by the definitions of $\Ur_r$ and $\Vr_r$, the ball $\Bin$ (resp.\ $\Bsup$) is the largest Euclidean ball centered at $\ur_{z,\rho r}$ which is contained in $\Ur_r$ (resp.\ $\Vr_r$). 
The reason why we care about $\Bout$ is that by Lemma~\ref{lem-excursion-hit}, $P_r|_{[a,b]}$ cannot exit the ball $B_{\sr_{\rho r} + ( \Anar + \Asup) r}(\ur_{z,\rho r}) \subset \Bout$. We need $\Bout$ to have a slightly larger radius than $\sr_{\rho r} + (\Anar  +\Asup ) r $ for the purposes of Lemma~\ref{lem-arc-loop} below. 

We also define
\eqb \label{eqn-ab'}
a' := \sup\{t \leq a : P_r(t) \in \bdy \Bout \} \quad \text{and} \quad b' := \inf\{t\geq b : P_r(t) \in \bdy \Bout\} .
\eqe
Then $a' < a < b < b'$. Furthermore, Lemma~\ref{lem-excursion-hit} implies that $P_r([a,b]) \subset \Bout$, so the definitions of $a'$ and $b'$ show that $P_r([a',b']) \subset \ol\Bout$ and $P_r((a',b')) \subset \Bout$. 

Recall that the point $u$ appearing in Lemma~\ref{lem-endpt-close} is contained in $\Bin$. 
Lemma~\ref{lem-endpt-close} holds vacuously if $u \in P_r([a',b'])$, so we can assume without loss of generality that
\eqb
u \notin P_r([a',b']) .
\eqe
The set  $\bdy \Bout \setminus \{ P_r(a') , P_r(b') \}$ consists of two disjoint arcs. 
Since $P_r|_{[a',b']}$ is a simple curve in $\ol\Bout$ which intersects $\bdy \Bout$ only at its endpoints, it follows that exactly one of these two arcs is disconnected from $u$  by $P_r|_{[a',b']}$. We assume without loss of generality that the clockwise arc of $\bdy\Bout$ from $P_r(a')$ to $P_r(b')$ is disconnected from $u$. 
Let
\allb \label{eqn-Imid-def}
\Iout &:= \left\{ \text{open clockwise arc of $\bdy\Bout$ from $P_r(a')$ to $P_r(b')$} \right\} \notag\\
\Imid &:= \left\{ \text{open clockwise arc of $\bdy\Bsup$ from $P_r(a)$ to $P_r(b)$} \right\} .
\alle
Note that $P_r([a',b'])$ disconnects $\Iout$ from $u$ in $\Bout$, but does not necessarily disconnect $\Imid$ from $u$ in $\Bout$. 
By Lemma~\ref{lem-excursion-hit}, we have $|P_r(b) - P_r(a)| \geq \sr_{\rho r}/8$, so the Euclidean length of $\Imid$ satisfies
\eqb \label{eqn-Imid}
|\Imid| \geq \sr_{\rho r}/8  . 
\eqe

We say that $x\in \Imid$ is \emph{accessible from $\Iout$ in $\ol\Bout\setminus (\Bsup \cup P_r([a',b']))$} if there is a path in $\ol\Bout\setminus (\Bsup \cup P_r([a',b']) ) $ from $x$ to a point of $\Iout$. Let
\eqb \label{eqn-dc-set}
X_{\op{acc}} :=  \left\{ x\in \Imid : \text{$x$ is accessible from $\Iout$ in $\ol\Bout\setminus (\Bsup \cup P_r([a',b']))$} \right\} .
\eqe
See Figure~\ref{fig-endpt-close-setup} for an illustration. 
One of the main reasons why we are interested in the set $X_{\op{acc}}$ is the following elementary topological fact.

\begin{figure}[ht!]
\begin{center}
\includegraphics[width=.6\textwidth]{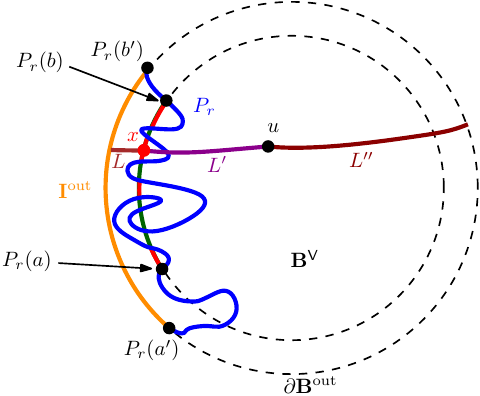} 
\caption{\label{fig-dc-join} Illustration of the proof of Lemma~\ref{lem-dc-join}. The path $P_r|_{[a',b']}$ must intersect $L\cup L' \cup L''$. By our choices of $L$ and $L''$, it must in fact intersect $L'$. 
}
\end{center}
\end{figure}

\begin{lem} \label{lem-dc-join}
If $x\in  X_{\op{acc}}$, then every path in $\ol\Bout$ from $u$ to $x$ hits $P_r([a',b'])$. 
\end{lem} 
\begin{proof}
See Figure~\ref{fig-dc-join} for an illustration. 
Recall that $\Iout$ and $\bdy\Bout\setminus \ol\Iout$ are the open clockwise and counterclockwise arcs of $\bdy\Bout$ from $P_r(a')$ to $P_r(b')$, respectively. 
By the assumption made just before~\eqref{eqn-Imid}, $P_r|_{[a',b']}$ disconnects $\Iout$ but not $\bdy\Bout\setminus \ol\Iout$ from $u$ in $\Bout$. 

By the definition~\eqref{eqn-dc-set} of $X_{\op{acc}}$, there is a path $L$ from $x$ to a point of $\Iout$ in $\ol\Bout$ which is disjoint from $\Bsup \cup P_r([a',b'])$. 
Furthermore, since $P_r|_{[a',b']}$ does not disconnect $\bdy\Bout\setminus \ol\Iout$ from $u$ in $\ol\Bout$, there is a path from $u$ to a point of $\bdy\Bout\setminus \ol\Iout$ in $\ol\Bout$ which is disjoint from $P_r([a',b'])$. 

Now consider a path $L'$ in $\ol\Bout$ from $u$ to $x$.  
The union $L\cup L' \cup L''$ contains a path in $\ol\Bout$ joining the two arcs of $\bdy\Bout \setminus \{P_r(a') , P_r(b')\}$. 
Since $P_r|_{[a',b']}$ is a path in $\ol\Bout$, topological considerations show that $P_r|_{[a',b']}$ must hit $L\cup L' \cup L''$.
Since $P_r|_{[a',b']}$ cannot hit $L $ or $L''$ by definition, we get that $P_r|_{[a',b']}$ must hit $L'$.  
\end{proof}

For $x\in \Imid$, we define
\eqb \label{eqn-in-pt}
x' := \frac{ \sr_{\rho r}  }{ \sr_{\rho r} +  \Asup r} (x - \ur_{z,\rho r}) + \ur_{z,\rho r} \in \bdy \Bin ,
\eqe
so that $x'$ is the unique point of $\bdy \Bin$ which lies on the line segment from the center point $\ur_{z,\rho r}$ to $x$. 
We also let
\eqb \label{eqn-dist-set}
X_{\op{dist}} := 
\left\{x \in \Imid \, : \,  D_h\left( x' , u  ; \ol\Bin \right)  \leq \llambda \wt D_h(u,v)  \right\} .
\eqe  
By condition~\ref{item-endpt-ball-leb'} in the definition of $\Fr_{z,\rho r}$, the set $\{x' \in \bdy \Bin : x \notin X_{\op{dist}}\}$ has one-dimensional Lebesgue measure at most $(\llambda/2)\sr_{\rho r}$. By scaling, we therefore have
\eqb \label{eqn-use-leb}
|X_{\op{dist}}| \geq |\Imid| - \llambda  \sr_{\rho r}    .
\eqe

\subsubsection{Proof of Lemma~\ref{lem-endpt-close} assuming that the accessible set is not too small}

The following lemma tells us that the conclusion of Lemma~\ref{lem-endpt-close} is satisfied provided $X_{\op{acc}}$ is not too small relative to $\sr_{\rho r}$. 
 
\begin{lem} \label{lem-dc-set}
If the one-dimensional Lebesgue measure of $X_{\op{acc}}$ satisfies $|X_{\op{acc}}| > 3\llambda \sr_{\rho r}  $, then there is a time $t \in [a',b'] \subset [\tau,\sigma]$ such that
\eqb \label{eqn-dc-set-conclusion}
 D_{h-\fr_r}\left( P_r(t) , u  ; \ol\Bin \right)  \leq 2\llambda e^{-\xi \Amax} \wt D_h(u,v) .
\eqe
\end{lem}

We note that Lemma~\ref{lem-dc-set} implies that if $|X_{\op{acc}}| > 3\llambda \sr_{\rho r}$, then the conclusion of Lemma~\ref{lem-endpt-close} holds with $\Cclose = 2$. This is because $P_r([a',b'])\subset \Bout$ and $\ol\Bin \subset \BB A_{r,4r}(0)$. 

The idea of the proof of Lemma~\ref{lem-dc-set} is that if $|X_{\op{acc}}| > 3\llambda \sr_{\rho r}  $, then by~\eqref{eqn-use-leb} there must be a point $x \in X_{\op{acc}} \cap X_{\op{dist}}$. 
By Lemma~\ref{lem-dc-join}, every path in $\Bout$ from $u$ to $x$ must hit $P_r([a',b'])$. 
We then want to use the definition~\eqref{eqn-dist-set} of $X_{\op{dist}}$ to upper-bound the $D_{h-\fr_r}$-distance from $u$ to the intersection point. 
There is a minor technicality arising from the fact that~\eqref{eqn-dist-set} only gives a bound for the distance from $u$ to $x' \in\bdy \Bin$, rather than from $u$ to $x$. 
To deal with this technicality, we will use condition~\ref{item-E-sup} (intersections of geodesics with a small neighborhood of the boundary) in the definition of $\Er_r$ to say that there are not very many points $x\in\Imid$ for which $P_r$ hits the segment $[x,x']$. 

\begin{proof}[Proof of Lemma~\ref{lem-dc-set}]
Define $x' \in \bdy\Bin$ for $x\in \Imid$ as in~\eqref{eqn-in-pt}. Let 
\eqb
Y := \left\{ x\in X_{\op{acc}} : P_r([a',b']) \cap [x,x'] \not=\emptyset \right\} .
\eqe 
If $x\in Y$, then $x'$ lies at Euclidean distance at most $ \Asup r$ from $P_r([a',b'])$. By condition~\ref{item-E-sup} in the definition of $\Er_r$ (in particular, we use the last sentence of the condition),  
the one-dimensional Lebesgue measure of the set $\{ x' \in \bdy \Bin : x \in Y\}$ is at most $\llambda \Aendpt \rho  r  \leq \llambda \sr_{\rho r  } $. 
By scaling, we get that the one-dimensional Lebesgue measure of $Y$ is at most $2 \llambda \sr_{\rho r}  $. 

Hence, if $|X_{\op{acc}}| > 3\llambda \sr_{\rho r}$, then $|X_{\op{acc}} \setminus Y|  > \llambda \sr_{\rho r}$. 
By~\eqref{eqn-use-leb}, this implies that the one-dimensional Lebesgue measure of $X_{\op{dist}} \cap (X_{\op{acc}} \setminus Y) $ is positive, so there exists $x\in X_{\op{dist}} \cap (X_{\op{acc}} \setminus Y)$.   

Since $x\in X_{\op{dist}}$, the definition~\eqref{eqn-dist-set} implies that there is a path $L$ in $\ol\Bin$ from $u$ to $x'$ such that
\eqbn
\op{len}\left( L ;  D_h\right) \leq 2\llambda \wt D_h(u,v) .
\eqen
The union of $L$ and $[x,x']$ gives a path in $\ol\Bsup$ from $u$ to $x$.
Since $x\in X_{\op{acc}}$, Lemma~\ref{lem-dc-join} implies that the path $P_r|_{[a',b']}$ must hit $L \cup [x,x']$. 
Since $x\notin Y$, the path $P_r|_{[a',b']}$ does not hit $[x,x']$. 

Therefore, $P_r|_{[a',b']}$ must hit $L$. Since $L \subset \ol\Bin$ is a path started from $u$ of $D_h$-length at most $2\llambda \wt D_h(u,v)$, we get that
\eqb \label{eqn-dc-set-dist}
 D_h\left( P_r(t) , u  ; \ol\Bin \right)  \leq 2\llambda \wt D_h(u,v) ,
\eqe
where $t\in [a',b']$ is chosen so that $P_r(t) \in L$. 

Since $\fr_r$ attains its maximum value $\Amax$ at each point of $\ol\Ur_r \supset \ol\Bin$, we infer from Weyl scaling (Axiom~\ref{item-metric-f}) that 
\eqbn
D_{h-\fr_r}(P_r(t) , u ; \ol\Bin) = e^{-\xi \Amax} D_h\left(P_r(t) , u ; \ol\Bin\right) .
\eqen
Combining this with~\eqref{eqn-dc-set-dist} gives~\eqref{eqn-dc-set-conclusion}. 
\end{proof}
 
\subsubsection{The set of arcs of $\Imid\setminus \ol X_{\op{acc}}$}

In light of Lemma~\ref{lem-dc-set}, for the rest of the proof of Lemma~\ref{lem-endpt-close} we can assume that
\eqb \label{eqn-dc-set-assume}
|X_{\op{acc}}| \leq 3\llambda \sr_{\rho r}    .
\eqe 
Intuitively, we do not expect~\eqref{eqn-dc-set-assume} to be the typical situation since it implies that $P_r([a',b'])$ disconnects ``most" points of $\Imid$ from $\Iout$ (recall~\eqref{eqn-dc-set}). This, in turn, means that a large portion of $P_r( [a',b'])$ is outside of $\Vr_r$. This is unexpected since $P_r$ is a $D_{h-\fr_r}$-geodesic and $\fr_r$ is non-negative and supported on $\Vr_r$, so $P_r|_{[a',b']}$ should want to spend most of its time in $\Vr_r$.  
However, we are not able to easily rule out~\eqref{eqn-dc-set-assume}. 
We note that Lemma~\ref{lem-excursion-tube} does not rule out~\eqref{eqn-dc-set-assume} since it could be that $P_r|_{[a',b']}$ has many small excursions outside of $\Vr_r$, each of Euclidean diameter at most $\Anar r$. 

Hence, we need to prove Lemma~\ref{lem-endpt-close} under the assumption~\eqref{eqn-dc-set-assume}. 
This will require a finer analysis of the structure of the set $X_{\op{acc}}$. 

The set $\Imid \setminus \ol X_{\op{acc}}$ is a countable union of disjoint open arcs of $\Imid$. 
Let $\mcl I$ be the set of all such arcs and for $I\in\mcl I$, write $|I|$ for its Euclidean length (equivalently, its one-dimensional Lebesgue measure). 
The elements of $\mcl I$ are the green arcs in Figure~\ref{fig-endpt-close-setup}. 

We now give an outline of the proof of Lemma~\ref{lem-endpt-close} subject to the assumption~\eqref{eqn-dc-set-assume}. 
As a consequence of~\eqref{eqn-dc-set-assume}, we get that ``most" points of $\Imid$ are contained in $\Imid \setminus \ol X_{\op{acc}}$, so $\sum_{I\in\mcl I} |I|$ is close to $|\Imid|$ (Lemma~\ref{lem-complement-size}). 
From this and~\eqref{eqn-use-leb}, we see that ``most" of the arcs $I\in \mcl I$ intersect $X_{\op{dist}}$ (Lemma~\ref{lem-good-arcs}). 
From condition~\ref{item-E-compare} (comparison of distances in small annuli) in the definition of $\Er_r$ (applied with $\delta = |I|/r$) and a geometric argument, we get the following. If $I\in\mcl I$ and $y_I$ is one of the endpoints of $I$, then there is a loop in $\BB A_{2|I| , 3|I|}(y_I)$ which disconnects the inner and outer boundaries and whose $ D_h$-length (hence also its $D_{h-\fr_r}$-length) is bounded above by $(|I|/r)^{-1/4}$ times (roughly speaking) the $D_h$-length of the segment of $P_r$ joining the endpoints of $I$.
By concatenating this loop with a path in $\ol\Bin$ from $u$ to $x'$, for a point $x'\in I\cap X_{\op{dist}}$, we obtain an upper bound for $D_{h-\fr_r}(u , P_r([a',b']))$ in terms of $|I|$ and the $D_h$-length of the segment of $P_r$ joining the endpoints of $I$ (Lemma~\ref{lem-arc-loop}). 
We will then use a pidgeonhole argument to say that there exists $I\in\mcl I$ for which this last quantity is much smaller than $e^{-\xi \Amax} \wt D_h(u,v)$. 

Let us now give the details. We start with a lower bound for the sum of the Lebesgue measures of the arcs in $\mcl I$. 

\begin{lem} \label{lem-complement-size}
The total one-dimensional Lebesgue measure of the arcs in $\mcl I $ satisfies
\eqb \label{eqn-complement-size}
\sum_{I\in\mcl I } |I| = |\Imid \setminus \ol X_{\op{acc}}| \geq |\Imid| - 3\llambda \sr_{\rho r}  .
\eqe
\end{lem}
\begin{proof}  
We first claim that each point of $\ol X_{\op{acc}} \setminus X_{\op{acc}}$ belongs to $P_r([a',b']) \cap \Imid$. Indeed, suppose $x\in \ol X_{\op{acc}}  $ and $x\notin P_r([a',b'])$. We need to show that $x\in X_{\op{acc}}$. Since $P_r([a',b'])$ is a Euclidean-closed set, $x$ lies at positive Euclidean distance from $P_r([a',b'])$. Since $x\in \ol X_{\op{acc}}$, there exists $y\in X_{\op{acc}}$ such that the arc of $\Imid$ between $x$ and $y$ is disjoint from $P_r([a',b'])$. By the definition of $X_{\op{acc}}$~\eqref{eqn-dc-set}, there is a path from a point of $\Iout$ to $y$ which is contained in $\ol\Bout \setminus (\Bsup \cup P_r([a',b']))$. The union of this path and the arc of $\Imid$ between $x$ and $y$ gives a path from $\Iout$ to $x$ which is contained in $\ol\Bout \setminus (\Bsup \cup P_r([a',b']))$.

By, e.g., Lemma~\ref{lem-geodesic-leb} (applied to the unit-speed parametrization of the circle $\bdy \Bsup$), a.s.\ the set $P_r([a',b']) \cap \Imid$ has zero one-dimensional Lebesgue measure. 
By this, the previous paragraph, and our assumption~\eqref{eqn-dc-set-assume}, 
\eqbn
\sum_{I\in\mcl I } |I| = |\Imid \setminus \ol X_{\op{acc}}| = |\Imid \setminus  X_{\op{acc}}| \geq |\Imid| - 3\llambda \sr_{\rho r}  .
\eqen
\end{proof}

We will also need the following elementary topological fact.

\begin{figure}[ht!]
\begin{center}
\includegraphics[width=.6\textwidth]{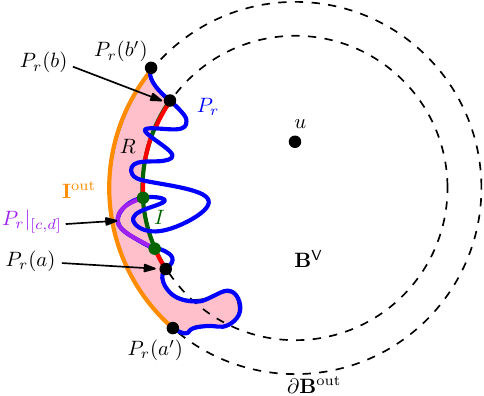} 
\caption{\label{fig-arc-path} Illustration of the proof of Lemma~\ref{lem-arc-path}. The region $R$ is shown in pink and the desired segment $P_r|_{[c,d]}$ of $P$ is shown in purple. 
}
\end{center}
\end{figure}

\begin{lem} \label{lem-arc-path}
For each $I\in \mcl I$, there is a segment of $P_r|_{[a,b]}$ joining the two endpoints of $I$ which is contained in $\Bout \setminus \Bsup$. 
\end{lem}
\begin{proof}
See Figure~\ref{fig-arc-path} for an illustration. 
Let $R \subset \Bout \setminus \ol\Bsup$ be the open region bounded by $\Iout$, $\Imid$, and the segments $P_r([a',a])$ and $P_r([b,b'])$. Then $R$ has the topology of the open unit disk and $I\subset \bdy R$. 
By the definition~\eqref{eqn-dc-set} of $X_{\op{acc}}$ and since $I\subset \Imid \setminus X_{\op{acc}}$, there is no path in $\ol R$ from $I$ to $\Iout$ which is disjoint from $P_r([a',b'])$. 
Hence $P_r([a',b'])$ disconnects $I$ from $\Iout$ in $R$.  

Since $P_r([a',a]) \cup P_r([b,b'])\subset \bdy R$ and $P_r([a,b]) \cap \bdy\Bout = \emptyset$, the set $P_r([a',b']) \cap R$ consists of countably many disjoint segments of $P_r|_{[a,b]}$ with endpoints in $\Imid$.  
Since $P_r$ is continuous, these segments accumulate only at points of $\Imid$. 
Since $I$ is connected and $P_r([a',b'])$ disconnects $I$ from $\Iout$ in $R$, there are times $c,d \in [a,b]$ with $c < d$ such that $P_r(c), P_r(d) \in \Imid$, $P_r((c,d)) \subset R$, and $P_r([c,d])$ disconnects $I$ from $\Iout$ in $R$. 

Let $\wh I$ be the set of points of $\Imid$ which are disconnected from $\Iout$ in $R$ by $P_r([c,d])$ (not including the endpoints of $P_r([c,d])$). Equivalently, $\wh I$ is the segment of $\Imid$ between $P_r(c)$ and $P_r(d)$. Then $\wh I$ is a connected open arc of $\Imid$ which contains $I$. Moreover, every path from $\wh I$ to $\Iout$ in $\ol\Bout \setminus \Bsup$ either hits $P_r([c,d])$ or exits $R$ (in which case it must intersect either $P_r([a',a])$ or $P_r([b,b'])$). Hence no such path can be disjoint from $P_r([a',b'])$. So, by the definition~\eqref{eqn-dc-set} of $X_{\op{acc}}$, we have $\wh I\subset \Imid\setminus X_{\op{acc}} $. Since $\wh I$ is an open arc of $\Imid$, also $\wh I \subset \Imid \setminus \ol X_{\op{acc}}$. Since $I$ is a connected component of $\Imid \setminus \ol X_{\op{acc}}$, it follows that $\wh I = I$. 
\end{proof}

\subsubsection{Regularity of arcs in $\mcl I$}

We will next record some bounds for the sizes of the individual arcs in $\mcl I$, starting with an upper bound. 

\begin{lem} \label{lem-micro-arcs}
For each $I \in \mcl I$, we have $|I| \leq \Anar r$.
\end{lem}
\begin{proof}
By Lemma~\ref{lem-arc-path}, for each $I\in\mcl I$ there is a segment of $P_r|_{[a,b]}$ joining the endpoints of $I$ which is contained in $\Bout\setminus \Bsup$. 
By Lemma~\ref{lem-excursion-hit}, $P_r|_{[a,b]}$ does not hit $\Vr_r \setminus \Bsup$, so this segment of $P_r|_{[a,b]}$ is disjoint from $\Vr_r$.
The Euclidean diameter of this segment is at least $|I|$. 
By Lemma~\ref{lem-excursion-tube}, the Euclidean diameter of the segment is at most $\Anar r$, so we get $|I| \leq \Anar r$, as required.
\end{proof}

We do not have a uniform lower bound for the sizes of the arcs in $\mcl I$. But, using condition~\ref{item-E-sup} (intersections of geodesics with a small neighborhood of the boundary) in the definition of $\Er_r$, we can say that the small arcs make a negligible contribution to the total one-dimensional Lebesgue measure of $\mcl I$.

\begin{lem} \label{lem-macro-arcs} 
Define the set of small arcs
\eqb \label{eqn-small-arcs}
\mcl I_{\op{small}} := \left\{I \in\mcl I : |I| \leq \Asup r\right\} .
\eqe
Then
\eqb \label{eqn-macro-arcs}
\sum_{I\in \mcl I_{\op{small}} } |I|   \leq   2 \llambda \sr_{\rho r} .
\eqe
\end{lem}
\begin{proof} 
By Lemma~\ref{lem-arc-path}, for each $I\in\mcl I $ the endpoints of $I$ are hit by $P_r|_{[a',b']}$. 
Hence the Euclidean distance from each point of $I$ to $P_r([a',b'])$ is at most $|I|$. 
In particular, if $I \in \mcl I_{\op{small}}$, then the Euclidean distance from each point of $I$ to $P_r([a',b'])$ is at most $\Asup r$.
This implies that the Euclidean distance from $P_r([a',b'])$ to each point of the arc $I' := \{x' : x\in I\} \subset \bdy\Bin$ is at most $2\Asup r$, where here we use the notation~\eqref{eqn-in-pt}.

The arcs $I' $ for $I\in   \mcl I_{\op{small}}$ are disjoint and we have $|I'| \geq |I| / 2$. 
Therefore, the one-dimensional Lebesgue measure of the set of points $x' \in \bdy \Bin$ which lie at Euclidean distance at most $2\Asup r$ from $P_r([a',b'])$ is at least
\eqbn  
\frac12 \sum_{I\in \mcl I_{\op{small}} } |I| .
\eqen 
By condition~\ref{item-E-sup} in the definition of $\Er_r$ (in particular, we use the last sentence of the condition), the one-dimensional Lebesgue measure of the set of $x' \in \bdy \Bin$ which lie at Euclidean distance at most $2\Asup r$ from $P_r([a',b'])$ is at most $\llambda \Aendpt \rho r $, so
\eqb  \label{eqn-small-arc-sum}
\frac12 \sum_{I\in \mcl I_{\op{small}} } |I|  \leq \llambda \Aendpt \rho r \leq \llambda \sr_{\rho r}  ,
\eqe 
where the last inequality comes from the definition of $\sr_{\rho r}$ (recall Lemma~\ref{lem-endpt-ball}).
\end{proof}

We will now consider a certain ``good" subset of $\mcl I$, and show that the arcs in this subset cover most of $\Imid$. 
Let
\eqb \label{eqn-good-arc-set}
\mcl I^* := \left\{I\in \mcl I : |I| \geq \Asup r \:\text{and}\: I\cap X_{\op{dist}} \not=\emptyset \right\}.
\eqe

\begin{lem} \label{lem-good-arcs}
The total one-dimensional Lebesgue measure of the arcs in $\mcl I^*$ satisfies 
\eqb \label{eqn-good-arcs} 
\sum_{I\in\mcl I^*} |I| \geq |\Imid| - 6 \llambda \sr_{\rho r} .
\eqe
\end{lem}
\begin{proof}
Let $\mcl I_{\op{small}}$ be as in~\eqref{eqn-small-arcs}. 
We can write $\Imid$ as the disjoint union of $X_{\op{acc}}$, the arcs in $\mcl I_{\op{small}}$, and the arcs in $\mcl I$ with $|I| \geq \Asup r$. 
By the definition~\eqref{eqn-good-arc-set} of $\mcl I^*$, 
\eqb \label{eqn-good-arcs-union}
X_{\op{dist}} \subset \ol X_{\op{acc}} \cup \bigcup_{I\in \mcl I_{\op{small}}} I \cup  \bigcup_{I\in\mcl I^*} I .
\eqe 
We therefore have the following string of inequalities:
\allb \label{eqn-good-arcs-dist}
|\Imid| - \llambda  \sr_{\rho r} 
&\leq |X_{\op{dist}} | \quad \text{(by~\eqref{eqn-use-leb})} \notag\\
&\leq |\ol X_{\op{acc}}| +  \sum_{I\in \mcl I_{\op{small}} } |I|    +    \sum_{I\in\mcl I^*} |I|  \quad \text{(by~\eqref{eqn-good-arcs-union})} \notag\\
&\leq  3 \llambda \sr_{\rho r} + 2 \llambda \sr_{\rho r} +  \sum_{I\in\mcl I^*} |I| \quad \text{(by Lemmas~\ref{lem-complement-size} and~\ref{lem-macro-arcs})} .
\alle
Re-arranging gives~\eqref{eqn-good-arcs}.
\end{proof}

\subsubsection{Building a path from a point of $P_r$ to $u$}

The following lemma is the main quantitative estimate needed for the proof of Lemma~\ref{lem-endpt-close}.

\begin{lem} \label{lem-arc-loop} 
Let $I\in\mcl I^*$ and let $y_I$ be the initial endpoint of $I$. 
There are times $a' < s_I < t_I < b'$ such that 
\eqb \label{eqn-arc-loop0}
P_r([s_I , t_I]) \subset B_{3 |I| }(y_I) ,\quad t_I - s_I \geq \left( \frac{|I|}{4r} \right)^{\xi(Q+2)+1/4} r^{\xi Q} e^{\xi h_r(0)} , \quad \text{and}
\eqe
\eqb \label{eqn-arc-loop} 
D_{h-\fr_r}\left( P_r(t_I) , u ; \BB A_{r,4r}(0) \right) 
\leq 2\llambda e^{-\xi \Amax} \wt D_h(u,v)  + 2(|I|/r)^{-1/4} (t_I - s_I)  .
\eqe
\end{lem}

\begin{figure}[ht!]
\begin{center}
\includegraphics[width=1\textwidth]{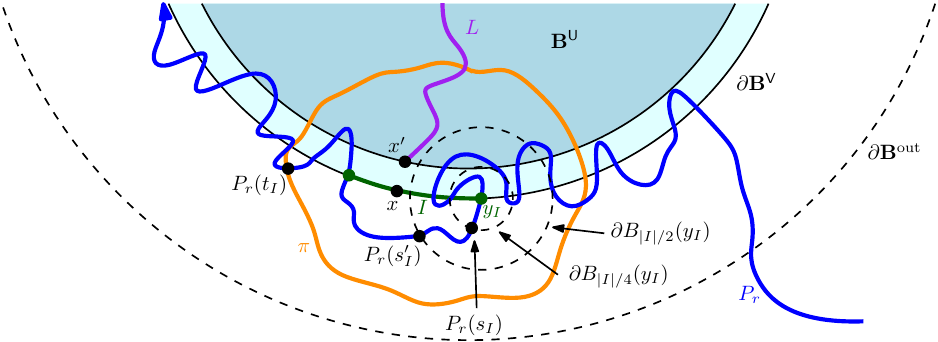} 
\caption{\label{fig-arc-loop} Illustration of the proof of Lemma~\ref{lem-arc-loop}. The orange loop $\pi$ has $D_h$-length at most $2 (|I|/r)^{-1/4} D_h\left(\text{across $\BB A_{ |I| /4 , |I| /2}(y_I)$}\right) $, and is provided by condition~\ref{item-E-compare} (comparison of distance in small annuli) in the definition of $\Er_r$. The point $x$ belongs to $I\cap X_{\op{dist}}$. The purple path $L$ goes from $u$ (not pictured) to $x'$, has $D_h$-length at most $2\llambda \wt D_h(u,v)$, and is provided by the definition~\eqref{eqn-dist-set} of $X_{\op{dist}}$. The bound~\eqref{eqn-arc-loop} is obtained by concatenating a segment of $\pi$ with a segment of $L$, then bounding $D_h\left(\text{across $\BB A_{ |I| /4 , |I| /2}(y_I)$}\right)$ in terms of $t_I - s_I$. 
}
\end{center}
\end{figure}

We will eventually deduce Lemma~\ref{lem-endpt-close} from Lemma~\ref{lem-arc-loop} by showing that there exists an $I\in \mcl I^*$ for which $2|I|^{-1/4} (t_I - s_I) $ is much smaller than $ e^{-\xi \Amax} \wt D_h(u,v)$.

\begin{proof}[Proof of Lemma~\ref{lem-arc-loop}]
See Figure~\ref{fig-arc-loop} for an illustration. 
Throughout the proof we fix $I\in\mcl I^*$. 
\medskip

\noindent\textit{Step 1: definition of $s_I$ and $t_I$.}
By Lemma~\ref{lem-micro-arcs} we have $|I| \leq \Anar r$.
Hence we can apply condition~\ref{item-E-compare} (comparison of distances in small annuli) in the definition of $\Er_r$ with $\delta= |I|/r$ to get that there is a path $\pi \subset \BB A_{2 |I|   , 3 |I| }(y_I)$ such that
\eqb \label{eqn-arc-loop-compare}
\op{len}\left( \pi ; D_h\right) 
\leq 2 (|I|/r)^{-1/4} D_h\left(\text{across $\BB A_{ |I| /4 , |I| /2}(y_I)$}\right)  .
\eqe 

We have $y_I \in \bdy \Bsup$ and $P_r(b') \in \bdy \Bout$. The Euclidean distance from $\bdy \Bout$ to $\bdy \Bsup$ is $  3\Anar   r \geq  3 |I|$. 
Therefore, the path $P_r$ must hit both $\bdy B_{|I| /4}(y_I)$ and $\pi$ between the (unique) time when it hits $y_I$ and the time $b'$.  
Let $s_I$ (resp.\ $t_I$) be the first time that $P_r$ hits $\bdy B_{|I| /4}(y_I)$ (resp.\ $\pi$) after the time when it hits $y_I$. 
Then $a' < s_I < t_I < b'$ and (since $P_r$ cannot travel from $y_I$ to $\bdy B_{3 |I| }(y_I)$ without hitting $\pi$),
\eqbn
P_r([s_I , t_I]) \subset B_{3 |I| }(y_I) . 
\eqen
We will check the other conditions in the lemma statement for this choice of $t_I$ and $s_I$. 
\medskip

\noindent\textit{Step 2: upper-bound for $D_{h-\fr_r}\left( P_r(t_I) , u ; \BB A_{r,4r}(0) \right)$ in terms of $D_h\left(\text{across $\BB A_{|I| /4 ,  |I| /2}(y_I)$}\right)$.}
By the definition~\eqref{eqn-good-arc-set} of $\mcl I^*$, there exists $x \in I \cap X_{\op{dist}}$. By the definition~\eqref{eqn-dist-set} of $X_{\op{dist}}$, if we let $x' \in \bdy \Bin$ be the point corresponding to $x$ as in~\eqref{eqn-in-pt}, then there is a path $L$ from $u$ to $x'$ in $\ol\Bin$ such that
\eqbn
\op{len}\left( L ;  D_h\right) \leq 2\llambda \wt D_h(u,v) .
\eqen 
Since $L$ is contained in $\ol\Bin$, which is contained in $\ol\Ur_r$, and $\fr_r\equiv \Amax$ on $\ol\Ur_r$,
\eqb \label{eqn-arc-loop-path}
\op{len}\left( L ;   D_{h-\fr_r} \right) \leq 2\llambda e^{-\xi \Amax}\wt D_h(u,v) .
\eqe

The definition~\eqref{eqn-good-arc-set} of $\mcl I^*$ gives $|I| \geq \Asup r$, so
\eqbn
|x' - y_I| \leq |I|  + |x-x'| = |I| + \Asup r \leq 2 |I| .
\eqen
Since $\pi \subset \BB A_{2|I| ,3|I|}(y_I)$, it follows that $\pi$ intersects $L$ and (since $3|I| \leq 3\Anar r$) also $\pi\subset \Bout$.  
Since $P_r(t_I) \in \pi$, the path $\pi\cup L$ contains a path from $u$ to $P_r(t_I)$. 
We have $\pi \cup L \subset \Bout \subset \BB A_{r,4r}(0)$. 
By~\eqref{eqn-arc-loop-compare} (and the fact that $\fr_r$ is non-negative) and~\eqref{eqn-arc-loop-path}, 
\allb \label{eqn-arc-loop-conc}
&D_{h-\fr_r}\left( P_r(t_I) , u ; \BB A_{r,4r}(0) \right) \notag\\
&\qquad\qquad \leq   \op{len}\left( L ;  D_{h-\fr_r} \right) + \op{len}\left( \pi ; D_{h-\fr_r} \right)   \notag\\ 
&\qquad\qquad \leq 2\llambda e^{-\xi \Amax} \wt D_h(u,v)
+ 2 (|I|/r)^{-1/4} D_h\left(\text{across $\BB A_{|I| /4 ,  |I| /2}(y_I)$}\right) .
\alle
\medskip

\noindent\textit{Step 3: comparing $t_I-s_I$ to $D_h\left(\text{across $\BB A_{|I| /4 ,  |I| /2}(y_I)$}\right)$.}
We claim that 
\eqb \label{eqn-arc-loop-show}
t_I - s_I \geq D_h\left(\text{across $\BB A_{|I|  /4 , |I| /2}(y_I)$}\right) .
\eqe
Once~\eqref{eqn-arc-loop-show} is established, the bound~\eqref{eqn-arc-loop-conc} immediately gives~\eqref{eqn-arc-loop}.
Furthermore, the lower bound for $t_I - s_I$ in~\eqref{eqn-arc-loop0} also follows from~\eqref{eqn-arc-loop-show} and the reverse H\"older continuity condition~\ref{item-E-narrow} in the definition of $\Er_r$ (applied with $z\in \bdy B_{|I|/4}(y_I)$ and $w\in \bdy B_{|I|/2}(y_I)$), which gives
\eqbn
D_h\left(\text{across $\BB A_{|I|  /4 , |I| /2}(y_I)$}\right) \geq \left( \frac{|I|}{4r} \right)^{\xi(Q+2)+1/4} r^{\xi Q} e^{\xi h_r(0)} .
\eqen

Hence it remains to prove~\eqref{eqn-arc-loop-show}.  
Let $s_I'$ be the first time after $s_I$ at which $P_r$ exits $B_{|I| /2}(y_I)$. 
Then $P_r|_{[s_I , s_I']}$ is a path between the inner and outer boundaries of $\BB A_{|I| /4 , |I|  /2}(y_I)$. 
We claim that 
\eqb \label{eqn-arc-loop-disjoint}
P_r([s_I ,s_I']) \cap \Vr_r =\emptyset .
\eqe
Since $\fr_r$ vanishes outside of $\Vr_r$,~\eqref{eqn-arc-loop-disjoint} implies that
\allb
t_I - s_I \geq s_I' - s_I = \op{len}\left(P_r|_{[s_I , s_I']} ; D_{h-\fr_r} \right) 
&= \op{len}\left(P_r|_{[s_I , s_I']} ; D_h \right) \notag\\
&\geq D_h\left(\text{across $\BB A_{|I|  /4 , |I| /2}(y_I)$}\right) ,
\alle
which is~\eqref{eqn-arc-loop-show}. 
 
To prove~\eqref{eqn-arc-loop-disjoint}, we first note that by Lemma~\ref{lem-arc-path}, the path $P_r$ does not enter $\Bsup$ between the time when it hits $y_I$ and the time when it hits the other endpoint of $ I$.  
Since the Euclidean distance between the endpoints of $I$ is at least $|I|/2$, $s_I'$ must be smaller than the time when $P_r$ hits the other endpoint of $I$.  
Hence $P_r([s_I,s_I']) \cap \Bsup =\emptyset$. 
In particular, Lemma~\ref{lem-arc-path} implies that $[s_I , s_I'] \subset [a,b]$. 
By Lemma~\ref{lem-excursion-length}, $P_r|_{[a,b]}$ does not hit $\Vr_r\setminus \Bsup$. 
Therefore,~\eqref{eqn-arc-loop-disjoint} holds. 
\end{proof}

\subsubsection{Pidgeonhole arguments}

In light of Lemma~\ref{lem-arc-loop}, we seek an arc $I \in \mcl I^*$ for which $t_I - s_I$ is much smaller than $(|I|/r)^{1/4} \wt D_h(u,v)$. 
To find such an arc, we will partition the set $\mcl I^*$ based on the Euclidean sizes of the arcs.
Let
\eqb
\ul K := \lfloor \log_2(1/\Anar)\rfloor  \quad \text{and} \quad \ol K := \lceil \log_2(1/\Asup) \rceil -1  .
\eqe
For $k\in [\ul K , \ol K ]_{\BB Z}$, let 
\eqb \label{eqn-arcs-k}
\mcl I_k^* := \left\{I\in\mcl I^* :   |I| \in [2^{-k-1}  r , 2^{-k }   r) \right\} .
\eqe  
By Lemma~\ref{lem-micro-arcs} and the definition~\eqref{eqn-good-arc-set} of $\mcl I^*$, we have $\Asup r \leq |I| \leq \Anar r$ for each $I\in\mcl I^*$. 
Hence $\mcl I^*$ is the disjoint union of $\mcl I_k^*$ for $k\in [\ul K , \ol K]_{\BB Z}$. 
 
The proof that there exists an arc $I\in\mcl I^*$ for which $t_I - s_I$ is small is based on a pidgeonhole argument. 
Lemma~\ref{lem-good-arcs} implies that the total Euclidean length of the arcs in $\mcl I^*$ is close to $|\Imid|$. 
Hence there must be some $k \in [\ul K , \ol K]_{\BB Z}$ for which $\#\mcl I_k^*$ is larger than a constant times $r^{-1} 2^{k/2} |\Imid|$: otherwise, the sum of the lengths of the arcs in $\mcl I^*$ would be too small (Lemma~\ref{lem-arc-collection}). In the proof of Lemma~\ref{lem-endpt-close}, we will then use an argument based on Lemma~\ref{lem-arc-loop} and Markov's inequality to show that there must be an $I\in\mcl I_k^*$ for which $t_I - s_I$ is sufficiently small.  

Let us start with the pidgeonhole argument for the Euclidean lengths of the arcs in $\mcl I^*$. 

\begin{lem} \label{lem-arc-collection}
Let $\Aendpt > 0$ be the constant appearing in Lemma~\ref{lem-endpt-ball}, so that the radius of $\Bin$ satisfies $\sr_{\rho r} \in [\Aendpt \rho r , \Aendpt^{1/2}\rho r]$. 
Almost surely, there exist a random $k\in [\ul K , \ol K]_{\BB Z}$ and a collection of arcs $\mcl I_k^{**} \subset \mcl I_k^*$ such that $\#\mcl I_k^{**} \succeq    2^{k/2 } \Aendpt \rho$, with a deterministic universal implicit constant, and the balls $B_{3 |I|}(y_I)$ for $I \in\mcl I_k^{**}$ are disjoint (here $y_I$ is the first endpoint of $I$ hit by $P_r$, as in Lemma~\ref{lem-arc-loop}).  
\end{lem}
\begin{proof}
We have
\allb \label{eqn-arc-decomp-sum} 
|\Imid|/2 
&\leq |\Imid| - 6 \llambda \sr_{\rho r}  \quad \text{(since $|\Imid| \geq \sr_{\rho r}/8$ by~\eqref{eqn-Imid})} \notag\\
&\leq \sum_{I\in\mcl I^*} |I| \quad \text{(by Lemma~\ref{lem-good-arcs})} \notag\\
&\leq \sum_{k=\ul K }^{\ol K } \sum_{I\in\mcl I_k^*} |I| \quad \text{(since $\mcl I^* = \bigcup_{k=\ul K}^{\ol K} \mcl I_k^*$)} \notag\\
&\leq r \sum_{k=\ul K }^{\ol K } 2^{   -k } \# \mcl I_k^* \quad \text{(by~\eqref{eqn-arcs-k})}  .
\alle
 
We claim that there exists $k\in [\ul K , \ol K ]_{\BB Z}$ such that $\# \mcl I_k^* \geq 2^{ k/2} r^{-1} |\Imid|  $. Indeed, if this is not the case then~\eqref{eqn-arc-decomp-sum} gives
\eqbn
|\Imid|/2  
\leq   |\Imid|  \sum_{k=\ul K }^{\ol K } 2^{  -k / 2 }   \quad \Rightarrow \quad  1/2 \leq \frac{1}{1-2^{-1/2}} 2^{-\ul K / 2 } 
\eqen
which is not true since $2^{-\ul K/2} \leq 2\Anar^{1/2} $, which is much smaller than $(1-2^{-1/2})/2$. 
  
Henceforth fix $k\in [\ul K , \ol K ]_{\BB Z}$ such that $\#\mcl I_k^* \geq 2^{k/2} r^{-1} |\Imid|$.
The arcs in $\mcl I_k^*$ are disjoint and have lengths in $[2^{-k-1} r , 2^{-k} r)$. 
Hence for each $I\in \mcl I_k^*$, the number of arcs in $\mcl I_k^*$ which are contained in $B_{3|I|}(y_I)$ is at most some universal constant.
It follows that we can find a subcollection $\mcl I_k^{**} \subset \mcl I_k^*$ such that $\#\mcl I_k^{**} \succeq 2^{k/2} r^{-1} |\Imid|$ and the balls $B_{3|I|}(y_I)$ for $I\in\mcl I_k^{**}$ are disjoint. 
We conclude by noting that by~\eqref{eqn-Imid} and our choice of $\sr_{\rho r}$ in Lemma~\ref{lem-endpt-ball}, 
\eqbn
r^{-1} |\Imid| \succeq r^{-1} \sr_{\rho r} \geq \Aendpt \rho .
\eqen 
\end{proof}

\begin{proof}[Proof of Lemma~\ref{lem-endpt-close}] 
Throughout the proof, all implicit constants are required to be deterministic and depend only on $\xi$. 

Let $k\in [\ul K ,\ol K]_{\BB Z}$ and $\mcl I_k^{**} \subset \mcl I_k^*$ be as in Lemma~\ref{lem-arc-collection}, so that $\#\mcl I_k^{**} \succeq    2^{k/2 } \Aendpt \rho$.  
For $I \in\mcl I_k^{**}$, let $a' < s_{I } < t_{I } < b'$ be as in Lemma~\ref{lem-arc-loop}. 
Lemma~\ref{lem-arc-loop} tells us that $P_r([s_{I} , t_{I}]) \subset B_{3 |I| }(y_{I})$. Lemma~\ref{lem-arc-collection} implies that the balls $B_{3 |I| }(y_{I})$ are disjoint for different choices of $I\in \mcl I_k^{**}$. Hence the intervals $[s_I , t_I] $ for $I \in\mcl I_k^{**}$ are disjoint. 

In light of Lemma~\ref{lem-arc-loop}, we seek $I\in\mcl I_k^{**}$ for which $t_I - s_I$ is much smaller than $(|I|/r)^{1/4}$. 
To find such an $I$, we will first choose a sub-collection of $\mcl I_k^{**}$, which is not too much smaller than $\mcl I_k^{**}$, such that the increments $t_I - s_I$ for $I\in\mcl I_k^{**}$ are all comparable (step 1). 
We will then use Lemma~\ref{lem-arc-loop} to upper bound the sum of the increments $t_I - s_I$ over all arcs $I$ in this collection (step 2). 
Finally, we will use a pidgeonhole argument to find an $I$ for which $t_I - s_I$ is small (step 3).  
\medskip

\noindent\textit{Step 1: finding a sub-collection on which $t_I - s_I$ is controlled.} 
We seek a collection of distinct arcs $I_1,\dots,I_N \in \mcl I_k^{**}$ such that $N$ is not too much smaller than $\#\mcl I_k^{**}$ and the geodesic time increments $t_{I_j}  - s_{I_j}$ for $j=1,\dots,N$ are all comparable. We will find such a collection via a pidgeonhole argument. 

The bound~\eqref{eqn-arc-loop0} of Lemma~\ref{lem-arc-loop} followed by the definition~\eqref{eqn-arcs-k} of $\mcl I_k^*$ shows that for $I\in\mcl I_k^{**}$,  
\eqb
t_I-s_I \geq \left(\frac{|I|}{4 r} \right)^{ \xi (Q+3) } r^{\xi Q} e^{\xi h_r(0)} \geq  2^{-(k+2)  \xi (Q+3)} r^{\xi Q} e^{\xi h_r(0)}  .
\eqe
By combining this with the crude bound $t_I - s_I \leq \sigma - \tau$ and Lemma~\ref{lem-excursion-length}, we get that for $I\in\mcl I_k^{**}$, 
\allb
t_I - s_I 
&\in [2^{-(k+2) \xi (Q+3)} r^{\xi Q} e^{\xi h_r(0)}  , \Anar^{\xi(Q+3)} r^{\xi Q} e^{\xi h_r(0)}]  \notag\\
&\subset [2^{-(k+2)\xi (Q+3)} r^{\xi Q} e^{\xi h_r(0)}  ,   r^{\xi Q} e^{\xi h_r(0)}].
\alle
The number of intervals of the form $[q,2q]$ for $q  > 0$ needed to cover $[2^{-(k+2)\xi (Q+3)} r^{\xi Q} e^{\xi h_r(0)}  , r^{\xi Q} e^{\xi h_r(0)}]$ is at most a constant (depending only on $\xi$) times $k $. Consequently, we can find a random $q  > 0$, an integer 
\eqb \label{eqn-ordered-arcs}
N\succeq k^{-1} \# \mcl I_k^{**} \succeq k^{-1} 2^{k/2} \Aendpt \rho  ,
\eqe 
and intervals $I_1,\dots,I_N \in \mcl I_k^{**}$ such that $t_{I_j} - s_{I_j} \in [q,2q]$ for each $j \in [1,N]_{\BB Z}$. 

Since the intervals $[s_{I_j} , t_{I_j}]$ for $j\in [1,N]_{\BB Z}$ are disjoint, we can choose our numbering so that
\eqb \label{eqn-disjoint-intervals}
s_{I_1} < t_{I_1} < s_{I_2} < t_{I_2} < \dots < s_{I_N} < t_{I_N} .
\eqe 
\medskip

\noindent\textit{Step 2: bounding $q$.} 
We will now use the estimate~\eqref{eqn-arc-loop} from Lemma~\ref{lem-arc-loop} to show that the number $q$ from the preceding paragraph must be small relative to $\wt D_h(u,v)$. 
For each $j\in [1,N]_{\BB Z}$, we have $|I_j| \in [2^{-k-1} r  ,2^{-k} r]$ and $t_{I_j} - s_{I_j} \in [q,2q]$. By plugging these bounds into~\eqref{eqn-arc-loop}, we get
\eqb  \label{eqn-use-arc-loop}
D_{h-\fr_r}\left( P_r(t_{I_j}) , u ; \BB A_{r,4r}(0) \right) 
\preceq  \llambda e^{-\xi \Amax} \wt D_h(u,v) + 2^{ k/4} q   ,
\quad \forall j\in [1,N]_{\BB Z}  
\eqe
with a universal implicit constant. 

By~\eqref{eqn-use-arc-loop} (with $j=1$ and $j=N$) and the triangle inequality for the points $P(t_{I_1}) , u , P(t_{I_N})$,
\eqb \label{eqn-arc-loop-tri}
t_{I_N} - t_{I_1}
= D_{h-\fr_r}\left( P_r(t_{I_1}) , P_r(t_{I_N}) ; \BB A_{r,4r}(0) \right) 
\preceq  \llambda e^{-\xi \Amax}\wt D_h(u,v) +   2^{ k/4} q   .
\eqe
On the other hand,~\eqref{eqn-disjoint-intervals} and our choices of $N$ and $q$ around~\eqref{eqn-ordered-arcs} shows that
\eqb \label{eqn-arc-loop-sum}
t_{I_N} - t_{I_1} 
\geq \sum_{j=2}^N (t_{I_j}   - s_{I_j}) 
\geq (N-1) q 
\succeq  k^{-1} 2^{k/2 } \Aendpt \rho  q  .
\eqe

Combining~\eqref{eqn-arc-loop-tri} and~\eqref{eqn-arc-loop-sum} gives
\eqb
k^{-1} 2^{k/2 } \Aendpt \rho  q \preceq  \llambda e^{-\xi \Amax} \wt D_h(u,v) +   2^{ k/4} q 
\eqe
which re-arranges to give
\eqb \label{eqn-quantum-size}
q \preceq \frac{\llambda}{k^{-1} 2^{k/2} \Aendpt \rho   - R 2^{k/4}} e^{-\xi \Amax} \wt D_h(u,v)
\eqe
for a constant $R >0$ which depends only on $\xi$. 
\medskip

\noindent\textit{Step 3: conclusion.} 
We have $2^k \geq 2^{\ul K} \geq  1/(2\Anar) $, which can be taken to be as large as we would like as compared to $1/(\Aendpt \rho) $ (recall from the discussion surrounding~\eqref{eqn-rho-parameter} that $\Anar$ is chosen after $\rho$ and the parameters from Lemma~\ref{lem-endpt-ball}).
Hence we can arrange that $k^{-1} 2^{k/2} \Aendpt \rho  q \geq 2 R 2^{k/4}$. 
Therefore,~\eqref{eqn-quantum-size} gives
\eqb \label{eqn-quantum-size'}
q \preceq \frac{k 2^{-k/2}}{\Aendpt \rho}  e^{-\xi \Amax} \wt D_h(u,v)  .
\eqe
Plugging~\eqref{eqn-quantum-size'} into~\eqref{eqn-use-arc-loop} shows that for each $j\in [1,N]_{\BB Z}$, 
\eqb  \label{eqn-use-arc-loop'}
D_{h-\fr_r}\left( P_r(t_{I_j}) , u ; \BB A_{r,4r}(0) \right) 
\preceq  \left( \llambda +  \frac{k 2^{-k/4}}{\Aendpt \rho}  \right)  e^{-\xi \Amax} \wt D_h(u,v)  .
\eqe
Since $k\geq \ul K \geq \log_2(1/\Anar) - 1$, the coefficient on the right side of~\eqref{eqn-use-arc-loop'} can be made to be smaller than $2\llambda$ provided the parameters are chosen appropriately. 
This yields~\eqref{eqn-endpt-close} for an appropriate choice of $\Cclose$. The inclusion~\eqref{eqn-endpt-close-eucl} holds since $t_I \in [a',b']$ and $P_r([a',b'])\subset \Bout$ by definition~\eqref{eqn-ab'}. 
\end{proof}

\subsection{Proof of Proposition~\ref{prop-shortcut}}
\label{sec-shortcut}

\begin{figure}[ht!]
\begin{center}
\includegraphics[width=.85\textwidth]{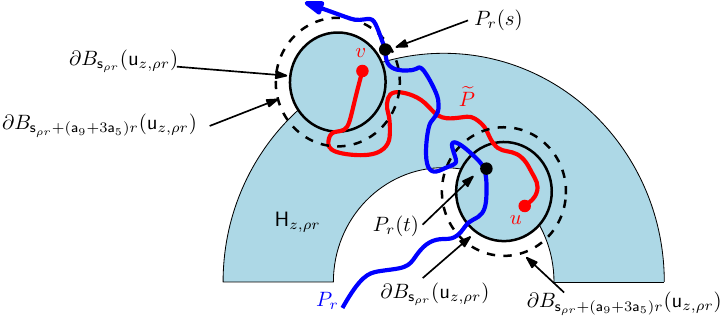} 
\caption{\label{fig-shortcut} Illustration of the proof of Proposition~\ref{prop-shortcut}. We consider a $z\in \Zr_r$ for which $\Fr_{z,\rho r}$ occurs as in Lemma~\ref{lem-excursion-hit}. We look at the corresponding pair of points $u,v$ such that $\wt D_h(u,v) \leq \Cmid_0 D_h(u,v)$ and there is a $\wt D_h$-geodesic $\wt P$ from $u$ to $v$ which is contained in $\ol\Hr_{z,\rho r} \subset \Ur_r$. Lemma~\ref{lem-endpt-close} tells us that there are times $s,t$ for $P_r$ such that $D_h(P_r(t) , u)$ and $D_h(P_r(s) , v)$ are each much smaller than $e^{-\xi \Amax} \wt D_h(u,v) = \wt D_{h-\fr_r}(u,v)$. We then use the triangle inequality to show that $\wt D_h(P_r(t) , P_r(s)) \leq \Cmid  |s-t|$. 
}
\end{center}
\end{figure}

\noindent\textit{Step 1: choice of $s$ and $t$.}
See Figure~\ref{fig-shortcut} for an illustration. 
Let $z\in \Zr_r$ and $u , v \in \bdy \Hr_{z,\rho r}$ be as in Section~\ref{sec-endpt-close}, so that $\Fr_{z,\rho r}$ occurs and $u,v$ are as in the definition of $\Fr_{z,\rho r}$. In particular,
\eqb \label{eqn-shortcut-uv}
\wt D_h(u,v) \leq \Cmid_0 D_h(u,v) .
\eqe
 
By Lemma~\ref{lem-endpt-close}, a.s.\ there exists $t  \subset [\tau,\sigma]$ such that
\eqb \label{eqn-shortcut-u}
P_r(t) \in B_{\sr_{\rho r} + (3\Anar + \Asup) r}(\ur_{z,\rho r}) \quad\text{and} \quad
D_{h-\fr_r}\left( P_r(t) , u ; \BB A_{r,4r}(0) \right) \leq  \Cclose \llambda e^{-\xi \Amax} \wt D_h(u,v)  .
\eqe 
By the definition of $\Fr_{z,\rho r}$, we have $u \in B_{\sr_{\rho r}/2}(\ur_{z,\rho r})$. 
By this,~\eqref{eqn-shortcut-u}, and the triangle inequality, 
\eqb \label{eqn-shortcut-u'}
|P_r(t) - u| \leq \sr_{\rho r} + (3\Anar + \Asup) r + \frac{\sr_{\rho r}}{2} \leq 2\Aendpt^{1/2} \rho r, 
\eqe 
where the second inequality comes from the fact that $\sr_{\rho r} \leq \Aendpt^{1/2} \rho r$ (Lemma~\ref{lem-endpt-ball}) and the fact that each of $\Anar$ and $\Asup$ can be chosen to be much smaller than $\Aendpt$.  
 
By Lemma~\ref{lem-endpt-close} with $\vr_{z,\rho r}$ and $v$ in place of $\ur_{z,\rho r} $ and $u$, there exists $s \in [\tau,\sigma]$ such that 
\eqb \label{eqn-shortcut-v}
D_{h-\fr_r}\left( P_r(s) , v  ; \BB A_{r,4r}(0) \right)   
\leq  \Cclose  \llambda e^{-\xi \Amax} \wt D_h(u,v)  \quad \text{and} \quad
|P_r(s) - v| \leq 2\Aendpt^{1/2} \rho r .
\eqe  
We will check the conditions of~\eqref{eqn-shortcut} for this choice of $s$ and $t$ (possibly with the order of $s$ and $t$ interchanged). 
\medskip

\noindent\textit{Step 2: lower bound for $|s-t|  $.}
Recall that the points $u$ and $v$ lie on the inner and outer boundaries, respectively, of the annulus $\BB A_{\Kann\rho r , \rho r}(z)$. 
From this, the inequalities for Euclidean distances in~\eqref{eqn-shortcut-u'} and~\eqref{eqn-shortcut-v}, and the triangle inequality, we get
\eqb \label{eqn-shortcut-eucl}
|P_r(t) - P_r(s)| \geq (1-\Kann) \rho r - 4\Aendpt^{1/2} \rho r \geq \frac{1-\Kann}{2} \rho r  ,
\eqe 
where in the last inequality we use that $\Aendpt^{1/2}$ is much smaller than $1-\Kann$ (Lemma~\ref{lem-endpt-ball}).

This right side of~\eqref{eqn-shortcut-eucl} is at least $\Anar r$, so the reverse H\"older continuity condition~\ref{item-E-narrow} in the definition of $\Er_r$ gives
\eqb \label{eqn-shortcut-lower}
D_h\left( P_r(t) , P_r(s) ; \BB A_{r,4r}(0) \right)  \geq \Anar^{\xi(Q+3)} r^{\xi Q} e^{\xi h_r(0)} .
\eqe 
By Lemma~\ref{lem-excursion-annulus}, $P_r|_{[\tau',\sigma']}$ is a $D_{h-\fr_r}(\cdot,\cdot;\ol{\BB A}_{r,4r}(0))$-geodesic.
In fact, since $P_r([s,t]) \subset \BB A_{r,4r}(0)$, we have that $P_r|_{[s,t]}$ is a $D_{h-\fr_r}(\cdot,\cdot; \BB A_{r,4r}(0))$-geodesic.
Since $\fr_r \leq \Amax$, we get from~\eqref{eqn-shortcut-lower} that
\allb
|s-t| 
&= D_{h-\fr_r}\left( P_r(t) , P_r(s) ; \BB A_{r,4r}(0) \right)  \notag\\ 
&\geq  e^{-\xi \Amax} D_h\left( P_r(t) , P_r(s) ; \BB A_{r,4r}(0) \right)    \notag\\
&\geq  \Anar^{\xi(Q+3)}  e^{-\xi \Amax}  r^{\xi Q} e^{\xi h_r(0)} 
\alle
which gives the first inequality in~\eqref{eqn-shortcut}. 
\medskip

\noindent\textit{Step 3: upper bound for $\wt D_{h-\fr_r}\left( P_r(t) , P_r(s) ; \BB A_{r,4r}(0) \right)  $.}
We now prove the second inequality in~\eqref{eqn-shortcut}. 
From the bi-Lipschitz equivalence of $D_h$ and $\wt D_h$ and Weyl scaling (Axiom~\ref{item-metric-f}), we get that $D_{h-\fr_r}$ and $\wt D_{h-\fr_r}$ are also bi-Lipschitz equivalent, with the same lower and upper bi-Lipschitz constants $\Clower$ and $\Cupper$. 
Therefore,~\eqref{eqn-shortcut-u} and~\eqref{eqn-shortcut-v} imply that
\eqb \label{eqn-to-shortcut}
\max\left\{ \wt D_{h-\fr_r}\left( P_r(t) , u ; \BB A_{r,4r}(0) \right)  , \wt D_{h-\fr_r}\left( P_r(s) ,v; \BB A_{r,4r}(0) \right)   \right\}   
\leq   \Cupper \Cclose \llambda e^{-\xi \Amax} \wt D_h(u,v)  .
\eqe

Let $\wt P$ be the $\wt D_h$-geodesic from $u$ to $v$ which is contained in $\ol{\Hr}_{z,\rho r}$, as in condition~\ref{item-endpt-ball-annulus'} in the definition of $\Fr_{z,\rho r}$.
Since $\wt P$ is a $\wt D_h$-geodesic, $\wt P\subset \Ur_r$, and $\fr_r $ attains its maximal value $ \Amax$ everywhere on $\Ur_r$, 
\eqb \label{eqn-shortcut-len}
\wt D_{h-\fr_r  }\left( u , v  ; \BB A_{r,4r}(0) \right) = e^{-\xi \Amax}\wt D_h(u,v)   .
\eqe
By~\eqref{eqn-to-shortcut}, \eqref{eqn-shortcut-len}, and the triangle inequality, followed by~\eqref{eqn-shortcut-uv},
\allb \label{eqn-shortcut-tildeD}
\wt D_{h-\fr_r }\left( P_r(t) , P_r(s) ; \BB A_{r,4r}(0) \right) 
&\leq \left(1 + 2 \Cupper \Cclose \llambda  \right) e^{-\xi \Amax} \wt D_h(u,v) \notag\\ 
&\leq \left(1 +2 \Cupper \Cclose \llambda   \right) \Cmid_0 e^{-\xi \Amax}   D_h(u,v) .
\alle 

On the other hand, since $\fr_r \leq \Amax$, Weyl scaling gives
\eqb \label{eqn-uv-D}
D_{h-\fr_r}(u,v) \geq e^{-\xi \Amax} D_h(u,v) .
\eqe
Hence
\allb \label{eqn-shortcut-D}
|s-t| 
&= D_{h-\fr_r}(P_r(t) , P_r(s)) \quad \text{(since $P_r$ is a $D_{h-\fr_r}$-geodesic)} \notag\\
&\geq D_{h-\fr_r}(u,v) - D_{h-\fr_r}( P_r(t) , u ) - D_{h-\fr_r}( P_r(s) , v)  \quad \text{(triangle inequality)} \notag\\
&\geq e^{-\xi \Amax} D_h(u,v)  - 2 \Cclose \llambda e^{-\xi \Amax} \wt D_h(u,v) \quad\text{(by~\eqref{eqn-shortcut-u}, \eqref{eqn-shortcut-v}, and~\eqref{eqn-uv-D})} \notag\\
&\geq e^{-\xi \Amax} D_h(u,v)  - 2 \Cclose \llambda e^{-\xi \Amax} \Cupper D_h(u,v) \quad \text{(bi-Lipschitz equivalence)} \notag\\
&= \left( 1 - 2 \Cupper \Cclose \llambda   \right) e^{-\xi \Amax} D_h(u,v)  .
\alle

Combining~\eqref{eqn-shortcut-tildeD} and~\eqref{eqn-shortcut-D} gives
\eqb \label{eqn-shortcut-end}
\wt D_{h-\fr_r }\left( P_r(t) , P_r(s) ; \BB A_{r,4r}(0) \right) 
\leq \frac{1 + 2 \Cupper \Cclose \llambda }{1 - 2\Cupper \Cclose \llambda } \Cmid_0 |s-t| . 
\eqe
Since $\Cmid_0 < \Cmid $ and $\Cmid_0 , \Cmid $ depend on the laws of $D_h$ and $\wt D_h$ (recall~\eqref{eqn-Cmid0-choice}), we can choose $\llambda$ to be small enough, in a manner depending only on laws of $D_h$ and $\wt D_h$, so that
\eqb
\frac{1 + 2 \Cupper \Cclose \llambda }{1 - 2\Cupper \Cclose \llambda } \Cmid_0 \leq \Cmid . 
\eqe 
Then~\eqref{eqn-shortcut-end} gives the second inequality in~\eqref{eqn-shortcut}. 
\qed

\bibliography{cibib}
\bibliographystyle{hmralphaabbrv}

\end{document}